\documentclass[10pt]{article}
\usepackage{amsfonts,color}
 \usepackage{amsmath,amssymb}
 \usepackage{epsfig}
\usepackage{lscape}
\usepackage{amssymb}
\usepackage{graphicx}
 \usepackage{footnote}
\usepackage[utf8]{inputenc}	
\usepackage{mathtools}
\usepackage{amsthm,amssymb,wasysym}
\usepackage{amsfonts}
\usepackage[english]{babel}

\allowdisplaybreaks[1]

\makeindex

 \newtheorem{theorem}{Theorem}[section]
 \newtheorem{lemma}[theorem]{Lemma}
 \newtheorem{remark}[theorem]{Remark}
 \newtheorem{proposition}[theorem]{Proposition}
 \newtheorem{cor}[theorem]{Corollary}
 \newtheorem{conjecture}[theorem]{Conjecture}

 \newcommand{\R}{\mathbb{R}}
 \newcommand{\N}{\mathbb{N}}
 \newcommand{\lam}{\lambda}
 \newcommand{\norm}[2][]{\|#2\|_{#1}}
 \newcommand{\matr}[1]{\begin{pmatrix}#1\end{pmatrix}}
 \newcommand{\tel}[1]{\frac{1}{#1}}

 \DeclareMathOperator{\diag}{diag}
 \DeclareMathOperator{\Sym}{Sym}
 \DeclareMathOperator{\PSym}{PSym}
 \DeclareMathOperator{\Cof}{Cof}
 \DeclareMathOperator{\dev}{dev}

\newcommand{\WH}{\widehat{W}_{H}}

\newcommand{\GL}{{\rm GL}}
\newcommand{\SO}{{\rm SO}}
\newcommand{\OO}{{\rm O}}

\def\barr{\begin{array}}
\def\id{1\!\!1}
\def\tr{\textrm{tr}}

\def\dd{\displaystyle}

\def\barr{\begin{array}}
\def\earr{\end{array}}
\def\bec#1{\begin{equation}\label{#1}}
\def\becn{\begin{equation*}}
\def\endec{\end{equation}}
\def\endecn{\end{equation*}}

\makeatletter
\let\@fnsymbol\@arabic
\makeatother
\newcommand{\GLpz}{GL^+(2)}

\begin{document}
\title{The exponentiated Hencky-logarithmic strain energy.\\ Part I: Constitutive issues and rank--one convexity}
\author{
Patrizio Neff\thanks{Corresponding author: Patrizio Neff,  \ \ Head of Lehrstuhl f\"{u}r Nichtlineare Analysis und Modellierung, Fakult\"{a}t f\"{u}r
Mathematik, Universit\"{a}t Duisburg-Essen,  Thea-Leymann Str. 9, 45127 Essen, Germany, email: patrizio.neff@uni-due.de}\quad
and \quad
Ionel-Dumitrel Ghiba\thanks{Ionel-Dumitrel Ghiba, \ \ \ \ Lehrstuhl f\"{u}r Nichtlineare Analysis und Modellierung, Fakult\"{a}t f\"{u}r Mathematik,
Universit\"{a}t Duisburg-Essen, Thea-Leymann Str. 9, 45127 Essen, Germany;  Alexandru Ioan Cuza University of Ia\c si, Department of Mathematics,  Blvd.
Carol I, no. 11, 700506 Ia\c si,
Romania; and  Octav Mayer Institute of Mathematics of the
Romanian Academy, Ia\c si Branch,  700505 Ia\c si, email: dumitrel.ghiba@uni-due.de, dumitrel.ghiba@uaic.ro} \quad
and \quad
Johannes Lankeit\thanks{Johannes Lankeit,  \ \ Institut f\"{u}r Mathematik, Universit\"{a}t Paderborn,  Warburger Str. 100,
33098 Paderborn, Germany, email: johannes.lankeit@math.uni-paderborn.de}  }

\maketitle
\begin{center}
{\it In memory of Albert Tarantola ($\star$\!\! 1949\ -- $\dagger$\!\! 2009),  lifelong advocate of  logarithmic measures}
\end{center}

\begin{abstract}
We investigate  a family of isotropic volumetric-isochoric decoupled strain energies
\begin{align*}
F\mapsto W_{_{\rm eH}}(F):=\widehat{W}_{_{\rm eH}}(U):=\dd\left\{\begin{array}{lll}
\dd\frac{\mu}{k}\,e^{k\,\|\dev_n\log {U}\|^2}+\dd\frac{\kappa}{{\text{}}{2\, {\widehat{k}}}}\,e^{\widehat{k}\,[\tr(\log U)]^2}&\text{if}& \det\, F>0,\vspace{2mm}\\
+\infty &\text{if} &\det F\leq 0,
\end{array}\right.\quad
\end{align*}
based on the Hencky-logarithmic (true, natural) strain tensor $\log U$, where $\mu>0$ is the infinitesimal shear modulus,
$\kappa=\frac{2\mu+3\lambda}{3}>0$ is the infinitesimal bulk modulus with $\lambda$ the first Lam\'{e} constant, $k,\widehat{k}$ are dimensionless
parameters, $F=\nabla \varphi$ is the gradient of deformation,  $U=\sqrt{F^T F}$ is the right stretch tensor and $\dev_n\log {U} =\log {U}-\frac{1}{n}
\tr(\log {U})\cdot\id$
 is the deviatoric part  of the strain tensor $\log U$. For small elastic strains, $W_{_{\rm eH}}$ approximates the classical quadratic Hencky strain energy
 \begin{align*}
 F\mapsto W_{_{\rm H}}(F):=\widehat{W}_{_{\rm H}}(U)&:={\mu}\,\|{\rm dev}_n\log U\|^2+\frac{\kappa}{2}\,[{\rm tr}(\log U)]^2,
 \end{align*}
  which is not everywhere rank-one convex. In plane elastostatics, i.e. $n=2$, we prove the everywhere rank-one convexity of the proposed family $W_{_{\rm eH}}$, for
  $k\geq \frac{1}{4}$ and $\widehat{k}\geq \frac{1}{8}$. Moreover, we show that the corresponding Cauchy (true)-stress-true{-}strain relation
is invertible for $n=2,3$ and we show the monotonicity of the Cauchy (true) stress tensor
as a function of the true strain tensor in a domain of bounded distortions. We also prove that the rank-one
convexity of the energies belonging to the family $W_{_{\rm eH}}$ is not preserved in
{dimension $n=3$} and that the energies
 \begin{align*}
F\mapsto \frac{\mu}{k}\,e^{k\,\|\log U\|^2},\quad F\mapsto \frac{\mu}{k}\,e^{\frac{k}{\mu}\left(\mu\,\|\dev_n\log U\|^2+\frac{\kappa}{2}[\tr(\log U)]^2\right)},\quad F \in{\rm GL}^+(n), \quad n\in \N,\quad  n\geq 2
\end{align*}
  are not rank-one convex.
\\
\vspace*{0.25cm}
\\
{\textbf{Key words:} idealized finite isotropic elasticity, Legendre-Hadamard ellipticity condition, hyperelasticity, constitutive inequalities,
stability, Hencky strain, logarithmic strain, natural strain, true strain, Hencky energy, convexity,   rank-one convexity,  volumetric-isochoric split,
plane elastostatics, monotonicity and invertibility of the constitutive law, homogeneous symmetric bifurcations, Baker-Ericksen inequality, bounded distortions, elastic domain, nonlinear Poisson's ratio.}
\end{abstract}

\newpage

\tableofcontents

\pagebreak

\section{Introduction}

\subsection{Logarithmic strain and geodesically motivated invariants}

We introduce a modification of the well-known isotropic quadratic Hencky strain energy
\[
	W_{_{\rm H}}(F) = \widehat{W}_{_{\rm H}}(U) \;=\; {\mu}\,\|{\rm dev}_n\log U\|^2+\frac{\kappa}{2}\,[{\rm tr}(\log U)]^2\,,
\]
where $\mu>0$ is the infinitesimal shear (distortional) modulus, $\kappa=\frac{2\mu+3\lambda}{3}>0$ is the bulk modulus with $\lambda$ the first Lam\'{e}
constant, $F=\nabla \varphi$ is the deformation gradient, $U=\sqrt{F^T F}$ is the right Biot stretch tensor, $\log U$ is the referential (Lagrangian)
logarithmic strain tensor, $\norm{\,.\,}$ is the Frobenius tensor norm, and $\dev_n X=X-\frac{1}{n}\, \tr(X)\cdot\id$ is the deviatoric part of a second
order tensor $X{\in\R^{n\times n}}$ {(see Section
\ref{auxnot} for other notations)}.

 It was recently discovered \cite{Neff_Osterbrink_Martin_hencky13,neff2013hencky} (see also  {\cite{Neff_log_inequality13,LankeitNeffNakatsukasa}})
 that the Hencky strain energy  enjoys a surprising property, which singles it out among all other isotropic strain energy functions. Indeed, the Hencky energy measures the geodesic distance of the deformation gradient $F\in\GL^+(n)$ to the special orthogonal group $\SO(n)$,
 i.e.
\begin{align}\label{eq:introductionDistance}
	{\rm dist}^2_{{\rm geod}}(F, {\rm SO}(n)) &\;=\; {\mu}\,\|{\rm dev}_n\log U\|^2+\frac{\kappa}{2}\,[{\rm tr}(\log U)]^2 \;=\; W_{_{\rm H}}(F)\,,\\
{\rm dist}^2_{{\rm geod}}(F, {\rm SO}(n))& \;=\;0\qquad {\text{if and only if}}\qquad \varphi(x)=\widehat{Q}\,x+\widehat{b} \quad { \text{for some fixed }
}\widehat{Q}\in {\rm SO}(3), \quad \widehat{b}\in \R^3,\notag
\end{align}
where the Lie-group $\GL^+(n)$ is viewed as a Riemannian manifold endowed with a  certain left-invariant metric which is also right $\OO(n)$-invariant\footnote{Although every
such  Riemannian  metric is uniquely characterized by three coefficients, the geodesic distance to $\SO(n)$ in fact depends on only two of them, corresponding to
the two material parameters $\mu$ and $\kappa$.} (isotropic). The use of the quadratic Hencky strain energy in nonlinear elasticity theory can therefore be
motivated by purely geometric reasoning.

\smallskip

In contrast, for  the case of the simple {E}uclidean distance on $\R^{n\times n}$ we note that
\begin{align}
{\rm dist}^2_{{\rm euclid}}(F, {\rm SO}(n)) &\overset{\rm (def)}{:=} \inf_{\overline{R}\in{\rm SO}(n)}\|F-\overline{R}\|^2=\inf_{\overline{R}\in{\rm
SO}(n)}\|\overline{R}^TF-\id\|^2=\|U-\id\|^2\,,
 \end{align}
 which yields the Biot-stretch measure $U-\id$ without any possibility of weighting the deviatoric and volumetric contributions independently
 \cite{neff2010stable}.
 {On the other hand}, the additive \emph{volumetric-isochoric split}
\begin{align}
	\widehat{W}_{_{\rm H}}(U) \;&=\; {\mu}\,\|{\rm dev}_n\log U\|^2+\frac{\kappa}{2}\,[{\rm tr}(\log U)]^2 \;=\; \underbrace{\mu\,\|\log \frac{U}{\det
U^{1/n}}\|^2}_{\widehat{W}_{_{\rm H}}^{\rm iso}\left(\frac{U}{\det U^{1/n}}\right)} \,+\, \underbrace{\frac{\kappa}{2}[\log \det U]^2}_{\widehat{W}_{_{\rm H}}^{\rm vol}(\det U)}
\end{align}
of $\WH$ into an isochoric term $\widetilde{W}_{_{\rm H}}^{\rm iso}$ depending only on $\frac{U}{\det U^{1/n}}$\,, i.e. on the isochoric part of $U$, and a volumetric term $\widetilde{W}_{_{\rm H}}^{\rm vol}$
depending only on $\det U$ is characterized by means of the same geodesic distance as well: it can be shown that \cite{Neff_Osterbrink_Martin_hencky13,%
neff2013hencky,Neff_Nagatsukasa_logpolar13}
\begin{align}
K_1^2&:={\rm dist}^2_{{\rm geod}}\left((\det F)^{1/n}\cdot \id, {\rm SO}(n)\right)={\rm dist}^2_{{\rm geod,\mathbb{R}_+\cdot \id}}\left((\det F)^{1/n}\cdot
\id, \id\right)=|\tr (\log  U)|^2 = \widetilde{W}_{_{\rm H}}^{\rm vol}(\det U)\,,\notag\\
K_2^2&:={\rm dist}^2_{{\rm geod}}\left( \frac{F}{(\det F)^{1/n}}, {\rm SO}(n)\right)={\rm dist}^2_{{\rm geod,{\rm SL}(n)}}\left( \frac{F}{(\det F)^{1/n}},
{\rm SO}(n)\right)=\|\dev_n \log U\|^2 = \widetilde{W}_{_{\rm H}}^{\rm iso}\left(\frac{U}{\det U^{1/n}}\right),\notag
\end{align}
where ${\rm dist}^2_{{\rm geod,\mathbb{R}_+\cdot \id}}$ and ${\rm dist}^2_{{\rm geod,{\rm SL}(n)}}$ are the canonical left invariant geodesic distances on
the Lie-group ${\rm SL}(n)$ and on the multiplicative group $\mathbb{R}_+\cdot\id$, respectively. This result strongly suggests that the two quantities $K_1^2=\|{\rm
dev}_n\log U\|^2$ and $K_2^2=[{\rm tr}(\log U)]^2$ should be considered separately as fundamental measures of elastic deformations, which motivates a
family of elastic energy functions stated in terms of these two quantities alone \cite{Mielke02b}. It is clear,  however,  that it is not the strain measure $\log U$
itself which has any importance in this regard\footnote{Truesdell writes \cite{truesdell1966mechanical}: ``It is important to realize that since each of the several
material tensors [the strain tensors like $U-\id$, $\id -U^{-1}$, $\log U$, $U-U^{-1}$] is an isotropic function of any one of the others, an exact
description
of strain in terms of any one is equivalent to a description in terms of any
other; only when an approximation is to be made may the choice of a
particular measure become important." }, but the fundamentally motivated scalar geodesic invariants $K_1^2$, $K_2^2$. They restrict the form of the constitutive
law.

Moreover, in $2D$, the purely isochoric term ${{\rm dist}^2_{{\rm geod}}\left( \frac{F}{\det F^{1/2}}, \,{\rm SL}(2)\right)}$ penalyzes the departure from
conformal  (shape preserving) mappings, i.e. the absolute minimizer in dimension $n=2$ is a deformation $\phi$ with $\nabla \phi$ satisfying
\begin{align}
\begin{array}{lcl}
\log \nabla \phi^T\nabla \phi=\alpha(x,y)\cdot \id_2, \quad \alpha(x,y)\in\mathbb{R} &\qquad \Leftrightarrow &\qquad
{\nabla \phi^T\nabla \phi
=e^{\alpha(x,y)\cdot \id_2}}{, \quad \alpha(x,y)\in\mathbb{R}}\notag\vspace{1mm}\\
\qquad \quad\ \  \nabla \phi\in\!\!\!\!\!\!\!\!\!\!\!\!\!\!\!\!\!\!\!\!\!\!\!\!\!\!\underbrace{\mathbb{R}_+\cdot {\rm SO}(2)}_{\quad \text{the special
conformal group}\  {\rm CSO}(2)}\!\!\!\!\!\!\!\!\!\!\!\!\!\!\!\!\!\!\!\!\!\!\!\!\!\!&\qquad \Leftrightarrow&\qquad\qquad\quad
\phi:\mathbb{R}^2\rightarrow\mathbb{R}^2 \quad \text{is holomorphic}.
\end{array}
\end{align}

Since $K_1^2$, $K_2^2$ have this inherently fundamental differential geometric motivation, we propose to investigate a new constitutive framework for ideal
isotropic elasticity.
Then it is  natural to consider the most primitive possible strain energy form satisfying:
\begin{itemize}
\item[i)] The elastic energy $W$ can be written as  a function of the \textit{geodesic invariants}
$$W=\widetilde{W}(K_1^2,K_2^2),\quad \text{where}\quad K_1^2:=\|{\rm dev}_n\log U\|^2\quad
    \text{and}\quad K_2^2:=[{\rm tr}(\log U)]^2;$$
    \item[ii)] The energy is  strictly increasing as a function of  $K_1^2, K_2^2$;
\item[iii)] The energy is strictly  convex  as a function of $\log U$ (Hill's inequality);
\item[iv)] Preferably, the energy should be  a rank-one convex (polyconvex, quasiconvex) function;
\item[v)] The energy should satisfy a coercivity  condition.
\end{itemize}
We observe that iv) necessitates that $W$ should grow at least exponentially (see \cite{Walton05}).

\subsection{Scope of investigation}

Many elastic materials show a completely different response regarding shape changing deformations
 and purely volumetric deformations. Therefore, in concordance with our just stated requirements, we  investigate in
this paper a family of isotropic exponentiated
Hencky-logarithmic strain type energies in which both contributions coming from dilatations and distortions  are a priori
additively separated\footnote{Such an assumption is especially suitable for only slightly compressible materials or under small elastic strains \cite{henann2011large}.} \cite{flory1961thermodynamic}
\begin{align}\label{the}
W_{_{\rm eH}}(F):=\widehat{W}_{_{\rm eH}}(U)&:=
\dd\left\{\begin{array}{lll}
\dd\frac{\mu}{k}\,e^{k\,K_2^2}+\frac{\kappa}{2\widehat{k}}\,e^{\widehat{k}\,K_1^2}&\qquad\qquad\qquad\quad\ \ \text{if}& \det\, F>0,\vspace{2mm}\\
+\infty &\ \qquad\qquad\qquad\quad\ \text{if} &\det F\leq 0,
\end{array}\right.\notag\\
&\ \  =\left\{\begin{array}{lll}
\underbrace{\frac{\mu}{k}\,e^{k\,\|\dev_n\log\,U\|^2}+\frac{\kappa}{2\widehat{k}}\,
e^{\widehat{k}\,(\tr(\log U){)}^2}}_{\text{volumetric-isochoric split}}&\ \
\!\!\text{if}& \det\, F>0,\vspace{2mm}\\
+\infty &\ \ \!\! \text{if} &\det F\leq 0,
\end{array}\right.\\
&\ \ =\dd\left\{\begin{array}{lll}
\dd\frac{\mu}{k}\,e^{k\,\|\log \frac{U}{\det U^{1/n}}\|^2}+\frac{\kappa}{2\widehat{k}}
\,e^{\widehat{k}\,(\log \det U)^2}&\text{if}& \det\,
F>0,\vspace{2mm}\\
+\infty &\text{if} &\det F\leq 0,
\end{array}\right.\notag
\end{align}
where  $U=\sum_{i=1}^n \lambda_i \, N_i\otimes N_i$, $\log U=\sum_{i=1}^n \log \lambda_i (N_i\otimes N_i)=\lim\limits_{r\rightarrow0}\frac{1}{r}(U^r-\id)$,
$\lambda_i$  and $N_i$ are the eigenvalues and eigenvectors of $U$, respectively.  The immediate importance of the family \eqref{the} of free-energy
{functions} is seen by looking at  small (but not infinitesimally small) elastic strains. Then the exponentiated Hencky energy $W_{_{\rm eH}}(\cdot)$ reduces to
first order to  the   quadratic Hencky energy based on the logarithmic strain tensor $\log U$
\begin{align}\label{th}
W_{_{\rm H}}(F):=\widehat{W}_{_{\rm H}}(U)&:={\mu}\,\|{\rm dev}_n\log U\|^2+\frac{\kappa}{2}\,[{\rm tr}(\log
U)]^2+\underbrace{\left(\frac{\mu}{k}+\frac{\kappa}{2\widehat{k}}\right)}_{\text{const.}},\\
\notag
\tau_{_{\rm H}}&:=D_{\log V} \widetilde{W}_{_{\rm H}}(V)=2\,\mu\, \dev_3\log V+{\kappa}\,\tr(\log V)\cdot \id,\qquad \tau_{_{\rm H}}=\det V\cdot \sigma_{_{\rm H}},
\end{align}
where $V=\sqrt{F\,F^T}$  is the  left stretch tensor, $\widetilde{W}_{_{\rm H}}(V)=\widehat{W}_{_{\rm H}}(U)$,  $\sigma_{_{\rm H}}$ is the Cauchy stress tensor  in the current
configuration  and $\tau_{_{\rm H}}$ is the Kirchhoff stress tensor. The Hencky energy  ${W}_{_{\rm H}}$ has been introduced by Heinrich Hencky \cite{tanner2003heinrich}
starting from 1928 \cite{Hencky28a,Hencky29a,Hencky29b,Hencky31,Walton05,vallee1978lois,skrzypek1985application,Bazant98,ortiz2001computation} (see
\cite{neff2014axiomatic} for a recent english  translation of Hencky's {G}erman original papers) and has since then acquired a unique status in finite strain
elastostatics \cite{Anand79,Anand86,adamov2001comparative} and especially in finite strain elasto-plasticity\footnote{In Hencky's first paper
\cite{hencky1924theorie}, the constitutive law $\sigma_{_{\rm H}}=2\,\mu\, \dev_3\log V+{\kappa}\,\tr(\log V)\cdot \id$ is proposed, which is Cauchy-elastic,
tensorially correct, but not hyperelastic. This has been corrected by Hencky in later papers. Incidentally,  Becker's law \eqref{cbec} is also
Cauchy-elastic,
tensorially
correct, but hyperelastic only
 for $\nu=0$ \cite{carroll2009must,blume1992form} (see also \cite{xiaoAM2002,NBecker}).}.
 Hencky himself used this constitutive law to
study finite elastic deformations of rubber in some simple cases \cite{Hencky28a,Hencky29a,Hencky29b,Hencky31,hencky1933elastic}.
The modern applications seem to begin with the study of finite
elastic and elasto–plastic bending of a long plate-strip (plane
strain) in the cases of incompressible and compressible deformations
\cite{de1967elastisch,de1969berechnung,bruhns1970berucksichtigung,bruhns1971elastoplastische,bruhns1969elastisch}. The formulation based on the Hencky strain energy provides the
greatest possible extent of
elastic determinacy \cite[page 19]{neff2014axiomatic}: the
Kirchhoff-stress response does not depend on a specific reference state
or previously applied coaxial deformations. A similar property was
postulated for an idealized law of elasticity by Murnaghan
\cite{Murnaghan,murnaghan1944}, who argued that the dependence of the stress
response on a specific {position of zero strain} was tantamount
to an {action at a distance} and should therefore be avoided.

The first axiomatic study on the nonlinear stress-strain function involving  a logarithmic strain tensor is, however, due to the famous geologist George
Ferdinand Becker \cite{BeckerBio,Becker1893} in 1893. Using a principle of superposition for the principal forces in the reference configuration he
concludes {with} a stress-strain law in the form
\begin{align}\label{cbec}
T_{\rm Biot}(U)=2\,\mu\, \dev_3\log U+{\kappa}\,\tr(\log U)\cdot \id,
\end{align}
where $T_{\rm Biot}(U)=R^T\cdot S_1(F)=U\cdot S_2(C)$ is the (symmetric in case of isotropy) Biot-stress tensor, $F= R\, U$ is the right polar decomposition, $S_2$ is the
symmetric second Piola-Kirchhoff stress tensor and $C=F^TF$, see \cite{NBecker} for detailed explanations. Even earlier, in 1880, Imbert
\cite{imbert1880recherches} in his doctoral thesis considered the uniaxial tension of vulcanized rubber bands and obtained as a best fit a constitutive
law\footnote{
Note that   \eqref{cimb} is the uniaxial specification of \eqref{cbec}, and \eqref{cbec} is closely {resembling \eqref{th}$_2$}. A small calculation
\cite{NBecker} shows $\tau_{\rm Becker}=V \cdot \tau_{_{\rm H}}$, where $\tau $ is the corresponding Kirchhoff stress $\tau=(\det F) \cdot\sigma=D_{\log V}W(\log
V)$ and  $V$ is the  left stretch tensor. Moreover $\|\tau_{\rm Becker}-  \tau_{_{\rm H}}\|{\leq 
\|V-\id\|\cdot \|\tau_{_{\rm H}}\|}$. Hence, for small elastic strains $\|V-\id\|\ll1$, Becker's law coincides with Hencky's model to first order in the nonlinear
strain measure $V-\id$.} \cite[page 53]{imbert1880recherches}, which in modern notations reads
\begin{align}\label{cimb}
\langle T_{{\rm Biot}}.\, e_1,e_1\rangle=E\,\langle \log U.\, e_1,e_1\rangle,
\end{align}
for recoverable (fully elastic extensional) stretches $\lambda\in[1,e)$. Three years later, in 1893, Hartig \cite{hartig} (see also \cite{bell1973}) used the same constitutive law for tension and
compression data of rubber.

In \cite{nadai1937plastic} Nadai  introduced the name ``natural strain" tensor for the logarithmic strain tensor $\log U$ and motivated application of this
concept in metal forming processes\footnote{In the {G}erman metal forming  literature the logarithmic strain is also called ``Umformgrad". In \cite[page
17]{ludwik1909elemente}   Ludwik uses the ``effective specific elongation" $\alpha=\int_{\ell_0}^{\ell} \frac{d \ell}{\ell}={\rm {ln}}\frac{\ell}{\ell_0}$
. It
can be motivated by considering the summation over the infinitesimal increase in length
 as referred to the current length, i.e. ${\rm ln} \frac{\ell}{\ell_0}=\lim_{N\rightarrow\infty} \sum_{i=0}^{N-1}\frac{\ell_{i+1}-\ell_i}{\ell_i}$
 \cite{Xiao2005,hanin1956isotropic}. The scalar Hencky-type measure $\|\dev_3\log U\|$ is sometimes used as ``equivalent strain" in order to represent the
 degree of plastic deformation \cite{onaka2012appropriateness,onaka2010equivalent}. Its use for severe shearing has been questioned in
 \cite{shrivastava2012comparison}. In our opinion the problematic issue is not the logarithmic measure itself, but its degenerate (sublinear) growth
 behaviour for large strains. The opposing views may be reconciled by using $e^{\|\dev_3\log U\|}$ as ``exponentiated equivalent strain" measure. }
in  metallurgy. The strain measure (natural strain) then has been extensively used over
the years to  report experimental true-stress-true-strain data.   More recently, in \cite{henann2011large} a modified Hencky  energy is proposed which is
motivated by in depth
 molecular dynamics simulations for a metalic glass\footnote{i.e. an amorphous metal which is very nearly isotropic with superior elastic deformability up
 to 1-2\% distortional strain, but which shows no ductility, in contrast to polycrystalline metals which typically show elastic strains up only to 0.1-0.2\%.
Recently, Murphy \cite{murphy2007linear} (see also \cite{zhang2014further}) has postulated a linear Cauchy stress-strain relation for some strain measure and
gets as well $W_{_{\rm H}}$ as a preferred solution. His corresponding ``strain measure" $E$ is then $E:=\frac{1}{\det V}\cdot \log V$, so that $\sigma=2\,
\mu\, E+\lambda\,\tr(E)\cdot \id$, which is Hencky's relation in disguise. However, $V\mapsto E(V)$ is not invertible, thus $E$ does not really qualify as
a strain measure.}.
 Hill \cite{hill1970constitutive,hill1968constitutive} (see also \cite{rakotomanana1991generalized,miehe2009finite,miehe2011coupled}) has discussed the advantage of the logarithmic strain
 measure\footnote{Tarantola noted  \cite[page 15]{tarantola2009stress} that ``Cauchy originally defined the strain as $E=\frac{1}{2}(C-\id)$, but many
 lines of thought suggest that this was just a guess, that, in reality is just the first order approximation to the more proper definition $E=\log
 \sqrt{C}=\frac{1}{2}(C-\id)-\frac{1}{4}(C-\id)^2+...$, i.e., in reality, $E=\log U=(U-\id)-\frac{1}{2}(U-\id)^2+...$".} in setting up a class of
 constitutive inequalities, based on a family of measures of finite strain and their corresponding conjugate stresses, for both elastic and elasto-plastic
 solids.   Hill showed that only one member of this class admits incompressibility, namely that corresponding to logarithmic strain. The special
 Hill's-inequality (which we will call KSTS-M$^+$ for reasons which become clear later) asserts in the hyperelastic
case that the strain energy should be a
convex function of logarithmic strain \cite{ogden1970compressible} and Hill argued that this inequality is the most suitable for compressible solids.
\v{S}ilhav\'{y} \cite{Silhavy97} remarks that Hill's inequality is, up to date, not found to be in conflict with experimental facts.

The Hencky strain tensor appears also in much more diverse fields, such as image registration \cite{pennec2005riemannian,pennec2006left} and relativistic
elastomechanics \cite{kijowski1992relativistic}. Extensions to the anisotropic hyperelastic  response based on the  Hencky-logarithmic strain were
investigated,  e.g. in \cite{dluzewski2000anisotropic,dluzewski2003logarithmic,germain2010inverse}.

Let us now summarize some well-known unique  features of the quadratic Hencky strain energy ${W}_{_{\rm H}}$ based exclusively on the natural strain tensor $\log
U$:
\begin{itemize}
\item[$\oplus_{1\ }$] The two isotropic Lam\'{e} constants $(\mu,\lambda)\rightsquigarrow (\mu,\kappa)$ (or $(E,\nu)$), the shear modulus, the bulk modulus and the
    second Lam\'{e} constant,  are determined in the infinitesimal strain regime, but the model based on the energy $W_{_{\rm H}}(F)$ can well describe the
    nonlinear deformation response for moderate principal stretches $\lambda_i\in(0.7,1.4)$ (see \cite{Anand79,Anand86,BohlkeBertram}). Of course, for a
    particular material, one may always get agreement with (a finite number of) experiments to any desired accuracy for another constitutive law  with
    more adjustable parameters, e.g. Ogden's strain energy \cite{Ogden83}.
\item[$\oplus_{2\ }$] The Hencky model outperforms other  {well known nonlinear} elasticity models with equally few constitutive parameters,  {like
    Neo-Hooke or Mooney-Rivlin  type elastic materials \cite{mooney1940theory,Ogden83,rivlin1948large,Ciarlet88}} in the {above-mentioned} strain range.
\item[$\oplus_{3\ }$] The geometrically nonlinear Poynting effect (a cylindrical bar of steel, copper, rubber or brass lengthens  in torsion proportional to the
    square of the twist) is correctly described
    \cite{Anand86,bruhns2000hencky,dell1998generalized,dell1997second,bell1973,panov2012using,batra1998second}.
\item[$\oplus_{4\ }$]  $W_{_{\rm H}}$ has the correct behaviour for extreme strains in
the sense that  $W(F_e)\rightarrow\infty$ as $\det F_e\rightarrow0$ and,
likewise{,} $W(F_e)\rightarrow\infty$ as $\det F_e\rightarrow\infty$.
\item[$\oplus_{5\ }$] The Hencky strain  tensor $\log U$ puts extension and contraction on the same footing, its principal values  vary from $-\infty$ to
    $\infty$, whereas those of $C=F^T\!F$ or $B=FF^T$ vary from $0$ to $\infty$ and those of $C-\id$ vary from $-1$ to $\infty$.
\item[$\oplus_{6\ }$] The Hencky strain defines a strictly monotone primary matrix function \cite{jog2013conditions,NBecker,Norris}, i.e.
\begin{align}
\langle \log U_1-\log U_2,U_1-U_2\rangle >0 \qquad \forall \, U_1,U_2\in{\rm PSym}(3), \ U_1\neq U_2,
\end{align}
even for non-coaxial arguments $U_1,U_2$.
\item[$\oplus_{7\ }$] Tension and compression are treated equivalently:
$W_{_{\rm H}}(F)=W_{_{\rm H}}(F^{-1})$, i.e. invariance w.r.t. the Lagrangian or Eulerian description.
Both the incompressible and compressible versions of $J_2$-finite strain
deformation theory  \cite{Hutchinson82} usually assume identical true-stress-true-strain relations
in tension and compression.
\item[$\oplus_{8\ }$] The linear and second-order behaviour of $W_{_{\rm H}}$ is in agreement with Bell's
experimental observations \cite{bell1973}, i.e. in general, under small strain conditions the instantaneous {elastic modulus} $E$ decreases
for tension and increases in the case of compression (c.f. Figure \ref{dumitrel_thirdOrderConstants}).
\item[$\oplus_{9\ }$] True strain for equivalent amounts of deformation in tension and compression is equal except for the sign: $\log V=-\log V^{-1}$.
\item[$\oplus_{10}$] The Eulerian strain tensor $\log V$ (and the Lagrangian strain tensor $\log U$) is additive for coaxial stretches, i.e. $\log
    (V_1V_2)=\log V_1+\log V_2$ for $V_1V_2=V_2V_1$.  This implies the superposition principle for the Kirchhoff stress $\tau_{_{\rm H}}$ for coaxial strains
    \cite{Becker1893,NBecker}.
\item[$\oplus_{11}$]
For incompressibility (e.g. for rubber \cite{hencky1933elastic,horgan2009volumetric,horgan2009constitutive}), only one parameter, the shear
(distortional) modulus $\mu=\frac{E}{3}${,} suffices, where $E=\frac{\mu\,(2\mu+3\lambda)}{\lambda+\mu}$ is Young's modulus.
\item[$\oplus_{12}$] The Hencky strain tensor $\log U$ has the advantage that it additively separates dilatation from pure distortion
    \cite{Criscione2000,Richter50,Richter52,richter1948isotrope,richter1949hauptaufsatze};  there is an  exact volumetric-isochoric decoupling by the
    properties of the logarithmic strain tensor:
    \begin{align*}
    \log \frac{U}{\det U^{1/n}}&=\log [U\cdot (\det U)^{-1/n}]=\log U+\log [(\det U)^{-1/n}\cdot \id]\\
    &=\log U-\frac{1}{n}(\log \det U)\cdot \id=\log U-\frac{1}{n}\tr(\log U)\cdot \id=\dev_n \log U.
    \end{align*}
 Among all finite strain measures from the Seth-Hill family \cite{seth1961generalized,hill1978aspects}, only the spherical and deviatoric parts of the Hencky strain
 quadratic energy can additively separate the volumetric and the isochoric deformation
 \cite{sansour2008physical,horgan2009constitutive,annaidh2012deficiencies}.
\item[$\oplus_{13 }$] The volumetric   expression $[\tr(\log U)]^2=(\log \det U)^2$  has been motivated independently  in
    \cite{Tarantola98,Tarantola06,Murnaghan} and found to be superior in describing the  pressure-volume equation of state
    (EOS) for geomaterials under extreme pressure (see Section \ref{secteos}).
\item[$\oplus_{14}$] The incompressibility condition $\det F=1$ is the simple statement $\tr(\log U)=0$.
\item[$\oplus_{15}$] For the Hencky  energy $W_{_{\rm H}}$, uniaxial tension leads to uniaxial lateral contraction and a planar  pure Cauchy shear stress
    produces biaxial pure shear strain \cite{vallee1978lois}, similar  as in linear elasticity (see Figure \ref{uxl} and also Section \ref{sectinvert}),
    i.e. planar pure shear stress $\sigma=\left(
                                             \begin{array}{ccc}
                                               \sigma_{11} & \sigma_{12} & 0 \\
                                               \sigma_{12} & \sigma_{22} & 0 \\
                                               0 & 0 & 0 \\
                                             \end{array}
                                           \right)$, $\tr(\sigma)=0$  corresponds to isochoric  planar stretch $V=\left(
                                             \begin{array}{ccc}
                                               V_{11} & V_{12} & 0 \\
                                               V_{12} & V_{22} & 0 \\
                                               0 & 0 & 1 \\
                                             \end{array}
                                           \right)$, $\det V=1$.  For Poisson's number $\nu=1/2$ exact incompressibility follows  and for $\nu=0$ there is no lateral contraction in uniaxial tension,
    exactly as in linear elasticity \cite{vallee1978lois}.
\item[$\oplus_{16}$] The Hencky  energy $W_{_{\rm H}}$ has constant nonlinear Poisson's ratio $\widehat{\nu}=-\frac{(\log V)_{22}}{(\log V)_{11}}=\nu$ as in linear elasticity and $\lambda_2=\lambda_1^{-\nu}$ in  uniaxial extension \cite{vallee1978lois}.
   \item[$\oplus_{17}$]  If $\Psi({\rm exp}{(S)}):=W(S)$, then for isotropic response $D_S\,W(S)=D \Psi({\rm exp}{(S)})\cdot {\rm exp}{(S)}$ (see
       \cite{vallee1978lois,vallee2008dual,Kochkin1986,sansour1997theory}). Thus, $2\,S_2(C)=D_C\, \Psi(C)=D_{\log C} \, W(\log C)\cdot C^{-1}$, while $D_C[\log C].\, H=C^{-1}\cdot H$ is not true in general. Therefore $\tau=D_{\log V} \, W(\log V)$, \, $\sigma=\frac{1}{\det V}\cdot \tau=\frac{1}{\det V}\cdot D_{\log V} \, W(\log V)$, see Appendix
       \ref{Appensansour}. Using this formula, the algorithmic tangent $D_F^2[W(F)].(H,H)$ for the isotropic Hencky energy in finite element simulations
       can be analyzed with knowledge of only the first Fr\'{e}chet-derivative $D_C[\log C].H$ (see  \cite{jog2002explicit}).

\item[$\oplus_{18}$] The Kirchhoff stress $\tau$ is conjugate to the strain measure $\log V$
    \cite{hoger1986material,hoger1987stress,scheidler1991timea,scheidler1991timeb,lehmann1993stress,Norris,sansour2001dual,vallee2008dual}, where
    $V=\sqrt{F\,F^T}$ is the  left stretch tensor, i.e.  $\langle \tau, \frac{\rm d}{\rm dt}\log V\rangle=\det V\cdot\langle \sigma, D\rangle$ is equal
    to the power per unit volume element in the reference configuration. Here, $D$ is the strain rate tensor $D={\rm sym} L={\rm sym}(\dot{F}{F}^{-1})$.
\item[$\oplus_{19}$] Contrary to the arbitrary number  of possible strain tensors in the Lagrangian setting, there is only one strain rate tensor $D$  in
    the Eulerian setting.  In the one dimensional case\footnote{In the one dimensional case ${\varphi(x_1,t)=(\varphi_1(x_1,t), x_2, x_3)^T
\Rightarrow F=\nabla} \varphi=\diag({\varphi}_{1,x_1},1,1)\Rightarrow D={\rm
sym}(\dot{F}{F}^{-1})=\diag\left(\frac{\dot{\varphi}_{1,x_1}}{\varphi_{1,x_1}},0,0\right)$ and $\int_0^t \frac{\dot{\varphi}_{1,x_1}}{\varphi_{1,x_1}} \,
ds=\log |\varphi_{1,x_1}|+C\cong \log U.$}, the logarithmic strain tensor $\log V$ is equal to the integrated strain rate. More
generally\footnote{Computing the rates $\frac{\rm d}{{\rm d t}} \log U$ is more complicated because, in addition to the principal strains being a
function of time, the principal directions also change in time \cite{jog2008explicit,gurtin1983relationship,hoger1986material,dui2006some}.}, $\frac{\rm d}{\rm
dt}[\log V(t)]=D(t)$ for any
coaxial stretch family $V(t)$.
\item[$\oplus_{20}$] The logarithmic strain  possesses certain intrinsic, far-reaching properties
that also suggest  its favored position among
all possible strain measures: the Eulerian
logarithmic strain $\log V$ is the unique strain measure whose corotational rate
(associated with the so-called logarithmic spin) is
the strain rate tensor $D$. In other words, the strain rate tensor  is the
co-rotational rate of the Hencky strain tensor associated
with the logarithmic spin tensor. Such a result has been introduced by Reinhardt and
Dubey \cite{reinhardt1996application} as $D$-rate and by Xiao
et al. \cite{xiao1997logarithmic,xiaoAM2002} as log-rate (see also \cite{ogden1970compressible,xiao2006elastoplasticity,Xiao2005}). This is  consistent
with  Truesdell's rate type concept of hypoelasticity based on a unique logarithmic strain rate
\cite{gurtin1983relationship,meyers2006choice,xiao1999existence}. We need to emphasize  that, contrary to hyperelastic models,  hypo-elastic models
\cite{meyers1999consistency,muller2006thermodynamic} ignore the potential character of the energy. Otherwise they are simply the hyperelastic models
rewritten in a suitable incremental form. In case of the logarithmic rate, the hypo-elastic model integrates exactly to the hyperelastic quadratic Hencky
model.
\item[$\oplus_{21}$] The  quadratic Hencky  energy $W_{_{\rm H}}$ satisfies the Baker-Ericksen (BE) inequalities everywhere, see Subsection \ref{BEsect} later in
    this paper.
\item[$\oplus_{22}$] The Cauchy stress $\sigma=\sigma(\log V)$ induces an invertible true-stress-true-strain relation up to  $\det F\leq e$
    \cite{vallee1978lois,vallee2008dual}.
\item[$\oplus_{23}$] The Kirchhoff stress $\tau=\tau(\log V)$  is invertible.
\item[$\oplus_{24}$] The quadratic Hencky energy $W_{_{\rm H}}$ satisfies Hill's inequality (KSTS-M$^+$) everywhere, i.e.  the corresponding Kirchhoff stress
    $\tau_{_{\rm H}}=(\det F) \cdot\sigma=D_{\log V}W_{_{\rm H}}(\log V)$ is a monotone function of the logarithmic strain tensor $\log V$ and $W_{_{\rm H}}$ is a convex function
    of $\log V$.
\item[$\oplus_{25}$] The Kirchhoff stress ${\tau}_{_{\rm H}}$ has the symmetry property $\tau_{_{\rm H}}(V^{-1})=-\tau_{_{\rm H}}(V)$. In fact, this relation is true whenever the
    energy satisfies the tension-compression symmetry.
    \item[$\oplus_{26}$] Since $\log B=\log V^2=2\log V${,} there is no need to compute the polar decomposition \cite{jog2002explicit,neff2013grioli} in
        order to evaluate $\log V$.
\item[$\oplus_{27}$] There is a representation of $\|\dev_3\log V\|^2$ and $[\tr(\log V)]^2$ in terms of principal invariants of $V$ available
    \cite{Fitzgerald,dui2006some}: $\log V=\alpha_0\, \id +\alpha_1 V+\alpha_2\ V^2=\beta_0\, \id +\beta_1 V+\beta_{-1}\ V^{-1}$, $\alpha_h=\alpha_h(i_1,i_2,i_3)$,
    $\beta_r=\beta_r(i_1,i_2,i_3)$, $i_h=i_h(V)$, $h=1,2,3$, $r=-1,0,1$. Moreover, it is always possible to express the strain energy terms via its
    representation in principal stretches from which we may infer, via Cardano's formula, a representation in terms of the principal invariants of $B$,
    i.e. $W_{_{\rm H}}=\widetilde{W}_{_{\rm H}}(i_1(B),i_2(B),i_3(B))$ \cite{xiao1997logarithmic}. Otherwise, calculation of $\log V$ needs diagonalization  and
    determination of the principal axes. Then $\log V=Q^T\,\left(\begin{array}{ccc}
                                   \log \lambda_1 & 0 & 0 \\
                                   0 &  \log \lambda_2 & 0 \\
                                   0 & 0 &  \log \lambda_3
                                 \end{array}\right) Q$ for $ V=Q^T\,\left(\begin{array}{ccc}
                                 \lambda_1 & 0 & 0 \\
                                   0 &  \lambda_2 & 0 \\
                                   0 & 0 &   \lambda_3
                                 \end{array}\right) Q$ and $Q\in {\rm SO}(3)$.
\item[$\oplus_{28}$] There are  efficient  methods for the  explicit evaluation
of the derivatives of the logarithm of an arbitrary tensor \cite{jog2008explicit,al2013computing}.
\item[$\oplus_{29}$] The use of the logarithmic strain tensor $\log U$ leads to simple additive structures in algorithmic computational elasto-plasticity theory
    \cite{alexander1971tensile,simo1992algorithms,peric1992model,weber1990finite,rolph1984large,sansour1997theory}.
\end{itemize}

\begin{figure}[h!]
\hspace*{1cm}
\begin{minipage}[h]{0.4\linewidth}
\centering
\includegraphics[scale=0.8]{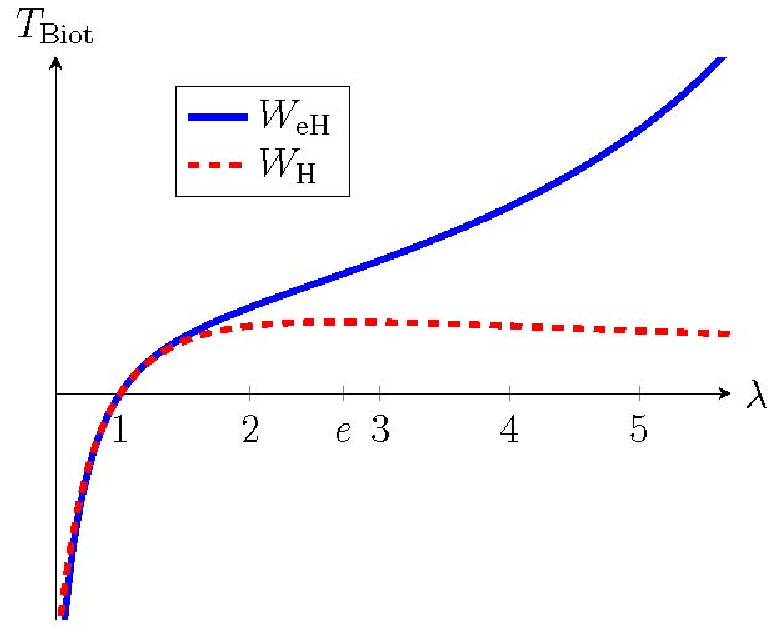}
\centering
\caption{\footnotesize{Nominal stress obtained from the exponentiated Hencky energy  $W_{_{\rm eH}}$ and the
classical Hencky energy $W_{_{\rm H}}$  for uniaxial deformation. Loss of monotonicity beyond $e$ for $W_{_{\rm H}}$.}}
\label{uxl}
\end{minipage}
\hspace{1cm}
\begin{minipage}[h]{0.4\linewidth}
\centering
\includegraphics[scale=0.8]{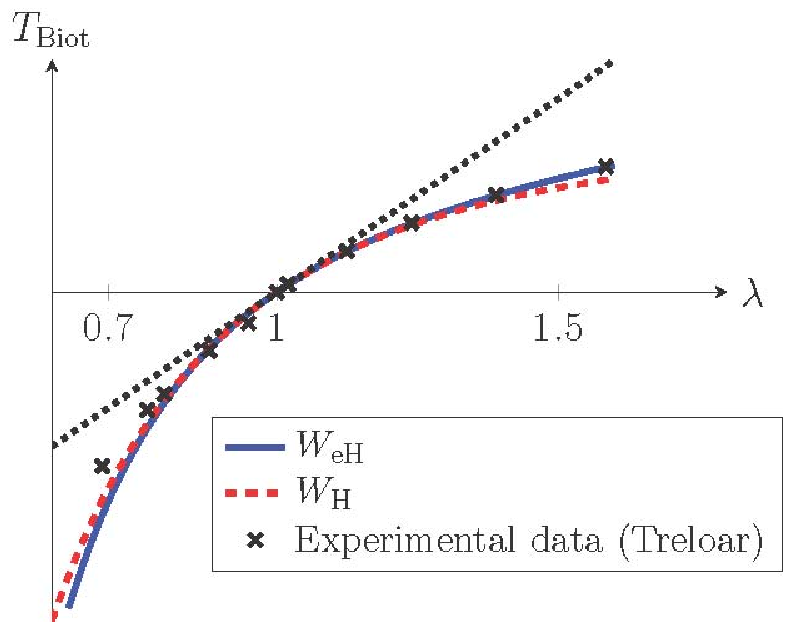}
\centering
\caption{\footnotesize{The generic infinitesimal strain nonlinearity and second-order behaviour in agreement with Bell's
observation \cite{bell1973,jones1975properties}: decreasing elasticity modulus in tension, increasing modulus in compression.}}
\label{dumitrel_thirdOrderConstants}
\end{minipage}
\end{figure}%

For these reasons the  quadratic Hencky model is used in theoretical investigations and in physical applications
\cite{MieheApel,AuricchioIJP11,Bruhns01,Bruhns02JE,Fitzgerald,Plesek06,xin1994finite,gao1997exact,gao2003finite,freed2014hencky}. We observe also a renewed
interest in this class of isotropic slightly compressible hyperelastic solids originally proposed by Hencky
\cite{Hencky28a,Hencky29a,Hencky29b,Hencky31,henann2011large}. The strain energy $W_{_{\rm H}}$ is also   often used in commercial  FEM-codes.

\medskip

  However, the quadratic Hencky energy has  some serious {shortcomings:}
\begin{itemize}
\item[$\ominus_{1\ }$] Beyond $\det F\leq e$, the Hencky energy $W_{_{\rm H}}$ leads to no globally invertible Cauchy stress-logarithmic strain relation and the
    possibility for multiple symmetric homogeneous bifurcations may arise  \cite[page 48]{jog2013conditions}, see  also \cite[page
    185]{truesdell1975inequalities},\cite{moon1974interpretation}.
    \item[$\ominus_{2\ }$] The Cauchy stress tensor is degenerate in the sense that $\sigma_{_{\rm H}}\nrightarrow+\infty$ for $V\rightarrow +\infty$ and there are Cauchy stress distributions which cannot be reached by the constitutive law, i.e. $V\mapsto \sigma(V)$ {is not surjective.}

\item[$\ominus_{3\ }$] The  energy $W_{_{\rm H}}$ does not satisfy the pressure-compression (PC) inequality (this is related to  the non-convexity of $\det
    F\mapsto(\log\det F)^2$ for $\det F> e$).
\item[$\ominus_{4\ }$] One may not guarantee  real wave speeds  over the entire deformation range  \cite{Bruhns01,Hutchinson82,Neff_Diss00}. Therefore,
    $W_{_{\rm H}}$ is not quasiconvex (weakly lower semicontinuous) and not Legendre-Hadamard (LH)-elliptic (rank-one convex).
\item[$\ominus_{5\ }$] The tension-extension (TE) inequalities (separate convexity) are not satisfied (see Proposition \ref{TEHencky}) .
\item[$\ominus_{6\ }$]  The  quadratic Hencky energy $W_{_{\rm H}}$ is not coercive, i.e. an estimate of the type
$$
W_{_{\rm H}}(F)\geq C_1 \|F\|^q-C_2, \quad q\geq 1, \quad C_1,\, C_2>0
$$
is not possible, since $W_{_{\rm H}}$ growths only sublinearly.
\item[$\ominus_{7\ }$] The   true-stress-stretch  invertibility (TSS-I) does not hold true everywhere.
\end{itemize}
These points being more or less well-known, it is clear that there cannot exist a general mathematical well-posedness result for the quadratic Hencky model
$W_{_{\rm H}}$. Of course, in the vicinity of the stress free reference configuration, an existence proof  for small loads based on the implicit function theorem
will always be possible \cite{Ciarlet88}. All in all,  however, the  status of Hencky's quadratic energy,  despite its many attractive features,  is thus
put into doubt.

\medskip

 For sufficiently  regular energies, Legendre-Hadamard ellipticity on ${\rm GL}^+(3)$ (LH-ellipticity, also known as rank-one-convexity\footnote{Since
 ${\rm GL}^+(3)$ is an open subset of $\R^{3\times 3}$, in accordance with \cite[page 352]{Ball77}  we say that $W$ is  rank-one convex on ${\rm GL}^+(3)$
 if it is convex on all closed line segments in ${\rm GL}^+(3)$ with end points differing by a matrix of rank one, i.e
 \begin{align}
W( F+(1-\theta)\, \xi\otimes \eta)\leq \theta \,W( F)+(1-\theta) W(F+\xi\otimes \eta),
\end{align}
for all $F\in {\rm GL}^+(3)$, $\theta\in[0,1]$, {and for all} $\,\, \xi,\, \eta\in\mathbb{R}^3$, with $F+t\, \xi\otimes \eta\in {\rm GL}^+(3)$ for all
$t\in[0,1]$. In other words, the energy function $W$ is  rank-one convex on ${\rm GL}^+(3)$ if and only if the {function} $t\mapsto W(F+t\xi\otimes\eta)$
is convex $\forall\,\, \xi,\, \eta\in\mathbb{R}^3$, on all
closed line segments in the set $\{t:F+t\, \xi\otimes \eta\in {\rm GL}^+(3)\}$.})
\cite{silhavy2002monotonicity,vsilhavy2002convexity,vsilhavy2001rank,Silhavy03,Ogden83,Neff_critique05} is tantamount to
\begin{align}
\label{def:lhellipt}
\langle D_FS_1(F).(\xi\otimes\eta),\xi\otimes\eta\rangle= D^2_F W(F)(\xi\otimes\eta,\xi\otimes\eta)>0, \quad \forall\,\, \xi,\,
\eta\in\mathbb{R}^3\setminus \{0\},\quad \forall \, F\in {\rm GL}^+(3).
\end{align}
This condition stems from the study of wave propagation\footnote{The condition $D^2_F W(F)(\xi\otimes\xi,\xi\otimes\xi)>0 \
\forall\,\xi\in\mathbb{R}^3\setminus \{0\}$, i.e. the convexity of $t\mapsto W(F+t\xi\otimes\xi)$ for all $\xi\in\mathbb{R}^3$ with $F+t\, \xi\otimes
\xi\in {\rm GL}^+(3)$ for all $t\in[0,1]$, is a necessary condition for the existence of at least one longitudinal acceleration wave
\cite{eremeyev2007constitutive,ZubovRudev,SawyersRivlin78}.} or hyperbolicity of the dynamic problem  and  it is just what is needed
 for a good existence and uniqueness theory for linear elastostatics and elastodynamics (see
 \cite{Ogden83,fosdick2007note,edelstein1968note,simpson2008bifurcation}).
 The failure of ellipticity \cite{silling1988numerical,merodio2006note} may be related to the emergence of discontinuous deformation gradients
 \cite{knowles1975ellipticity,ernst1998ellipticity,wang1996reformulation}. Strict rank-one convexity in the solution of the boundary value problem is also
 necessary for the smoothness of  weak solutions. While strong ellipticity apparently holds over wide ranges, including buckling, and is physically rather compelling,
 it is not necessarily universal  \cite[page 20]{Marsden83} (see also \cite{sidoroff1974restrictions}). However, from a numerical point of view in finite element simulations, loss of ellipticity manifests itself by a pathological dependence of the computed results on the size and distortion of the finite elements and should therefore be avoided.

\medskip

Concerning our new formulation, it is clear that, up to moderate strains,  {for principal stretches $\lambda_i\in(0.7,1.4)$}, our  exponentiated Hencky
formulation \eqref{the} is de facto as good as the quadratic  Hencky model $W_{_{\rm H}}$ and in the large strain region it will improve several important features
from a mathematical point of view\footnote{The domain where the Hencky energy $W_{_{\rm H}}$ is rank-one convex is included in the domain for which the eigenvalues
$\lambda_1,\lambda_2,\lambda_3$ of $U$ satisfy $
\lambda_1^2\leq e^2 \lambda_2\lambda_3,\qquad \lambda_2^2\leq e^2 \lambda_3\lambda_1, \qquad \lambda_3^2\leq e^2 \lambda_1\lambda_2
$ (see Corollary \ref{TEHenckycor}{)}. Moreover, this domain is included in the domain defined by $\|\dev_3\log U\|^2\leq \frac{4}{3}$. Numerical computations
reveal that the exponentiated Hencky energy is rank-one convex in a domain for which $\|\dev_3\log U\|^2\leq a$ with $a> \frac{4}{3}$ (see Subsection
\ref{rank-dom}).}.

Having identified $K_2^2=\|{\rm dev}_n\log U\|^2$ and $K_1^2=[{\rm tr}(\log U)]^2$  as the basic input variables   for a nonlinear  elasticity formulation,
this investigation started by numerically checking the ellipticity conditions for $e^{\|\log U\|^2}$ in the two-dimensional case\footnote{In this paper we
also show that for planar elastostatics $F\mapsto e^{\|\log U\|^2}$ is {\bf not rank-one convex}, a surprising  observation which is difficult to obtain,
since ellipticity is lost  for extremely large principal stretches only.}. In one space dimension it is  readily observed that $t\mapsto (\log t)^2$ is not
convex, but $t\mapsto e^{(\log t)^2}$ is convex (see Figure{s} \ref{devlogfig} and \ref{logfig}). A similar effect appears for the Hencky energy
\eqref{th}: $W_{_{\rm H}}$ is not {LH-elliptic} \cite{Bruhns01,Silhavy02b,Silhavy05} (it is of the type given by Figure \ref{devlogfig}), but we show that our
energy $W_{_{\rm eH}}(U)$ is  { LH-elliptic}   {in the two dimensional case} (it is of the type given by Figure \
\ref{logfig}).
\begin{figure}[h!]
\hspace*{1cm}
\begin{minipage}[h]{0.4\linewidth}
\centering
\includegraphics[scale=0.65]{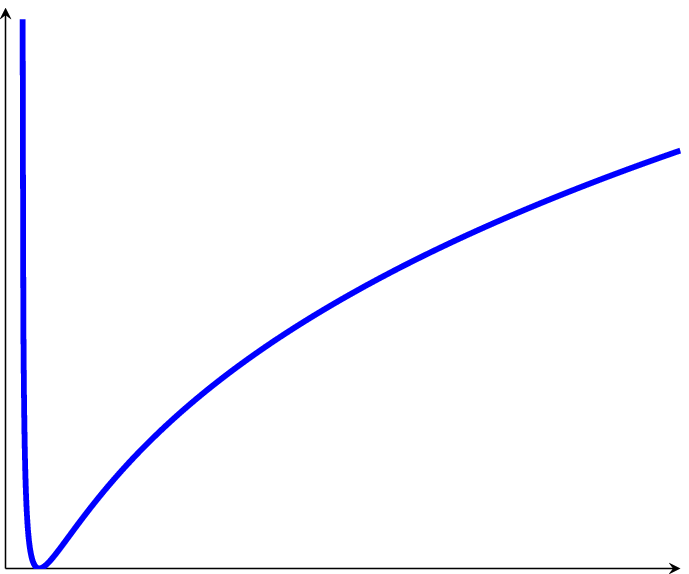}
\centering
\caption{\footnotesize{$W_{_{\rm H}}(F)$ is not rank-one convex}}
\label{devlogfig}
\end{minipage}
\hspace{1cm}
\begin{minipage}[h]{0.4\linewidth}
\centering
\includegraphics[scale=0.65]{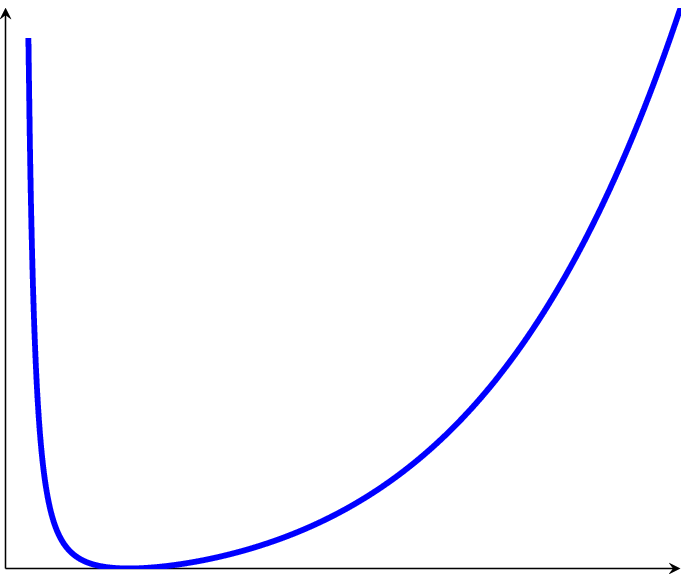}
\centering
\caption{\footnotesize{$W_{_{\rm eH}}(F)$ is rank-one convex.}}
\label{logfig}
\end{minipage}
\end{figure}%

In this paper, then,  we {prove} that the functions
$
W_{_{\rm eH}}(F):={\frac{\mu}{k}\,e^{k\,\|\dev_n\log\,U\|^2}+\frac{\kappa}{2\widehat{k}}\,e^{\widehat{k}\,(\tr(\log U){)}^2}}$ from the family of energies defined
in \eqref{the} have the following attractive properties beyond those of $W_{_{\rm H}}$\footnote{The idea of considering
the exponential function in modelling of nonlinear elasticity is not entirely new. In
fact $W(F)=\frac{\mu}{2\, k}\left[e^{k\,(I_1-3)}-1\right]$, where
$I_1=\tr(F\,F^T)$, is a Fung-type model which is often used in the biomechanics
literature to describe the nonlinearly elastic response of biological tissues
\cite{fung1979inversion,beatty1987topics}. In the limit
$\lim\limits_{k\rightarrow0}\frac{\mu}{2\,
k}\left[e^{k\,(I_1-3)}-1\right]=\frac{\mu}{2}(I_1-3)$, we recover the
Neo-Hookean energy for elastic incompressible materials. Another Fung-type energy \cite{fung1979inversion,beatty1987topics} is $W(F)=\frac{\mu}{2\, k}\left[e^{k\,\|C-\id\|^2}-1\right]$.}{:}
 \begin{itemize}
 \item[$\oplus_{1\ }$] For nonlinear incompressible material (like rubber) the new energy $\frac{\mu}{k} \, e^{k\,\|\dev_3\log U\|^2}$ has only two
     independent constants{, which furthermore have }a clear physical meaning, the infinitesimal shear modulus $\mu>0$ and the distortional
     strain-stiffening parameter $k>0$ (see Figure \ref{graphic-dupa-k} and \ref{tauetauh}). For nonlinear (slightly) compressible material{,} in
     addition{,} there is the infinitesimal bulk modulus $\kappa>0$ and also the volumetric stiffening parameter $\widehat{k}>0$.
    \item[$\oplus_{2\ }$]  The Cauchy stress tensor  satisfies $\sigma_{_{\rm eH}}\rightarrow+\infty$ for $V\rightarrow +\infty$.
 \item[$\oplus_{3\ }$] We have: $\lim_{k,\widehat{k}\rightarrow 0}\sigma_{_{\rm eH}}=\sigma_{_{\rm H}}$,\ \  $\lim_{k,\widehat{k}\rightarrow 0}\tau_{_{\rm eH}}=\tau_{_{\rm H}}$.
     \item[$\oplus_{4\ }$] At very large stretch ratios the model exhibits the strain stiffening behaviour common to many elastomers.
 \item[$\oplus_{5\ }$]{They} satisfy the BE-inequalities.
  \item[$\oplus_{6\ }$] {They} satisfy the PC-inequalities.
   \item[$\oplus_{7\ }$] {They} satisfy the TE-inequalities in the planar case if $k\geq\frac{1}{4}$.
    \item[$\oplus_{8\ }$] {They} satisfy the TE-inequalities in the three dimensional case if $k\geq\frac{3}{16}$.
     \item[$\oplus_{9\ }$] {They} are {rank one convex} (LH-elliptic) in the planar case if $k\geq\frac{1}{4}$, in the entire deformation range.
     \item[$\oplus_{10}$]  {T}he corresponding  Kirchhoff stress $\tau_{_{\rm eH}}$ has the property: $\tau_{_{\rm eH}}(V^{-1})=-\tau_{_{\rm eH}}(V)$.
       \item[$\oplus_{11}$] {T}he   true-stress-true-strain  invertibility (TSTS-I) holds true everywhere.
         \item[$\oplus_{12}$]  {T}he   true-stress-stretch  invertibility (TSS-I) holds true everywhere.
         \item[$\oplus_{13}$]  {T}he   true-stress-true-strain monotonicity (TSTS-M) is satisfied for bounded distortions.
             \item[$\oplus_{14}$]  {Hill's} inequality  (KSTS-M$^+$) is satisfied, in the entire deformation range and $\tau_{_{\rm eH}}$ is invertible.
              \item[$\oplus_{15}$]  Planar pure Cauchy shear stress produces biaxial pure shear strain and  $\nu=\frac{1}{2}$ corresponds to exact
                  incompressibility.
             \item[$\oplus_{16}$] For $n=3$ among the family $W_{_{\rm eH}}$ there exists  a special  ($k=\frac{2}{3}\, \widehat{k}\;$) three parameter subset
                 such that uniaxial  tension leads to no lateral contraction if and only if the Poisson's ratio $\nu=0$, as in linear elasticity.
         \item[$\oplus_{17}$]  {T}here is no number $k>0$ such that  $W_{_{\rm eH}}$ is  rank one convex everywhere in the three dimensional case, but there is
             a built-in failure criterion active on extreme distortional strains: the energy seems to be  rank-one convex in the cone-like elastic
             domain
$
 \mathcal{E}^+(W_{_{\rm eH}}^{\rm iso},{\rm LH}, U,27)=\big\{U\in {\rm PSym}(3) \,\big|\,\|\dev_3 \log U\|^2\leq 27\big\}.$
 \end{itemize}

\begin{figure}[h!]\begin{center}
\begin{minipage}[h]{0.75\linewidth}
\centering
\includegraphics[scale=0.75]{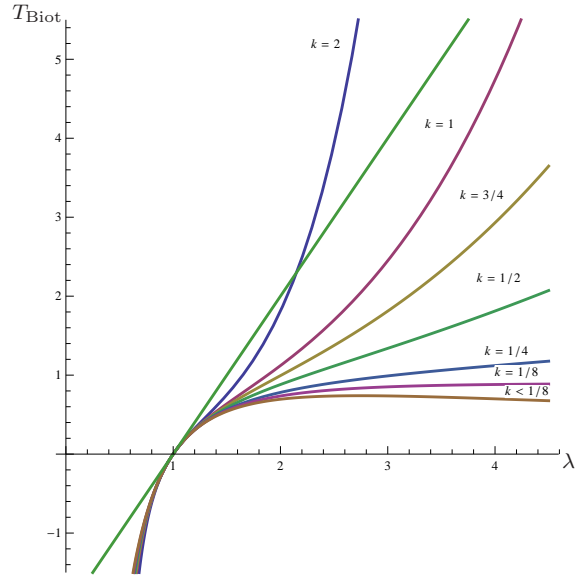}
\put(0,40){\footnotesize $\lambda$}
\put(-208,210){\footnotesize $T_{\rm Biot}$}
\centering
\caption{\footnotesize{Qualitative picture of nominal stress response of $W_{_{\rm eH}}$ for uniaxial elongation, for different values of $k$, typical S-shape
entropic elasticity response. We remark that for $k<\frac{1}{8}$ the function $\frac{{\rm d}}{{\rm d} \lambda}\left[\frac{\mu}{k}\,e^{k\log^2
\lambda}\right]$  is not everywhere monotone increasing. The value of $k$ determines the shape of the strain-hardening and strain-softening response, with
larger $k$ implying strong strain-hardening. Specific values of $k$ only change the response for large elastic  strains $|\log \lambda|>0.1$. Classical
Hencky's response is retrieved for $k\rightarrow 0.$ Invertibility of the Cauchy stress-stretch response needs $k>\frac{1}{8}$. In the uniaxial case it is
intuitively reasonable that there should be a bijective constitutive relationship between stress and strain that defines the mechanical properties of an
idealized elastic body. For large stretch values $\lambda$  the function $\frac{{\rm d}}{{\rm d} \lambda}\left[\frac{\mu}{k}\,e^{k\log^2 \lambda}\right]$
is monotone increasing for all $k>0$. The influence of $k$ on the response is  easily understood and we have negligible effect of $k$ under pressure.}}
\label{graphic-dupa-k}
\end{minipage}
\end{center}
\end{figure}

\begin{figure}[h!]\begin{center}
\begin{minipage}[h]{0.7\linewidth}
\centering
\includegraphics[scale=1]{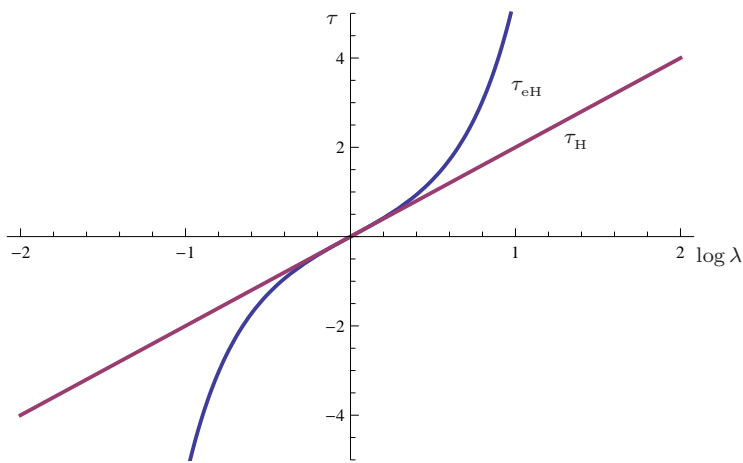}
\put(0,75){\footnotesize $\log \lambda$}
\put(-50,120){\footnotesize $\tau_{_{\rm H}}$}
\put(-70,140){\footnotesize $\tau_{_{\rm eH}}$}
\put(-140,165){\footnotesize $\tau$}
\centering
\caption{\footnotesize{Uniaxial Kirchhoff stress tensor as a function of logarithmic strain for the classical Hencky model $W_{_{\rm H}}$ and our exponentiated family
$W_{_{\rm eH}}$.}}
\label{tauetauh}
\end{minipage}
\end{center}
\end{figure}
 These results {completely settle} the status of the quadratic Hencky energy as a useful approximation in plane elasto-statics and lead to new perspectives
 for the three-dimensional idealized isotropic setting.

The contents of this paper in the order of their appearance {are: i)} a further short discussion of the existing literature; ii) notation; iii) introduction of  general constitutive requirements in idealized nonlinear elasticity; iv) the
invertible true-stress-true-strain relation; v) rank-one convexity in the two-dimensional case; vi) domains of rank-one convexity in  the
three-dimensional case;  vii) summary; viii) extensive list of references; ix) appendix.

\subsection{Previous work in the spirit of our investigation}
Roug{\'e}e \cite[pages 131, 302]{rougee2006intrinsic} (see also \cite{freed1995natural,rougee1997mecanique} and later extensions by Fiala
\cite{fiala2004time,fiala2008,fiala2010,fiala2011geometrical}) identifies  Hencky's logarithmic strain measure $2\,\log U=\log C$ as having (as its
Frobenius tensor norm) the length of a geodesic joining two metric states: he endows the set of positive definite matrices  {\rm PSym}$(3)$ (which is
not a Lie-group w.r.t. matrix multiplication) with a Riemannian structure (see also \cite{bhatia2006riemannian,moakher2005differential,moakher2002means}). In this case, geodesics joining the identity $\id$ with any metric tensor $C=F^TF$
are simply one-parameter groups $t\mapsto {\rm exp}(t\log C)$.  This interpretation is fundamentally different from   ours given in
\cite{NeffEidelOsterbrinkMartin_Riemannianapproach,neff2013hencky,Neff_Osterbrink_Martin_hencky13} and hinted at in the introduction.

 Criscione et al. \cite{Criscione2000} proposed a new invariant basis for the natural strain $\log U,$ which leads to a representation for the Cauchy
 stress $\sigma $  as the sum of three response terms that are mutually orthogonal (see also \cite{ogden2004fitting}). In fact, Criscione et al.  \cite{Criscione2000,Criscione2005}  (see also
 \cite{Diani005}) consider energies $W_{\rm Crisc}(K_1,K_2,K_3)$ based on the Hencky-logarithmic strain, where $(K_1,K_2,K_3)$ is a  set of invariants for
 the isotropic case\footnote{Richter in 1949 \cite{richter1949hauptaufsatze} already  considers the following complete set of isotropic invariants: $K_1=\tr(\log U),
 K_2^2= \tr((\dev_3\log U)^2)$ and $\tr((\dev_3\log U)^3)$, see also \cite{lurie2012non}. A similar list of invariants was used by Lurie \cite[page
 189]{lurie2012non}: $K_1$, $K_2$ and $\widetilde{K}_3=\arcsin (K_3)$. }:
\begin{equation}\left\{\begin{array}{ll}
\text{``the amount-of-dilatation":} &K_1=\tr(\log U)=\log \det U=\log \det V=\log \det F,\vspace{2mm}\\
\text{``the magnitude-of-distortion":}&K_2=\|\dev_3\log U\|=\|\dev_3\log V\|,\vspace{2mm}\\
\text{``the mode-of-distortion":}&K_3=3\sqrt{6}\, \det \left(\dd\frac{\dev_3 \log U}{\|\dev_3\log U\|}\right).
\end{array}\right.
\end{equation}
As it turns out, any isotropic energy can also be represented as a function $W_{\rm Crisc}=W_{\rm {C}risc}(K_1,K_2,K_3)$ of Criscione's invariants (see
\cite{Criscione2000,criscione2002direct,Bechir14}). In this paper, we use exclusively  $|K_1|^2$ (which we   call  accordingly the
``magnitude-of-dilatation") and the magnitude-of-distortion $K_2^2$,  but with our different geometric  motivation.

  In \cite{Walton05} some necessary conditions for  the LH-ellipticity versus exponential-growth are discussed for energies depending on the Hencky strain
  $\log U$.  In fact, Sendova and Walton \cite{Walton05} have considered the energy $W$ to be a function of $K_2=\|\dev_3 \log U\|$ and  proved that $W$
  has to grow at least exponentially as a function of  $K_2$. They note, however,  that ``constructing conditions that are both necessary and sufficient
  for strong ellipticity to hold for all deformations still seem[s to be] a daunting task". In \cite{sansour2008physical} Sansour has discussed the
  multiplicative decomposition of the deformation gradient into its volumetric and isochoric parts and
its implications in the case of anisotropy.   Sansour's statement for isotropy is already contained  in the paper by Richter \cite[page 209]{richter1948isotrope}. This
decomposition problem was studied later in the papers \cite{annaidh2012deficiencies,horgan2009constitutive}. Gearing and Anand \cite{gearing2004notch} (see
also \cite{henann2011large,dluzewski2003numerical}) recently propose{d} an energy of the form\footnote{The energy \eqref{anandenergy} does not satisfy the tension-compression
symmetry.}
\begin{align}\label{anandenergy}
W_{{\rm Anand}}(\log U)=\mu(\log \det U)\cdot \|\dev_3\log U\|^2+h(\log\det U),
\end{align}
where $\det U\mapsto h(\log (\det U))$ is highly non-convex. The energy $W_{{\rm Anand}}(\log U)$ couples volumetric and distortional response and is based
on molecular dynamics simulations. The molecular dynamics simulation\footnote{The numerical results given by Hennan and Anand  \cite{henann2011large} correspond
to the large volumetric strain range $0.75\leq\det F\leq 1.16$  ($-0.3\leq\log\det F\leq 0.15$) but  small shear strain range
$\|\dev_3\log V\|\leq 0.035$.} is not in contradiction with an increasing generalized shear modulus {$\mu$ as} $\det U\rightarrow 0 $, see
\cite{henann2011large,henann2009fracture}.

 The Baker-Ericksen (BE) inequalities express the requirement that the greater principal Cauchy stress should occur in the direction of the greater
 principle stretch, while the tension-extension (TE) inequalities  demand that each principal stress is a strictly increasing function of the corresponding
 principal stretch. The BE-inequalities and TE-inequalities arise in connection with propagation of waves in principal direction of strain
 \cite{Truesdell65}. The strong ellipticity condition for hyperelastic  materials  \cite{zhilin2013material} was studied in
 \cite{knowles1975ellipticity,SawyersRivlin78,Hill79,aubert1995necessary,wang1996reformulation}, but the complete study seems to be presented  first in
 Ogden's Ph.D.-thesis \cite{Ogden00}.  For an incompressible hyperelastic material corresponding conditions were given in \cite{SawyersRivlinIncomp}. A
 family of universal solutions in plane elastostatics for the quadratic Hencky model is obtained in \cite{Aron2006}.

 A stronger constitutive requirement than rank-one convexity is Ball's-polyconvexity condition \cite{Ball77,Ball78}. A free energy function $W(F)$ is
 called  polyconvex if and
only if it is expressible in the form
$W(F) =P(F, \Cof F, \det F)$, $P:\mathbb{R}^{19}\rightarrow\mathbb{R}$, where $P(\cdot,\cdot,\cdot)$ is convex. Polyconvexity implies weak lower
semicontinuity, quasiconvexity and rank-one convexity.  Quasiconvexity of the energy function $W$ at $\overline{F}$   means that
\begin{align}
\int_{\Omega}W(\overline{F}+\nabla \vartheta)dx\geq \int_{\Omega}W(\overline{F})dx=W(\overline{F})\cdot |\Omega|, \quad \text{for every bounded open set}
\quad \Omega\subset\mathbb{R}^3
\end{align}
holds for all $ \vartheta\in C_0^\infty (\Omega)$ such that $\det(\overline{F}+\nabla \vartheta)>0$. It implies that the homogeneous solution
$\varphi(x)=\overline{F}.\, x, \ x\in\R^3$ is always a global energy minimizer subject to its own Dirichlet boundary conditions.

In fact, polyconvexity is the cornerstone notion for a proof of the existence of minimizers by the direct methods of the calculus of variations
for energy functions satisfying no polynomial growth conditions. This is  typically the case in nonlinear elasticity since one has the natural requirement
$W(F)\rightarrow\infty$ as $\det F\rightarrow0$. Polyconvexity is best understood for isotropic energy functions, but it  is  not restricted to
isotropic response. It was a long standing open question how to extend the notion of polyconvexity in a meaningful way to anisotropic materials
\cite{Ball02}. The  answer   has been provided in a series  of papers
\cite{Neff_Vortrag_Schroeder_IUTAM02,HutterSFB02,Balzani_Schroeder_Gross_Neff05,Schroeder_Neff_Ebbing07,cism_book_schroeder_neff09,Schroeder_Neff01,Hartmann_Neff02,Schroeder_Neff04,merodio2006note,Balzani_Neff_Schroeder05,Schroeder_Neff_Ebbing07,Ebbing_Schroeder_Neff_AAM09}.
For  isotropic strain energies, the polyconvexity condition in the case of space dimension 2 was conclusively discussed  by Rosakis \cite{Rosakis98} and
\v{S}ilhav\'{y} \cite{Silhavy99b}, while the case of arbitrary space dimension was studied by Mielke \cite{Mielke05JC}, by Dacorogna and Marcellini
\cite{DacorognaMarcellini}, Dacorogna and Koshigoe \cite{DacorognaKoshigoe} and  Dacorogna and Marechal \cite{DacorognaMarechal}.

\subsection{Notation}\label{auxnot}

 For $a,b\in\R^n$ we let $\langle {a},{b}\rangle_{\R^n}$  denote the scalar product on $\R^n$ with
associated vector norm $\|a\|_{\R^n}^2=\langle {a},{a}\rangle_{\R^n}$.
We denote by $\R^{n\times n}$ the set of real $n\times n$ second order tensors, written with
capital letters.
The standard Euclidean scalar product on $\R^{n\times n}$ is given by
$\langle {X},{Y}\rangle_{\R^{n\times n}}=\tr{(X Y^T)}$, and thus the Frobenius tensor norm is
$\|{X}\|^2=\langle {X},{X}\rangle_{\R^{n\times n}}$. In the following we do not adopt any summation convention and we omit {the   subscript} $\R^{n\times
n}$ in writing the Frobenius tensor norm. The identity tensor on $\R^{n\times n}$ will be denoted by $\id$, so that
$\tr{(X)}=\langle {X},{\id}\rangle$. We let $\Sym(n)$ and $\rm PSym(n)$ denote the symmetric and positive definite symmetric tensors respectively. We adopt
the usual abbreviations of Lie-group theory, i.e.
${\rm GL}(n):=\{X\in\R^{n\times n}\;|\det{X}\neq 0\}$ denotes the general linear group,
${\rm SL}(n):=\{X\in {\rm GL}(n)\;|\det{X}=1\},\;
\mathrm{O}(n):=\{X\in {\rm GL}(n)\;|\;X^TX=\id\},\;{\rm SO}(n):=\{X\in {\rm GL}(n,\R)\;| X^T X=\id,\;\det{X}=1\}$, ${\rm GL}^+(n):=\{X\in\R^{n\times
n}\;|\det{X}>0\}$  is the group of invertible matrices with positive determinant,
 $\mathfrak{so}(3):=\{X\in\mathbb{R}^{3\times3}\;|X^T=-X\}$ is the Lie-algebra of  skew symmetric tensors
and $\mathfrak{sl}(3):=\{X\in\mathbb{R}^{3\times3}\;| \tr({X})=0\}$ is the Lie-algebra of traceless tensors. {Here and i}n the following the superscript
$^T$ is used to denote transposition, and $\Cof A = (\det A)A^{-T}$ is the cofactor of $A\in {\rm GL}(n)$. The set of positive real numbers is
denoted by $\R_+:=(0,\infty)$, while $\overline{\R}_+{:}=\R_+\cup \{\infty\}$. For all  vectors $\xi,\eta\in\R^3$ we have the (dyadic) tensor product
$(\xi\otimes\eta)_{ij}=\xi_i\,\eta_j$.

Let us consider $W(F)$ to be the strain energy function of an elastic material in which $F$ is the gradient of a deformation from a reference configuration
to a configuration in the Euclidean 3-space; $W(F)$ is measured per unit volume of the reference configuration. The domain of $W(\cdot)$ is ${\rm
GL}^+(n)$.  We denote by $C=F^T F$ the right Cauchy-Green strain tensor, by $B=F\, F^T$ the left Cauchy-Green (or Finger) strain tensor, by $U$ the right
stretch tensor, i.e. the unique element of ${\rm PSym}(n)$ for which $U^2=C$ and by  $V$ the  left stretch tensor, i.e. the unique element of ${\rm
PSym}(n)$ for which $V^2=B$. Here,  we are only concerned with rotationally symmetric functions (objective and isotropic), i.e.
 $
 W(F)={W}(Q_1^T\, F\, Q_2) \ \forall \, F=R\,U=V R\in {\rm GL}^+(n),\ Q_1,Q_2,R\in{\rm SO}(n).
 $ We define $J=\det{F}$ and we denote by $S_1=D_F[W(F)]$ the first Piola-Kirchhoff stress tensor, by $S_2=F^{-1}S_1=2\,D_C[W(C)]$
 the second  Piola-Kirchhoff stress tensor, by $\sigma=\frac{1}{J}\,  S_1\, F^T$ the Cauchy stress tensor, and by $\tau=J\cdot \sigma$ the Kirchhoff stress
 tensor.

\section{Constitutive requirements in idealized nonlinear elasticity}\label{auxnot}\setcounter{equation}{0}

\subsection{The Baker-Ericksen inequalities}\label{BEsect}

An ``ellipticity criterion'' much weaker than the LH-ellipticity criterion {\eqref{def:lhellipt}} are the so called Baker-Ericksen (BE) inequalities. The
Baker-Ericksen inequalities are arguably  an   absolutely necessary requirement for reasonable material behaviour. Baker and Ericksen \cite{BakerEri54}
considered a unit cube of isotropic elastic material to undergo a pure homogeneous deformation with principal directions parallel to the edges of the cube.
They showed that the BE-inequalities are necessary and sufficient for  the greater princip{al} Cauchy stress to occur in the direction of the greater
princip{al} stretch. For an isotropic material, Rivlin \cite{rivlin2004restrictions} supposed that the unit cube considered by Baker and Ericksen is
further subjected to a superposed infinitesimal simple shear with direction of shear parallel to an edge of the deformed cube and plane parallel to one of
its faces. Rivlin \cite{rivlin2004restrictions} proved that the BE-inequalities are
necessary and sufficient conditions
for the incremental shear modulus to be positive. The order relation for Cauchy stresses requested by the cube problem considered by Baker and Ericksen
then follows.

Let $\widehat{W}:{\rm GL}^+(3)\rightarrow\mathbb{R}$ be a function that can be written as a function of the singular values of $U$ via
$\widehat{W}(U)=g(\lambda_1,\lambda_2,\lambda_3)$.
Then the BE-inequalities express the requirement that
  \cite{Marsden83,BakerEri54,Criscione2005,Diani005,fosdick2006generalized}:
 \begin{align}
 ( \sigma_i-\sigma_j)\,(\lambda_i-\lambda_j) \geq 0,
 \end{align}
  where $\sigma_i =\dd\frac{1}{\lambda_1\lambda_2\lambda_3}\dd\lambda_i\frac{\partial g}{\partial
  \lambda_i}=\dd\frac{1}{\lambda_j\lambda_k}\dd\frac{\partial g}{\partial \lambda_i}, \ \ i\neq j\neq k \neq i$, are the principal Cauchy stresses.
  Usually, in the literature, the BE-inequalities mean that the above inequalities are strict. In this paper, we prefer to denote these strict inequalities
  as BE$^+$-inequalities. The BE-inequalities are equivalent (see \cite[page 17]{Marsden83}) {to}
  \begin{align}
  \frac{\lambda_i\frac{\partial g}{\partial \lambda_i}-\lambda_j\frac{\partial g}{\partial \lambda_j}}{\lambda_i-\lambda_j}\geq 0{,\qquad \text{ for all
  }\lambda_i,\lambda_j\in\R^+, \quad \lambda_i\neq\lambda_j}.
   \end{align}
We may also view the BE-inequalities as {C}auchy \textbf{t}rue-\textbf{s}tress-\textbf{o}rder-\textbf{c}ondition (TS-OC).
\subsection{Relation of Baker-Ericksen inequalities to other constitutive requirements}\label{BEal}
Marzano has  shown \cite{marzano1983interpretation}  that the BE-inequalities  are necessary and sufficient conditions for a simple extension (a
deformation in which two, but not three, principal stretches are equal) to correspond to simple tension.
In \cite{mihai2011positive} it was proved that for a homogeneous isotropic hyperelastic material subject to a
pure  Cauchy shear stress (a state of pure shear: $\tr(\sigma)=0$ \cite{Norris}) of the form
\begin{align}
\sigma=\left(
         \begin{array}{ccc}
           0 & s & 0 \\
           s & 0 & 0 \\
           0 & 0 & 0 \\
         \end{array}
       \right)
\end{align}
the BE-inequalities are satisfied if and only if the corresponding left Cauchy-Green strain tensor $B=F\,F^T$ has the representation\footnote{Since
$\tr(\sigma)=0$ one might rather expect the stronger statement $B=\left(
         \begin{array}{ccc}
           B_{11} &  B_{12} & 0 \\
            B_{12} &  B_{22} & 0 \\
           0 & 0 &  1 \\
         \end{array}
       \right)$, i.e. $B_{33}=1$, as well as $\det B=1$. However, this is not true  in general for isotropic energies, e.g. it is not satisfied for
       Neo-Hooke or Mooney-Rivlin type materials. }
\begin{align}\label{mihaistrain}
B=\left(
         \begin{array}{ccc}
           B_{11} &  B_{12} & 0 \\
            B_{12} &  B_{22} & 0 \\
           0 & 0 &  B_{33} \\
         \end{array}
       \right){,}
\end{align}
where
$
 B_{11}+ B_{12}> B_{33}> B_{11}- B_{12}>0.
$

   In general, there are  many different possible ways of expressing the physically plausible requirement (the Drucker postulate) that  stresses should increase
   with increasing stretch or strain\footnote{In the literature, all these concepts are defined using  strict inequalities for $\lambda_i\neq \lambda_j$,
   $i\neq j$. In this paper these common cases will be denoted by TE${}^+$, OF${}^+$, E${}^+$ and PC${}^+$, respectively.}
   \cite{truesdell1956ungeloste,truesdell1963static}:
   \begin{itemize}
   \item TE-inequalities (\textbf{t}ension-\textbf{e}xtension-inequalities):  each principal Cauchy stress is a strictly increasing function of the
       corresponding principal stretch, i.e. $
      \dd \frac{\partial \sigma_i}{\partial \lambda_i}> 0, \  i=1,2,3 .
       $
      Since $\sigma_i =\dd\frac{1}{\lambda_j\lambda_k}\frac{\partial g}{\partial \lambda_i}$, $i\neq j\neq k\neq i$, we obtain $\dd\frac{\partial
      \sigma_i}{\partial \lambda_i}=\dd\frac{1}{\lambda_j\lambda_k}\frac{\partial^2 g}{\partial \lambda_i^2}$ and the TE-inequalities are equivalent to
      the \textbf{s}eparate \textbf{c}onvexity (SC) of the function $g$, namely $\frac{\partial^2 g}{\partial \lambda_i^2}>0$, $i=1,2,3$.
   \item OF-inequalities\footnote{These inequalities appear also, but not as  strict inequalities, in the following theorem:
   \begin{theorem} {\rm  \cite[Theorem 6.5]{ball1984differentiability}}
 Let  ${W}:{\rm GL}^+(n)\rightarrow\mathbb{R}$ be an objective-isotropic function of class $C^2$ with the representation in terms of the singular values
 of $U$ via $W(F)=\widehat{W}(U)=g(\lambda_1,\lambda_2,...,\lambda_n)$. Let $F\in {\rm GL}^+(n)$ be given with the $n$-tuple of singular values
 $\lambda_1,\lambda_2,...,\lambda_n$. Then $D^2 W(F)[H,H]\geq 0$ for every $H\in \R^{n\times n}$  if and only if the following conditions hold
 simultaneously:
 \begin{itemize}
 \item[i)] $\dd\sum_{i,j=1}^n \frac{\partial^2 g}{\partial \lambda_i \partial \lambda_j}a_ia_j\geq 0\quad$  for every $(a_1,a_2,...,a_n)\in\R^n$
     (convexity of $g$);
 \item[ii)] for every $i\neq j,\quad $
 $\dd
 \underbrace{\dd\frac{\frac{\partial g}{\partial \lambda_i}-\frac{\partial g}{\partial \lambda_j}}{\lambda_i-\lambda_j}\geq 0}_{\rm
 ``OF-inequality"}\quad \text{if} \quad \lambda_i\neq \lambda_j, \qquad\frac{\partial^2 g}{\partial \lambda_i^2}-\frac{\partial^2 g}{\partial
 \lambda_i\partial \lambda_j}\geq 0 \quad \text{if} \quad \lambda_i=\lambda_j.
 $
 \item[iii)] $\dd\frac{\partial g}{\partial \lambda_i}+\frac{\partial g}{\partial \lambda_j}\geq 0$ for every $i\neq j$.
 \end{itemize}
\end{theorem}

Hence, if the function $F\mapsto W(F)$
is convex in $F\in{\rm GL}^+(n)$, then the OF-inequalities hold true. However, the convexity of $F\mapsto W(F)$ is physically not acceptable, since it
precludes buckling.}
(\textbf{o}rdered-\textbf{f}orce-inequalities \cite{Truesdell65}): the greater principal force $T_i=\sigma_i\lambda_j\lambda_k= \dd\frac{\partial
g}{\partial \lambda_i},\ i\neq j\neq k \neq i$, which is associated with the greater principal stretch is such that
       \begin{align}
      (T_i-T_j)\,(\lambda_i-\lambda_j)\geq 0.
       \end{align}
       The physical meaning of the OF-inequalities is the following: if a block of isotropic material is supposed to be in equilibrium subject to pairs
       of equal and oppositely directed normal forces acting upon its faces, then the greater stretch will occur in the direction of the greater force.
       The  OF-inequalities are therefore similar to the  BE-inequalities, only that principal forces instead of principal stresses are concerned
       \cite[page 158]{Truesdell65}. We observe that
 \begin{align}\label{bIFS}
R\,T_{\rm Biot}(U)\,R^T=\tau(V)\, V^{-1}=J\, \sigma(V) \, V^{-1},
\end{align}
where $F=R\, U=V\, R, \quad F\, R^T=V, \quad V^T=R\, F^T=V$. Hence, the Biot stress tensor  $T_{\rm Biot}$ is symmetric \cite[page 144]{Truesdell65} and
represents ``the principal forces acting in the reference system". Therefore, we may also denote the OF-inequalities as  \textbf{B}iot
\textbf{s}tress{-}\textbf{o}rder{-}\textbf{c}ondition (BS-OC). Using nearly incompressible materials like rubber, Ball \cite{Ball78} has described a
reasonable situation for which the BE-inequalities are valid, while the OF-condition is violated. This fact was  previously proved by Sidoroff \cite[page
380]{sidoroff1974restrictions}.
\item Convexity type conditions. The convexity of $W$ as function of $F$ means
 $
D_F^2W(F).(H,H)>0\,,$ {for all} $ H\neq 0,
$
and implies the monotonicity of the Piola-Kirchhoff stress
\begin{align}
\langle S_1(F+H)-S_1(H),H\rangle>0, \qquad \forall\, H\neq 0.
\end{align}
 This condition yields unqualified uniqueness of boundary value problems, it excludes therefore buckling and is unphysical \cite{hill1968constitutive}.
 For diagonal deformation gradients, the above convexity condition implies the monotonicity of the Biot stress tensor as a function of stretch
 (BSS-M$^+$).

\item GCN-inequality (\textbf{G}eneralized-\textbf{C}oleman-\textbf{N}oll-inequality) \cite[page 18]{Marsden83}: if  $\Lambda\neq\id$ is  a
    positive-definite symmetric matrix (a pure stretch), then the first Piola-Kirchhoff stress $S_1(F)$ satisfies
\begin{align}
\langle S_1(\Lambda F)-S_1(F),\Lambda F-F\rangle>0.
\end{align}
For homogeneous isotropic hyperelastic materials, the  GCN-inequality implies strict convexity of $g(\lambda_1,\lambda_2,\lambda_3)$ in all variables,
implying convexity in $U$ \cite{Ball78,Marsden83},   which is known to be unreasonable \cite{hill1968constitutive,hill1970constitutive}. Moreover, the
GCN-inequality implies the OF-condition which according to Sidoroff \cite[page 380]{sidoroff1974restrictions} is inadmissible for compressible materials.
In order to circumvent  the problems of the GCN-inequality, Sidoroff \cite{sidoroff1974restrictions} proposed the condition
\begin{align}
W(F)=\widehat{g}(\log \lambda_1,\log \lambda_2,\log \lambda_3), \qquad \text{where} \quad \widehat{g} \quad \text{should be strictly convex}.
\end{align}
This is nothing else than Hill's condition, i.e. our KSTS-M$^+$.

   \item E-TSS-inequalities (``\textbf{e}mpirical"-inequalities): using the general form of the Cauchy stress tensor for isotropic materials
   \begin{align}
   \sigma=\sigma(B)=\beta_0 \,\id+\beta_1 \, B+\beta_{-1} \, B^{-1},
   \end{align}
   where $\beta_0,\beta_1, \beta_{-1}$ are functions depending on the principal invariants of $B$, $I_1(B)=\tr(B),$ $I_2(B)=\tr({\rm Cof} B),$
   $I_3(B)=\det B$, the E-TSS-inequalities require
   \begin{align}\label{eineq}
 \text{ E-TSS:}\quad  \beta_0\leq 0,\quad\qquad \beta_1> 0, \quad\qquad \beta_{-1}\leq 0\,,
   \end{align}
while the strengthen{ed} E$^+$-TSS-inequalities require
 \begin{align}\label{eineq00}
  \text{ E$^+$-TSS:}\quad   \beta_0\leq 0,\quad\qquad \beta_1> 0, \quad\qquad \beta_{-1}< 0\,.
   \end{align}
Some experimental data seem to support the{se} inequalities in certain bounded deformation ranges. However, no theoretical motivation has been found for
the empirical inequalities \cite{beatty1996introduction}.  The connection of the E-TSS-inequalities {with} the Poynting effect is discussed in
\cite{angela2013numerical}. Batra \cite{batra1976deformation} (see also \cite{batra1975coincidence}) proved that the E-TSS-inequalities are sufficient
conditions for the simple extension to correspond to simple tension. Batra's result  has been improved later by Marzano
\cite{marzano1983interpretation}{, who} proved  that the BE-inequalities   are necessary and sufficient to have the equivalence between simple extension
and simple tension. The  BE-inequalities   are weaker than  the E-TSS-inequalities, because    for $\beta_0\leq 0$ the BE-inequalities imply
\cite{marzano1983interpretation} only that
\begin{align}
 \beta_1> 0.
   \end{align}
Assuming that $\beta_{-1}<0$, Johnson and Hoger \cite{johnson1993dependence} have shown that one may uniquely write
\begin{equation}
B=\psi_0\,\id+\psi_1\,\sigma+\psi_2\, \sigma^2,
\end{equation}
where $\psi_i=\psi_i(\beta_0(I_1(B),I_2(B),I_3(B)), \beta_1(I_1(B),I_2(B),I_3(B)), \beta_{-1}(I_1(B),I_2(B),I_3(B)), i=0,1,2$. This means that E$^+$-TSS
implies invertibility of the Cauchy stress-stretch relation  if $\beta_0,\beta_1,\beta_{-1}$ do not depend on $B$.

Nothing can be said about the validity of the third inequality from  \eqref{eineq}, beyond their logical relation to the BE and OF inequalities,
which may be abbreviated as follows \cite{truesdell1963static}:
\begin{center}
E-TSS $\Rightarrow$ BE \ and \  OF.
\end{center}
While the OF and BE inequalities are equivalent in the linearized theory,
in general \cite{truesdell1963static,Truesdell60}
\begin{center}
OF $\nRightarrow$ BE \ and \ BE $\nRightarrow$  OF.
\end{center}
Hence, the empirical inequalities imply the OF-{inequality}. Since the OF-condition is not a valid assumption in general {(see above)}, the
E-TSS-inequalities in general  cannot be  a valid assumption either. Rivlin \cite{rivlin2004restrictions} pointed out that the OF-conditions do not, in
general,
provide an appropriate restriction on the strain-energy function for an isotropic elastic material. Hence, OF is in general an independent statement and
Rivlin \cite{rivlin2004restrictions} proved  that it is unacceptable.
 \item E-BSS-inequalities: in view of the general form of the Biot stress tensor for isotropic materials
   \begin{align}
  T_{\rm Biot}=\beta_0 \,\id+\beta_1 \, U+\beta_{-1} \, U^{-1},
   \end{align}
   the E-BSS-inequalities require
   \begin{align}\label{eineq}
   \beta_0\leq 0,\quad\qquad \beta_1> 0, \quad\qquad \beta_{-1}\leq 0\,,
   \end{align}
while the E$^+$-BSS-inequalities require
 \begin{align}\label{eineq00}
   \beta_0\leq 0,\quad\qquad \beta_1> 0, \quad\qquad \beta_{-1}< 0\,.
   \end{align}
\item IFS (\textbf{i}nvertible-\textbf{f}orce-\textbf{s}tretch relation):  the invertibility of the map $(\lambda_1,\lambda_2,\lambda_3)\mapsto
    T_i(\lambda_1,\lambda_2,\lambda_3)$, where $T_i$ are the principal forces  \cite{truesdell1963static,Truesdell60,Krawietz75}. We remark that if {the}
    GCN-inequality holds, then $g(\lambda_1,\lambda_2,\lambda_2)$ is  strictly convex and IFS follows.   IFS expresses the invertibility of the Biot
    stress tensor $T_{\rm Biot}(U)$ \cite{Rivlin73Proc}, since{,} up to a superposed rotation{,} the Biot stress tensor  defines the principal forces
    acting in the reference system \cite[page 144]{Truesdell65} (see also \cite{Truesdell60}), see \eqref{bIFS}. Rivlin contested this condition  since for Neo-Hookean incompressible
    materials {different $(\lambda_1,\lambda_2,\lambda_3)$ may correspond to the same $T_{\rm Biot}$ stress tensor \cite{Rivlin73Proc}}.
\item  PC-inequality (\textbf{p}ressure-\textbf{c}ompression-inequality): the condition that the volume of a compressible
isotropic material should be decreased by  uniform pressure but increased by uniform tension
is expressed by requiring the hydrostatic tension $\sigma=\sigma_1=\sigma_2=\sigma_3$ to be a strictly increasing function of the stretch
$\lambda=\lambda_1=\lambda_2=\lambda_3$, i.e.
$
\,\frac{\partial \sigma}{\partial \lambda}\geq 0.
$
\item TSTS-M$^+$  (Jog and Patil's \textbf{t}rue-\textbf{s}tress-\textbf{t}rue-\textbf{s}train \textbf{m}onotonicity \cite{jog2013conditions}): the
    monotonicity of the Cauchy stress tensor as a function of $\log B$ or $\log V$ (see Remark \ref{remarkmon}), i.e.
  \begin{align}\label{Jogines1}
 \langle\sigma(\log B_1)-\sigma(\log B_2),\log B_1-\log B_2\rangle> 0, \qquad \forall\, B_1, B_2\in \PSym^+(3), \ B_1\neq B_2.
 \end{align}
  \item TSTS-I (\textbf{t}rue-\textbf{s}tress-\textbf{t}rue-\textbf{s}train-\textbf{i}nvertibility): the map $\log B\mapsto \sigma(\log B)$ is invertible
      (see Sections \ref{sectinvert0} and \ref{sectinvert}).
 \item TSS-M$^+$ (\textbf{t}rue-\textbf{s}tress-\textbf{s}tretch-\textbf{m}onotonicity): the monotonicity of the Cauchy stress tensor as a function of $
     B$ or $\log V$, i.e.
  \begin{align}
 \langle\sigma(B_1)-\sigma(B_2), B_1-B_2\rangle> 0, \qquad \forall\, B_1, B_2\in \PSym^+(3), \ B_1\neq B_2.
 \end{align}
 We remark that subtracting the two stretch tensors $B_1, B_2$ is in principle a problematic issue: $B_1, B_2$  do not belong to  a linear space.
\item   TSS-I (\textbf{t}rue-\textbf{s}tress-\textbf{s}tretch-\textbf{i}nvertibility): the map $B\mapsto \sigma(B)$ is invertible
    \cite{carroll1973controllable}. Since $\log: \PSym(n)\to\Sym(n)$ is invertible, TSTS-I and  TSS-I are clearly equivalent. However,  in order to be
    more precise we keep  both definitions. Truesdell and Moon relate  TSS-I with ``semi-invertibility" \cite{truesdell1975inequalities}. There, they
    also implicitly show that the E-TSS-inequalities are not in general sufficient for TSS-I. Johnson and Hoger \cite{johnson1993dependence} have shown
    that E$^+$-TSS-inequalities together with constant coefficients are sufficient for TSS-I (see also \cite{destrade2012simple}). Taking a compressible
    Neo-Hooke model in the form
\begin{align}
W_{NH}^{\rm iso}(F)=\frac{\mu}{2}\langle \frac{B}{\det B^{1/3}}-\id,\id\rangle+\kappa\, h(\det F),
\end{align}
which additively separates the isochoric and volumetric contributions it can be shown \cite{NGS} that $B\mapsto \sigma(B)$ is invertible.  Here
$h:\R_+\rightarrow\R$ must be a strictly convex function satisfying
$
\lim\limits_{J\rightarrow0}h^\prime(J)=-\infty$ and $\lim\limits_{J\rightarrow\infty}h^\prime(J)=\infty.
$
For instance, suitable  convex functions are $h:\R_+\rightarrow\R$,
$
h(t)=e^{\log^2 t}$ and $
h(t)=t^2-2\,{\log t}$. Therefore, TSS-I merits further investigation (see the discussion of IFS).

 \item KSTS-M$^+$ (Hill's \textbf{K}irchhoff-\textbf{s}tress-\textbf{t}rue-\textbf{s}train-\textbf{m}onotonicity \cite{hill1968constitutive}): the
     monotonicity of the {K}irchhoff stress tensor as a function of  $\log V$, i.e.
  \begin{align}\label{Jogines1}
 \langle\tau(\log V_1)-\tau(\log V_2),\log V_1-\log V_2\rangle> 0, \qquad \forall\, V_1, V_2\in \PSym^+(3), \ V_1\neq V_2.
 \end{align}
 In \cite{ogden1970compressible} Ogden has proved that the later called Odgen's energy does not satisfy the KSTS-M$^+$ inequality, but it may satisfy
 KSTS-M$^+$ under some restrictions on deformations confirmed by experiments.
  \item KSTS-I (\textbf{K}irchhoff \textbf{s}tress-\textbf{t}rue-\textbf{s}train-\textbf{i}nvertibility): the map $\log V\mapsto \tau(\log V)$ is
      invertible.
\item KSS-M$^+$ (\textbf{K}irchhoff \textbf{s}tress-\textbf{s}tretch \textbf{m}onotonicity): the monotonicity of the Kirchhoff stress tensor as a
    function of $V$, i.e.
  \begin{align}\label{tauJogines1}
 \langle\tau( V_1)-\tau(V_2),V_1-V_2\rangle> 0, \qquad \forall\, V_1, V_2\in \PSym^+(3), \ V_1\neq V_2.
 \end{align}
\item   KSS-I (\textbf{K}irchhoff \textbf{s}tress-\textbf{s}tretch-\textbf{i}nvertibility): the map $V\mapsto \tau(V)$ is invertible.
\item BSTS-M$^+$ (\textbf{B}iot \textbf{s}tress-\textbf{t}rue \textbf{s}train \textbf{m}onotonicity):  the monotonicity of the Biot stress tensor $T_{\rm
    Biot}(U)=R^T S_1(F)$ as a function of $\log B$, i.e.
  \begin{align}\label{Biot1}
 \langle T_{\rm Biot}(\log U_1)-T_{\rm Biot}(\log U_2),\log U_1-\log U_2\rangle> 0, \qquad \forall\, U_1, U_2\in \PSym^+(3), \ U_1\neq U_2.
 \end{align}

\item BSTS-I (\textbf{B}iot \textbf{s}tress-\textbf{t}rue \textbf{s}train-\textbf{i}nvertibility):  the map $\log U\mapsto T_{\rm Biot}(\log U)$ is
    invertible.
\item BSS-M$^+$ (\textbf{B}iot \textbf{s}tress-\textbf{s}tretch-\textbf{m}onotonicity):  the monotonicity of the Biot stress tensor \cite{Krawietz75}
    $T_{\rm Biot}(U)=R^T S_1(U)$ as a function of $U$, i.e.
  \begin{align}\label{Biotst1}
 \langle T_{\rm Biot}(U_1)-T_{\rm Biot}(U_2),U_1-U_2\rangle> 0, \qquad \forall\, U_1, U_2\in \PSym^+(3), \ U_1\neq U_2.
 \end{align}
 Krawietz \cite{Krawietz75} has shown that BSS-M$^+$ implies the generalized Colleman-Noll  (GCN) inequality. The GCN-inequality in turn is known to be
 not acceptable from physical grounds \cite{Ball77}. Therefore BSS-M$^+$  is not an admissible requirement in general. However, Ogden \cite[page
 361]{Ogden83} remarks that ``there is a good physical  reason for supposing that the inequality [\eqref{Biotst1}] holds for real elastic materials, at
 least for some bounded domain which encloses the stress free origin $U=\id$".
\item BSS-I (\textbf{B}iot \textbf{s}tress-\textbf{s}tretch-\textbf{i}nvertibility):  the map $U\mapsto T_{\rm Biot}(U)$ is invertible. In
    \cite{ogden1977inequalities}, Ogden suggested  that $T_{\rm Biot}$ should be invertible in the domain of elastic response. However, BSS-I  is in fact
    equivalent to Truesdell's notion IFS and to BSTS-I. This seems to have been overlooked in the literature
    \cite{Ogden83,ogden1977inequalities,rivlin2004restrictions}. In a forthcoming paper we will show that BSS-I excludes bifurcations in Rivlin's cube
    problem which is not necessarily a problematic feature.
   \end{itemize}

In the following $X$STS-M$^+$, $X$STS-I, $X$SS-M$^+$, $X$SS-I, E-$X$SS, E$^+$-$X$SS have the obvious meaning once the stress tensor $X$ is defined. It is
easy to see that BE and TE are necessary for rank-one convexity (see Theorem \ref{silhavy318}), i.e.
  \begin{center}
LH-ellipticity \quad  $\Rightarrow$ \quad BE \ and \ TE.
\end{center}

Moreover, because the constitutive inequalities are indifferent to superposed rotations, we have
\begin{center}
BSTS-M$^+$ \quad $\Rightarrow$  \quad BSTS-I \quad $\Leftrightarrow$  \quad BSS-I \quad $\Leftrightarrow$ \quad IFS\,.
\end{center}

In Figure \ref{diagrama}, we give a  diagram showing the relation between some {of the} introduced constitutive requirements.

\medskip

The KSTS-M$^+$ condition does not exclude loss of rank-one convexity (consider e.g. the quadratic Hencky energy) but it is also in principle not in
conflict with rank-one convexity.
\begin{figure}[h!]\begin{center}
\begin{minipage}[h!]{0.7\linewidth}
\hspace*{-1cm}\includegraphics[scale=0.7]{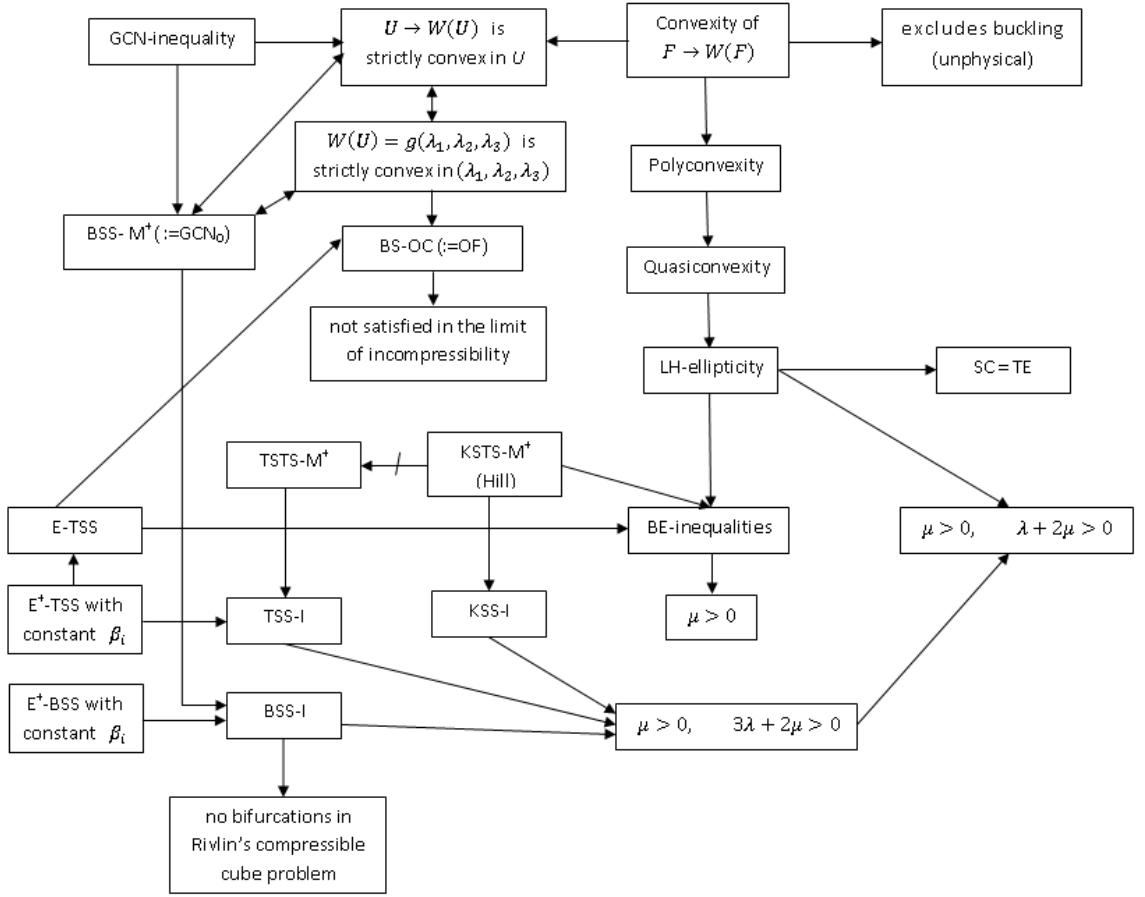}
\centering
\caption{\footnotesize{Relations between some constitutive requirements in idealized isotropic compressible nonlinear elasticity. Whether TSTS-M$^+$
implies KSTS-M$^+$ is not clear.}}
\label{diagrama}
\end{minipage}
\end{center}
\end{figure}
In order to prove this fact we consider a special Ciarlet-Geymonat energy (linear Poisson's ratio $\nu=0$)
\begin{align}
W_{\rm CG}^{\nu=0}(F)&=\frac{\mu}{2}\left[\|F\|^2-2\,\log (\det F)-3\right].
\end{align}
This uni-constant compressible Neo-Hooke energy was considered in \cite{lehmich2012convexity} and it has been speculated that it has some advantageous
properties. The energy $ W_{\rm CG}^{\nu=0}$ is LH-elliptic (it is even polyconvex). In the following we show that the energy $ W_{\rm CG}^{\nu=0}$
satisfies the KSTS-M$^+$ condition. First of all let us remark that
\begin{align}
W_{\rm CG}^{\nu=0}(F)&=\frac{\mu}{2}\left[\|e^{\log U}\|^2-2\,\tr (\log U)-3\right]=\frac{\mu}{2}\left[\|e^{S}\|^2-2\,\tr (S)-3\right],\notag \quad
\text{where} \quad S=\log U\in {\rm Sym}(3),
\end{align}

\pagebreak

and further
\begin{align}
W_{\rm CG}^{\nu=0}(F)&=g_1(\mu_1,\mu_2,\mu_3)=\frac{\mu}{2}\left[e^{2\,\mu_1}+e^{2\,\mu_2}+e^{2\,\mu_3}
-2\,(\mu_1+\mu_2+\mu_3)-3\right],
\end{align}
where $\mu_1, \mu_2, \mu_3$ are the eigenvalues of $S=\log U$. The function $g_1$ being  convex and nondecreasing in each variable $\mu_i$, using the
Davis-Lewis theorem \cite{Davis57,Lewis03,Lewis96,Lewis96b,Borwein2010} we have that $W_{\rm CG}^{\nu=0}$ is convex in $S=\log U$.  Thus, the energy
$W_{\rm CG}^{\nu=0}$ satisfies the KSTS-M$^+$ condition everywhere.  Moreover, the BSS-M$^+$ condition is also satisfied, since
\begin{align}
W_{\rm CG}^{\nu=0}(F)&=\frac{\mu}{2}\left[\|U\|^2-2\,\log (\det U)-3\right]
\end{align}
is convex\footnote{Similarly, as shown in \cite{lehmich2012convexity} the energy $C\mapsto\frac{\mu}{4}\left[\|C\|^2-2\,\log (\det C)-3\right]$ is convex
in $C$ and indeed polyconvex. The convexity in $C$ has been used by Fung \cite{fung1979inversion} to invert the second Piola-Kirchhoff stress tensor $S_2=2\, D_C[W(C)]$.} in $U$ \cite{lehmich2012convexity}. On the other hand, the Mooney-Rivlin variant of the energy $W_{\rm CG}^{\nu=0}(F)$,
\begin{align}
W_{\rm CGMR}(F)&=\alpha_1\,\|F\|^2+\alpha_2\,\|{\rm Cof}\,F\|^2-\,\log (\det F)+e^{(\log \det F)^2}-3\,\alpha_1-3\,\alpha_2-1
\end{align}
is not convex considered as  a function of  $\log U$. We give the following conjecture:

\begin{conjecture}
The energy $W_{_{\rm eH}}$ does not satisfy the E$^+$-TSS-inequalities.
\end{conjecture}
However, we will show in this paper that:
\begin{remark}
 The energy   $W_{_{\rm eH}}$  satisfies the TSS-I condition (see Section \ref{sectinvert0}).
\end{remark}
\subsection{Baker-Ericksen inequalities and Schur convexity}

The BE-inequalities related to the function $g$ can be reformulated in terms of Schur-convexity. The connection between Schur-convexity and the
Baker-Ericksen inequalities has been clearly pointed out by \v{S}ilhav\'{y} in \cite[page 310]{Silhavy97} and in full explicitness in \cite[pages
421,429]{silhavy2002monotonicity}.
For our purpose here and in order to  see the relation between Schur-convexity and BE-inequalities it is sufficient to know the following characterizations
of Schur-convex functions (further information on Schur-convexity can be found in \cite{MarshallOlkinArnold}):

 \begin{proposition}{\rm \cite[page 84]{MarshallOlkinArnold}} \label{S-caract}
Let $I$ be an open interval in $\R$ and let $\ell:I^n\to \R$ be continuously differentiable. Then $\ell$ is Schur convex if and only if  $\ell$ is
symmetric and
$
(x_i-x_j)\left(\frac{\partial \ell}{\partial x_i}-\frac{\partial \ell}{\partial x_j}\right)\geq 0
$
for all $i\neq j$.
\end{proposition}

 \begin{proposition}\label{SchurTh}{\rm \cite[page 97]{MarshallOlkinArnold}} Let $I$ be an open interval in $\R$ and let $\ell:I^n\to \R$. If the function
 $\ell$  is symmetric and convex in each pair of arguments, the
other arguments being fixed, then $\ell$ is Schur-convex.
 \end{proposition}

This notion relates to the BE-inequalities as follows:
\begin{proposition}{\rm  \cite[Remark 5.1]{Buliga}}
\label{prop:be-schur}
Schur-convexity of the function
\begin{align}\label{elldef}
\ell:\R_+^3\to\R,\qquad \ell(x,y,z)= g(e^x,e^y,e^z)
\end{align}
is equivalent to the fulfilment of the Baker-Ericksen inequalities in terms of the function $g$.
\end{proposition}

This characterization makes the following theorem quickly conceivable.
 \begin{theorem}\label{conlBE}
 Convex isotropic  functions of $\log U$ always satisfy  the BE-inequalities.
\end{theorem}
\begin{proof}
 Convex (isotropic) functions of $\log U$ lead to \begin{equation}\label{eq:darstellungvorbuliga} g(\lam_1,\lam_2,\lam_3)=\ell(\log \lam_1, \log
 \lam_2,\log \lam_3),\end{equation} where $\ell$ is a convex function.
 To see this, we apply Proposition \ref{prop:be-schur}: An energy function given by $g$ satisfies BE if and only if the function $\ell:\R_+^3\to\R$, \
 $\ell(x,y,z)= g(e^x,e^y,e^z)$
 is Schur-convex, hence it is sufficient to show that $\ell$ is  convex and symmetric.
 Convexity follows from
 \[
  g(e^x,e^y,e^z)=\ell(\log e^x, \log e^y, \log e^z)=\ell(x,y,z)
 \]
 and convexity of $\ell$, while the symmetry is obtained from the isotropy of $W$. From the  {Schur-}convexity of $\ell$ it follows that the functions $g$
 satisf{ies} the Baker-Ericksen-inequalities.
\end{proof}
\begin{remark} {\rm (Optimality of logarithmic strain and Baker-Ericksen inequalities)}
Theorem \ref{conlBE}  shows that {(Schur-)}convex dependence on the logarithmic strain tensor somehow is the ideal form for BE.
(Isotropic functions $W$ of $\log U$ satisfy the BE-inequalities if and only if $\ell$ from \eqref{elldef} is Schur-convex.)
\end{remark}
In the following remark we gather a few simple convexity properties, some of which can be derived with the results of this section:
\begin{remark}\label{remarkconlog}
\begin{itemize}
\item[]
\item[i)] $
 e^{\|\dev_n\log U\|^2}
$, $
  e^{\norm{\log U}^2}
 $,  $
 \|\dev_n\log U\|^2
$, $
  \norm{\log U}^2
  $
   are all convex functions of $\log U$, i.e. satisfy Hill's inequality (KSTS-M).
\item[ii)] $
 e^{\|\log U\|^2}
$
satisfies BE, because $\|\cdot\|^2$ is convex and hence so is $e^{\|\cdot\|^2}$.
\item[iii)] $
  e^{\norm{\dev_n\log U}^2}
 $
 satisfies the Baker-Ericksen inequalities  in any dimension because $\norm{\dev_n\cdot}^2$ is convex and $t\mapsto e^t$ is monotone {increasing} and
 convex.
 \item[iv)] $
 e^{\|\dev_n\log U\|^2}
$, $
  e^{\norm{\log U}^2}
 $ are  SC (\textbf{s}eparately \textbf{c}onvex) in $\lambda_i, i=1,2,3$ (direct calculations) but not convex in $(\lambda_1,\lambda_2,\lambda_3)$. Therefore, $
 e^{\|\dev_n\log U\|^2}
$, $
  e^{\norm{\log U}^2}
 $ is not convex in $U$ and the energy terms do not satisfy BSS-M$^+$.
 \item[v)]  $
 \|\dev_n\log U\|^2
$, $
  \norm{\log U}^2
 $ are not SC (\textbf{s}eparately \textbf{c}onvex) \cite{Silhavy97} in $\lambda_i, i=1,2,3$ (they do not satisfy the TE-inequalities) and therefore are
 not  rank-one convex \cite{Neff_Diss00,Bruhns01}.
\item[vi)] $W_{{\rm Becker}}^{\nu=0}(U)=2\, \mu\,\langle U, \log U-\id\rangle$ (the maximum entropy function) \cite{NBecker} does not satisfy the
    BE-inequalities but satisfies the TE-inequalities. The formulation of Becker \cite{Becker1893} is hyperelastic for Poisson's ratio $\nu=0$
    (exclusively), which is the case for the modelling of cork. Moreover, since $T_{\rm Biot}^{\rm Becker}(U)=D_U W_{{\rm Becker}}^{\nu=0}(U)=2\, \mu\,
    \log U$ and since $\log$ is  monotone, it follows $\langle T_{\rm Biot}^{\rm Becker}(U_1)-T_{\rm Biot}^{\rm Becker}(U_2),U_1-U_2\rangle>0$ which is
    BSS-M$^+$ for $\nu=0$. Hence, it is clear that IFS (BSS-I) hold. Moreover, $T_{\rm Biot}^{\rm Becker}$ satisfies BSS-I for arbitrary $-1<\nu\leq
    \frac{1}{2}$.
 \end{itemize}
\end{remark}

\section{The invertible true-stress-true-strain relation}\label{sectinvert0}\setcounter{equation}{0}

We consider the exponentiated Hencky energy
\begin{align}\label{hee}
W_{_{\rm eH}}(\log V):=\frac{\mu}{k}\,e^{k\,\|\dev_3\log\,V\|^2}+\frac{\kappa}{2\,\widehat{k}}\,e^{\widehat{k}\,(\tr(\log V))^2}{.}
\end{align} Here, we first show that the corresponding true-stress-true-strain relation
$$
\sigma_{_{\rm eH}}:{\rm Sym}(3)\to {\rm Sym}(3), \qquad \sigma_{_{\rm eH}} =\sigma_{_{\rm eH}}(\log V)
$$
is invertible for the exponentiated  energy $W_{_{\rm eH}}$. Then we prove that a pure planar Cauchy shear stress $\sigma$ produces a biaxial shear strain  for
general Hencky type energies. The invertibility of the  true-stress-true-strain relation, i.e. of the map $\log V \mapsto \sigma(\log V)$,  is denoted by
TSTS-I as introduced previously. In the older literature, the requirement of an invertible stress-strain relation is tacitly assumed  to  always hold
generally, even for nonlinear materials response \cite{richter1948isotrope}.

The  Kirchhoff stress tensor corresponding to \eqref{hee} is given \cite{Ogden83} by
\begin{align}
D_{\log V}W_{_{\rm eH}}(\log V)=\tau_{_{\rm eH}}=(\det F) \cdot\sigma_{_{\rm eH}}=e^{\log \det V}\cdot \sigma_{_{\rm eH}}=e^{\tr(\log  V)}\cdot \sigma_{_{\rm eH}},
\end{align}
 where $\sigma_{_{\rm eH}}$ is the Cauchy stress tensor. Hence, the Kirchhoff stress $\tau_{_{\rm eH}}$ has the expression
 \begin{align}\label{exptau}
 \tau_{_{\rm eH}}=2\,{\mu}\,e^{k\,\|\dev_3\log\,V\|^2}\cdot \dev_3\log\,V+{\kappa}\,e^{\widehat{k}\,[\tr(\log V)]^2}\,\tr(\log V)\cdot \id,
 \end{align}
 while the Cauchy stress tensor is
 \begin{align}\label{eqsigma}
 \sigma_{_{\rm eH}}&=e^{-\tr(\log  V)}\cdot \tau_{_{\rm eH}}=2\,{\mu}\,e^{k\,\|\dev_3\log\,V\|^2-\tr(\log  V)}\cdot \dev_3\log\,V+{\kappa}\,e^{\widehat{k}\,[\tr(\log
 V)]^2-\tr(\log  V)}\,\tr(\log V)\cdot \id.
 \end{align}
 Moreover, by orthogonal projection onto the Lie-algebra $\mathfrak{sl}(3)$ and $\R\cdot \id$, respectively, we find
 \begin{align}\label{devtrlog}
 \dev_3\sigma_{_{\rm eH}}&=2\,{\mu}\,e^{k\,\|\dev_3\log\,V\|^2-\tr(\log  V)}\cdot \dev_3\log\,V,\quad
 \tr(\sigma_{_{\rm eH}})=3\,{\kappa}\,e^{\widehat{k}\,[\tr(\log V)]^2-\tr(\log  V)}\,\tr(\log V).
 \end{align}
 Let us use the notation
$
 x:=\tr(\log V).
$ In this notation, from \eqref{devtrlog}, we have
 \begin{align}\label{eqtr}
 \frac{\tr(\sigma_{_{\rm eH}})}{3\kappa}=e^{\widehat{k}\,x^2-x}\,x.
 \end{align}
 The function $x\mapsto e^{\widehat{k}\,x^2-x}x$, $x\in \R${,} is strictly monotone {if $\widehat{k}>\frac18$}. Thus, {in this case,} equation
 \eqref{eqtr} has a unique solution $x=\tr(\log V)$ as a function of $\tr(\sigma_{_{\rm eH}})$.
 We substitute the solution $x$ of equation \eqref{eqtr}  in equation
 \eqref{devtrlog}$_1$, to obtain
 \begin{align}\label{devtrlog2}
\frac{e^x\cdot \dev_3\sigma_{_{\rm eH}}}{\mu}&=2\,e^{k\,\|\dev_3\log\,V\|^2}\cdot \dev_3\log\,V,
 \end{align}
 and further
  \begin{align}\label{devtrlog3}
{k}\,\frac{e^x\cdot \dev_3\sigma_{_{\rm eH}}}{{\mu}}&=D_{\dev_3\log V}\,e^{k\,\|\dev_3\log\,V\|^2}.
 \end{align}
 Using the substitution
$
 Y=\dev_3\log V
$
 we have
   \begin{align}\label{devtrlog4}
{k}\,\frac{e^x\cdot \dev_3\sigma_{_{\rm eH}}}{{\mu}}&=D_{Y}\,e^{k\,\|Y\|^2}.
 \end{align}
 Because $Y\mapsto e^{k\,\|Y\|^2}$, $Y\in {\rm Sym}(3)$ is uniformly convex {with respect to }$Y$, it follows that $D^2_{Y}\,e^{k\,\|Y\|^2}(H,H)>0$, {for
 all }$Y\in {\rm Sym}(3)$, {and for all } $H\in {\rm Sym}(3)$. Hence, the function $Y\mapsto D_{ Y}\,e^{\|Y\|^2}$ is a strictly monotone tensor function.
 Theref{ore, equ}ation \eqref{devtrlog4} has a unique solution $Y=\dev_3\log V$ as a function of $\dev_3\sigma_{_{\rm eH}}$ and $x=\tr(\log V)$. Hence, given
 the Cauchy stress $\sigma_{_{\rm eH}}$, we can always uniquely find $\tr(\log V)$ and $\dev_3\log V$, i.e. $\log V$, such that \eqref{eqsigma} is satisfied.
 Therefore TSTS-I is true in the three-dimensional case. simple changes of the computations show that  TSTS-I  is also true in the two-dimensional case.

Whether well known elastic strain energies like compressible Neo-Hooke, Mooney-Rivlin or Ogden  type energies \cite{Ciarlet88}   give rise to an overall
invertible Cauchy-stress-stretch relation $\sigma=\sigma(B)$ is not clear. This is connected to possible homogeneous  bifurcations, e.g. in  a hydrostatic
loading  problem  \cite{chen1996stability,kearsley1986asymmetric}.

Let us consider, in the following{,}  three particular cases for our energy $W_{_{eH}}$: pure Cauchy shear stress, uniaxial tension and simple shear.
\subsection{Pure  Cauchy shear stress}
 In this subsection  we consider the case of pure Cauchy shear stress, i.e.
 \begin{align}\label{pureshear}
 \sigma_{_{\rm eH}}=\left(
          \begin{array}{ccc}
            0 & s & 0 \\
            s & 0 & 0 \\
            0 & 0 & 0 \\
          \end{array}
        \right), \qquad 0=\tr( \sigma_{_{\rm eH}})=\tr( \tau_{_{\rm eH}}).
 \end{align}
We aim to find the corresponding form of the stretch tensor $V$. From \eqref{exptau},  by considering the trace on both sides, it follows that in the case
of pure shear stress, we must have
 $
 \tr(\log V)=0 \  \Leftrightarrow \   \det V=1.
 $
We need to remark that the conclusion that pure Cauchy shear stresses lead to an incompressible response is not verified e.g. for Neo-Hooke, Mooney-Rivlin or
Ogden-type materials.
 In our case,  however, it remains to solve
 \begin{align}\label{eq3}
 2\,\mu\, e^{k\,\|\dev_3\log V\|^2}\dev_3\log V&=\left(
          \begin{array}{ccc}
            0 & s & 0 \\
            s & 0 & 0 \\
            0 & 0 & 0 \\
          \end{array}
        \right)
          \quad  \Leftrightarrow \quad 2\,\mu\, e^{k\,\|\log V\|^2}\log V=\left(
          \begin{array}{ccc}
            0 & s & 0 \\
            s & 0 & 0 \\
            0 & 0 & 0 \\
          \end{array}
        \right).
 \end{align}
 Inspired  by Vall\'{e}e's result  in \cite{vallee1978lois}, a solution {of equ}ation \eqref{eq3} can be found in the form of pure biaxial stretch\footnote{This is
 suggested by the formula presented in \cite[page 736]{Bernstein2009}:
$
e^{\alpha \cdot \widehat{A}}=\left(
     \begin{array}{cc}
       \cosh \alpha & \sinh \alpha  \\
     \sinh \alpha & \cosh \alpha
     \end{array}
   \right)\ \ \ \text{for} \ \ \ \widehat{A}=\left(
     \begin{array}{cc}
       0 & \alpha  \\
      \alpha & 0
     \end{array}
   \right).
$}
 \begin{align}\label{purestrain}
 V=\left(
     \begin{array}{ccc}
       \cosh \frac{\gamma}{2} & \sinh \frac{\gamma}{2} & 0 \\
     \sinh \frac{\gamma}{2} & \cosh \frac{\gamma}{2} & 0 \\
       0 & 0 & 1 \\
     \end{array}
   \right).
 \end{align}
 Corresponding to this ansatz  for $V$, we have
  \begin{align}
 \log\, V=\left(
     \begin{array}{ccc}
       0 &  \frac{\gamma}{2} & 0 \\
      \frac{\gamma}{2} & 0 & 0 \\
       0 & 0 & 0 \\
     \end{array}
   \right), \qquad  \det\, V=1,
 \end{align}
 an{d equation} \eqref{eq3} becomes
 \begin{align}
\sigma_{_{\rm eH}}=\left(
          \begin{array}{ccc}
            0 & s & 0 \\
            s & 0 & 0 \\
            0 & 0 & 0 \\
          \end{array}
        \right)=2\,\mu\, e^{k\,\frac{\gamma^2}{2}} \left(
     \begin{array}{ccc}
       0 &  \frac{\gamma}{2} & 0 \\
      \frac{\gamma}{2} & 0 & 0 \\
       0 & 0 & 0 \\
     \end{array}
   \right).
 \end{align}
 For all $s\in \R$ we always have a solution $\gamma=\gamma(s)$ of the above equation, because $\gamma\mapsto  e^{k\,\frac{\gamma^2}{2}} \frac{\gamma}{2}$
 is monotone increasing. Thus,  we  recover completely the classical statement that in linear elasticity,  pure shear stresses \eqref{pureshear}  produces
 pure biaxial shear strains
 \eqref{purestrain}, i.e.
 \begin{align}\label{pureshearc}
 \sigma=2\,\mu\,\varepsilon=2\,\mu\!\!\!\!\!\!\!\!\!\!\!\!\underbrace{\left(
          \begin{array}{ccc}
            0 &  \frac{\gamma}{2} & 0 \\
            \frac{\gamma}{2} & 0 & 0 \\
            0 & 0 & 0 \\
          \end{array}
        \right)}_{\textrm{``pure infinitesimal shear stress"}}\quad \Leftrightarrow \quad  \varepsilon=\!\!\!\!\!\!\!\!\!\!\!\underbrace{\left(
     \begin{array}{ccc}
       0 &   \frac{\gamma}{2} & 0 \\
      \frac{\gamma}{2} & 0 & 0 \\
       0 & 0 & 0 \\
     \end{array}
   \right)}_{\textrm{``pure infinitesimal shear strain"}}\!\!\!\!\!\!\!\!\!\!\!\!,\qquad\quad  \underbrace{\tr(\varepsilon)=0}_{\textrm{``linearized
   incompressibility"}}\!\!\!\!\!\!\!\!\!\!\!\!\!\!\!\!\!\,,
 \end{align}
 where $\varepsilon={\rm sym}\,\nabla u$.   For the finite strain case, this equivalence seems  to be true only for Hencky type energies
 \cite{vallee1978lois}.

\subsection{Uniaxial Cauchy tension}\label{uniaxialsection}
 Next we consider the case of uniaxial tension
 \begin{align}\label{puretension}
 \sigma_{_{\rm eH}}=\left(
          \begin{array}{ccc}
            s & 0 & 0 \\
            0 & 0 & 0 \\
            0 & 0 & 0 \\
          \end{array}
        \right).
 \end{align}
From  \eqref{eqsigma},  by projection on the Lie-algebra{s} $\mathfrak{sl}(n)$ and $\mathbb{R}\cdot \id$, we have
 \begin{align}\label{eqt3}
 2\,\mu\, e^{k\,\|\dev_3\log V\|^2-\tr(\log  V)}\dev_3\log V&=\dev_3\sigma_{_{\rm eH}}=\left(
          \begin{array}{ccc}
            \frac{2}{3}s & 0 & 0 \\
            0 & -\frac{1}{3}s & 0 \\
            0 & 0 & -\frac{1}{3}s \\
          \end{array}
        \right).\\
 3\,{\kappa}\,e^{\widehat{k}\,[\tr(\log V)]^2-\tr(\log  V)}\,\tr(\log V)&=\tr(\sigma_{_{\rm eH}})=s.\notag
 \end{align}
 This means that a suitable ansatz for $ V$ is similar to that considered by Vall\'{e}e \cite{vallee1978lois}
 \begin{align}\label{purestrain}
 V=\left(
     \begin{array}{ccc}
       e^{a+\frac{1}{3}x} & 0 & 0 \\
     0 & e^{-\frac{1}{2}a+\frac{1}{3}x} & 0 \\
       0 & 0 & e^{-\frac{1}{2}a+\frac{1}{3}x} \\
     \end{array}
   \right)=e^{\frac{1}{3}x}\left(
     \begin{array}{ccc}
       e^{a} & 0 & 0 \\
     0 & e^{-\frac{1}{2}a} & 0 \\
       0 & 0 & e^{-\frac{1}{2}a} \\
     \end{array}
   \right).
 \end{align}
 It is easy to compute that, corresponding to this ansatz  for $V$, we have
  \begin{align}
  \label{eq:eq3_19}
 &\det\, V=e^x, \qquad \log\, V=\left(
     \begin{array}{ccc}
       a+\frac{1}{3}\,x & 0 & 0 \\
     0 & -\frac{1}{2}\,a+\frac{1}{3}\,x & 0 \\
       0 & 0 & -\frac{1}{2}\,a+\frac{1}{3}\,x \\
     \end{array}
   \right),  \\
  &\tr(\log\, V)=x,  \qquad  \dev_3\log\, V=\left(
     \begin{array}{ccc}
       a & 0 & 0 \\
     0 & -\frac{1}{2}\,a & 0 \\
       0 & 0 & -\frac{1}{2}\,a \\
     \end{array}
   \right), \qquad  \|\dev_3\log V\|^2=\frac{3}{2}\, a^2\notag
 \end{align}
 a{nd equat}ion \eqref{eqt3} becomes
 \begin{align}
3\, \mu\, e^{k\,\frac{3}{2}\, a^2-x}\, a&=s,\qquad\qquad
 3\,{\kappa}\,e^{\widehat{k}\,x^2-x}\,x=s.
 \end{align}
In terms of  Poisson's ratio $\nu\,{\in(-1,\frac12)}$ and Young's modulus $E{>0}$, we have
 \begin{align}\label{det0}
 \, e^{k\,\frac{3}{2}\, a^2-x}\,\frac{3}{2}\, a&=\,\frac{1+\nu}{E}\,s,\qquad\qquad
 \,e^{\widehat{k}\,x^2-x}\,x=\frac{1-2\,\nu}{E}\,s.
 \end{align}
 For all $s\in \R$ we always have a solution $x=x(s)$ of the second equation and the function $s\mapsto x(s)$ is monotone strictly increasing if $\kappa>0$
 and  ${\rm sgn}[x(s)]={\rm sgn}[s]$. Having $x(s)$ from \eqref{det0}$_2$,  we then find the unique solution $a(s)$ of \eqref{det0}$_1$. Moreover,   for
 $\mu>0$  the function $s\mapsto a(s)$ is also  monotone strictly increasing and ${\rm sgn}[a(s)]={\rm sgn}[s]$.

 Therefore,   the ansatz
 \begin{align}
 \log\, V=\left(
     \begin{array}{ccc}
       a+\frac{1}{3}\,x & 0 & 0 \\
     0 & -\frac{1}{2}\,a+\frac{1}{3}\,x & 0 \\
       0 & 0 & -\frac{1}{2}\,a+\frac{1}{3}\,x \\
     \end{array}
   \right)
 \end{align}
 corresponds to
 \begin{align}
\sigma_{_{\rm eH}}=\left(
          \begin{array}{ccc}
            s & 0 & 0 \\
            0 & 0 & 0 \\
            0 & 0 & 0 \\
          \end{array}
        \right)=\frac{3}{2}\,e^{k\,\frac{3}{2}\, a^2-x}\,\frac{E}{1+\nu} \left(
     \begin{array}{ccc}
       a & 0 & 0 \\
      0 & 0 & 0 \\
       0 & 0 & 0 \\
     \end{array}
   \right).
 \end{align}

 \medskip

  In the limit case $\nu=\frac{1}{2}$ (linear incompressibility),  we observe that  \eqref{det0}$_2$ implies $x=0$. Therefore
  \begin{align}
 \log\, V\Big|_{\nu=\frac{1}{2}}=\left(
     \begin{array}{ccc}
       \gamma & 0 & 0 \\
     0 & -\frac{1}{2}\, \gamma  & 0 \\
       0 & 0 & -\frac{1}{2}\, \gamma  \\
     \end{array}
   \right),\qquad \underbrace{\det V\Big|_{\nu=\frac{1}{2}}=1}_{\nu=\frac{1}{2}: \text{ exact incompressibility}}\!\!\!\!\!\!\!\!\!\!, \qquad
   e^{k\,\frac{3}{2}\,  \gamma ^2}\,  \gamma =\frac{1}{E}\,s
 \end{align}
 and this corresponds to
 \begin{align}
\sigma_{_{\rm eH}}\Big|_{\nu=\frac{1}{2}}=\left(
          \begin{array}{ccc}
            s & 0 & 0 \\
            0 & 0 & 0 \\
            0 & 0 & 0 \\
          \end{array}
        \right)=E \,e^{k\,\frac{3}{2}\,  \gamma ^2}\left(
     \begin{array}{ccc}
        \gamma  & 0 & 0 \\
      0 & 0 & 0 \\
       0 & 0 & 0 \\
     \end{array}
   \right).
 \end{align}
On the other hand, $\sigma_{_{\rm eH}}=0$ ($s=0$) is equivalent with $\log V=0$ ($x=0, a=0$).

 \medskip

In the  case $\nu=0$,  the nonlinear system  \eqref{det0} becomes
\begin{align}\label{n0xa}
\, e^{k\,\frac{3}{2}\, a^2-x}\,\frac{3}{2}\, a&=\,\frac{1}{E}\,s,\qquad\qquad
 \,e^{\widehat{k}\,x^2-x}\,x=\frac{1}{E}\,s,
\end{align}
which implies
$
e^{k\,\frac{3}{2}\, a^2}\, a=e^{\widehat{k}\,x^2}\,\frac{2}{3}\,x.
$ Using the substitution $x=\frac{3}{2}\, y$, we have $
e^{k\,\frac{3}{2}\, a^2}\, a=e^{\widehat{k}{\frac94}\,y^2}\,y.
$ We choose the entry parameters $k, \widehat{k}$ such that  $3\,\widehat{k}=2\,k$ and we further deduce that $x=\frac{3}{2}\, a$. Thus, with the
substitution $\gamma =\frac{3}{2}\, a$, we deduce
\begin{align}
 \log\, V\Big|_{\nu=0}=\,\left(
     \begin{array}{ccc}
     \gamma  & 0 & 0 \\
     0 & 0 & 0 \\
       0 & 0 & 0 \\
     \end{array}
   \right), \qquad\qquad e^{k\,\frac{2}{3}\, \gamma ^2-\gamma }\,\gamma &=\,\frac{1}{E}\,s,
 \end{align}
 which corresponds to
 \begin{align}\label{pureshearch}
\sigma_{_{\rm eH}}\Big|_{\nu=0}=\left(
          \begin{array}{ccc}
            s & 0 & 0 \\
            0 & 0 & 0 \\
            0 & 0 & 0 \\
          \end{array}
        \right)=E \,e^{k\,\frac{2}{3}\, \gamma ^2-\gamma }\,\left(
     \begin{array}{ccc}
       \gamma  & 0 & 0 \\
      0 & 0 & 0 \\
       0 & 0 & 0 \\
     \end{array}
   \right).
 \end{align}
 Moreover, if there is no lateral contraction in uniaxial tension in the case $\nu=0$, then from \eqref{purestrain} we deduce that we must  have
 $x=\frac{3}{2}\, a$. On the other hand, for $\nu=0$, if $x=\frac{3}{2}\, a$, then using \eqref{n0xa} we obtain that $3\,\widehat{k}=2\,k$ must hold
 necessarily.

  Thus,  we have shown that uniaxial tension produces extension/contraction, as in linear elasticity, since for linear elasticity, using the inverted law $
  \varepsilon=\dd\frac{1+\nu}{E}\,\sigma
-\dd\frac{\nu}{E}\,\tr(\sigma)\cdot \id$, we have
  \begin{align}\label{pureshearc}
 \sigma=2\,\mu\,\varepsilon+\lambda\, \tr(\varepsilon)\cdot \id=E\underbrace{\left(
          \begin{array}{ccc}
           \gamma & 0 & 0 \\
           0 & 0 & 0 \\
            0 & 0 & 0 \\
          \end{array}
        \right)}_{\textrm{``uniaxial tension"}}\quad \Leftrightarrow \quad  \varepsilon=\!\!\underbrace{\left(
     \begin{array}{ccc}
       \gamma &   0 & 0 \\
      0 & -\nu \, \gamma & 0 \\
       0 & 0 & -\nu\, \gamma \\
     \end{array}
   \right)}_{\textrm{``extension/lateral  contraction"}},
 \end{align}
 where $\varepsilon={\rm sym}\,\nabla u$. In the limit case $\nu=\frac{1}{2}$,  we have $\tr (\varepsilon)=0$, while for $\nu=0$  there is no lateral
 contraction in uniaxial tension as in \eqref{pureshearch}. In linear elasticity, the Poisson's ratio is defined by $\nu=-\frac{\varepsilon_{22}}{\varepsilon_{11}}$ \cite{poisson1811traite}, where the transverse strain $\varepsilon_{22}$ and the longitudinal strain $\varepsilon_{11} $ are computed in uniaxial extension.
  \begin{remark}{\rm ($W_{_{\rm eH}}$ with no lateral contraction for $\nu=0$)}
The above formula \eqref{pureshearch} is true  if and only if the  distortional stiffening parameter $k$ and the volumetric strain stiffening  parameter
$\widehat{k}$ are such that  $3\,\widehat{k}=2\,k$. In this case $\nu=0$ implies  no lateral contraction for the exponentiated Hencky energy (3 parameter
energy: $\nu,\,E,\,k$)
\begin{align}
W_{_{\rm eH}}^{\sharp}(\log V)&:=\frac{\mu}{k}\,e^{k\,\|\dev_3\log\,V\|^2}+\frac{3\,\kappa}{4\,{k}}\,e^{\frac{2}{3}\,k\,({\rm tr}(\log
V))^2}=\frac{1}{2\,k}\left\{\frac{E}{1+\nu}\,e^{k\,\|\dev_3\log\,V\|^2}+\frac{E}{2(1-2\,\nu)}\,e^{\frac{2}{3}\,k\,({\rm tr}(\log V))^2}\right\}.\notag
\end{align}

\end{remark}

\subsection{On the nonlinear Poisson's ratio}

 We define the nonlinear Poisson's ratio as negative ratio of the lateral contraction and axial extension measured in the logarithmic strain, i.e., according to \eqref{eq:eq3_19}
\begin{align}\label{hatnuax}
\widehat{\nu}(s)=-\frac{(\log V)_{22}}{(\log V)_{11}}=\frac{\frac{1}{2}\, a-\frac{1}{3}\, x}{ a+\frac{1}{3}\, x}.
\end{align}
The nonlinear Poisson's ratio \cite{poisson1811traite}    is a purely kinematical quantity which can be measured in the simple tension test. In \cite[page 75]{fu2001nonlinear}  {it} is defined as $\widehat{\nu}(s)=-\frac{\lambda_2-1}{\lambda_1-1}$.  The (linear) Poisson's ratio\footnote{In terms of the Young's modulus and the shear modulus $\nu$ is given by $\nu=\frac{E}{2\,\mu}-1$, while in terms of the Young's modulus and the bulk modulus $\kappa$ it is given by $\nu=\frac{1}{2}-\frac{E}{6\,\kappa}$.} $\nu=-\frac{\varepsilon_{22}}{\varepsilon_{11}}$ \cite{poisson1811traite} for many materials is {positive and not} strain sensitive until nonelastic effects intervene \cite{smith1999interpretation,kakavas2000prediction}. In view of our definition,  we have
\begin{align}\label{anux}
a\, \left(\frac{1}{2}-\widehat{\nu}\right)=\frac{x}{3}\,(1+\widehat{\nu}).
\end{align}
Since ${\rm sgn}[a(s)]={\rm sgn}[s]={\rm sgn}[x(s)]$ we deduce that $\widehat{\nu}\in (-\frac{1}{2},1)$, which is in concordance with the definition from linear elasticity. From \eqref{anux} we have $
a=\frac{2}{3}\,\frac{1+\widehat{\nu}}{1-2\,\widehat{\nu}}\, x.
$ Moreover, the system \eqref{det0} becomes
\begin{align}\label{det0nn1}
 \, e^{^{k\,\frac{2}{3}\,\left(\frac{1+\widehat{\nu}}{1-2\,\widehat{\nu}}\right)^2 \, x^2-x}}\, x&=\,\frac{1-2\,\widehat{\nu}}{1+\widehat{\nu}}\, \frac{1+\nu}{E}\,s,\qquad
 \,e^{\widehat{k}\,x^2-x}\,x=\frac{1-2\,\nu}{E}\,s.
 \end{align}
 This system is also equivalent to
 \begin{align}\label{det0nn2}
\left[k\,\frac{2}{3}\,\left(\frac{1+\widehat{\nu}}{1-2\,\widehat{\nu}}\right)^2-\widehat{k}\right] \, x^2&=\log\left(\,\frac{1-2\,\widehat{\nu}}{1+\widehat{\nu}}\, \frac{1+\nu}{1-2\,\nu}\right), \qquad
 \,e^{\widehat{k}\,x^2-x}\,x=\frac{1-2\,\nu}{E}\,s.
 \end{align}

 In the following we consider the case of the three parameter energy $W_{_{\rm eH}}^{\sharp}$, i.e. the case $k\,\frac{2}{3}=\widehat{k}$. In this case we obtain the system
 \begin{align}\label{det0nn2}
\frac{\widehat{\nu}(2-\widehat{\nu})}{(1-2\,\widehat{\nu})^2} \, x^2&=\frac{1}{2\,k}\,\log\left(\,\frac{1-2\,\widehat{\nu}}{1+\widehat{\nu}}\, \frac{1+\nu}{1-2\,\nu}\right), \qquad
 \,e^{\widehat{k}\,x^2-x}\,x=\frac{1-2\,\nu}{E}\,s.
 \end{align}
 From the above equations we deduce that
$$
 \widehat{\nu}>0\ \Leftrightarrow \  \frac{1+\nu}{1-2\,\nu}>\frac{1+\widehat{\nu}}{1-2\,\widehat{\nu}} \ \Leftrightarrow \  \nu>\widehat{\nu}.
$$
 Hence, $ \widehat{\nu}>0$ implies ${\nu}>0$. On the other hand, if we assume that there is ${\nu}>0$ such that $\widehat{\nu}<0$, then we obtain
  \begin{align}\label{contradictie}
  \frac{1+\nu}{1-2\,\nu}<\frac{1+\widehat{\nu}}{1-2\,\widehat{\nu}}.
 \end{align}
 {But $\nu>0$ implies} $\frac{1+\nu}{1-2\,\nu}>1$, while $\widehat{\nu}<0$ implies $\frac{1+\widehat{\nu}}{1-2\,\widehat{\nu}}<1$. This is in clear  contradiction with \eqref{contradictie}.
 Therefore
$
 \widehat{\nu}>0\quad \Leftrightarrow \quad   \nu>0.
$
 If  $\widehat{\nu}=0$,  then  from \eqref{det0nn2} it results that we have to have $\nu=0$ and $x$ is determined only by $e^{\widehat{k}\,x^2-x}\,x=\frac{s}{E}$ (see the discussion from Subsection \ref{uniaxialsection} about the particular case $\nu=0$).
 \\
 If  $\widehat{\nu}\neq0$, then, since $\widehat{\nu}\in (-\frac{1}{2},1)$, we deduce {that $\widehat{\nu}$} is given as solution of the following equation { if $s>0$}:
 \begin{align}\label{hatnunus}
 \sqrt{\frac{\log\left(\,\frac{1-2\,\widehat{\nu}}{1+\widehat{\nu}}\, \frac{1+\nu}{1-2\,\nu}\right)}{\widehat{k}\,\left(\frac{1+\widehat{\nu}}{1-2\,\widehat{\nu}}\right)^2
 -\widehat{k}}}\,\, e^{{\frac{\log\left(\,\frac{1-2\,\widehat{\nu}}{1+\widehat{\nu}}\, \frac{1+\nu}{1-2\,\nu}\right)}{\left(\frac{1+\widehat{\nu}}{1-2\,\widehat{\nu}}\right)^2
 -1}-\sqrt{\frac{\log\left(\,\frac{1-2\,\widehat{\nu}}{1+\widehat{\nu}}\, \frac{1+\nu}{1-2\,\nu}\right)}{\widehat{k}\,\left(\frac{1+\widehat{\nu}}{1-2\,\widehat{\nu}}\right)^2-\widehat{k}}
 }}}\,=(1-2\,\nu)\frac{s}{E},
 \end{align}
 while for $s<0$, $\widehat{\nu}$ is  solution of the equation:
 \begin{align}\label{hatnunus}
 \sqrt{\frac{\log\left(\,\frac{1-2\,\widehat{\nu}}{1+\widehat{\nu}}\, \frac{1+\nu}{1-2\,\nu}\right)}{\widehat{k}\,\left(\frac{1+\widehat{\nu}}{1-2\,\widehat{\nu}}\right)^2
 -\widehat{k}}}\,\,e^{{\frac{\log\left(\,\frac{1-2\,\widehat{\nu}}{1+\widehat{\nu}}\, \frac{1+\nu}{1-2\,\nu}\right)}{\left(\frac{1+\widehat{\nu}}{1-2\,\widehat{\nu}}\right)^2
 -1}+\sqrt{\frac{\log\left(\,\frac{1-2\,\widehat{\nu}}{1+\widehat{\nu}}\, \frac{1+\nu}{1-2\,\nu}\right)}{\widehat{k}\,\left(\frac{1+\widehat{\nu}}{1-2\,\widehat{\nu}}\right)^2-\widehat{k}}
 }}}\,=-(1-2\,\nu)\frac{s}{E},
 \end{align}
 with $x$ given by the independent equation $\,e^{\widehat{k}\,x^2-x}\,x=(1-2\,\nu)\frac{s}{E}.$ In Figure \ref{hatnuposnu} and \ref{hatnuneg} we give the representation of the nonlinear Poisson's ratio $\widehat{\nu}$ as function of $\frac{s}{E}$, corresponding to different values of the  (linear) Poisson's ratio. We also represent (see Figure  \ref{nuhatk}) the influence of the parameter $\widehat{k}$ on  the nonlinear Poisson's ratio $\widehat{\nu}$.

 We notice three particular cases. If $\nu=-1$, then {it follows from \eqref{det0}$_1$ that} $a=0$ and further from \eqref{hatnuax} that $\widehat{\nu}=-1$. If $\nu=\frac{1}{2}$, then {\eqref{det0}$_2$ leads to }$x=0$, while  \eqref{hatnuax} implies $\widehat{\nu}=\frac{1}{2}$. Moreover, if $\nu=0$, then {\eqref{det0} shows} $x=\frac{3}{2}\, a$. Therefore,  from \eqref{hatnuax}  we obtain $\widehat{\nu}=0$.

\begin{figure}[h!]
\begin{center}
\begin{minipage}[h]{0.9\linewidth}
\centering
\includegraphics[scale=0.8]{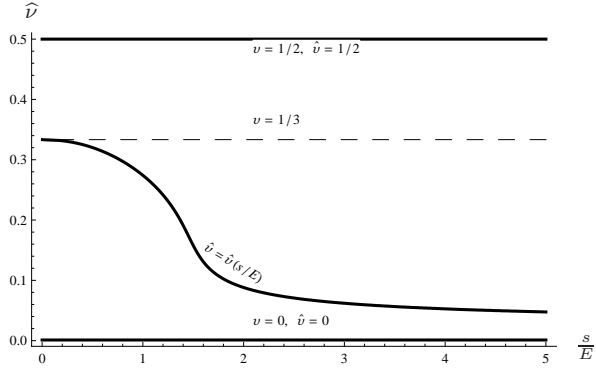}
\centering
\put(6,7){\footnotesize $\frac{s}{E}$}
\put(-202,133){\footnotesize $\widehat{\nu}$}
\caption{\footnotesize{ The nonlinear Poisson's ratio $\widehat{\nu}$ for $\widehat{k}=\frac{1}{6}$ and for the following values of the (linear) Poisson's ratio: $\nu=0$, $\nu=\frac{1}{3}$ and $\nu=\frac{1}{2}$. For $\nu=0$ and $\nu=\frac{1}{2}$ the nonlinear Poisson's ratio is equal to the (linear) Poisson's ratio, while for $\nu\in(0,\frac{1}{2})$ the nonlinear Poisson ratio  $\widehat{\nu}\left(\frac{s}{E}\right)=-\frac{(\log V)_{22}}{(\log V)_{11}}$ approximates the (linear) Poisson's ratio only in a small neighborhood of $\frac{s}{E}=0$. The graphic of the map  $\frac{s}{E}\mapsto \widehat{\nu}(\frac{s}{E})$ is tangent to the line $\widehat{\nu}(0)=\nu$, decreases  and  it is smaller than $\nu$ for non-infinitesimal values of the load parameter $s$. Moreover, the nonlinear Poisson's ratio $\widehat{\nu}$ remains positive whenever  $\nu=\widehat{\nu}(0)$ is positive  and $\widehat{\nu}\in\left(-1,\frac{1}{2}\right)$. }}
\label{hatnuposnu}
\end{minipage}
\end{center}
\end{figure}
\begin{figure}[h!]
\begin{center}
\begin{minipage}[h]{0.9\linewidth}
\centering
\includegraphics[scale=0.8]{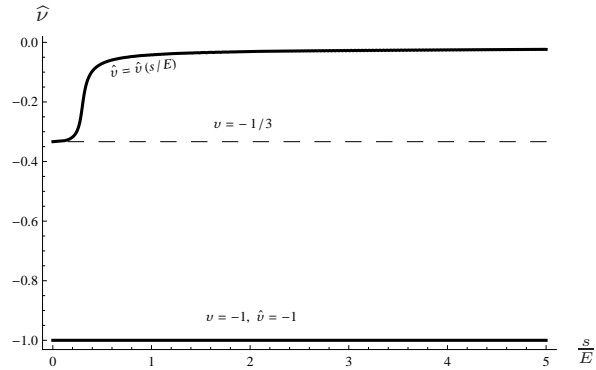}
\centering
\put(6,7){\footnotesize $\frac{s}{E}$}
\put(-198,133){\footnotesize $\widehat{\nu}$}
\caption{\footnotesize{ The variation of the nonlinear Poisson's ratio  $\widehat{\nu}$ for  $\widehat{k}=\frac{1}{6}$ and negative (linear) Poisson ratio (e.g. auxetic materials). For $\nu=-1$  the nonlinear Poisson's ratio is equal to the (linear) Poisson's ratio, while for $\nu\in(-1,0)$ the nonlinear Poisson's ratio approximates the (linear) Poisson's ratio only in a small neighborhood of $\frac{s}{E}=0$. For negative (linear) Poisson's ratio the map  $\frac{s}{E}\mapsto \widehat{\nu}(\frac{s}{E})$ is tangent to the line $\widehat{\nu}(0)=\nu$, increases  and  it is bigger than $\nu$ for non-infinitesimal values of the load parameter $s$. Moreover, the nonlinear Poisson's ratio $\widehat{\nu}$ remains negative whenever $\nu=\widehat{\nu}(0)$ is negative.  }}
\label{hatnuneg}
\end{minipage}
\end{center}
\end{figure}
\begin{figure}[h!]\begin{center}
\begin{minipage}[h]{0.9\linewidth}
\centering
\includegraphics[scale=0.8]{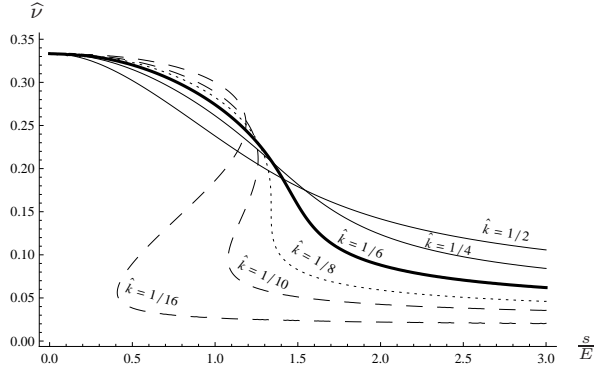}
\centering
\put(6,7){\footnotesize $\frac{s}{E}$}
\put(-200,133){\footnotesize $\widehat{\nu}$}
\caption{\footnotesize{ Influence of the parameter $\widehat{k}$ on the nonlinear Poisson's ratio $
\widehat{\nu}$. For $\widehat{k}<\frac{1}{8}$ the map $\frac{s}{E}\mapsto \widehat{\nu}(\frac{s}{E})$ is not well-defined. For $k\geq\frac{1}{8}$ the map $\frac{s}{E}\mapsto \widehat{\nu}(\frac{s}{E})$ is bijective.}}
\label{nuhatk}
\end{minipage}
\end{center}\end{figure}%

\subsection{Cauchy stress in simple shear for $W_{_{\rm H}}$ and $W_{_{\rm eH}}$}\label{simpleshearcauchy}

Consider a simple glide deformation of the form
\begin{align}
F=\left(
\begin{array}{ccc}
1&\gamma& 0\\
0&1&0\\
0&0&1
\end{array}\right)
\end{align}
with $\gamma>0$. Then the polar decomposition of $F=R\cdot U=V\cdot R$ into the right Biot stretch tensor $U=\sqrt{F^T F}$ of the deformation and the orthogonal polar factor $R$ is given by
\begin{align}
U=\frac{1}{\sqrt{\gamma^2+4}}\left(
\begin{array}{ccc}
2&\gamma& 0\\
\gamma&\gamma^2+2&0\\
0&0&\sqrt{\gamma^2+4}
\end{array}\right), \qquad
R=\frac{1}{\sqrt{\gamma^2+4}}\left(
\begin{array}{ccc}
2&\gamma& 0\\
-\gamma&2&0\\
0&0&\sqrt{\gamma^2+4}
\end{array}\right).
\end{align}
Further, $U$ can be orthogonally diagonalized to
\begin{align}
U=Q\cdot \left(
\begin{array}{ccc}
1&0& 0\\
0&\frac{1}{2}(\sqrt{\gamma^2+4}+\gamma)&0\\
0&0&\frac{1}{2}(\sqrt{\gamma^2+4}-\gamma)
\end{array}\right)\cdot Q^T=Q\cdot \left(
\begin{array}{ccc}
1&0& 0\\
0&\lambda_1&0\\
0&0&\frac{1}{\lambda_1}
\end{array}\right)\cdot Q^T,
\end{align}
where
\begin{align}
Q= \left(
\begin{array}{ccc}
2& -2&0\\
\sqrt{\gamma^2+4}+\gamma&\sqrt{\gamma^2+4}-\gamma&0\\
0&0&1
\end{array}\right)
\end{align}
and $\lambda_1=\frac{1}{2}(\sqrt{\gamma^2+4}+\gamma)$ denotes the first eigenvalue of $U$. Hence, the principal logarithm of $U$ is
\begin{align}
\log U=Q\cdot \left(
\begin{array}{ccc}
1&0& 0\\
0&\log\lambda_1&0\\
0&0&-\log \lambda_1
\end{array}\right)\cdot Q^T=\frac{1}{\sqrt{\gamma^2+4}}\cdot \left(
\begin{array}{ccc}
-\gamma \,\log\lambda_1&2\, \log\lambda_1& 0\\
2\, \log\lambda_1&\gamma\,\log\lambda_1&0\\
0&0&0
\end{array}\right),
\end{align}
while the principal logarithm of $V$ is given by
\begin{align}
\log V&=R\cdot \log U\cdot R^{-1}=\frac{1}{\sqrt{\gamma^2+4}}\,R\cdot \left(
\begin{array}{ccc}
-\gamma \,\log\lambda_1&2\, \log\lambda_1& 0\\
2\, \log\lambda_1&\gamma\,\log\lambda_1&0\\
0&0&0
\end{array}\right)\cdot R^{-1}\\&=
\frac{1}{({\gamma^2+4})\sqrt{\gamma^2+4}}\left(
\begin{array}{ccc}
2&\gamma& 0\\
-\gamma&2&0\\
0&0&\sqrt{\gamma^2+4}
\end{array}\right)\cdot \left(
\begin{array}{ccc}
-\gamma \,\log\lambda_1&2\, \log\lambda_1& 0\\
2\, \log\lambda_1&\gamma\,\log\lambda_1&0\\
0&0&0
\end{array}\right)\cdot \left(
\begin{array}{ccc}
2&-\gamma& 0\\
\gamma&2&0\\
0&0&\sqrt{\gamma^2+4}
\end{array}\right)\notag\\&=
\frac{1}{\sqrt{\gamma^2+4}}\left(
\begin{array}{ccc}
0&\log \lambda_1& 0\\
\log \lambda_1&0&0\\
0&0&0
\end{array}\right)\cdot \left(
\begin{array}{ccc}
2&-\gamma& 0\\
\gamma&2&0\\
0&0&\sqrt{\gamma^2+4}
\end{array}\right)=
\frac{\log \lambda_1}{\sqrt{\gamma^2+4}}\left(
\begin{array}{ccc}
\gamma\,&2& 0\\
2&-\gamma&0\\
0&0&0
\end{array}\right).\notag
\end{align}

\noindent\textbf{Cauchy stress in simple shear for $W_{_{\rm H}}$}

\bigskip

The Kirchhoff tensor $\tau_{_{\rm H}}$ corresponding to the Hencky energy $W_{_{\rm H}}$ is given by
 \begin{align}\label{eqsigmatauH}
 \tau_{_{\rm H}}(\log\,V)&=2\,{\mu}\, \dev_3\log\,V+{\kappa}\,\tr(\log V)\cdot \id.
 \end{align}
Hence, in the case of simple shear, we have
\begin{align}
\tau_{_{\rm H}}=2\,\mu\, \frac{\log \lambda_1}{\sqrt{\gamma^2+4}}\left(
\begin{array}{ccc}
\gamma\,&2& 0\\
2&-\gamma&0\\
0&0&0
\end{array}\right).
\end{align}
Moreover, since  $\det F=1$ and
$
\sigma=\frac{1}{\det F}\,\tau,
$ we obtain
\begin{align}
\sigma_{_{\rm H}}=2\,\mu\, \frac{\log \lambda_1}{\sqrt{\gamma^2+4}}\left(
\begin{array}{ccc}
\gamma\,&2& 0\\
2&-\gamma&0\\
0&0&0
\end{array}\right)=2\,\mu\, \frac{\log \left[\frac{1}{2}(\sqrt{\gamma^2+4}+\gamma)\right]}{\sqrt{\gamma^2+4}}\left(
\begin{array}{ccc}
\gamma\,&2& 0\\
2&-\gamma&0\\
0&0&0
\end{array}\right).
\end{align}
In particular, the simple shear stress $[{\sigma_{_{\rm H}}}]_{_{12}}$
corresponding to the amount of shear
 is given by
 \begin{align}
 [{\sigma_{_{\rm H}}}]_{_{12}}=4\,\mu\, \frac{\log \left[\frac{1}{2}(\sqrt{\gamma^2+4}+\gamma)\right]}{\sqrt{\gamma^2+4}}=
 2\,\frac{E}{1+\nu}\, \frac{\log \left[\frac{1}{2}(\sqrt{\gamma^2+4}+\gamma)\right]}{\sqrt{\gamma^2+4}}.
 \end{align}
The quadratic Hencky energy looses ellipticity in simple shear, see Subsection \ref{henckynotelliptic}.
\medskip

 \bigskip

\noindent\textbf{Cauchy stress in simple shear for $W_{_{\rm eH}}$}

\bigskip

In view of \eqref{exptau}, the Kirchhoff tensor $\tau_{_{\rm eH}}$ is given by
 \begin{align}\label{eqsigmatau1nou}
 \tau_{_{\rm eH}}(\log\,V)&=2\,{\mu}\,e^{k\,\|\dev_3\log\,V\|^2}\cdot \dev_3\log\,V+{\kappa}\,e^{\hat{k}\,[\tr(\log V)]^2}\,\tr(\log V)\cdot \id.
 \end{align}
 Since for simple shear    $\det F=1$ and  $\tr(\log V)=0$, we deduce
   \begin{align}\label{eqsigmatau1nou}
 \sigma_{_{\rm eH}}(\log\,V)&=2\,{\mu}\,e^{2\,k\,\log^2\lambda_1}\cdot \, \frac{\log \lambda_1}{\sqrt{\gamma^2+4}}\left(
\begin{array}{ccc}
\gamma\,&2& 0\\
2&-\gamma&0\\
0&0&0
\end{array}\right)\\
&=2\,{\mu}\,e^{2\,k\,\log^2 \left[\frac{1}{2}(\sqrt{\gamma^2+4}+\gamma)\right]}\cdot \, \frac{\log \left[\frac{1}{2}(\sqrt{\gamma^2+4}+\gamma)\right]}{\sqrt{\gamma^2+4}}\left(
\begin{array}{ccc}
\gamma\,&2& 0\\
2&-\gamma&0\\
0&0&0
\end{array}\right).
 \end{align}
 For the exponentiated energy $W_{_{\rm eH}}$  the simple shear stress ${\sigma_{_{\rm eH}}}_{_{12}}$
corresponding to the amount of shear $\gamma$
 is given by
 \begin{align}
 [{\sigma_{_{\rm eH}}}]_{_{12}}&=
 2\,\frac{E}{1+\nu}\,e^{2\,k\,\log^2 \left[\frac{1}{2}(\sqrt{\gamma^2+4}+\gamma)\right]}\cdot \, \frac{\log \left[\frac{1}{2}(\sqrt{\gamma^2+4}+\gamma)\right]}{\sqrt{\gamma^2+4}}.
  \end{align}

 The response of some rubbers is (more or less) linear under simple shear loading conditions (this is the raison d'\^{e}tre of the Mooney-Rivlin model \cite{mooney1940theory}, where $[{\sigma_{_{\rm MR}}}]_{_{12}}=2(C_1+C_2)\,\gamma=\frac{E}{2(1+\nu)}\, \gamma$). Let us therefore  compare  (Figure \ref{shearHEL}) the simple shear stress $\sigma_{12}$
corresponding to the amount of shear for the energies $W_{_{\rm eH}}, W_{_{\rm H}}$, for the Mooney-Rivlin  energy and for Neo-Hooke energy.

\begin{figure}[h!]\begin{center}
\begin{minipage}[h]{0.7\linewidth}
\centering
\includegraphics[scale=0.8]{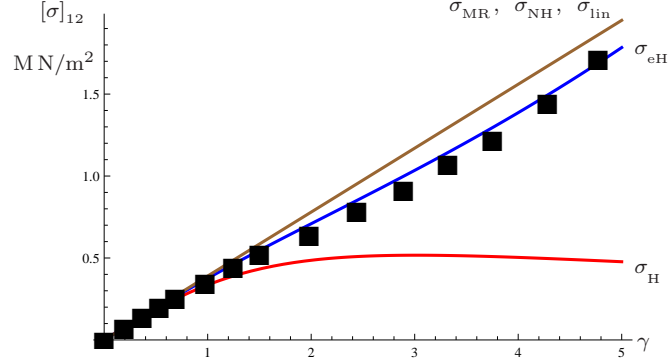}
\put(1,5){\footnotesize $\gamma$}
\put(-225,130){\footnotesize $[\sigma]_{_{12}}$}
\put(-235,110){\footnotesize ${\rm M}\, {\rm N}/{\rm m}^2$}
\put(-70,132){\footnotesize $\sigma_{_{\rm MR}},\ \, \sigma_{_{\rm NH}},\ \, \sigma_{_{\rm lin}}$}
\put(0,116){\footnotesize $\sigma_{_{_{\rm eH}}}$}
\put(0,33){\footnotesize $\sigma_{_{_{\rm H}}}$}
\centering
\caption{\footnotesize{The  shear stress $\sigma_{12}$
corresponding to the amount of shear $\gamma$ for the energies $W_{_{\rm eH}}, W_{_{\rm H}}$, the Neo-Hooke energy $W_{_{\rm NH}}$, the Mooney-Rivlin energy $W_{_{\rm MR}}$ and the infinitesimal case corresponding to rubber: $\mu=\frac{E}{2(1+\nu)}=0.39\, {\rm M}\, {\rm N}/{\rm m}^2$ (according to Treloar's data \cite{treloar1973elasticity}). For the exponentiated energy $W_{_{\rm eH}}$ we have chosen $\frac{3}{16}<k=0.243$. The squares ($\blacksquare$) represent the experimental data for the simple shear deformation of vulcanized rubber, measured in 1944
by L.R.G. Treloar \cite{treloar1944stress} and in 1975 by L.R.G. Treloar and D.F. Jones \cite{jones1975properties}  (see also \cite{treloar1973elasticity,treloar1975physics})  and provided by courtesy of
R. Ogden, in the form of tab-separated ASCII-files (see \cite{web}).}}
\label{shearHEL}
\end{minipage}
\end{center}
\end{figure}

Later in this paper we will implicitly show that $W_{_{\rm eH}}$ remains rank-one convex in simple shear.
 Rubber becomes harder to deform at large strains, probably because of limited chain extendability. Many rubber materials are normally subjected to fairly small deformation, rarely exceeding 25\%, in tension/compression or 75\% in simple shear.

\bigskip

\noindent\textbf{Cauchy stress in simple shear in the infinitesimal case}

\bigskip

It is well known that in the infinitesimal case the Cauchy stress tensor is given by
\begin{align}
\sigma_{_{\rm lin}}=2\, \mu\, \dev_3\varepsilon +\kappa\, \tr(\varepsilon)\cdot \id,
\end{align}
where $\varepsilon={\rm sym} \nabla u$ is the linearized strain tensor of the deformation $\varphi(x)=x+u(x)$ with the displacement $u:\Omega\subset \R^3\rightarrow\R^3$. In the infinitesimal case,  simple shear corresponds to the pure shear strain
\begin{align}
\varepsilon=
\left(
\begin{array}{ccc}
0&\frac{\gamma}{2}& 0\\
\frac{\gamma}{2}&0&0\\
0&0&0
\end{array}\right).
\end{align}
The Cauchy stress tensor in simple shear is given by
\begin{align}
\sigma_{_{\rm lin}}=2\, \mu\,
\left(
\begin{array}{ccc}
0&\frac{\gamma}{2}& 0\\
\frac{\gamma}{2}&0&0\\
0&0&0
\end{array}\right) \qquad \Rightarrow \qquad [\sigma_{_{\rm lin}}]_{_{12}}=\mu \,\gamma=\frac{E}{2(1+\nu)} \, \gamma.
\end{align}

\subsection{Response of rubber under large pressure. Equation of state.}\label{secteos}

Rubber, if considered as a linear, isotropic solid {very nearly satisfies} $\nu=0.5$ (i.e. for small loads, rubber responds practically incompressible). However, rubber under large pressure allows for an appreciable volume change \cite{bell1973}. This can be seen by experimentally determined equations of states (EOS), relating the mean stress (the pressure) $\frac{1}{3}\tr(\sigma)$ to the relative volume change $\det F$. For the exponentiated Hencky energy this relation  is given by
\begin{align}
\frac{1}{3}\,\tr(\sigma_{_{\rm eH}})=\frac{\rm d}{\rm d\,t}\left[\frac{\kappa}{2\,\widehat{k}}\, e^{\widehat{k}\, (\log t)^2}\right]\bigg|_{
\,t=\det F}=\left({\kappa}\, e^{\widehat{k}\, (\log \det F)^2}\frac{\log \det F}{\det F}\right),
\end{align}
  while for the quadratic Hencky energy we have
\begin{align}
\frac{1}{3}\,\tr(\sigma_{_{\rm H}})=\frac{\rm d}{\rm d\,t}\left[\frac{\kappa}{2}\, \, (\log t)^2\right]\bigg|_{
\,t=\det F}=\left({\kappa}\,\frac{\log \det F}{\det F}\right).
\end{align}

We have found that the analytical expression of the pressure $\frac{1}{3}\,\tr(\sigma)$ is in concordance with the  classical Bridgman's compression data for natural rubber as reported in  \cite[page 497, Fig. 4.47]{bell1973} with $\kappa=2.5\cdot 10^9\,{\rm Pa}=2.5\cdot 10^9\,{\rm GPa}$ (see Figures \ref{eos0},\ref{eos1}).  Tabor \cite{tabor1994bulk} showed that the bulk modulus of rubber is of the order $1$ GPa and found the value of the bulk modulus $\kappa$ to be about $2$ GPa. Recently, Zimmermann and Stommel \cite{zimmermann2013mechanical} found determined experimentally  that $\kappa$ is of the order  $\kappa=2.5$ GPa, which can be found in the literature as well (see e.g. \cite{horgan2009volumetric}).
\begin{figure}[h!]\begin{center}
\begin{minipage}[h!]{0.8\linewidth}
\centering
\includegraphics[scale=0.8]{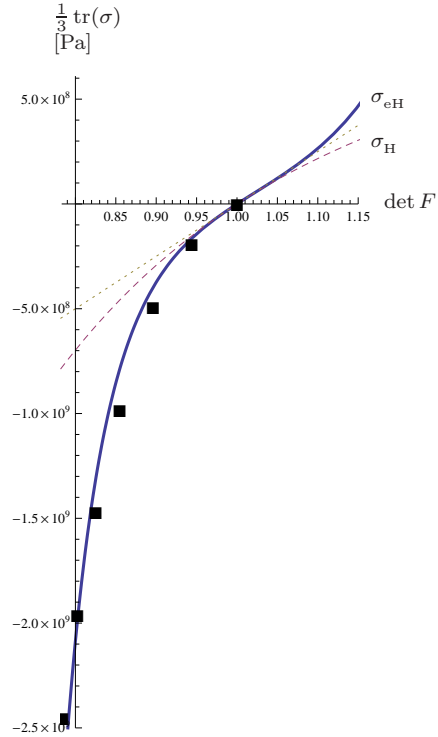}
\centering
\put(-120,270){\footnotesize $\frac{1}{3}\,\tr(\sigma)$}
\put(-120,260){\footnotesize [Pa]}
\put(0,240){\footnotesize $\sigma_{_{\rm eH}}$}
\put(0,223){\footnotesize $\sigma_{_{\rm H}}$}
\put(5,200){\footnotesize $\det F$}
\caption{\footnotesize{ The pressure $\frac{1}{3}\,\tr(\sigma)$ as function of $\det F$:  Bridgman's
experimental data \cite{bell1973} in compression ($\blacksquare$), analytical form corresponding to the exponentiated  volumetric Hencky energy $\frac{\kappa}{2\,\widehat{k}}\, e^{\widehat{k}\, (\log \det F)^2}$ with $\widehat{k}=22$ (continuous line) and the analytical form corresponding to the volumetric quadratic Hencky energy $\frac{\kappa}{2}\,  (\log \det F)^2$ (dashed line). The dotted line represents the  tangent  to these curves. The value of the bulk modulus  of rubber  is chosen  to be  $\kappa=2.5\,{\rm GPa}$.  We point out that  in the experimental data reported in  \cite[page 487]{bell1973}  the magnitude of the  pressure $\frac{1}{3}\,\tr(\sigma)$ is expressed in  $\frac{\rm kg}{{\rm cm}^2}$  (see Figure 4.47 from \cite[page 487]{bell1973}) which means in fact $9.81\cdot 10^4\frac{\rm kg}{{\rm m\, s} ^2}$=$9.81\cdot 10^4 {\rm Pa}$.}}
\label{eos0}
\end{minipage}
\end{center}
\end{figure}%

 \begin{figure}[h!]\begin{center}
\begin{minipage}[h]{0.8\linewidth}
\centering
\includegraphics[scale=1]{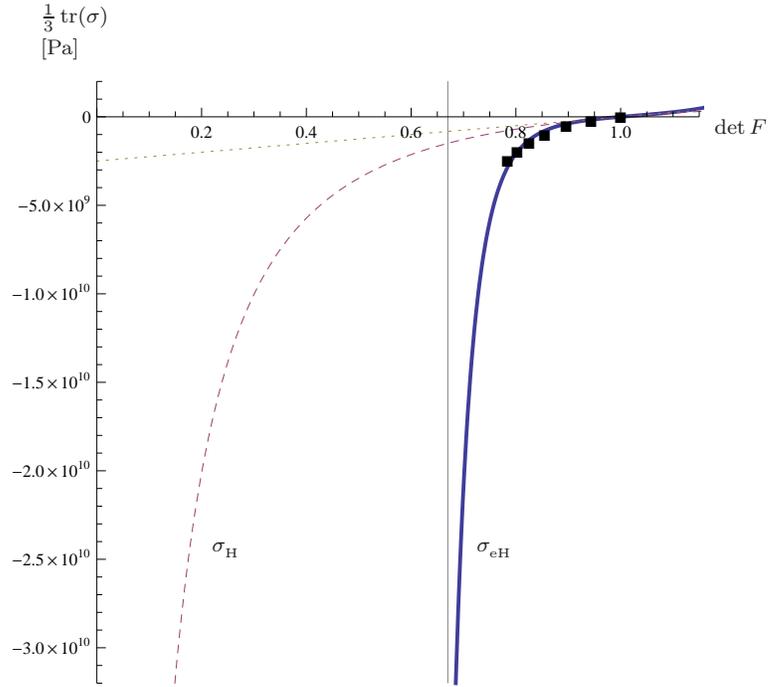}
\centering
\put(-250,250){\footnotesize $\frac{1}{3}\,\tr(\sigma)$}
\put(-250,238){\footnotesize [Pa]}
\put(5,208){\footnotesize $\det F$}
\put(-85,50){\footnotesize $\sigma_{_{\rm eH}}$}
\put(-185,50){\footnotesize $\sigma_{_{\rm H}}$}
\caption{\footnotesize{ The pressure $\frac{1}{3}\,\tr(\sigma)$ as function of $\det F$. It seems that there is a singularity at $\det F=0.67$, meaning that this model would preclude compression beyond $\det F=0.67$. However, the pressure does not have a singularity in  $(0,\infty)$. Moreover the  mean stress (the pressure) corresponding to $W_{_{\rm eH}}$  is invertible as function of the volume change. The considered values and the legend are the same as in Figure \ref{eos0}. }}
\label{eos1}
\end{minipage}
\end{center}
\end{figure}%

From Figure \ref{eos1}, certain threshold values seem unreachable by compression, unless in infinite amount of energy is spent. However, this impression is misleading: stresses and energy remain finite for any stretch $V\in {\rm PSym}(3)$. Therefore, in our model the assumption of limited chain extensibility is not needed.

\begin{figure}[h!]\begin{center}
\begin{minipage}[h]{0.8\linewidth}
\includegraphics[scale=1]{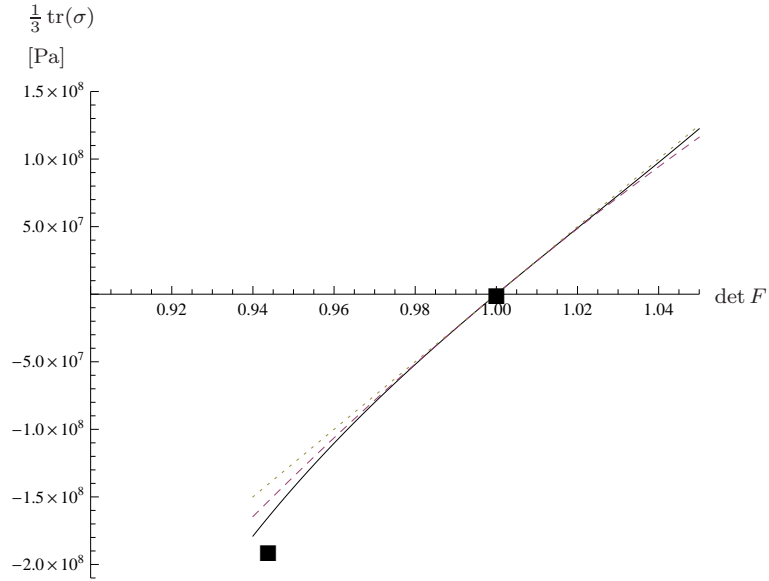}
\centering
\put(-255,210){\footnotesize $\frac{1}{3}\,\tr(\sigma)$}
\put(-255,195){\footnotesize [Pa]}
\put(5,104){\footnotesize $\det F$}
\caption{\footnotesize{ The pressure $\frac{1}{3}\,\tr(\sigma)$ as function of $\det F$ in the neighbourhood of identity $F=\id$. The considered values and the legend are the same as in Figure \ref{eos0}. }}
\label{eos2}
\end{minipage}
\end{center}
\end{figure}%

In Figure \ref{eos2} we represent the pressure $\frac{1}{3}\,\tr(\sigma)$ as function of $\det F$ in the neighbourhood of the identity $F=\id$ and we compare the analytical results obtained for the  exponentiated  volumetric Hencky energy $\frac{\kappa}{2\,\widehat{k}}\, e^{\widehat{k}\, (\log \det F)^2}$ with $\widehat{k}=22$ with the analytical form corresponding to the volumetric quadratic Hencky energy $\frac{\kappa}{2}(\log \det F)^2$, as well with Bell's
experimental data \cite{bell1973}. In the neighbourhood of the identity $F=\id$, the quadratic Hencky energy gives also good results, while in large compression  the values obtained using the quadratic Hencky energy are not in agreement with the experimental data (see Figures \ref{eos0},\ref{eos1}). Moreover, the EOS relation corresponding to the quadratic Hencky   is not invertible for $\det F>e$ and it is not able to predict the response for $\frac{1}{3}\, \tr(\sigma)>\frac{1}{e}$ \cite{vallee1978lois}.

   \section{Monotonicity of the Cauchy stress tensor $\sigma$ as a function of $\log B$}\label{sectinvert}\setcounter{equation}{0}

Motivated by \cite{jog2013conditions} we consider  a novel constitutive requirement for an isotropic material, namely that the Cauchy stress tensor
$\sigma$ should be a monotone tensor function of $\log B$, $B=V^2$, i.e.
  \begin{align}\label{Jogine}
 \hspace{-1cm}{\rm TSTS-M}:\qquad \langle\sigma(\log B_1)-\sigma(\log B_2),\log B_1-\log B_2\rangle\geq 0, \qquad \forall\, B_1, B_2\in
 \PSym^+(3).\qquad\quad
 \end{align}
 We will  refer to \eqref{Jogine} as true-stress-true{-}strain monotonicity (TSTS-M), and to
  \begin{align}\label{Jogines}
 {\rm TSTS-M}^+:\qquad \langle\sigma(\log B_1)-\sigma(\log B_2),\log B_1-\log B_2\rangle> 0, \qquad \forall\, B_1, B_2\in \PSym^+(3), \ B_1\neq B_2,
 \end{align}
 as strict true-stress-true{-}strain monotonicity (TSTS-M$^+$). In {a} forthcoming paper \cite{NeffMartin14} (see also
\cite{Norris}), it is shown that
   \begin{align}\label{NeffM}
 \qquad \langle\log B_1-\log B_2, B_1- B_2\rangle> 0, \qquad \forall\, B_1, B_2\in \PSym^+(3), \ B_1\neq B_2.
 \end{align}

Recall  that Hill's monotonicity condition (KSTS-M) is monotonicity of the Kirchhoff stress tensor in terms of the logarithmic strain tensor, i.e.
\begin{align}\label{HMC}
 {\rm KSTS-M}: \qquad \langle\tau(\log B_1)-\tau(\log B_2),\log B_1-\log B_2\rangle\geq  0, \qquad \forall\, B_1, B_2\in \PSym^+(3),
 \end{align}
 where $\tau$ is the Kirchhoff  stress. The strict Hill's monotonicity condition is denoted by KSTS-M$^+$. Also, Hill has shown that convexity of the
 quadratic Hencky energy $W_{_{\rm H}}$ in terms of $\log B$ implies the BE-inequalities.

In the linear theory of elasticity, $\sigma(\varepsilon)=2\, \mu\, \dev_3 \varepsilon+\kappa\, \tr(\varepsilon)\cdot \id$, $\varepsilon={\rm sym} \nabla
u$, and the TSTS-M$^+$  condition implies, after linearization{,}
$\langle \sigma(\varepsilon_1)-\sigma(\varepsilon_2),\varepsilon_1-\varepsilon_2\rangle >0$ for all $\varepsilon_1,\varepsilon_2\in {\rm Sym}(3)$,
$\varepsilon_1\neq \varepsilon_2$, and it is satisfied if and only if $\mu,\kappa>0$. Therefore, in the linear setting, TSTS-M$^+$ is stronger {than}
rank-one convexity which {only implies} $\mu>0, 2\, \mu+\lambda>0$.

The TSTS-M$^+$  condition caught our attention because of its possible relevance for the stability of nonlinear isotropic elastic bodies. Initially, {its}
relation to loss of stability or loss of rank-one convexity was left unclear. Jog and Patil \cite{jog2013conditions} have given a family of energies,
including Neo-Hooke and Mooney-Rivlin energies, {which} does not satisfy TSTS-M$^+$. In this work we show (for the first time) that there exist free
energies (namely $W_{_{\rm eH}}$) which do not satisfy TSTS-M$^+$ throughout but which are rank-one convex, while we also provide examples (namely $F\mapsto
\frac{\mu}{k}\,e^{k\,\|\log\,V\|^2}$, $k\geq \frac{3}{8}$) which satisfy  TSTS-M$^+$  but which are not   rank-one convex. In \cite[page
671]{jog2013conditions} it is conjectured that the
  TSTS-M$^+$ condition is stronger than polyconvexity, which, however,  is not true since  TSTS-M$^+$ is  not even stronger  than  rank-one convexity. The
  TSTS-M$^+$  condition implies that the Cauchy stress is an invertible function of the left stretch tensor (TSS-I); a property which could become
  important in FEM-computations based on the least squares finite element method
  \cite{cai2003first,cai2004least,starke2007adaptive,schwarz2010modified,starke2011analysis}.

   For isotropic materials, TSTS-M$^+$ (and TSS-I) leads to a unique stress free reference (natural) configuration, up to a rigid deformation, i.e.
   $\sigma=0$ implies $B=\id$ (or, equivalently, $\log B=0$), since taking $B_2=\id$ in \eqref{Jogines} we deduce at once
  \begin{align}
  \langle\sigma(\log B_1),\log B_1\rangle> 0, \qquad \forall\, B_1\in \PSym^+(3), \ B_1\neq \id \quad \Rightarrow\quad \sigma(\log B_1)\neq 0\,.
  \end{align}
   We note the simple implications
  \begin{center}
  TSTS-M$^+$ \quad $\Rightarrow$ \ \ $\left\{\begin{array}{ll}
  \text{ TSTS-M} \vspace{1mm}\\  \text{ TSTS-I}\quad\Leftrightarrow \quad \text{TSS-I}.\end{array}\right.$
  \end{center}

The TSTS-M$^+$  and KSTS-M$^+$ condition are frame-indifferent in the following sense: superposing one time dependent rigid rotation field $Q(t)\in {\rm
SO}(3)$, we have
\begin{align}
&F_1\mapsto F_1^*=Q(t) F_1, \qquad F_2\mapsto F_2^*=Q(t) F_2, \notag\\& B_1=F_1\,F_1^T\mapsto B_1^*= Q(t) B_1 Q^T(t),  \qquad B_2=F_2\,F_2^T\mapsto B_2^*=
Q(t) B_2 Q^T(t),\notag\\
& \log B_1\mapsto \log B_1^*= Q(t) (\log B_1) Q^T(t), \qquad \log B_2\mapsto\log B_2^*= Q(t) (\log B_2) Q^T(t),
\end{align}
and the identity
 \begin{align}
 &\langle\sigma( \log B_1^*)-\sigma(  \log B_2^*),\log B_1^*-\log B_2^*\rangle=\langle\sigma(  \log B_1)-\sigma(  \log B_2),\log B_1-\log
 B_2\rangle,\notag
 \end{align}
 holds, due to the isotropy of the formulation.

 In Section \ref{sectinvert0} we have shown that
\begin{align}
\tau=D_{\log V}W(\log V)=(\det V) \cdot\sigma=e^{\tr(\log V)}\cdot \sigma,
\end{align}
where $\sigma $ is the Cauchy stress and $\tau$ the Kirchhoff stress corresponding to the energy $F\mapsto W(\log V)$.
\begin{remark}\label{remarkmon}
Sufficient for TSTS-M$^+$ is Jog and Patil's \cite{jog2013conditions} constitutive requirement that
\begin{align}
\mathbb{Z}:=D_{\log V}\,\sigma(\log V)
\end{align}
is positive definite.
\end{remark}
\begin{proof}
Let us remark that for all $B_1, B_2\in {\rm PSym}^+(3)$ and $0\leq t\leq 1$, we have $2\,\log V_1=\log B_1, \, 2\,\log V_2=\log B_2$ and $t\, (\log
V_1-\log V_2)+\log V_2\in{\rm Sym}(3)$, where $V_1^2=B_1,\, V_2^2=B_2$ . Moreover, we have
 \begin{align}\label{Joginespr}
  \langle\sigma(\log B_1)&-\sigma(\log B_2),\log B_1-\log B_2\rangle=2\,\langle\sigma(2\,\log V_1)-\sigma(2\,\log V_2),\log V_1-\log
  V_2\rangle\notag\\&=2\,\left\langle\left[\int_0^1 \frac{\rm d}{\rm dt}\,\sigma \bigg(2\,t\, (\log V_1-\log V_2)+2\,\log V_2\bigg)dt\right],\log V_1-\log
  V_2\right\rangle\\&=4\,\int_0^1 \left\langle\left[D_{\log V}\,\sigma \bigg(2\,t\, (\log V_1-\log V_2)+2\,\log V_2\bigg).\,(\log V_1-\log V_2)\right],\log
  V_1-\log V_2\right\rangle dt\,.\notag
 \end{align}
 Using that the integrand is non-negative, due to the assumption that $\mathbb{Z}=D_{\log V}\sigma(\log V)$ is positive definite, the TSTS-M$^+$ condition
 follows.
\end{proof}

\medskip

 With the substitution $X=\log V$, the monotonicity of $\sigma$ as a function of $X\in \Sym(3)$ means
 \begin{align}
 \langle\sigma(X+H)-\sigma(X),H\rangle\geq 0 \qquad \forall\, X, H\in \Sym(3),
 \end{align}
 and sufficient for monotonicity of $\sigma$ is (proof as in Remark \ref{remarkmon})
 \begin{align}\label{gradsigma}
 \langle D_{X}\sigma(X).\,H,H\rangle\geq 0 \qquad \forall\, X, H\in \Sym(3).
 \end{align}

\begin{remark}
Since $e^{\|\log U\|^2}$ is uniformly convex in $\log U$, KSTS-M$^+$ is satisfied everywhere.
\end{remark}

\subsection{TSTS-M$^+$ for the energy $F\mapsto \frac{\mu}{k}\,e^{k\,\|\log\,V\|^2}+\frac{\lambda}{2\widehat{k}}\,e^{\widehat{k}\,[{\rm tr}
(\log\,V)]^2}$}
\begin{proposition}
 The Cauchy stress tensor $\sigma$ corresponding to the energy $F\mapsto \frac{\mu}{k}\,e^{k\,\|\log\,V\|^2}$ satisfies TSTS-M  for $k\geq \frac{3}{8}$ and
 TSTS-M$^+$  for $k> \frac{3}{8}$.
 \end{proposition}
 \begin{proof}
 In order to show this, let us remark that  for the energy $F\mapsto \frac{\mu}{k}\,e^{k\,\|\log\,V\|^2}$ we have
 \begin{align}\label{eqsigmataut}
 \widetilde{\tau}(\log\,V)&=\,2\,{\mu}\,e^{k\,\|\log\,V\|^2}\cdot \log\,V,\qquad
  \widetilde{\sigma}(\log\,V)=\,2\,{\mu}\,e^{k\,\|\log\,V\|^2-\tr(\log  V)}\cdot \log\,V.
 \end{align}
  We compute
 \begin{align}
 \langle D_{X}\widetilde{\sigma}(X).\,H,H\rangle=&\,2\,{\mu}\,e^{k\,\| X\|^2-\tr(X)}[2k\langle  X,H\rangle-\tr(H)]\langle
 X,H\rangle\notag+2\,{\mu}\,e^{k\,\|X\|^2-\tr(X)}\| H\|^2\notag\\
 =&\,2\,{\mu}\,e^{k\,\| X\|^2-\tr(X)}\{2\,k\,\langle  X,H\rangle^2-\tr(H)\langle  X,H\rangle+\| H\|^2\}.
\end{align}
If $\tr(H)\langle  X,H\rangle<0$, then obviously $
 \langle D_{X}\widetilde{\sigma}(X).\,H,H\rangle>0\notag
$. Otherwise,  for  $k\geq \frac{3}{8}$ it follows
\begin{align}
 \langle D_{X}\widetilde{\sigma}(X).\,H,H\rangle\geq &2\,{\mu}\,e^{k\,\| X\|^2-\tr(X)}\{2k\langle  X,H\rangle^2-2\,\sqrt{\frac{2k}{3}}\tr(H)\langle
 X,H\rangle+\| H\|^2\}\\
 =&\,2\,{\mu}\,e^{k\,\| X\|^2-\tr(X)}\langle H-\sqrt{\frac{2\,k}{3}}\langle  X,H\rangle\cdot \id,H-\sqrt{\frac{2k}{3}}\langle  X,H\rangle\cdot
 \id\rangle\notag\\
 =&\,2\,{\mu}\,e^{k\,\| X\|^2-\tr(X)}\left\| H-\sqrt{\frac{2\,k}{3}}\langle  X,H\rangle\cdot \id\right\|^2\geq 0\notag.
 \end{align}
Moreover, for $k> \frac{3}{8}$ we have
$
 \langle D_{X}\widetilde{\sigma}(X).\,H,H\rangle>0\notag
$ and   the proof is complete.
 \end{proof}
 \begin{cor}
 The Cauchy stress tensor corresponding to the energy $F\mapsto \frac{\mu}{k}\,e^{k\,\|\log\,V\|^2}+\frac{\lambda}{2\widehat{k}}\,e^{\widehat{k}\,[{\rm tr}
 (\log\,V)]^2}$ satisfies TSTS-M  for $k\geq \frac{3}{8}$, $\widehat{k}\geq \frac{1}{8}$ and $\mu,\lambda>0$ and TSTS-M$^+$ for $k> \frac{3}{8}$,
 $\widehat{k}\geq \frac{1}{8}$ (or $k\geq \frac{3}{8}$, $\widehat{k}> \frac{1}{8}$) and $\mu,\lambda>0$ .
 \end{cor}
 \begin{proof}
 From direct calculations we have
 \begin{align}
 \langle D_{X}e^{\widehat{k}\,(\tr (X))^2-\tr(X)}\, \tr(X)\cdot \id.\,H,H\rangle=e^{\widehat{k}\,(\tr (X))^2-\tr(X)} \{2\,\widehat{k}\,
 [\tr(X)]^2-\tr(X)+1\}\,[\tr(H)]^2.
 \end{align}
 Thus, if $\widehat{k}\geq \frac{1}{8}$, then
  \begin{align}
 \langle D_{X}e^{\widehat{k}\,(\tr (X))^2-\tr(X)}\, \tr(X)\cdot \id.\,H,H\rangle\geq e^{\widehat{k}\,(\tr (X))^2-\tr(X)}
 \left(\frac{1}{2}\,\tr(X)-1\right)^2[\tr(H)]^2\geq 0.
 \end{align}
The above inequality is strict for $\widehat{k}> \frac{1}{8}$.  The rest of the  proof follows from the previous theorem.
 \end{proof}
  Since, however, we  prove in Subsection \ref{Buligalog} that $F\mapsto e^{\|\log V\|^2}$ is not LH-elliptic, we note that in general
 \begin{center}
 TSTS-M$^+$ \quad $\nRightarrow$ \quad LH-ellipticity\,,
 \end{center}
 answering a conjecture arising in \cite{jog2013conditions}. It is also clear that
 \begin{center}
 LH-ellipticity  \quad $\nRightarrow$ \quad TSTS-M or TSTS-I,
 \end{center}
 as already implied by some examples from the development in \cite{jog2013conditions}. As a preliminary conclusion on the status of the TSTS-M-condition we
 can note that TSTS-M is an additional plausible criterion, basically independent of other conditions.  It is compatible, in principle, with rank-one
 convexity, but does not imply it. It can be speculated that TSTS-M$^+$ should hold for some domain of bounded distortions.

The same remarks hold for the KSTS-M$^+$ condition, i.e. the notion is frame-indifferent and
\begin{center}
  KSTS-M$^+$ \quad $\Rightarrow\quad \text{KSS-I}$, \qquad KSTS-M$^+$ \quad $\nRightarrow\quad \text{LH}$, \qquad
  LH \quad $\nRightarrow\quad \text{KSTS-M}^+$.
  \end{center}

\subsection{TSTS-M$^+$ for the family of energies $W_{_{\rm eH}}$}
Let us  consider  our  exponentiated Hencky energy with volumetric-isochoric decoupled format
\begin{align}\label{heem}
W_{_{\rm eH}}(\log V):=\frac{\mu}{k}\,e^{k\,\|\dev_3\log\,V\|^2}+\frac{\kappa}{2\widehat{k}}\,e^{\widehat{k}\,(\tr(\log V))^2}.
\end{align}
\begin{proposition}
The TSTS-M$^+$ condition \eqref{gradsigma} is not everywhere satisfied for the energy function $W_{_{\rm eH}}$ defined by \eqref{heem} for $n=2,3$.
\end{proposition}
\begin{proof}
 In  Section \ref{sectinvert0} we have shown that
 \begin{align}\label{eqsigmatau}
 \tau_{_{\rm eH}}(\log\,V)&=2\,{\mu}\,e^{k\,\|\dev_3\log\,V\|^2}\cdot \dev_3\log\,V+{\kappa}\,e^{\widehat{k}\,[\tr(\log V)]^2}\,\tr(\log V)\cdot \id,\notag\\
  \sigma_{_{\rm eH}}(\log\,V)&=2\,{\mu}\,e^{k\,\|\dev_3\log\,V\|^2-\tr(\log  V)}\cdot \dev_3\log\,V+{\kappa}\,e^{\widehat{k}\,[\tr(\log V)]^2-\tr(\log
  V)}\,\tr(\log V)\cdot \id.
 \end{align}

 We compute
 \begin{align}\label{Dsigma}
 \langle D_{X}\sigma_{_{\rm eH}}(X).\,H,H\rangle=&2{\mu}\,e^{k\,\|\dev_3 X\|^2-\tr(X)}[2k\langle \dev_3 X,H\rangle-\tr(H)]\langle \dev_3
 X,H\rangle\notag+2{\mu}\,e^{k\,\|\dev_3 X\|^2-\tr(X)}\|\dev_3 H\|^2\notag\\
 &+{\kappa}\,e^{\widehat{k}\,[\tr(X)]^2-\tr(X)}[2\,\widehat{k}\,\tr(X)\tr(H)-\tr(H)]\tr(X)\tr(H)+{\kappa}\,e^{\widehat{k}\,[\tr(X)]^2-\tr(X)}\,[\tr(H)]^2.\notag\\
 =&2{\mu}\,e^{k\,\|\dev_3 X\|^2-\tr(X)}\{2k\langle \dev_3 X,H\rangle^2-\tr(H)\langle \dev_3 X,H\rangle+\|\dev_3 H\|^2\}\notag\\
 &+{\kappa}\,e^{\widehat{k}\,[\tr(X)]^2-\tr(X)}\{2\,\widehat{k}\,[\tr(X)]^2-\tr(X)+1\}\cdot\,[\tr(H)]^2.
 \end{align}
 For $\widehat{k}>\frac{1}{8}$,  it is easy to see that
 \begin{align}
 \{2\,\widehat{k}\,[\tr(X)]^2-\tr(X)+1\}>0, \qquad {\text{for all }} X\in\Sym(3).
 \end{align}
 On the other hand, the first summand in   \eqref{Dsigma}
 \begin{align}\label{Jogn}
 \langle D_{X}[e^{k\,\|\dev_3 X\|^2-\tr(X)}\cdot \dev_3 X].\,H,H\rangle=\,2\,k\,\langle \dev_3 X,H\rangle^2-\tr(H)\langle \dev_3 X,H\rangle+\|\dev_3 H\|^2
 \end{align}
 is not positive for all $H\in \Sym(3)$. For instance, we may choose
 \begin{align}
 H_0=\dev_3 X+a\cdot \id, \quad a\in\R_+,
 \end{align}
 and we obtain
 \begin{align}
 2k\langle \dev_3 X,H_0\rangle^2-\tr(H_0)\langle \dev_3 X,H_0\rangle+\|\dev_3 H_0\|^2=\,2\,k\,\| \dev_3 X\|^4-3 \,a\|\dev_3 X\|^2+\|\dev_3 X\|^2,
 \end{align}
 which is negative for large values of $a$ (in the   two-dimensional case we may consider $H_0=\dev_2 X+a\cdot \id, \quad a\in\R_+$). Hence, the TSTS-M
 condition is not satisfied for the energy $F\mapsto \,e^{k\,\|\dev_3 \log V\|^2}$ alone.

 The next question is if one may  control the negative part in \eqref{Jogn}  by adding the  volumetric function $F\mapsto \,e^{\widehat{k}\,(\tr(\log V)^2}$. The
 answer is negative as we may see in the following.  Let us consider the matrices
 \begin{align}
 X_1=\left(
     \begin{array}{ccc}
       0 & t & 0 \\
       t & 0 & 0 \\
       0 & 0& 0
     \end{array}
   \right)\in \Sym(3), \qquad
   H_1=\left(
     \begin{array}{ccc}
       q/3 & 1 & 0 \\
       1 & q/3 & 0 \\
       0 & 0& q/3
     \end{array}
   \right)\in \Sym(3),
 \end{align}
where, for large values of $t>0$, $q$ is chosen such that
\begin{align}\label{gcondt}
\frac{8\, k\, t^2+4}{t}<q< \frac{2\mu}{\kappa}\,e^{2\,k\, t^2}t.
\end{align}
For the considered matrices, we deduce
\begin{align}
\|\dev_3 X_1\|^2=2\, t^2,\quad \tr (X_1)=0, \quad \|\dev_3 H_1\|^2=2,\quad \tr (H_1)=q, \quad \langle \dev_3 X_1,H_1\rangle =2\, t,
\end{align}
and
\begin{align}\label{Dsigma1}
 \langle D_{X}\sigma_{_{\rm eH}}(X_1).\,H_1,H_1\rangle=&2\,{\mu}\,e^{2\,k\,t^2}\{8\,k\, t^2-2\, q\, t+4\}+{\kappa}\,q^2\notag\\
 =&2\,{\mu}\,e^{2\,k\,t^2}\{8\,k\, t^2-\, q\, t+4\}+q\left\{-2\mu\, e^{2\,k\,t^2}+\kappa\,q\right\}<0.
 \end{align}

In  the two-dimensional case,  as counter-example we may consider the matrices
 \begin{align}
 X_1=\left(
     \begin{array}{ccc}
       0 & t  \\
       t & 0
     \end{array}
   \right)\in \Sym(2), \qquad
   H_1=\left(
     \begin{array}{ccc}
       q/2 & 1  \\
       1 & q/2 \\
     \end{array}
   \right)\in \Sym(2),
 \end{align}
where, for large values of $t>0$, $q$ satisfies \eqref{gcondt}.
 Therefore, the monotonicity condition is not satisfied and the proof is complete.
\end{proof}

However, the energy $W_{_{\rm eH}}$ satisfies the TSTS-M$^+$ condition  by restricting it to some  ``elastic domain"  in \underline{stretch space} (a cone in ${\rm PSym}(3)$)
of bounded distortions
 \begin{align}\label{margdevs}
 \mathcal{E}^+(W_{_{\rm eH}},\textrm{TSTS-M}^+,  V,\frac{2}{3}\, \widetilde{\boldsymbol{\sigma}}_{\!\mathbf{y}}^2):=\left\{  Y\in {\rm PSym}(3) \big|\,\ \|\dev_3
 \log Y\|^2\leq\frac{2}{3}\, \widetilde{\boldsymbol{\sigma}}_{\!\mathbf{y}}^2\,\right\}\subset  {\rm PSym}(3),
 \end{align}
 which is equivalent to restrict the energy $\overline{W}_{_{\rm eH}}(\log V)=W_{_{\rm eH}}(V)$   to the  ``elastic domain"  in \underline{strain space}
 \begin{align}\label{margdev}
 \mathcal{E}(W_{_{\rm eH}},\textrm{TSTS-M}^+, \log V,\frac{2}{3}\, \widetilde{\boldsymbol{\sigma}}_{\!\mathbf{y}}^2):=\left\{  X\in {\rm Sym}(3) \big|\,\ \|\dev_3
 X\|^2\leq\frac{2}{3}\, \widetilde{\boldsymbol{\sigma}}_{\!\mathbf{y}}^2\,\right\}\subset  {\rm Sym}(3),
 \end{align}
  where $\widetilde{\boldsymbol{\sigma}}_{\!\mathbf{y}}$ is  a dimensionless quantity related to the so called yield stress
  $\boldsymbol{\sigma}_{\!\mathbf{y}}$, whose dimension is [MPa], i.e. a critical value of shear stress, below which a plastic or viscoplastic material
  behaves like an elastic  solid; above this value, a plastic material deforms and a viscoplastic material flows. This assumption is in complete
  concordance with the  {Huber-von-Mises-Hencky}  distortional strain energy hypothesis  \cite{hill1948theory}.

We also need to introduce the elastic domain in the \underline{Kirchhoff-stress space}
\begin{align}\label{contintau2}
 \mathcal{E}(W_{_{\rm eH}},\textrm{TSTS-M}^+, \tau_{_{\rm eH}},\frac{2}{3}\, {\boldsymbol{\sigma}}_{\!\mathbf{y}}^2):=\left\{ \tau\in {\rm Sym}(3) \big|\,\ \|\dev_3
 \tau\|^2\leq\frac{2}{3}\, {\boldsymbol{\sigma}}_{\!\mathbf{y}}^2\right\}\subset  {\rm Sym}(3).
 \end{align}

\begin{proposition}\label{prop1mon}  {\rm (TSTS-M$^+$  is  satisfied for the energy function $W_{_{\rm eH}}$ for bounded distortions)}
 If the material parameters $\mu,\kappa>0, \widehat{k}>\frac{1}{8}$ and $\widetilde{\boldsymbol{\sigma}}_{\!\rm \mathbf{y}}\in\R$ are such that
  \begin{align}\label{cinp-1}
 0<\widetilde{\boldsymbol{\sigma}}_{\!\mathbf{y}}^2\leq\frac{6}{e}\frac{\kappa}{\mu} \frac{8\widehat{k}-1}{8\widehat{k}},
 \end{align}
  holds true, then there exists $k>0$  such that
 \begin{align}\label{Dsigmaf}
 \forall\, X\in \mathcal{E}(W_{_{\rm eH}},\textrm{TSTS-M}^+, \log V,\frac{2}{3}\, \widetilde{\boldsymbol{\sigma}}_{\!\mathbf{y}}^2), \quad \forall \, H\in
 \Sym(3):\qquad \langle D_{X}{\sigma}_{_{\rm eH}}(X).\,H,H\rangle> 0\, \quad
 \end{align}
i.e. the TSTS-M$^+$ inequality  is satisfied in $\mathcal{E}(W_{_{\rm eH}},\textrm{TSTS-M}^+, \log V,\frac{2}{3}\,
\widetilde{\boldsymbol{\sigma}}_{\!\mathbf{y}}^2)$ (or equivalently, the TSTS-M$^+$ inequality  is satisfied in $\mathcal{E}^+(W_{_{\rm eH}},\textrm{TSTS-M}^+,
V,\frac{2}{3}\, \widetilde{\boldsymbol{\sigma}}_{\!\mathbf{y}}^2)$).
\end{proposition}
\begin{proof}
 Let us rewrite  equation \eqref{Dsigma}  as
 \begin{align}\label{Dsigma0}
 \langle D_{X}{\sigma}_{_{\rm eH}}(X).\,H,H\rangle=&\,e^{k\,\|\dev_3 X\|^2-\tr(X)}\Big\{4\,\mu\,k\langle \dev_3 X,H\rangle^2-2\,\mu\, \tr(H)\,\langle \dev_3
 X,H\rangle\\
 &+{\kappa}\,e^{\widehat{k}\,[\tr(X)]^2-k\,\|\dev_3 X\|^2}\{2\,\widehat{k}\,[\tr(X)]^2-\tr(X)+1\}[\tr(H)]^2\Big\}\notag\\&+2{\mu}\,e^{k\,\|\dev_3
 X\|^2-\tr(X)}\|\dev_3 H\|^2.\notag
 \end{align}
 If
  $
 \widehat{k}>\frac{1}{8},
 $
 then
 $
  2\,\widehat{k}\,[\tr(X)]^2-\tr(X)+1> 0
 $
 for all $X\in \Sym(3)$. Hence, for
 \begin{align}
 4\,\mu\,k\langle \dev_3 X,H\rangle^2-2\,\mu\, \tr(H)\,\langle \dev_3 X,H\rangle+{\kappa}\,e^{\widehat{k}\,[\tr(X)]^2-k\,\|\dev_3
 X\|^2}\{2\,\widehat{k}\,[\tr(X)]^2-\tr(X)+1\}[\tr(H)]^2>0\notag
 \end{align}
 to hold for all $X,H\in \Sym(3)$, it is sufficient to have
 \begin{align}\label{idelta}
 4\,\mu^2-16 \,\mu\,k \,{\kappa}\,e^{\widehat{k}\,[\tr(X)]^2-k\,\|\dev_3 X\|^2}\{2\,\widehat{k}\,[\tr(X)]^2-\tr(X)+1\}< 0 \quad \text{for all}\quad
 X\in\Sym(3).
 \end{align}
 Because $\mu,\kappa>0$,  for matrices $X\in\Sym(3)$ which belong  to  the ``elastic domain" $\mathcal{E}(W_{_{\rm eH}},\textrm{TSTS-M}^+, \log V,\frac{2}{3}\,
 \widetilde{\boldsymbol{\sigma}}_{\!\mathbf{y}}^2)$ defined by \eqref{margdev} the above inequality is satisfied if
 \begin{equation}\label{deltanegkg}
 \frac{\mu}{4\kappa}< k
 \,\,e^{\widehat{k}\,[\tr(X)]^2-k\,\frac{2}{3}\,\widetilde{\boldsymbol{\sigma}}_{\!\mathbf{y}}^2}\,\{2\,\widehat{k}\,[\tr(X)]^2-\tr(X)+1\}.
 \end{equation}

 On the other hand, for $\widehat{k}>\frac{1}{8}$, we find
 \begin{align}\label{inftrceva}
 \inf_{X\in \Sym(3)}\{2\,\widehat{k}\,[\tr(X)]^2-\tr(X)+1\}=\frac{8\,\widehat{k}-1}{8\,\widehat{k}}>0.
 \end{align}

 Taking $\inf_{X\in \Sym(3)}$ of the right hand side of  \eqref{deltanegkg}, we obtain that if there exist $\widehat{k}>\frac{1}{8}$ and $k>0$ such that
 \begin{align}\label{deltanegk2}
 \frac{e^{k\,\frac{2}{3}\,\widetilde{\boldsymbol{\sigma}}_{\!\mathbf{y}}^2}}{k}\,\frac{\mu}{4\kappa}\leq\frac{8\,\widehat{k}-1}{8\,\widehat{k}}<1
 \end{align}
 holds, then  the inequality  \eqref{deltanegkg} follows.
 The question is whether there    are always numbers $k>0$  satisfying the above inequality.  We have
 \begin{align}\label{deltanegk3}
  \inf_{k>0}\left\{\frac{e^{k\,\frac{2}{3}\,\widetilde{\boldsymbol{\sigma}}_{\!\mathbf{y}}^2}}{k}\right\}=\inf_{k>0}\left\{\frac{e^{k\,\frac{2}{3}\,\widetilde{\boldsymbol{\sigma}}_{\!\mathbf{y}}^2}}{k\,\frac{2}{3}\,\widetilde{\boldsymbol{\sigma}}_{\!\mathbf{y}}^2}\,\frac{2}{3}\,\widetilde{\boldsymbol{\sigma}}_{\!\mathbf{y}}^2\right\}
  =\inf_{r>0}\left\{\frac{e^{r}}{r}\right\}\,\frac{2}{3}\,\widetilde{\boldsymbol{\sigma}}_{\!\mathbf{y}}^2=\,\frac{2}{3}\,e\,\widetilde{\boldsymbol{\sigma}}_{\!\mathbf{y}}^2\,,
  \qquad \lim\limits_{k\rightarrow\infty}\frac{e^{k\,\frac{2}{3}\,\widetilde{\boldsymbol{\sigma}}_{\!\mathbf{y}}^2}}{k}=\infty.
 \end{align}
 In view of \eqref{deltanegk3} and using the continuity of the function $t\mapsto
 \frac{e^{t\,\frac{2}{3}\,\widetilde{\boldsymbol{\sigma}}_{\!\mathbf{y}}^2}}{t}$, we conclude: if the material parameters $\mu,\kappa>0$,
 $\widehat{k}>\frac{1}{8}$ and $\widetilde{\boldsymbol{\sigma}}_{\!\mathbf{y}}\in\R$ are chosen such that
 \begin{align}\label{cinp}
 0<\,\frac{2}{3}\,e\,\widetilde{\boldsymbol{\sigma}}_{\!\mathbf{y}}^2\leq\frac{4\kappa}{\mu}\,\frac{8\,\widehat{k}-1}{8\,\widehat{k}} ,
 \end{align}
 then  we may find a  constant $k>0$   which satisfies
\begin{align}\label{deltanegk30}
  \,\frac{2}{3}\,
  e\,\widetilde{\boldsymbol{\sigma}}_{\!\mathbf{y}}^2\,\frac{\mu}{4\kappa}=\inf_{k>0}\left\{\frac{e^{k\,\frac{2}{3}\,\widetilde{\boldsymbol{\sigma}}_{\!\mathbf{y}}^2}}{k}\right\}\leq
  \frac{e^{k\,\frac{2}{3}\,\widetilde{\boldsymbol{\sigma}}_{\!\mathbf{y}}^2}}{k}\,\frac{\mu}{4\kappa}\leq\frac{8\,\widehat{k}-1}{8\,\widehat{k}} .
 \end{align}
Using \eqref{inftrceva} we obtain that there is a constant $k>0$ such that \eqref{deltanegk2} is satisfied. Hence,  there is a constant $k>0$ such that
\eqref{idelta} holds true, which in view of \eqref{Dsigma0} implies \eqref{Dsigmaf} and the proof is complete.
 \end{proof}
 We remark that Proposition \ref{prop1mon} is unspecific about the values for $k>0$. Written in terms of Poisson's ratio\footnote{We use that $\kappa=\frac{2\,\mu\,(1+\nu)}{3\,(1-2\nu)}$, $\nu=\frac{3\,\kappa-2\,\mu}{2(3\,\kappa+\mu)}$.}
 $-1<\nu\leq\frac{1}{2}$, the extra constitutive assumption \eqref{cinp-1} becomes
 \begin{align}\label{cinpP}
 0<\,\widetilde{\boldsymbol{\sigma}}_{\!\mathbf{y}}^2\leq \frac{4}{e}\frac{1+\nu}{1-2\nu}\,\frac{8\,\widehat{k}-1}{8\,\widehat{k}}\qquad \Leftrightarrow
 \qquad
 \widehat{k}\geq \frac{1}{8-\widetilde{\boldsymbol{\sigma}}_{\!\mathbf{y}}^2\, \frac{e}{2}}> \frac{1}{8}.
 \end{align}

 Heinrich Hencky \cite{hencky1924theorie} offered a physical interpretation of the von Mises criterion suggesting that yielding begins when the elastic
 energy of distortion reaches a critical value \cite{Hill50} (see also \cite{fosdick1993normality,cleja2003consequences,cleja2005yield}). For this, the von Mises criterion is also known as the
 maximum distortional strain energy criterion. This stems from the relation between  the second deviatoric stress invariant $J_2$ and the elastic strain
 energy of distortion $W_D = \frac{J_2}{2\,\mu}$, with the elastic shear modulus $\mu = \frac{E}{2(1+\nu)}$, Young's modulus $E$ and Poisson's ratio
 $\nu$.

 In the following we express the constitutive assumption \eqref{cinp} in terms of the yield stress ${\boldsymbol{\sigma}}_{\!\mathbf{y}}$ and the Kirchhoff
 stress  tensor $\tau_{_{\rm eH}}$.
\begin{proposition}\label{prop2mon} {\rm ($W_{_{\rm eH}}$ satisfies TSTS-M$^+$ for bounded distortions)}
 There exist $\widehat{k}>\frac{1}{8}$ and $k>0$, such that for all  ${\boldsymbol{\sigma}}_{\!\rm \mathbf{y}}\in\R$ for which
 \begin{align}\label{cclm}
0<{\boldsymbol{\sigma}}_{\mathbf{y}}^2\leq\frac{3\,\mu \, \kappa}{e}\, \,\frac{8\,\widehat{k}-1}{\widehat{k}}
e^{k\,\frac{\kappa}{e\,\mu} \,\frac{8\,\widehat{k}-1}{\widehat{k}}},
\end{align}
 holds true,
 the TSTS-M$^+$ inequality is satisfied for all $V\in {\rm PSym}(3)$  (for all  $\log V\in {\rm Sym}(3)$) for which $\tau_{_{\rm eH}}(\log V)\in
 \mathcal{E}(W_{_{\rm eH}},\textrm{TSTS-M}^+, \tau_{_{\rm eH}},\frac{2}{3}\, {\boldsymbol{\sigma}}_{\!\mathbf{y}}^2)$.
\end{proposition}
\begin{proof}
 Let us remark that {any} $X\in {\rm Sym}(3)$ for which $\tau_{_{\rm eH}}(X)$ lies in the set $\mathcal{E}(W_{_{\rm eH}},\textrm{TSTS-M}^+, \tau_{_{\rm eH}},\frac{2}{3}\,
 {\boldsymbol{\sigma}}_{\!\mathbf{y}}^2)$ satisfies
 $$
 \|2{\mu}\,e^{k\,\|\dev_3 X\|^2}\cdot \dev_3X\|\leq \sqrt{\frac{2}{3}}\,{\boldsymbol{\sigma}}_{\!\mathbf{y}}.
 $$
Hence,
 $
 \|\dev_3\, X\|\,e^{k\,\|\dev_3\,X\|^2}\leq \sqrt{\frac{2}{3}}\frac{{\boldsymbol{\sigma}}_{\!\mathbf{y}}}{2\,{\mu}\,}.
$
 If the yield limit ${\boldsymbol{\sigma}}_{\!\mathbf{y}}$ is chosen such that \eqref{cclm} is satisfied, then there is
 $\widetilde{\boldsymbol{\sigma}}_{\mathbf{y}}>0$ such that
  $
 0<{\boldsymbol{\sigma}}_{\mathbf{y}}\leq{2\,{\mu}\,}\widetilde{\boldsymbol{\sigma}}_{\mathbf{y}}\,
 e^{k\,\frac{2}{3}\,\widetilde{\boldsymbol{\sigma}}_{\mathbf{y}}^2},
$
and
$
\,\,\widetilde{\boldsymbol{\sigma}}_{\!\mathbf{y}}^2\leq\frac{6\,\kappa}{e\,\mu}\,\frac{8\,\widehat{k}-1}{8\,\widehat{k}}{.}
$
Hence,
 $
 0<\sqrt{\frac{2}{3}}\frac{{\boldsymbol{\sigma}}_{\!\mathbf{y}}}{2\,{\mu}\,}\leq\sqrt{\frac{2}{3}}\,\widetilde{\boldsymbol{\sigma}}_{\!\mathbf{y}}\,
 e^{k\,\frac{2}{3}\,\widetilde{\boldsymbol{\sigma}}_{\!\mathbf{y}}^2},
$
 which implies
$
 \|\dev_3\, X\|\,e^{k\,\|\dev_3\,X\|^2}\leq \sqrt{\frac{2}{3}}\,\widetilde{\boldsymbol{\sigma}}_{\!\mathbf{y}} \,
 e^{k\,\frac{2}{3}\,\widetilde{\boldsymbol{\sigma}}_{\!\mathbf{y}}^2}.
 $
 In view of  the monotonicity  of $t\mapsto t\, e^{k\,t^2}$,  we deduce
\begin{align}\label{marginesigma}
 \|\dev_3\, X\|\leq \sqrt{\frac{2}{3}}\,\widetilde{\boldsymbol{\sigma}}_{\!\mathbf{y}}\,,
\end{align}
and $X\in \mathcal{E}(W_{_{\rm eH}},\textrm{TSTS-M}^+, \log V,\frac{2}{3}\, \widetilde{\boldsymbol{\sigma}}_{\!\mathbf{y}}^2)$. Since we have assumed that $\widetilde{\boldsymbol{\sigma}}_{\!\mathbf{y}}^2$ and  $\widehat{k}>\frac{1}{8}$ satisfy \eqref{cinp-1},
then   Proposition \ref{prop1mon}
ensures the existence of $k>0$  such that the TSTS-M$^+$ inequality is satisfied  and the proof is complete.
\end{proof}

\begin{remark}
\begin{itemize}
\item[]
\item[i)]In terms of Young's modulus $E$ and Poisson's ratio $\nu$ the condition imposed on the yield limit ${\boldsymbol{\sigma}}_{\!\mathbf{y}}$ by
Proposition {\rm  \ref{prop2mon}} is
\begin{align}
0<{\boldsymbol{\sigma}}_{\mathbf{y}}^2\leq\frac{1}{2\,e}\,\frac{\,E^2}{(1+\nu)(1-2\, \nu)}\, \,\frac{8\,\widehat{k}-1}{\widehat{k}}\,
e^{k\, \frac{2}{3\,e}\, \frac{1+\nu}{1-2\, \nu} \,\frac{8\,\widehat{k}-1}{\widehat{k}}}.
\end{align}
\item[ii)]
In the incompressible limit $\kappa\rightarrow\infty$, it follows that $W_{_{\rm eH}}$ satisfies TSTS-M$^+$ everywhere since then\linebreak
$\mathcal{E}(W_{_{\rm eH}},\textrm{TSTS-M}^+, \tau_{_{\rm eH}},\frac{2}{3}\, {\boldsymbol{\sigma}}_{\!\mathbf{y}}^2)={\rm Sym}(3)$ and TSTS-M$^+\Leftrightarrow $
KSTS-M$^+$.
\end{itemize}
\end{remark}

\subsection{TSTS-M$^+$ for three-parameter energies  $W_{_{\rm eH}}^{\sharp}$}

In this subsection we consider the set of energies of the family $W_{_{\rm eH}}$ for which $\widehat{k}=\frac{2}{3}\, k$.
\begin{proposition}\label{prop3mon} {\rm (The exponentiated 3-parameter energy  $W_{_{\rm eH}}$ satisfies TSTS-M$^+$ for bounded distortions)}
Let $\widetilde{\boldsymbol{\sigma}}_{\!\mathbf{y}}>0$ be such that
$
\widetilde{\boldsymbol{\sigma}}_{\!\mathbf{y}}^2\, e^{\frac{1}{8}\,\widetilde{\boldsymbol{\sigma}}_{\!\mathbf{y}}^2+1}\leq \,\frac{6\,\kappa}{\mu}
$
  holds true.  Then there exists $k>\frac{3}{16}$  such that for all ${\boldsymbol{\sigma}}_{\!\mathbf{y}}$  satisfying
$
 0<{\boldsymbol{\sigma}}_{\!\mathbf{y}}\leq{2\,{\mu}\,}\widetilde{\boldsymbol{\sigma}}_{\!\mathbf{y}}\,
 e^{k\,\frac{2}{3}\,\widetilde{\boldsymbol{\sigma}}_{\!\mathbf{y}}^2},
$
 the  exponentiated 3-parameter energy
 \begin{align}\label{3pareneg}
W_{_{\rm eH}}^{\sharp}(\log V)&:=\frac{\mu}{k}\,e^{k\,\|\dev_3\log\,V\|^2}+\frac{3\,\kappa}{4\,{k}}\,e^{\frac{2}{3}\,k\,({\rm tr}(\log V))^2},
\end{align}
 satisfies  TSTS-M$^+$  for all $V\in {\rm PSym}(3)$   for which  $\tau_{_{\rm eH}}(\log V)\in   \mathcal{E}(W_{_{\rm eH}},\textrm{TSTS-M}^+, \tau_{_{\rm eH}},\frac{2}{3}\,
 {\boldsymbol{\sigma}}_{\!\mathbf{y}}^2)$.

\end{proposition}

\begin{proof}
Similar as in the  proof of Proposition \ref{prop2mon},   we deduce that  $\tau_{_{\rm eH}}(X)\in
\mathcal{E}(W_{_{\rm eH}},\textrm{TSTS-M}^+, \tau_{_{\rm eH}},\frac{2}{3}\, {\boldsymbol{\sigma}}_{\!\mathbf{y}}^2)$ implies
 \begin{align}\label{Mise00}
 \|\dev_3\, X\|\,e^{k\,\|\dev_3\,X\|^2}\leq \sqrt{\frac{2}{3}}\frac{{\boldsymbol{\sigma}}_{\!\mathbf{y}}}{2\,{\mu}\,}.
 \end{align}

 On the other hand, in view of \eqref{Dsigma0} --\eqref{cinp}, in order to have
$
 \langle D_{X}{\sigma}_{_{\rm eH}}(X).\,H,H\rangle>0
$
 for  $X\in\Sym(3)$ which belong also to  the ``elastic domain" $\mathcal{E}(W_{_{\rm eH}},\textrm{TSTS-M}^+, \log V,\frac{2}{3}\,
 \widetilde{\boldsymbol{\sigma}}_{\!\mathbf{y}}^2)$  defined by \eqref{margdev}, we already know  that it is sufficient to prove that there are
   $k>0$ and $\widehat{k}>\frac{1}{8}$ which satisfy \eqref{deltanegk2}, that is
 \begin{align}\label{deltanegk20}
 \frac{e^{k\,\frac{2}{3}\,\widetilde{\boldsymbol{\sigma}}_{\!\mathbf{y}}^2}}{k}\,\frac{\mu}{4\,\kappa}\leq\frac{8\,\widehat{k}-1}{8\,\widehat{k}}<1.
 \end{align}
 For the 3-parameter energy we have $2\, k=3\, \widehat{k}$. Hence, in this case we have to prove that there is $k>\frac{3}{16}$ such that
 \begin{equation}\label{deltanegk21}
 \frac{e^{k\,\frac{2}{3}\,\widetilde{\boldsymbol{\sigma}}_{\!\mathbf{y}}^2}}{k}\,\frac{\mu}{4\,\kappa}\leq\frac{16\,{k}-3}{16\,{k}}.
 \end{equation}
  Let us rewrite \eqref{deltanegk21} in the form
   \begin{align}\label{deltanegk22}
 \frac{e^{k\,\frac{2}{3}\,\widetilde{\boldsymbol{\sigma}}_{\!\mathbf{y}}^2}}{16\,{k}-3}\leq \,\frac{\kappa}{4\,\mu} \qquad \Leftrightarrow\qquad
  \frac{e^{\left(k\,\frac{2}{3}-\frac{1}{8}\right)\widetilde{\boldsymbol{\sigma}}_{\!\mathbf{y}}^2}}{\left(k\,\frac{2}{3}-\frac{1}{8}\right)\widetilde{\boldsymbol{\sigma}}_{\!\mathbf{y}}^2}\,\frac{1}{24}\,\widetilde{\boldsymbol{\sigma}}_{\!\mathbf{y}}^2\,
  e^{\frac{1}{8}\,\widetilde{\boldsymbol{\sigma}}_{\!\mathbf{y}}^2}\leq \,\frac{\kappa}{4\,\mu} .
 \end{align}
  We have
 \begin{align}\label{deltanegk31}
  \inf_{k>\frac{3}{16}}\left\{\frac{e^{\left(k\,\frac{2}{3}-\frac{1}{8}\right)\widetilde{\boldsymbol{\sigma}}_{\!\mathbf{y}}^2}}{\left(k\,\frac{2}{3}-\frac{1}{8}\right)\widetilde{\boldsymbol{\sigma}}_{\!\mathbf{y}}^2}\right\}=e\,,
  \qquad
  \lim\limits_{k\rightarrow\infty}\frac{e^{\left(k\,\frac{2}{3}-\frac{1}{8}\right)\widetilde{\boldsymbol{\sigma}}_{\!\mathbf{y}}^2}}{\left(k\,\frac{2}{3}-\frac{1}{8}\right)\widetilde{\boldsymbol{\sigma}}_{\!\mathbf{y}}^2}=\infty.
 \end{align}
 In view of \eqref{deltanegk31} and using the continuity of the function $t\mapsto
 \frac{e^{\left(t\,\frac{2}{3}-\frac{1}{8}\right)\widetilde{\boldsymbol{\sigma}}_{\!\mathbf{y}}^2}}{\left(t\,\frac{2}{3}-\frac{1}{8}\right)\widetilde{\boldsymbol{\sigma}}_{\!\mathbf{y}}^2}$,
 we conclude: if the material parameters $\mu,\kappa>0$ and $\widetilde{\boldsymbol{\sigma}}_{\!\mathbf{y}}\in\R$ are chosen such that
 \begin{align}\label{cinp11}
  0<\,\widetilde{\boldsymbol{\sigma}}_{\!\mathbf{y}}^2\, e^{\frac{1}{8}\,\widetilde{\boldsymbol{\sigma}}_{\!\mathbf{y}}^2+1}\leq \,\frac{6\,\kappa}{\mu},
 \end{align}
 then we may find a  constant $k>\frac{3}{16}$   which satisfies \eqref{deltanegk21}.  If the yield limit ${\boldsymbol{\sigma}}_{\!\mathbf{y}}$ is chosen
 such that
 \begin{align}\label{Mise01}
 0<&\sqrt{\frac{2}{3}}\frac{{\boldsymbol{\sigma}}_{\!\mathbf{y}}}{2\,{\mu}\,}\leq\sqrt{\frac{2}{3}}\,\widetilde{\boldsymbol{\sigma}}_{\!\mathbf{y}} \,
 e^{k\,\frac{2}{3}\,\widetilde{\boldsymbol{\sigma}}_{\!\mathbf{y}}^2},
 \end{align}
   then in view of the monotonicity $t\mapsto t\, e^{k\, t^2}$, by \eqref{Mise00} and \eqref{Mise01} we have
 that
 $
 \|\dev_3\, X\|\leq\sqrt{\frac{2}{3}}\,\widetilde{\boldsymbol{\sigma}}_{\!\mathbf{y}},
$
 which means that $X\in \mathcal{E}(W_{_{\rm eH}},\textrm{TSTS-M}^+, \log V,\frac{2}{3}\, \widetilde{\boldsymbol{\sigma}}_{\!\mathbf{y}}^2)$.  Since
 $\widetilde{\boldsymbol{\sigma}}_{\!\mathbf{y}}$ satisfies \eqref{cinp11}, it follows that there is  $k>\frac{3}{16}$  satisfying \eqref{deltanegk21}. For
 the 3-parameter energy  ($2\, k=3\, \widehat{k}$),  in view of \eqref{Dsigma0} --\eqref{cinp}, if  \eqref{deltanegk21} is satisfied, then it follows that
 the TSTS-M$^+$ inequality is satisfied  and the proof is complete.
 \end{proof}

 \begin{remark}
 \begin{itemize}
 \item[i)] In terms of Young's modulus and Poisson's ratio, the condition imposed on the yield limit ${\boldsymbol{\sigma}}_{\!\mathbf{y}}$ by
     Proposition {\rm  \ref{prop3mon}} may be written in the form
\begin{align}
 0<{\boldsymbol{\sigma}}_{\!\mathbf{y}}\leq{\frac{E}{1+\nu}\,}\widetilde{\boldsymbol{\sigma}}_{\!\mathbf{y}}\,
 e^{k\,\frac{2}{3}\,\widetilde{\boldsymbol{\sigma}}_{\!\mathbf{y}}^2},\quad  \text{where}\quad
\widetilde{\boldsymbol{\sigma}}_{\!\mathbf{y}}^2\, e^{\frac{1}{8}\,\widetilde{\boldsymbol{\sigma}}_{\!\mathbf{y}}^2+1}\leq \,\frac{4\,(1+\nu)}{1-2\,
\nu}.
\end{align}
\item[ii)]  For illustrating purposes let us consider the case of $\nu=1/3$ and an extremely large domain of roughly 10\% distortional strain, i.e.  $\|\dev_3\log U\|\leq 0.1$. To this specification
 corresponds  $\widetilde{\boldsymbol{\sigma}}_{\!\mathbf{y}}=\sqrt{\frac{3}{2}}\,\,0.1$, which is in concordance
with the
   values considered for the yield stress $\widetilde{\boldsymbol{\sigma}}_{\!\mathbf{y}}=\sqrt{\frac{3}{2}}\,0.1$ ($\|\dev_3\log U\|\leq 0.1$), since
 $
{\frac{3}{2}}\,0.01\, e^{\frac{1}{8}\,{\frac{3}{2}}\,0.01+1}\backsimeq 0.041$. Moreover,  the required inequality
\eqref{deltanegk22}, $\frac{e^{0.01\,\cdot k}}{16\, k-3}\leq \frac{1+\nu}{6(1-2\,\nu)}$ is satisfied if the parameter $k$ belongs to the interval
$[0.29 ,919]\subset [\frac{3}{16},919]$.
\item[iii)] We will encounter $k>\frac{3}{16}$ also later on with regard to rank-one convexity conditions for $W_{_{\rm eH}}$.
\end{itemize}
\end{remark}

\subsection{TSTS-M$^+$ for the quadratic Hencky energy}
For comparison, we also consider the quadratic Hencky energy
\begin{align}
\widehat{W}_{_{\rm H}}(U)&:={\mu}\,\|{\rm dev}_n\log U\|^2+\frac{\kappa}{2}\,[{\rm tr}(\log U)]^2.
\end{align}
We recall that the corresponding Kirchhoff and the Cauchy stress tensors are given by
\begin{align}
\tau_{_{\rm H}}&=D_{\log V} \widetilde{W}_{_{\rm H}}(V)=2\,\mu\, \dev_3\log V+{\kappa}\,\tr(\log V)\cdot \id,
\\
\sigma_{_{\rm H}}&=[2\,\mu\, \dev_3\log V+{\kappa}\,\tr(\log V)\cdot \id]\, e^{-\tr(\log V)}.\notag
\end{align}
The monotonicity inequality \eqref{gradsigma} becomes
 \begin{align*}
 \langle D_{X}{\sigma}_{_{\rm H}}(X).\,H,H\rangle=\{2\,\mu\, \|\dev_3 H\|^2+\kappa[\tr(H)]^2-2\, \mu\, \tr(H)\langle \dev_3 X, \dev_3 H\rangle-\kappa\,
 \tr(X)\,[\tr(H)]^2\}\, e^{-\tr( X)}\geq 0.
 \end{align*}
In $X=\id$ we have
\begin{align}\label{Dsigmaff00}
 \langle D_{X}{\sigma}_{_{\rm H}}(\id ).\,H,H\rangle=[2\,\mu\, \|\dev_3 H\|^2-2\kappa\, \,[\tr(H)]^2]\,e^{-3},
 \end{align}
which is negative for all $H$ such that $\|\dev_3 H\|^2<[\tr(H)]^2$. Jog and Patil \cite[page 676]{jog2013conditions} have proved that the quadratic Hencky
energy satisfies the TSTS-M$^+$ conditions only for those deformations for which
$
\det V<e{.}
$
{This} bound coincides, incidentally, with the loss of ellipticity for $W_{_{\rm eH}}$ in a uniaxial setting.

\subsection{TSTS-M$^+$ for the energy $F\mapsto \mu\, e^{{a}\,\|\dev_3\log\,V\|^2+\frac{\widehat{a}}{2}\,(\tr(\log V))^2}$}

At the end of this subsection  we consider the energy
\begin{align}\label{heemf}
W(\log V):=\mu\, e^{{a}\,\|\dev_3\log\,V\|^2+\frac{\widehat{a}}{2}\,(\tr(\log V))^2}
\end{align}
with the corresponding  Kirchhoff and Cauchy  stress, respectively
 \begin{align}\label{eqsigmatau}
 \widehat{\tau}(\log\,V)&=\mu\,e^{{a}\,\|\dev_3\log\,V\|^2+\frac{\widehat{a}}{2}\,(\tr(\log V)^2}\{2\,{a}\, \dev_3\log\,V+{\widehat{a}}\,\tr(\log V)\cdot
 \id\},\notag\\
  \widehat{\sigma}(\log\,V)&=\mu\,e^{{a}\,\|\dev_3\log\,V\|^2+\frac{\widehat{a}}{2}\,(\tr(\log V)^2 -\tr(\log  V)}\{2\,{a}\,
  \dev_3\log\,V+{\widehat{a}}\,\tr(\log V)\cdot \id\},
 \end{align}
 and we try to determine $a,\widehat{a}$ such that this energy satisfies the TSTS-M condition.

 The monotonicity inequality \eqref{gradsigma} becomes
 \begin{align}\label{Dsigmaff}
 \langle D_{X}\widehat{\sigma}(X).\,H,H\rangle=&\mu\,e^{{a}\,\|\dev_3 X\|^2+\frac{\widehat{a}}{2}\,(\tr(X)^2 -\tr(X)}\{[2\,{a} \langle\dev_3
 X,H\rangle+{\widehat{a}}\,\tr(X)\tr(H)]^2\notag\\
 &-\tr(H)[2\,{a} \langle\dev_3 X,H\rangle+{\widehat{a}}\,\tr(X)\tr(H)]\}\notag \\
 &+\mu\,e^{{a}\,\|\dev_3 X\|^2+\frac{\widehat{a}}{2}\,(\tr(X)^2 -\tr(X)}\{2\,{a} \|\dev_3 H\|^2+{\widehat{a}}\,[\tr(H)]^2\}.\notag\\
=&\mu\,e^{{a}\,\|\dev_3 X\|^2+\frac{\widehat{a}}{2}\,(\tr(X)^2 -\tr(X)}\Big\{[2\,{a} \langle\dev_3 X,H\rangle+{\widehat{a}}\,\tr(X)\tr(H)]^2\\
 &-\tr(H)[2\,{a} \langle\dev_3 X,H\rangle+{\widehat{a}}\,\tr(X)\tr(H)]+2\,{a} \|\dev_3 H\|^2+{\widehat{a}}\,[\tr(H)]^2\Big\}\notag.
 \end{align}
 Using the  inequality of means, $x\, y <\frac{\alpha}{2}\,x^2+\frac{1}{2\alpha}\,y^2$, $\alpha>0$, we deduce
 \begin{align}\label{Dsigmaff2}
 [2{a} \langle\dev_3 X,H\rangle&+{\widehat{a}}\,\tr(X)\tr(H)]^2-\tr(H)[2\,{a}\, \langle\dev_3
 X,H\rangle+{\widehat{a}}\,\tr(X)\,\tr(H)]+{\widehat{a}}\,[\tr(H)]^2\\
 &\geq \left(1-\frac{\alpha}{2}\right)[2\,{a}\, \langle\dev_3 X,H\rangle+{\widehat{a}}\,\tr(X)\,\tr(H)]^2+\left(
 \widehat{a}-\frac{1}{2\alpha}\right)\,[\tr(H)]^2\quad \forall \, \alpha>0.\notag
 \end{align}
Hence, choosing the dimensionless parameters $\widehat{a},\alpha>0$ such that
\begin{align}\label{condhata}
\frac{1}{4}<\widehat{a},\qquad \frac{1}{2\,\widehat{a}}<\alpha<2\,,
\end{align}
we have
 \begin{align}\label{Dsigmaff}
 \langle D_{X}\widehat{\sigma}(X).\,H,H\rangle\geq 0 \qquad \forall \, X,H\in \Sym(3).
 \end{align}
 Therefore, if $\widehat{a}>\frac{1}{4}$, then the energy $F\mapsto \mu\, e^{{a}\,\|\dev_3\log\,V\|^2+\frac{\widehat{a}}{2}\,(\tr(\log V)^2}$ satisfies the
 TSTS-M condition.  For instance, if we choose $\widehat{a}=\frac{\kappa}{\mu}$, the condition $\widehat{a}>\frac{1}{4}$ is equivalent to
 \begin{align}\label{cinpP}
\frac{1}{4}<\frac{\kappa}{\mu}\quad \Leftrightarrow \quad 1<\frac{8\,(1+\nu)}{3\,(1-2\nu)} \quad \Leftrightarrow \quad -\frac{5}{14}<\nu .
 \end{align}
 If  we choose $\widehat{a}=\frac{\kappa}{4\,\mu}$, the condition $\widehat{a}>\frac{1}{4}$ is equivalent to\footnote{In \cite{mott2009limits} it is
 claimed that the classical elasticity formulation is applicable only for $\frac{1}{5}<\nu<\frac{1}{2}$.
}
 \begin{align}\label{cinpP}
\mu<{\kappa}\quad \Leftrightarrow \quad 1<\frac{2\,(1+\nu)}{3\,(1-2\nu)} \quad \Leftrightarrow \quad \frac{1}{8}<\nu<\frac{1}{2}.
 \end{align}
In conclusion, we observe that we do not need to consider a restricted domain for the energy \eqref{heemf} in order to enforce the TSTS-M$^+$ condition.

\section{Rank-one convexity}\setcounter{equation}{0}

\subsection{Criteria for rank-one convexity}\label{ROsect}

In this subsection we recall some criteria for rank-one-convexity that we will use throughout  the rest of this  paper.
 Knowles and Sternberg  \cite{knowles1976failure,knowles1978failure} (see also \cite{aubert1985conditions,aubert1987faible,knowles1975ellipticity}) have
 given the following result:
\begin{theorem}\label{silhavy318}{\rm  (Knowles and Sternberg \cite[page 318]{Silhavy97})}
 Let  ${W}:{\rm GL}^+(n)\rightarrow\mathbb{R}$ be an objective-isotropic function of class $C^2$ with the representation in terms of the singular values of
 $U$ via $W(F)=\widehat{W}(U)=g(\lambda_1,\lambda_2,...,\lambda_n)$, where $g\in C^2(\R_+^n,\R)$. Let $F\in {\rm GL}^+(n)$ be given with  an $n$-tuple of
 singular values $\lambda_1,\lambda_2,...,\lambda_n$. If $D^2 W(F)[a\otimes b,a\otimes b]\geq 0$ for every $a,b\in\R^n$ (i.e. $F\mapsto W(F)$ is rank-one
 convex), the following conditions hold:
 \begin{itemize}
 \item[i)] $\dd \frac{\partial^2 g}{\partial \lambda_i^2 }\geq 0$ for every $i=1,2,...,n$\,,\, i.e. separate convexity (SC) and the TE-inequalities
     hold;
 \item[ii)] for every $i\neq j$,
 \begin{align}\label{SBE}
 &\dd\underbrace{\frac{\lambda_i\frac{\partial g}{\partial \lambda_i}-\lambda_j\frac{\partial g}{\partial \lambda_j}}{\lambda_i-\lambda_j}\geq
 0}_{\text{\rm ``BE-inequalities"}}\quad \text{if} \quad \lambda_i\neq \lambda_j,\quad  \text{and}\quad\frac{\partial^2 g}{\partial
 \lambda_i^2}-\frac{\partial^2 g}{\partial \lambda_i\partial \lambda_j}+\frac{\partial g}{\partial \lambda_i}\frac{1}{\lambda_i}\geq 0 \quad \text{if}
 \quad \lambda_i=\lambda_j,\\
 &\dd\sqrt{\frac{\partial^2 g}{\partial \lambda_i^2}\frac{\partial^2 g}{\partial \lambda_j^2}}+\frac{\partial^2 g}{\partial \lambda_i\partial
 \lambda_j}+\frac{\frac{\partial g}{\partial \lambda_i}-\frac{\partial g}{\partial \lambda_j}}{\lambda_i-\lambda_j}\geq 0 \quad \text{if} \quad
 \lambda_i\neq\lambda_j,\qquad
 \sqrt{\frac{\partial^2 g}{\partial \lambda_i^2}\frac{\partial^2 g}{\partial \lambda_j^2}}-\frac{\partial^2 g}{\partial \lambda_i\partial
 \lambda_j}+\frac{\frac{\partial g}{\partial \lambda_i}+\frac{\partial g}{\partial \lambda_j}}{\lambda_i+\lambda_j}\geq 0.\notag
 \end{align}
 \end{itemize}
 If $n=2$, then conditions i) and ii) are also sufficient.\hfill$\Box$
\end{theorem}

From the above theorem we can easily see that LH-ellipticity implies the BE-inequalities and TE-inequalities. Necessary and sufficient conditions for
LH-ellipticity in  the three-dimensional case are given in \cite{simpson1983copositive,rosakis1990ellipticity} and more recently by Dacorogna
\cite{Dacorogna01}{, also for compressible materials}.

\begin{theorem}\label{dacorogna5}{\rm (Dacorogna \cite[page 5]{Dacorogna01})}
 Let  ${W}:{\rm GL}^+(3)\rightarrow\mathbb{R}$ be an objective-isotropic function of class $C^2$ with the representation in terms of the singular values of
 $U$ via $W(F)=\widehat{W}(U)=g(\lambda_1,\lambda_2,\lambda_3)$, where $g\in C^2(\R_+^3,\R)$ and $g$ is symmetric. Then $F\mapsto W(F)$ is rank one convex
 if and only if the following four sets of conditions hold for every $\lambda_1,\lambda_2,\lambda_3\in \R_+$
 \begin{itemize}
 \item[i)] $\dd \frac{\partial^2 g}{\partial \lambda_i^2 }\geq 0$ for every $i=1,2,3$\,,\, i.e. separate convexity (SC) and the TE-inequalities hold;
 \item[ii)] for every $i\neq j$,
 \begin{align}\label{DBE}
 &\dd\underbrace{\frac{\lambda_i\frac{\partial g}{\partial \lambda_i}-\lambda_j\frac{\partial g}{\partial \lambda_j}}{\lambda_i-\lambda_j}\geq
 0}_{\text{\rm ``BE-inequalities"}}\quad \text{if} \quad \lambda_i\neq \lambda_j,\quad  \text{and}\quad \sqrt{\frac{\partial^2 g}{\partial \lambda_i^2
 }\frac{\partial^2 g}{\partial \lambda_j^2 }}+m_{ij}^\varepsilon\geq 0, \quad  \text{and either}\notag\\
 &\dd m_{12}^\varepsilon\sqrt{\frac{\partial^2 g}{\partial \lambda_3^2 }}+m_{13}^\varepsilon\sqrt{\frac{\partial^2 g}{\partial \lambda_2^2
 }}+m_{23}^\varepsilon\sqrt{\frac{\partial^2 g}{\partial \lambda_1^2 }}+\sqrt{\frac{\partial^2 g}{\partial \lambda_1^2 }\frac{\partial^2 g}{\partial
 \lambda_2^2 }\frac{\partial^2 g}{\partial \lambda_3^2 }}\geq 0 \quad \text{or}\\
 & \det M^\varepsilon\geq 0,\notag
 \end{align}
 where $M^\varepsilon=(m_{ij}^\varepsilon)$ is symmetric and
 \begin{align}
 m_{ij}^\varepsilon=\left\{\begin{array}{cc}
                             \frac{\partial^2 g}{\partial \lambda_i^2 } & \text{if} \ i=j \ \text{or if } \ i<j \ \text{and} \ \lambda_i=\lambda_j, \\
                             \varepsilon_i\varepsilon_j \frac{\partial^2 g}{\partial \lambda_i \partial \lambda_j} +\frac{\frac{\partial g}{\partial
                             \lambda_i }-\varepsilon_i\varepsilon_j \frac{\partial g}{\partial \lambda_j}}{ \lambda_i -\varepsilon_i\varepsilon_j
                             \lambda_j}& \text{if} \ i<j \ \text{and} \ \lambda_i\neq\lambda_j \ \text{or}\ \varepsilon_i\varepsilon_j\neq 1,
                           \end{array}
 \right.
 \end{align}
  for any choice of $\varepsilon_i\in\{\pm 1\}$.\hfill $\Box$
 \end{itemize}
\end{theorem}

The last  one is taken from Buliga \cite{Buliga}:
\begin{theorem}\label{Buligacrit}{\rm (Buliga \cite[page 1538]{Buliga})}
 A twice continuously differentiable function ${W}:{\rm GL}^+(n)\rightarrow\mathbb{R}$ that can be written as a function of the singular values of $U$ via
 $W(F)=\widehat{W}(U)=g(\lambda_1,\lambda_2,...,\lambda_n)$ is rank-one-convex if and only if
 \begin{itemize}
 \item[i)] the function  $({\lambda}_1,{\lambda}_2,...,{\lambda}_n)\mapsto g(e^{{\lambda}_1},e^{{\lambda}_2},...,e^{{\lambda}_n})$ is Schur-convex and
  \item[ii)] for all $a=(a_1,a_2,...,a_n)\in\R^n$, $({\lambda}_1,{\lambda}_2,...,{\lambda}_n)\in\R_+^n$
 \begin{align}\label{Buligainegn}
  \sum_{i,j} H_{ij}({\lambda}_1,{\lambda}_2,...,{\lambda}_n)\,a_i\,a_j + G_{ij}({\lambda}_1,{\lambda}_2,...,{\lambda}_n)\,|a_i|\,|a_j|\geq 0,
 \end{align}
where
\begin{align}
&G_{ij}({\lambda}_1,{\lambda}_2,...,{\lambda}_n)=\frac{{\lambda}_i\frac{\partial g}{\partial
{\lambda}_i}({\lambda}_1,{\lambda}_2,...,{\lambda}_n)-{\lambda}_j\frac{\partial g}{\partial
{\lambda}_j}({\lambda}_1,{\lambda}_2,...,{\lambda}_n)}{{\lambda}_i^2-{\lambda}_j^2} \ \text{for} \  i\neq j,\quad
G_{ii}({\lambda}_1,{\lambda}_2,...,{\lambda}_n)=0,\notag\\
&H_{ij}({\lambda}_1,{\lambda}_2,...,{\lambda}_n))=\overline{H}_{ij}({\lambda}_1,{\lambda}_2,...,{\lambda}_n)+(D^2\,g({\lambda}_1,{\lambda}_2,...,{\lambda}_n))_{ij},\\
& \overline{H}_{ij}({\lambda}_1,{\lambda}_2,...,{\lambda}_n)=\frac{{\lambda}_j\frac{\partial g}{\partial
{\lambda}_i}({\lambda}_1,{\lambda}_2,...,{\lambda}_n)-{\lambda}_i\frac{\partial g}{\partial
{\lambda}_j}({\lambda}_1,{\lambda}_2,...,{\lambda}_n)}{{\lambda}_i^2-{\lambda}_j^2}\ \text{for} \  i\neq j,\quad
H_{ii}({\lambda}_1,{\lambda}_2,...,{\lambda}_n)=0.\notag\ \ \ \ \ \Box
\end{align}
\end{itemize}
\end{theorem}

 \v{S}ilhav\'{y} \cite{SilhavyPRE99} has previously given a similar  result  in terms of the copositivity of some
matrices (see also \cite{Dacorogna01}).

\subsection{The LH-condition for incompressible media}

In this subsection we consider the case of incompressible materials, i.e. we consider objective-isotropic energies  ${W}:{\rm SL}(3)\rightarrow\mathbb{R}$.
The restrictions imposed by  rank-one convexity are less strict in this case. The rank-one convexity for such a function $W$ means that $W$ still has to
satisfy
\begin{align}
D^2 W(F)(\xi\otimes\eta,\xi\otimes\eta)>0,
\end{align}
(similar to the LH-ellipticity condition), but now only for all vectors $\xi,\eta\neq 0$ with the additional property that
$$\det(F+\,\xi\otimes\eta)=1.$$
For $F,H\in \R^{3\times 3}$ we have
\begin{align}
\det(F+H)=\det F+\langle \Cof F, H\rangle+\langle F,\Cof  H\rangle +\det H.
\end{align}
Thus, for $F\in \R^{3\times 3}$
\begin{align}
\notag\det(F+\xi\otimes\eta)&=\det F+\langle \Cof F, \xi\otimes\eta\rangle+\langle F,\underbrace{\Cof  [\xi\otimes\eta]}_{=0}\rangle +\underbrace{\det
[\xi\otimes\eta]}_{=0}\\&=\det F\,[1+\langle F^{-T}, \xi\otimes\eta\rangle]=\det F\,[1+\tr(F^{-1}\xi\otimes\eta)],
\end{align}
since $\text{rank}(\xi\otimes\eta)=1$. Hence, it follows  that $\xi,\eta\neq 0$ have to satisfy
\begin{align}
\det F\,\cdot \tr(F^{-1}\xi\otimes\eta)=\det F\,\cdot \langle F^{-1} \xi,\eta\rangle =0 \ \ \Leftrightarrow\ \ \langle F^{-1} \xi,\eta\rangle=0\, .
\end{align}
Necessary conditions for LH-ellipticity of incompressible, isotropic hyperelastic solids were obtained by Sawyers and Rivlin
\cite{SawyersRivlin78,sawyers1973instability}, while necessary and sufficient conditions were established by Zubov and Rudev
\cite{ZubovRudev,zubov1995necessary}.

\begin{theorem}\label{zubovinc}{\rm (Zubov's LH-ellipticity criterion for incompressible materials \cite[page 437]{ZubovRudev})}
Let  ${W}:{\rm SL}(3)\rightarrow\mathbb{R}$ be an objective-isotropic function of class $C^2$ with the representation in terms of the singular values of
$U$ via $W(F)=\widehat{W}(U)=g(\lambda_1,\lambda_2,\lambda_3)$, where $g\in C^2(\R_+^3,\R)$ and $g$ is symmetric. Then $F\mapsto W(F)$ is rank one convex
on ${\rm SL}(3)$  if and only if the following  nine inequalities  hold for every $\lambda_1,\lambda_2,\lambda_3\in \R_+$:
\begin{itemize}
 \item[i)] for every $i\neq j$ and for any arbitrary permutation $ (i,j,k)$ of the numbers $1,2,3$
 \begin{align}\label{ZBEi}
 &\dd\alpha_k:=\underbrace{\frac{\lambda_i\frac{\partial g}{\partial \lambda_i}-\lambda_j\frac{\partial g}{\partial \lambda_j}}{\lambda_i-\lambda_j}>
 0}_{\text{\rm ``BE-inequalities"}}\quad \text{if} \quad \lambda_i\neq \lambda_j;
 \end{align}
\item[ii)] $\dd \delta_k:=\beta_i\lambda_i^2+\beta_j \lambda_j^2+2\gamma_k^-\lambda_i\lambda_j>0$, where  $ \quad (i,j,k) \ \text{is any  arbitrary
    permutation of the numbers} \ 1,2,3,$ and
    \begin{align}
    \dd \beta_i:=\frac{\partial^2 g}{\partial \lambda_i^2 },\qquad
    \gamma_k^{\pm}=\pm\frac{\partial^2 g}{\partial \lambda_i\partial \lambda_j}+\frac{\frac{\partial g}{\partial \lambda_i}\mp\frac{\partial g}{\partial
    \lambda_j}}{\lambda_i\mp\lambda_j};
    \end{align}
 \item[iii)] $\epsilon_k+\sqrt{\delta_i\delta_j}>0$, where $\dd
     \epsilon_k:=\beta_k\lambda_k^2+\gamma_k^+\lambda_i\lambda_j+\gamma_i^-\lambda_k\lambda_j+\gamma_j^-\lambda_k\lambda_i$ and $ (i,j,k)$ is an{y}
     arbitrary permutation of the numbers $1,2,3.$ \hfill $\Box$
 \end{itemize}
\end{theorem}

\subsection{The quadratic Hencky energy  $W_{_{\rm H}}$ is not rank-one convex}\label{henckynotelliptic}

In this subsection we re-examine a counter-example first considered  by Neff \cite{Neff_Diss00} in order to prove that the quadratic Hencky energy function
$W_{_{\rm H}}$ defined by \eqref{th} is not rank-one convex even when restricted to ${\rm SL}(3)$. A domain where $W_{_{\rm H}}$ is LH-elliptic has been given  in
\cite{Bruhns01} under some strong conditions upon the constitutive coefficients, i.e. $\mu,\lambda>0$.  The first proof of  the non-ellipticity of a
related energy expression $\|\dev_3\log U\|^N$, $0<N\leq 1$ seems to be due to Hutchinson et al. \cite{Hutchinson82}.

\begin{proposition}\label{neffdis}
The function $W:{\rm SL}(3)\to\R$, $W(F)=\|\dev_3 \log U\|^2$ is not LH-elliptic.
\end{proposition}
\begin{proof} The proof of this remark is  adapted  from \cite{Neff_Diss00}.
We consider the function $h:\R\to \R$,
\begin{align}
h(t)=W(\id+t(\eta \otimes\xi)).
\end{align}
We choose the vectors $\eta, \xi\in \R^3$ so that (i.e. the family of simple shears)
\begin{align}
\eta=\left(\begin{array}{c}
      1 \\
0 \\
 0
    \end{array}\right),\quad \xi=\left(\begin{array}{c}
      0 \\
1 \\
 0
    \end{array}\right),\quad
\eta\otimes \xi =\left(
                   \begin{array}{ccc}
                     0 & 1 & 0 \\
                     0 & 0 & 0 \\
                     0 & 0 & 0 \\
                   \end{array}
                 \right).
\end{align}
Hence
\begin{align}
&(\id+t(\eta \otimes\xi))^T(\id+t(\eta \otimes\xi)) =\left(
                   \begin{array}{ccc}
                     1 & t & 0 \\
                     t & 1+t^2 & 0 \\
                     0 & 0 & 1 \\
                   \end{array}
                 \right),
\end{align}
and from
\begin{align}&\det\left(
                   \begin{array}{ccc}
                     1-\lambda & t & 0 \\
                     t & 1+t^2-\lambda & 0 \\
                     0 & 0 & 1-\lambda \\
                   \end{array}
                 \right)=(1-\lambda)[\lambda^2-\lambda(2+t^2)+1]\notag
\end{align}
the eigenvalues of the matrix $(\id+y(\eta \otimes\xi))^T(\id+y(\eta \otimes\xi))$ can be seen to be
\begin{align}\label{eigenNe}
\lambda_1=1,\qquad \lambda_2=\frac{1}{2}\left(2+t^2+t\sqrt{4+t^2}\right),\qquad \lambda_3=\frac{1}{2}\left(2+t^2-t\sqrt{4+t^2}\right).
\end{align}
The matrix $U$ is positive definite and symmetric and therefore can be assumed diagonal, to obtain
\begin{align}\label{exprimaredev3}
 \norm{\dev_3\log U}^2&=\frac{1}{3}\left(\log^2\frac{\lam_1}{\lam_2}+\log^2\frac{\lam_2}{\lam_3}+\log^2\frac{\lam_3}{\lam_1}\right).
\end{align}
An analogous expression for $\norm{\dev_n \log U}^2$ can be given in any dimension $n\in \N$, see Appendix \ref{identitiesapp}. {I}n terms of the
eigenvalues, the function $h$ is given by
\begin{align}
h(t)&=\frac{1}{3}\left(\log^2\frac{\lambda_1}{\lambda_2}+\log^2\frac{\lambda_2}{\lambda_3}
+\log^2\frac{\lambda_3}{\lambda_1}\right).
\end{align}
Since $\lambda_1\lambda_2\lambda_3=1$, see \eqref{eigenNe},  it follows
$
0=\log(\lambda_1\lambda_2\lambda_3)=\log\lambda_2+\log\lambda_3,\qquad \log \lambda_2=-\log \lambda_3.
$
Thus,
 $
h(t)={2}\,\log^2{\lambda_2}={2}\,[\,\log (2+t^2+t\sqrt{4+t^2})-\log 2]^2.
$
\begin{figure}[h!]\begin{center}
\begin{minipage}[h]{0.9\linewidth}
\centering
\includegraphics[scale=0.7]{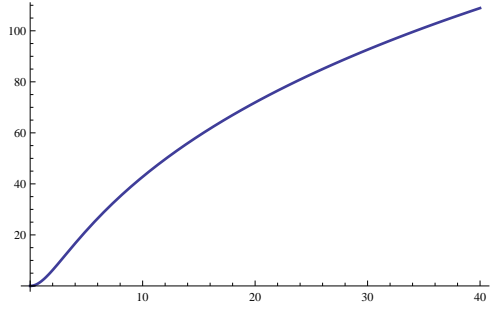}
\centering
\caption{\footnotesize{The graphical representation of $h:\R\to\R\,,\  h(t)={2}\,\log^2{\lambda_2}(t)={2}\,[\,\log (2+t^2+t\sqrt{4+t^2})-\log 2]^2$.}}
\label{hfunctie}
\end{minipage}
\end{center}
\end{figure}%
This function is not convex in $t$, as can be easily deduced. Let us remark that $\id\in{\rm SL}(3)$ and also $(\id+t(\eta \otimes\xi))\in{\rm SL}(3)$.
Therefore, the function $W$ is not rank-one convex in ${\rm SL}(3)$. Hence, $W$ is not elliptic in ${\rm SL}(3)$.
\end{proof}

A direct consequence of the previous  proposition is
\begin{remark}{\rm (three-dimensional case)}
The function $W:{\rm SL}(3)\to\R$, $W(F)=\mu \|\dev_3 \log U\|^2+\frac{\kappa}{2}({\rm tr}(\log U))^2$, for any $\mu,\kappa>0$, is not LH-elliptic.
\end{remark}
\begin{proof}
The  counterexample is the one as in the proof of the previous remark because corresponding to this counterexample we have $\frac{\kappa}{2}(\tr(\log
U))^2=\log(\lambda_1\lambda_2\lambda_3)=\log 1=0$.
\end{proof}

\begin{remark}{\rm (two-dimensional case)}
The function $W:{\rm SL}(2)\to\R$, $W(F)=\mu \|\dev_2 \log U\|^2+\frac{\kappa}{2}({\rm tr}(\log U))^2$, for any $\mu,\kappa>0$, is not LH-elliptic.
\end{remark}
\begin{proof} The proof is similar {to} the proof in the 3D case. The vectors
$\eta, \xi\in \R^2$ are now
$$\hspace{6cm}
\eta=\left(\begin{array}{c}
      1 \\
0
    \end{array}\right),\qquad \xi=\left(\begin{array}{c}
      0 \\
1  \end{array}\right).\hspace{6.6cm}\qedhere
$$
\end{proof}

\begin{proposition}\label{TEHencky}
The energy $W:{\rm SL}(3)\to\R$, $W(F)=\|\dev_3 \log U\|^2$  satisf{ies} the TE-inequalities (SC) only for those $U$ such that the eigenvalues
$\mu_1,\mu_2,\mu_3$ of $\dev_3\log U$ are smaller than $\frac{2}{3}$.
\end{proposition}
\begin{proof}
The corresponding function $g:\mathbb{R}^3_+\rightarrow\mathbb{R}$ for the isotropic  energy $W:{\rm SL}(3)\to\R$, $W(F)=\|\dev_3 \log U\|^2$ is
\begin{align}
  g(\lambda_1,\lambda_2,\lambda_3):=
\frac{1}{3}\left[\log^2\frac{\lam_1}{\lam_2}+
\log^2\frac{\lam_2}{\lam_3}+\log^2\frac{\lam_3}{\lam_1}\right].
\end{align}
Hence, we have to check where the function $g$ is separately convex. We deduce
\begin{align}\label{deriv2H}
\frac{\partial^2 g}{\partial \lambda_1^2}&=\frac{2}{\lambda_1^2}\,\left(2- \log \frac{\lambda_1}{\lambda_2}+ \log
\frac{\lambda_3}{\lambda_1}\right)=\,\frac{6}{\lambda_1^2}\left(\frac{2}{3}- \frac{2}{3}\log \lambda_1+\frac{1}{3}\log\lambda_2+ \frac{1}{3}\log
\lambda_3\right)\notag\\
\frac{\partial^2 g}{\partial \lambda_2^2}&=\frac{2}{\lambda_2^2}\left(2- \log \frac{\lambda_2}{\lambda_3}+ \log
\frac{\lambda_1}{\lambda_2}\right)=\,\frac{6}{\lambda_2^2}\left(\frac{2}{3}- \frac{2}{3}\log \lambda_2+\frac{1}{3}\log\lambda_3+ \frac{1}{3}\log
\lambda_1\right)\\
\frac{\partial^2 g}{\partial \lambda_3^2}&=\frac{2}{\lambda_3^2}\left(2- \log \frac{\lambda_3}{\lambda_1}+ \log
\frac{\lambda_2}{\lambda_3}\right)=\,\frac{6}{\lambda_3^2}\left(\frac{2}{3}- \frac{2}{3}\log \lambda_3+\frac{1}{3}\log\lambda_1+ \frac{1}{3}\log
\lambda_2\right)\notag.
\end{align}
On the other hand,  the eigenvalues $\mu_1,\mu_2,\mu_3$ of $\dev_3\log U$ are
\begin{align}
{\mu}_1=&\quad\,\, \frac{2}{3}\log{\lambda}_1-\frac{1}{3}\log{\lambda}_2-\frac{1}{3}\log{\lambda}_3,\notag\\
{\mu}_2=&-\frac{1}{3}\log{\lambda}_1+\frac{2}{3}\log{\lambda}_2-\frac{1}{3}\log{\lambda}_3,\\
{\mu}_3=&-\frac{1}{3}\log{\lambda}_1-\frac{1}{3}\log{\lambda}_2+\frac{2}{3}\log{\lambda}_3,\notag
\end{align}
and the proof is complete.
\end{proof}
We can obtain a similar condition in terms of the eigenvalues of $U$ instead of those of $\dev_3\log U$:
\begin{cor}\label{TEHenckycor}
The energy $W:{\rm SL}(3)\to\R$, $W(F)=\|\dev_3 \log U\|^2$  satisfies the TE-inequalities  only for those $U$ such that the eigenvalues
$\lambda_1,\lambda_2,\lambda_3$ of $U$ satisfy
\begin{align}\label{condlH}
\lambda_1^2\leq e^2 \,\lambda_2\lambda_3,\qquad \lambda_2^2\leq e^2\, \lambda_3\lambda_1, \qquad \lambda_3^2\leq e^2 \,\lambda_1\lambda_2.
\end{align}
\end{cor}
\begin{proof}
From  \eqref{deriv2H} we find that $g(\lambda_1,\lambda_2,\lambda_3)=\frac{1}{3}\left[\log^2\frac{\lam_1}{\lam_2}+
\log^2\frac{\lam_2}{\lam_3}+\log^2\frac{\lam_3}{\lam_1}\right]$ is separate{ly} convex if and only if
\begin{align}
2- \log \frac{\lambda_1^2}{\lambda_2\lambda_3}\geq 0, \qquad 2- \log \frac{\lambda_2^2}{\lambda_1\lambda_3}\geq 0, \qquad 2- \log
\frac{\lambda_3^2}{\lambda_1\lambda_2}\geq 0,
\end{align}
which are equivalent to the inequalities \eqref{condlH}.
\end{proof}

\subsection{Convexity of the volumetric response $F\mapsto e^{\widehat{k}\,(\log\det F)^m}$}\label{polytr}

In the family of energies \eqref{the} which we consider, the volumetric response is modelled by a term of the form \linebreak $F\mapsto
e^{\widehat{k}\,(\log\det F)^m}$.
In deriving convexity conditions, we first examine the conditions under which the more general form $\det F\mapsto h(\log \det F)$ is convex in $\det F$,
which is clearly sufficient for LH-ellipticity (details can be found in Appendix \ref{Appenhdet}, see also \cite[page 213]{Dacorogna08} and
\cite{lehmich2012convexity}).
Hence, we ask for the second derivative of $t\mapsto h(\log t)$ to be positive:
\begin{align}
 \frac{d^2}{dt^2}\,h(\log t)=\frac{\rm d}{\rm dt}[h'(\log t)\tel{t}]=h''(\log t)\tel{t^2}-h'(\log t)\tel{t^2}\geq 0.
\end{align}
Obviously, this is the case if and only if
 $
 h''(\log t)\geq h'(\log t)
$
{ for all $t>0$} and hence, if and only if for all $\xi\in\R$
 $
 h''(\xi)\geq h'(\xi).
$
 Thus, $t\mapsto h(\log t)$ is convex if and only if  $h$  grows at least exponentially (see also Appendix \ref{Appenhdet}). This result is in concordance
 with the  necessary conditions derived in the paper of Sendova and Walton \cite{Walton05}.

{Fix $m\in \N$. } We want to find $\widehat{k}$ such that $h(\xi)=e^{\widehat{k}\,\xi^m}$ matches this criterion, i.e.
\begin{align}
 \widehat{k}^2\,m^2\,\xi^{2m-2}e^{\widehat{k}\,\xi^m}+\widehat{k}\,m\,(m-1)\,\xi^{m-2}\,e^{\widehat{k}\,\xi^m}\geq
 \widehat{k}\,m\,\xi^{m-1}e^{\widehat{k}\,\xi^m},
\end{align}
which is equivalent to
 $
 \widehat{k}\,m\,\xi^m-\xi+(m-1)\geq 0.
$
We compute the minimum of this expression. To this aim we solve the equation
$
  \widehat{k}\,m^2\,\xi^{m-1}-1=0
$
and we obtain
$
  \xi=\widehat{k}^{-\tel{m-1}}m^{-\frac{2}{m-1}}.
$
Therefore
\begin{align}
 \min_{\xi\in \mathbb{R}}\,\{
 \widehat{k}\,m\,\xi^m-\xi+(m-1)\}&=\widehat{k}\,m\,\widehat{k}^{-\frac{m}{m-1}}m^{-\frac{2m}{m-1}}-\widehat{k}^{-\tel{m-1}}m^{-\frac{2}{m-1}}+(m-1)\\
 &=\widehat{k}^{-\tel{m-1}}m^{-\frac{m+1}{m-1}}(1-m)+(m-1).\notag
\end{align}
This minimum is nonnegative if and only if
$
 -\widehat{k}^{-\tel{m-1}}m^{-\frac{m+1}{m-1}}+1\geq 0.
$
Thus $\widehat{k}$ has to be chosen such that
$
\widehat{k}\geq m^{-(m+1)}.
$
In conclusion:
\begin{lemma}\label{lemaJH}Let $m\in\N$. Then the
 function
$
 t\mapsto e^{\widehat{k}\,(\log(t))^m}
$
is convex if and only if
$
 \widehat{k}\geq \tel{m^{(m+1)}}.
$
\end{lemma}
\newpage
This implies our next result:
\begin{proposition}
The function
\[
 \det F\mapsto e^{\widehat{k}\,(\log\det F)^m}, \quad F\in {\rm GL}^+(n)
\]
is convex in $\det F$ for $\widehat{k}\geq \tel{m^{(m+1)}}$. ({More explicitly, for $m=2$ this means $\widehat{k}\geq \tel8$, in case of $m=3$ convexity holds for $\widehat{k}\geq \tel{81}$.)}
\end{proposition}
In view of Proposition \ref{propDacdet} (see also \cite[page 213]{Dacorogna08}), we have
\begin{cor}
The function
\[
 F\mapsto e^{\widehat{k}\,(\log\det F)^m}, \quad F\in {\rm GL}^+(n)
\]
is rank-one convex in $F$ for $\widehat{k}\geq \tel{m^{(m+1)}}$. ({More explicitly, for $m=2$ this means $\widehat{k}\geq \tel8$, in case of $m=3$ rank-one convexity holds for $\widehat{k}\geq \tel{81}$.})
\end{cor}

\subsection{Rank-one convexity of the isochoric  exponentiated Hencky \\energy in plane elastostatics}\label{rank-isochoric}

In this subsection we consider  a variant of the exponentiated Hencky energy in plane strain, with isochoric part
\begin{align}\label{wf}
W_{_{\rm eH}}^{\rm iso}(F)=e^{\,k\,\|{\rm dev}_2 \log U\|^2}=e^{k\,\|\log \frac{U}{\det U^{1/2}}\|^2}.
\end{align}

Let us first recall that for small strains the exponentiated Hencky energy turns into the well-known quadratic Hencky energy:
\begin{align}\label{th22}
W_{_{\rm eH}}(F)-\left(\frac{\mu}{k}+\frac{\kappa}{2\widehat{k}}\right)&=\underbrace{\frac{\mu}{k}\,e^{k\,\|{\rm dev}_n\log
U\|^2}+\frac{\kappa}{2\widehat{k}}\,e^{\widehat{k}\,[{\rm tr}(\log U)]^2}}_{\text{fully nonlinear
elasticity}}-\left(\frac{\mu}{k}+\frac{\kappa}{2\widehat{k}}\right)\notag\\
&=\underbrace{\mu\,\|\,{\rm dev}_n\log U\|^2+\frac{\kappa}{2}\,[(\log \det U)]^2}_{\text{materially linear, geometrically nonlinear
elasticity}}{+\,\,\text{h.o.t.}}=\!\!\!\!\!\!\!\!\underbrace{W_{_{\rm H}}(F)}_{\text{quadratic Hencky energy}}\!\!\!\!\!\!\!\!{+\,\text{h.o.t.}}\\
&=\underbrace{\mu\,\|\,{\rm dev}_n\,{\rm sym} \nabla u\|^2+\frac{\kappa}{2}\,[\tr({\rm sym} \nabla u)]^2}_{\text{linear
elasticity}}{+\,\text{h.o.t.}},\notag
\end{align}
where  $u:\mathbb{R}^n\rightarrow\mathbb{R}^n$ is the displacement and
$
F=\nabla \varphi=\id +\nabla u
$ is the gradient of the deformation $\varphi:\mathbb{R}^n\rightarrow\mathbb{R}^n$ and h.o.t. denotes higher order terms of $\|\,{\rm dev}_n\log U\|^2$ and
$\frac{\kappa}{2}\,[(\log \det U)]^2$.

\begin{remark}\label{remarSU}
\begin{itemize}
\item[]
\item[i)] If $F\mapsto W(F)$ is rank-one convex  in ${\rm GL}^+(n)$ and if $Z:\R_+\rightarrow \R$ is a convex and  monotone non-decreasing function, then
    the composition function $F\mapsto (Z\circ W)(F)$ is also rank-one convex in ${\rm GL}^+(n)$. This follows from the fact that if  $t\mapsto h(t),\
    t\in\R$,
$
h(t)=W(F+t(\eta \otimes\xi))
$
is convex, then $t\mapsto Z({h(t)}), \ t\in\R${,} is also convex.
\item[ii)] If $F\mapsto W(F)$ is quasi-convex in ${\rm GL}^+(n)$ and if $Z:\R_+\rightarrow\R$ a convex and  monotone non-decreasing function, then the
    function $F\mapsto (Z\circ W)(F)$ is also quasi-convex in ${\rm GL}^+(n)$. To prove this fact, let us recall that quasiconvexity of the energy
    function $W$ at $\overline{F}$   means that
$
\frac{1}{|\Omega|}\int_{\Omega}W(\overline{F}+\nabla \vartheta)dx\geq W(\overline{F}),$ $\text{holds, for every bounded open set} \quad
\Omega\subset\mathbb{R}^n
$
 and for all $ \vartheta\in C_0^\infty (\Omega)$ such that $\det(\overline{F}+\nabla \vartheta)>0$. Using the monotonicity of $Z$ we deduce $
Z\left(\frac{1}{|\Omega|}\int_{\Omega}W(\overline{F}+\nabla \vartheta)dx\right)\geq Z( W(\overline{F})).
$  Hence, using the convexity and  Jensen's inequality, we obtain $
\frac{1}{|\Omega|}\int_{\Omega}Z\left(W(\overline{F}+\nabla \vartheta)\right)dx\geq Z( W(\overline{F})).
$
\item[iii)] If  $F\mapsto W(F)$ is polyconvex in ${\rm GL}^+(3)$ and if $Z:\R_+\rightarrow\R$ is a convex and   monotone non-decreasing function, then
    the  function $F\mapsto(Z\circ W)(F)$ is also polyconvex in ${\rm GL}^+(3)$. A free energy function $W(F)$ is called  polyconvex if and
only if it is expressible in the form
$W(F) =P(F, \Cof F, \det F)$, $P:\mathbb{R}^{19}\rightarrow\mathbb{R}$, where $P(\cdot,\cdot,\cdot)$ is convex. If $P$ is convex, {then} $e^P$ is also
convex. In this case, we have $Z({W(F)}) =(Z\circ P)(F, \Cof F, \det F)$, which means that $F\mapsto (Z\circ W)(F)$ is polyconvex in ${\rm GL}^+(3)$.
\end{itemize}
\end{remark}

An example of  a convex and  monotone non-decreasing function $Z:\R_+\rightarrow \R$ is the exponential function $Z(\xi)=e^\xi$.

We prove in this subsection that although $F\mapsto \|\dev_2\log U\|^2$  is not rank-one convex, the function \break $F\mapsto e^{k\,\|\dev_2\log U\|^2}$,
$k>\frac{1}{4}$  is indeed rank-one convex.

\begin{lemma}
\label{thm:w}
 Let $F\in\GLpz$ with singular values $\lam_1,\lam_2$. Then
 \begin{align}\label{wf3}
W(F)=e^{k\,\|{\rm dev}_2 \log U\|^2}=e^{k\|\log \frac{U}{\det U^{1/2}}\|^2}=g(\lambda_1,\lambda_2),\ \ \text{where}\
g:\mathbb{R}^2_+\rightarrow\mathbb{R},\ \  g(\lambda_1,\lambda_2):=e^{\frac{k}{2}\left(\log \frac{\lambda_1}{\lambda_2}\right)^2}.
\end{align}
\end{lemma}
\begin{proof} The matrix $U$ is positive definite and symmetric and therefore can be assumed diagonal, and we obtain
\begin{align*}
 \norm{\dev_2\log U}^2&=\norm{\log U-\tel2(\log \lam_1+\log \lam_2)\id}^2\\
 &=\norm{\matr{\tel2\log\lam_1-\tel2\log\lam_2&0\\0&\tel2\log\lam_2-\tel2\log\lam_1}}^2=\tel4\left[2(\log
 \lam_1-\log\lam_2)^2\right]=\tel2\left(\log\frac{\lam_1}{\lam_2}\right)^2.
\end{align*}
With this, the proof is complete.
\end{proof}

In this subsection we apply Theorem  \ref{silhavy318} in order to prove that the function $F\mapsto e^{k\,\|\dev_2\log U\|^2}$ is LH-elliptic. Thus,
according to Lemma \ref{thm:w}, we have to prove that the function
$$g:\mathbb{R}^2_+\rightarrow\mathbb{R},\ \  g(\lambda_1,\lambda_2):=e^{\frac{k}{2}\left(\log \frac{\lambda_1}{\lambda_2}\right)^2}$$
 satisfies all the necessary and sufficient conditions established by Knowles and Sternberg's Theorem \ref{silhavy318}.  The first condition  from Theorem
 \ref{silhavy318} requests separate convexity in each variable $\lambda_1,\lambda_2$.

\begin{lemma}\label{lemmasLH}The function $g$ is  separately convex  in each variable $\lambda_1,\lambda_2$, i.e.
$
\frac{\partial ^2 g}{\partial \lambda_1^2}\geq 0,\ \frac{\partial ^2 g}{\partial \lambda_2^2}\geq 0,\notag
$
if and only if $k\geq \frac{1}{4}$.
\end{lemma}
\begin{proof}
We need to compute
\begin{align}
&\frac{\partial  g}{\partial \lambda_1}=\frac{ k \log \frac{\lambda_1}{\lambda_2} e^{\frac{k}{2} \log ^2\frac{\lambda_1}{\lambda_2}}}{\lambda_1},
\quad\quad \frac{\partial  g}{\partial \lambda_2}=-\frac{ k \log \frac{\lambda_1}{\lambda_2} e^{\frac{k}{2} \log
^2\frac{\lambda_1}{\lambda_2}}}{\lambda_2}\,,\\
&\frac{\partial ^2 g}{\partial \lambda_1^2}=\frac{k e^{\frac{k}{2} \log ^2\frac{\lambda_1}{\lambda_2}}}{\lambda_1^2}\left(k \log
^2\frac{\lambda_1}{\lambda_2}-\log \frac{\lambda_1}{\lambda_2}+1\right),\quad
\frac{\partial ^2 g}{\partial \lambda_2^2}=\frac{k e^{\frac{k}{2} \log ^2\frac{\lambda_1}{\lambda_2}}}{\lambda_2^2}\left(k \log
^2\frac{\lambda_1}{\lambda_2}+\log \frac{\lambda_1}{\lambda_2}+1\right).\notag
\end{align}

We introduce the function $r:\R\rightarrow\mathbb{R}$ given by $ r(t)= k\, t^2-  t+1.
$ It is clear that if $k\geq\dd\frac{1}{4}$, then
$
 r(t)= k\, t^2-  t+1\geq\left(\frac{1}{2}t-1\right)^2{\geq0 \  }$ { for all }$\, t\in\mathbb{R}.
$
Moreover, if $r(t)\geq0$ for all $t\in \R$, then
$
k\geq\frac{1}{4}=\max\limits_{t\in(0,\infty)}\left\{\frac{t-1}{t^2}\right\}.
$
Thus, $r(t)\geq 0$ for all $t\in\R$ if and only if  $k\geq\dd\frac{1}{4}$. In consequence, we deduce
\begin{align}\label{devsepconv}
 &\frac{\partial ^2 g}{\partial \lambda_1^2}(\lambda_1,\lambda_2)=k\, e^{\frac{k}{2} \log ^2\frac{\lambda_1}{\lambda_2}}
 \frac{1}{\lambda_1^2}\,r\left(\log\left(\frac{\lambda_1}{\lambda_2}\right)\right)\geq0\ \ \ \text{if and only if} \ \ \ k\geq\dd\frac{1}{4}.
\end{align}
Analogously, we have
$
 \frac{\partial ^2 g}{\partial \lambda_2^2}(\lambda_1,\lambda_2)\geq0\ \ \ \text{if and only if} \ \ \ k\geq\dd\frac{1}{4}.
$
\end{proof}

\begin{lemma}
The function $g$ satisfies the BE-inequalities.
\end{lemma}
\begin{proof} For the function $g$ defined by \eqref{wf3}, the BE-inequalities become
\begin{align}\label{BELH}
\frac{\lambda_1 \frac{\partial g}{\partial \lambda_1}-\lambda_2 \frac{\partial g}{\partial \lambda_2}}{\lambda_1-\lambda_2}=\frac{2k \log
\frac{\lambda_1}{\lambda_2} e^{\frac{k}{2} \log ^2\frac{\lambda_1}{\lambda_2}}}{\lambda_1-\lambda_2}\geq {0 \qquad  \text{if}}\quad  \lambda_1\neq
\lambda_2,
\end{align}
which is always true.
{Indeed, this fact also }follows directly from Theorem \ref{conlBE} because $g$ is convex as a function of $\log U$ (see Remark \ref{remarkconlog}).
\end{proof}
Let us also compute
\begin{align}
&\frac{\partial^2  g}{\partial \lambda_1\partial \lambda_2}=-\frac{ k e^{\frac{k}{2} \log ^2\frac{\lambda_1}{\lambda_2}}}{\lambda_1 \lambda_2}\left( k \log
^2\frac{\lambda_1}{\lambda_2}+1\right).
\end{align}
The next set of inequalities from Knowles and Sternberg's criterion requires  that the following quantities
\begin{align}
\frac{\partial ^2 g}{\partial \lambda_1^2}-\frac{\partial ^2 g}{\partial \lambda_1\, \partial \lambda_2}+\frac{1}{\lambda_1}\frac{\partial g}{\partial
\lambda_1}=\frac{k (\lambda_1+\lambda_2) e^{\frac{k}{2} \log ^2\frac{\lambda_1}{\lambda_2}} \left(k \log
^2\frac{\lambda_1}{\lambda_2}+1\right)}{\lambda_1^2 \lambda_2},\\
\frac{\partial ^2 g}{\partial \lambda_2^2}-\frac{\partial ^2 g}{\partial \lambda_1\, \partial \lambda_2}+\frac{1}{\lambda_2}\frac{\partial g}{\partial
\lambda_2}=\frac{k (\lambda_1+\lambda_2) e^{\frac{k}{2} \log ^2\frac{\lambda_1}{\lambda_2}} \left(k \log ^2\frac{\lambda_1}{\lambda_2}+1\right)}{\lambda_1
\, \lambda_2^2},\notag
\end{align}
are positive for $\lambda_1=\lambda_2$. This condition is always satisfied because $\lambda_1,\,\lambda_2,k>0$.

In order to show that the last  two inequalities from  Knowles and Sternberg's Theorem \ref{silhavy318} are satisfied, we compute
\begin{align}
&\sqrt{\frac{\partial ^2 g}{\partial \lambda_1^2} \frac{\partial ^2g}{\partial \lambda_2^2}}+\frac{\partial ^2 g}{\partial \lambda_1\, \partial
\lambda_2}+\frac{\frac{\partial g}{\partial \lambda_1}-\frac{\partial g}{\partial \lambda_2}}{\lambda_1-\lambda_2}=\frac{k e^{\frac{k}{2} \log
^2\frac{\lambda_1}{\lambda_2}} }{\lambda_1 \lambda_2} \widetilde{g}(\lambda_1,\lambda_2),\quad \lambda_1\neq \lambda_2,
\\
&\sqrt{\frac{\partial ^2 g}{\partial \lambda_1^2} \frac{\partial ^2g}{\partial \lambda_2^2}}-\frac{\partial ^2 g}{\partial \lambda_1\, \partial
\lambda_2}+\frac{\frac{\partial g}{\partial \lambda_1}+\frac{\partial g}{\partial \lambda_2}}{\lambda_1+\lambda_2}=\frac{ k e^{\frac{k}{2} \log
^2\frac{\lambda_1}{\lambda_2}} }{\lambda_1 \lambda_2 } \widehat{g}(\lambda_1,\lambda_2)\notag,
\end{align}
where the functions $\widetilde{g}:\R_+^2\setminus\{(x,x){; x\in\R}\}\to \R$, $\widehat{g}:\R_+^2\to \R$ are defined by

\begin{align}
&\widetilde{g}(\lambda_1,\lambda_2)= \sqrt{ \left(k \log ^2\left(\frac{\lambda_1}{\lambda_2}\right)+1\right)^2- \log ^2\frac{\lambda_1}{\lambda_2}}-k  \log
^2\frac{\lambda_1}{\lambda_2}-1+\frac{(\lambda_1+\lambda_2)}{ (\lambda_1-\lambda_2)} \log \frac{\lambda_1}{\lambda_2},
\\
&\widehat{g}(\lambda_1,\lambda_2) =\sqrt{ \left(k \log ^2\left(\frac{\lambda_1}{\lambda_2}\right)+1\right)^2- \log ^2\frac{\lambda_1}{\lambda_2}}+ k  \log
^2\frac{\lambda_1}{\lambda_2}+1-\frac{(\lambda_1-\lambda_2)}{ (\lambda_1+\lambda_2)} \log \frac{\lambda_1}{\lambda_2}.\notag
\end{align}

Let us remark that the functions $\widetilde{g}$ and $\widehat{g}$ can be written in terms of   functions of a single variable only, i.e.
\begin{align}
&\widetilde{g}(\lambda_1,\lambda_2)= \widetilde{r}\left(\frac{\lambda_1}{\lambda_2}\right),\qquad \widehat{g}(\lambda_1,\lambda_2)
=\widehat{r}\left(\frac{\lambda_1}{\lambda_2}\right),
\end{align}
where $\widetilde{r}:\R_+\setminus\{1\}\to\R$,  $\widehat{r}:\R_+\to\R$ are defined by
\begin{align}
&\widetilde{r}(t)= \sqrt{ \left(k \log ^2 t+1\right)^2-\log ^2 t }-(k  \log ^2\,t+1)+\frac{t+1}{ t-1} \log t,
\\
&\widehat{r}(t) =\sqrt{ \left(k \log ^2 t +1\right)^2-\log ^2 t }+ (k  \log ^2\, t+1)-\frac{t-1}{ t+1} \log \,t.\notag
\end{align}
Hence,  Knowles and Sternberg's criterion is completely satisfied if and only if
\begin{align}
\label{eq:rtildegeqnull}
\widetilde{r}(t){\geq}0 \quad \qquad \text{for all} \quad t\in\R_+\setminus\{1\}\quad \text{and}\qquad
\widehat{r}(t){\geq}0 \quad \qquad \text{for all} \quad t\in\R_+.
\end{align}

We have to {show} the following inequality, which is the same as \eqref{eq:rtildegeqnull}$_1$:
\begin{align}\label{lh10}
\sqrt{ \left(k \log ^2 t+1\right)^2-\log ^2 t }\,\,{+1}\geq\Big(k  \log ^2\,t+1\Big)-\frac{t+1}{ t-1} \log t{+1}.
\end{align}
In order to transform it equivalently by squaring both sides, first we prove the following lemma:
\begin{lemma}\label{lemmaLH2} The inequality
\begin{align}\label{lh1}
 \Big(k  \log ^2\,t+1\Big)-\frac{t+1}{ t-1} \log t {+ 1\geq 0}
\end{align}
is satisfied for all $t\in\R_+\setminus\{1\}$ if and only if $k\geq \frac{1}{6}$.
 \end{lemma}
 \begin{proof}
 Let us consider the function $\tilde{s}:\R{_+}\setminus\{1\}\to\R$ by
$
\widetilde{s}(t):=\left(\frac{1}{6}  \log ^2\,t+1\right)-\frac{t+1}{ t-1} \log t.
$
For the function $\widetilde{s}$ we compute
\begin{align}\label{shsw}
\widetilde{s}^\prime(t)=\frac{1 }{3 (-1 + t)^2 t}\,\,\widehat{s}(t),
\end{align}
where $\widehat{s}:\R_+\to\R$,
$
\widehat{s}(t)=3 (1-  t^2) + (1 +  4\,t + t^2) \log t.
$

On the other hand
$
\widehat{s}^\prime(t)=-5\, t+\frac{1}{t}+2\, (t+2)\, \log t+4,\   \widehat{s}^{\prime\prime}(t)=\frac{4\, t-1}{t^2}+2\, \log t-3,
\ \widehat{s}^{\prime\prime\prime}(t)=\frac{2\, (t-1)^2}{t^3}\geq0 \ \text{for all}\   t\in {\R_+},\
\widehat{s}(1)=0,\  \widehat{s}^\prime(1)=0, \  \widehat{s}^{\prime\prime}(1)=0.
$
Thus
$
\widehat{s}^{\prime\prime}(t)\geq 0  \  \text{if}\   t\geq 1\ \ \text{and}\    \widehat{s}^{\prime\prime}(t)\leq 0 \ \text{if}\quad {0<}t\leq 1,
$
which implies further that $\widehat{s}^{\prime}$ is monotone decreasing on $(0,1)$ and monotone increasing on $(1,\infty)$. We deduce
$
\widehat{s}^{\prime}(t)\geq\widehat{s}^{\prime}(1)= 0 \   \text{for all }\   0<t\leq1\ \ \text{and}\
\widehat{s}^{\prime}(t)\geq\widehat{s}^{\prime}(1)= 0 \   \text{for all }\   t\geq 1.
$

Hence, $\widehat{s}$ is monotone increasing in $\R_+$, i.e.
$
\widehat{s}(t)\leq \widehat{s}(1) =0 \   \text{for all }\   0<t<1\   \text{and}\   \widehat{s}(t)\geq \widehat{s}(1) =0\  \text{for all }\   t> 1.
$
In view of \eqref{shsw}, we have
$
\widetilde{s}^{\prime}(t)\leq 0 \   \text{for all }\  0<t<1\   \text{and}\   \widetilde{s}^{\prime}(t)\geq 0 \   \text{for all }\   t> 1.
$
Because
$
\lim\limits_{t\rightarrow 1}\widetilde{s}(t)=-1,
$
the monotonicity of $\widetilde{s}(t)$ implies
$
\widetilde{s}(t)=\left(\frac{1}{{6}}  \log ^2\,t+1\right)-\frac{t+1}{ t-1} \log t\geq \lim\limits_{t_0\rightarrow 1}\widetilde{s}(t_0)=-1, \ \ \text{for
all} \ \ t\in\R_+\setminus\{1\}.
$
For $k{\geq}\frac{1}{6}$, we have
\begin{align}
(k  \log ^2\,t+1)-\frac{t+1}{ t-1} \log t\geq\left(\frac{1}{6}  \log ^2\,t+1\right)-\frac{t+1}{ t-1} \log t\geq -1\,,
\end{align}
for all $t\in\R_+\setminus\{1\}$. On the other hand,  if
$
(k  \log ^2\,t+1)-\frac{t+1}{ t-1} \log t\geq -1
$
 for all $t\in \R_+\setminus\{1\}$, then
\begin{align}\label{conditie13}
k\geq\frac{1}{6}=\sup\limits_{t\in\R_+}\left\{\frac{1}{\log^2 t}\left(-2+\frac{t+1}{t-1}\log t\right)\right\},
\end{align}
since the function $\widetilde{s}_1:\R_+\rightarrow \R$, $\widetilde{s}_1(t)=\log^2t-6\, \frac{t+1}{t-1}\,\log t+12$ is monotone decreasing on $(0,1]$,
monotone increasing on $[1, \infty)$, $\widetilde{s}_1(1)=0$, and $\lim\limits_{t\in\R_+}\left\{\frac{1}{\log^2 t}\left(-2+\frac{t+1}{t-1}\log
t\right)\right\}=\frac{1}{6}$. Thus, the inequality \eqref{lh1} holds for all $t\in\R_+\setminus\{1\}$ if and only if  $k\geq\dd\frac{1}{6}$,
\end{proof}

\begin{lemma}\label{lemmaLH1} The inequality
$
 \widetilde{r}(t)\geq 0
$
is satisfied for all $t\in\R_+\setminus\{1\}$   if \, $k\geq\frac{1}{4}$.
 \end{lemma}
 \begin{proof} Let us first remark that, in view of Lemma \ref{lemmaLH2},  we obtain
$
 \Big(k  \log ^2\,t+1\Big)-\frac{t+1}{ t-1} \log t+1\geq 0
$
for all  $k\geq \frac{1}{6}$. Hence, the  inequality  $\widetilde{r}(t)\geq 0$  is equivalent to the inequality
\begin{align}\label{dlg2}
\sqrt{ \left(k \log ^2 t+1\right)^2-\log ^2 t }+1\geq\Big(k  \log ^2\,t+1\Big)-\frac{t+1}{ t-1} \log t+1\geq 0,
\end{align}
 for $t\in\R_+\setminus\{1\}$, which can, by squaring and multiplication with $\frac{(t-1)^2}2$, equivalently  be written in the following form:
\begin{align}\label{dlg}
 k \, (t-1) [1-t+(t+1) \log t]\log ^2 t &- \left\{ \left[2\,(1- t^2)+\left(t^2+1\right) \log t\right]\log t+(t-1)^2\right\}\\&\qquad\qquad\quad+
 \,{(t-1)^2}\,\sqrt{ \left(k \log ^2 t+1\right)^2-\log ^2 t }\geq 0.\notag
 \end{align}
 Our next step is to prove that
$
s(t)\leq 0 \ \  \text{if} \ \  t<1\ \ \text{and}\ \
 s(t)\geq 0 \ \  \text{if} \quad t> 1,
$
 where $s:\R_+\to\R$ is defined by
$
  s(t)=1-t+(t+1) \log t.
 $
 This follows from
 $
 s^\prime(t)=\frac{1}{t}+\log t,\ \
 s^{\prime\prime}(t)=\frac{t-1}{t^2}, \ \  s^\prime(1)=1, \ \  s(1)=0.
$  Moreover, if $k\geq\frac{1}{4}$, we deduce
\begin{align}
\sqrt{ \left(k \log ^2 t+1\right)^2-\log ^2 t }\geq\sqrt{ \left(\frac{1}{4}\log ^2 t+1\right)^2-\log ^2 t }=\sqrt{ \left(\frac{1}{4}\,\log ^2 t-1\right)^2
}=\left|\frac{1}{4}\,\log ^2 t-1\right|,
\end{align}
and, due to the nonnegativity of $(t-1)s(t)$,
\begin{align}
 k \, (t-1)\, [1-t&+(t+1) \log t]\,\log ^2 t - \left\{ \left[2(1- t^2)+\left(t^2+1\right) \log t\right]\,\log t+(t-1)^2\right\}\\&\geq  \frac{1}{4} \,
 (t-1) [1-t+(t+1) \log t]\,\log ^2 t - \left\{ \left[2\,(1- t^2)+\left(t^2+1\right)\, \log t\right]\,\log t+(t-1)^2\right\}\notag\\
 &=\frac{1}{4}\,\left\{ 8 \left(t^2-1\right)\,\log t+ \left(t^2-1\right)\, \log^3 t+(-5\, t^2+2\, t-5)\,\log^2 t-4\, (t-1)^2\right\}.\notag
\end{align}

Hence, it is sufficient to prove that
\begin{align}\label{inegci}
{(t-1)^2}&\,\left(\frac{1}{4}\,\log ^2 t-1\right)+\frac{1}{{4}}\,\left\{ 8 \,\left(t^2-1\right)\,\log t+ \left(t^2-1\right)\, \log^3 t+(-5\, t^2+2\,
t-5)\,\log^2 t-4\, (t-1)^2\right\}\\
 &\qquad\qquad=\frac{t^2-1}{4}\, \left(\log^3 t-4 \log^2 t + 8\log t -\frac{8(t-1)}{t+1} \right)\geq 0.\notag
\end{align}

Employing the substitution $x=\log t$, we are   going to show that
\[
 s_0(x)=x^3-4\,x^2+8\,x-8\,\frac{e^x-1}{e^x+1}=x^3-4\,x^2+8\,x-8+\frac{16}{e^x+1}
\]
is negative for $x<0$ and positive for $x>0$.

Firstly, we observe that $s_0(0)=0$ and $\lim_{x\to\infty} s_0(x)=\infty$. We then compute
$s_0'(x)=s_1(x)-s_2(x)$, where we denote $s_1(x)=3(x-\frac43)^2+\frac83$, $s_2(x)=16\frac{e^x}{(e^x+1)^2}$.
Due to the fact that $\frac{y}{(1+y)^2}\in(0,4]$ for $y=e^x>0$,
\[
 s_0'(x)\geq 3(\frac43)^2+\frac83-4=4>0 \qquad \text{for }x<0,
\]
so that clearly $s_0(x)<0$ for $x<0$.
To deduce $s_0(x)>0$ for $x>0$, we will prove that all local minima of $s_0$ are located in $(1,\infty)$ and that the value of $s_0$ is positive there.
Because $s_2''(x)=\frac{e^x}{(1+e^x)^4}(1-4e^x+e^{2x})$ is negative on $(-\infty,\log(\sqrt{3}+2))\supset(0,1)$ and hence $s_2$ is concave and $s_1$ convex
on $(0,1)$, $s_1$ and $s_2$ can intersect in at most two points in $(0,1)$.
{Thanks to the fact that $s_1(0)>s_2(0)$ and $s_1(1)<s_2(1)$,
there is only one $x_m\in(0,1)$, where $s_1(x_m)=s_2(x_m)$ and hence $s_0'(x_m)=0$.} In $x_m$, $s_0$ attains a maximum ($s_0'$ is positive for smaller and
negative for larger values of $x$), hence local minima of $s$ must lie in $(1,\infty)$.
In any such place $x_0$, from $s_0'(x_0)=0$ we know $\dd\frac{16}{e^{x_0}+1}=\frac{e^{x_0}+1}{e^{x_0}}(3x_0^2-8x_0+8)$ and hence
\[
 s_0(x_0)=x_0^3-4x_0^2+8x_0-8+(1+e^{-x_0})(3x_0^2-8x_0+8)
=x_0^2(x_0-1)+3e^{-x_0}((x_0-\frac43)^2+\frac83) > 0,
\]
because $x_0\geq 1$. In conclusion, $s_0$ is positive on all of $(0,\infty)$, and negative in $(-\infty,0)$.

Thus, the inequality \eqref{inegci} is satisfied. Therefore \eqref{dlg2} is also satisfied and the proof is complete. \qedhere
\end{proof}

\begin{lemma}\label{lemmaLH3} If $k\geq\frac{1}{4}$, then the inequality
$
 \widehat{r}(t)\geq  0
$
is satisfied for all $t\in\R_+$.
 \end{lemma}
 \begin{proof}

It is easy to see that for all
$t\in\R_+\setminus\{1\}$ and if $k\geq \frac{1}{4}$, we have
\begin{align}
&(k\,  \log ^2\, t+1)-\frac{t-1}{ t+1} \log \,t\geq \frac{1}{4} \, \log ^2\, t-\frac{t-1}{ t+1}\, \log \,t+1\,.\notag
\end{align}
Let us remark that
$
\frac{1}{4}\,  \xi^2-\frac{t-1}{ t+1}\, \xi+1>0 \ \ \text{for all} \ \ \xi\in \R,
$
since
$
\left(\frac{t-1}{ t+1}\right)^2-1<0 \ \
\text{and} \ \ \frac{1}{4}>0.
$
Hence, taking $\xi=\log t\in\R$, we  have
$
\frac{1}{4} \, \log ^2\, t-\frac{t-1}{ t+1}\, \log \,t+1>0 \ \ \text{for all} \ \ t\in \R_+.
$
Therefore,
\begin{align}
&\widehat{r}(t) =\sqrt{ \left(k\, \log ^2 t +1\right)^2-\log ^2 t }+ (k \, \log ^2\, t+1)-\frac{t-1}{ t+1}\, \log \,t>0 \qquad \text{for all} \qquad t\in
\R_+\,,\notag
\end{align}
which completes the proof.
\end{proof}

Collecting  Lemmas \ref{lemmasLH}, \ref{lemmaLH1}, \ref{lemmaLH3} and Eq. \eqref{BELH}, we  can finally  conclude:
\begin{proposition}\label{rocS}
If $k\geq\frac{1}{4}$, then the function $F\mapsto e^{k\,\| \dev_2 \log U\|^2}$ is rank-one convex in ${\rm GL}^+(2)$.
\end{proposition}

\subsection{The main  rank-one convexity statement}

In view of the results established in Subsection \ref{rank-isochoric} and \ref{polytr} we conclude that:

\begin{theorem} {\rm (planar rank-one convexity)} The functions $W_{_{\rm eH}}:\R^{n\times n}\to \overline{\R}_+$ from the family of exponentiated Hencky type
energies
\begin{align}\label{thdefHen}\hspace{-2mm}
 W_{_{\rm eH}}(F)=W_{_{\rm eH}}^{\text{\rm iso}}(\frac F{\det F^{\frac{1}{n}}})+W_{_{\rm eH}}^{\text{\rm vol}}(\det F^{\tel n}\cdot \id) = \left\{\begin{array}{lll}
\dd\frac{\mu}{k}\,e^{k\,\|{\rm dev}_n\log U\|^2}+\frac{\kappa}{2\widehat{k}}\,e^{\widehat{k}\,[(\log \det U)]^2}&\text{if}& \det\, F>0,\vspace{2mm}\\
+\infty &\text{if} &\det F\leq 0,
\end{array}\right.
\end{align}
are {\bf rank-one convex} for the two-dimensional situation $n=2$, $\mu>0, \kappa>0$, $k\geq\dd\frac{1}{4}$ and $\widehat{k}\dd\geq \tel8$.
\end{theorem}

\begin{conjecture} {\rm (planar polyconvexity)} The functions $W_{_{\rm eH}}:\R^{n\times n}\to \overline{\R}_+$ from the family of exponentiated Hencky type
energies defined by \eqref{thdefHen}
are  {\bf polyconvex}\footnote{We use the definition of polyconvexity given by Ball \cite{Ball77} (see also
\cite{Schroeder_Neff_Ebbing07,Schroeder_Neff01}). Polyconvexity implies LH-ellipticity and may lead to an existence theorem based on the direct methods of
the calculus of variations, provided that proper growth conditions are satisfied
\cite{Hartmann_Neff02,Balzani_Neff_Schroeder05,Neff_critique05,Neff_Cosserat_plasticity05,balzani2006materially}. } for the two-dimensional situation
$n=2$, $\mu>0, \kappa>0$, $k\geq\dd\frac{1}{4}$ and $\widehat{k}\dd\geq \tel8$.
\end{conjecture}

In plane elasto-statics, the rank-one convex energy $W_{_{\rm eH}}(F)$ is applicable to the bending or shear of long strips and to all cases in which symmetry
arguments can be applied to reduce the formulation to a planar deformation.

\subsection{Formulation of the dynamic problem in the planar case}
For the convenience of the reader we state the complete dynamic setting. The dynamic problem in the planar case consists in finding the solution
$\varphi:\Omega\times (0,\infty)\rightarrow\mathbb{R}^2$, $\Omega\subset\mathbb{R}^2$ of the equation of motion
\begin{align}
{\varphi}_{{}_{,tt}}={\rm Div} \,S_1(\nabla \varphi)\qquad  \text{in} \qquad \Omega\times (0,\infty),
\end{align}
where the first Piola-Kirchhoff stress tensor $S_1=D_F[W(F)]$ corresponding to the energy $W_{_{\rm eH}}(F)$ is given by the constitutive equation
\begin{align}
S_1&=D_F[W(F)]=J\,\sigma\, F^{-T}=\tau\, F^{-T}\notag\\
&=\left[2\,{\mu}\,e^{k\,\|\dev_2\log\,U\|^2}\cdot \dev_2\log\,U+{\kappa}\,e^{\widehat{k}\,[\tr(\log U)]^2}\,\tr(\log U)\cdot \id\right]F^{-T}  \qquad
\text{in} \qquad \overline{\Omega}\times [0,\infty),
\end{align}
with $F=\nabla \varphi, \ U= \sqrt{F^TF}$. The above equations are supplemented, in the case of the mixed problem,  by the boundary conditions
\begin{align}\label{cl1}
{\varphi}({x},t)&=\widehat{\varphi}_i({x},t) \qquad \text{ on  }\qquad \Gamma_D\times [0,\infty),
\\\
{S}_1({x},t).\,n&=\widehat{s}_1({x},t) \qquad \text{ on  }\qquad \Gamma_N\times [0,\infty),\notag
\end{align}
and the initial conditions
\begin{align}
\varphi(x,0)=\varphi_0(x), \qquad {\varphi}_{{}_{,t}}(x,0)={\psi}_0(x) \qquad \text{in} \qquad \Omega,
\end{align}
where $\Gamma_D,\Gamma_N$  are subsets of the boundary $\partial \Omega$, so that $\Gamma_D\cup\overline{\Gamma}_N=\partial \Omega$,
$\Gamma_D\cap{\Gamma}_N=\emptyset$, ${n}$ is the unit outward normal to the boundary and  $\widehat{\varphi}_i, \widehat{s}_1, \varphi_0, {\psi}_0$ are
prescribed fields.

\subsection{The non-deviatoric planar case: $F\mapsto e^{\|\log U\|^2}$}
\label{Buligalog}
We consider the function $W:{\rm GL}^+(2)\rightarrow\mathbb{R}$, defined by $W(F):=\widehat{W}(U)=e^{\|\log U\|^2}$. We have
\begin{align}
e^{\|\log U\|^2}=g(\lambda_1,\lambda_2),
\end{align}
where $\lambda_1,\lambda_2$ are the singular values of $U$ and $g:\mathbb{R}_+^2\rightarrow\mathbb{R}$ is defined by
\begin{align}
  g(\lambda_1,{\lambda_2})=e^{\log^2\lambda_1+\log^2\lambda_2}.
\end{align}

In order to check the rank-one convexity of the function $F\mapsto e^{\|\log U\|^2}$, we will use  Buliga's criterion given by Theorem \ref{Buligacrit}. As
we will need the derivatives of $g$, we compute:
\begin{align*}
 \frac{\partial g}{\partial {\lambda}_1}&=e^{\log^2 {\lambda}_1+\log^2{\lambda}_2}\frac{2\log  {\lambda}_1}{ {\lambda}_1}, \quad
 \frac{\partial g}{\partial {\lambda}_2}=e^{\log^2 {\lambda}_1+\log^2{\lambda}_2}\frac{2\log {\lambda}_2}{{\lambda}_2},\\
 \frac{\partial^2 g}{\partial {\lambda}_1^2}&=e^{\log^2 {\lambda}_1+\log^2{\lambda}_2}\left(\frac{4\log^2  {\lambda}_1}{ {\lambda}_1^2}+\frac{2-2\log
{\lambda}_1}{ {\lambda}_1^2}\right),\\
 \frac{\partial^2 g}{\partial {\lambda}_1 \partial {\lambda}_2}&=e^{\log^2 {\lambda}_1+\log^2{\lambda_2}}\,\frac{4\log  {\lambda}_1\,\log
 {\lambda}_2}{{\lambda}_1{\lambda}_2},\\
 \frac{\partial^2 g}{\partial {\lambda}_2^2}&=e^{\log^2 {\lambda}_1+\log^2{\lambda}_2}\left(\frac{4\log^2 {\lambda}_2}{{\lambda}_2^2}+\frac{2-2\log
 {\lambda}_2}{{\lambda}_2^2}\right).
\end{align*}

For our function, the matrices $G({\lambda_1},{\lambda_2})$ and $H({\lambda_1},{\lambda_2})$ from Theorem \ref{Buligacrit} are then
\begin{align}
 G({\lambda_1},{\lambda_2})&=2e^{\log^2{\lambda_1}+\log^2{\lambda_2}}\matr{0&\frac{\log {\lambda_1}-\log
 {\lambda_2}}{{\lambda_1}^2-{\lambda_2}^2}\\\frac{\log {\lambda_1}-\log {\lambda_2}}{{\lambda_1}^2-{\lambda_2}^2}&0},\notag\\
 H({\lambda_1},{\lambda_2})&=2e^{\log^2{\lambda_1}+\log^2{\lambda_2}}\left[ \matr{0&\frac{\frac{{\lambda_2}\log
 {\lambda_1}}{{\lambda_1}}-\frac{{\lambda_1}\log {\lambda_2}}{{\lambda_2}}}{{\lambda_1}^2-{\lambda_2}^2}\\\frac{\frac{{\lambda_2}\log
 {\lambda_1}}{{\lambda_1}}-\frac{{\lambda_1}\log {\lambda_2}}{{\lambda_2}}}{{\lambda_1}^2-{\lambda_2}^2}&0}+\matr{\frac{2\log^2{\lambda_1}-\log
 {\lambda_1}+1}{{\lambda_1}^2}&\frac{2\log {\lambda_1}\log {\lambda_2}}{{\lambda_1}{\lambda_2}}\\\frac{2\log {\lambda_1}\log
 {\lambda_2}}{{\lambda_1}{\lambda_2}}&\frac{2\log^2{\lambda_2}-\log {\lambda_2}+1}{{\lambda_2}^2}}\right]\notag\\
&=2e^{\log^2{\lambda_1}+\log^2{\lambda_2}}\matr{\frac{2\log^2{\lambda_1}-\log {\lambda_1}+1}{{\lambda_1}^2}&\frac{2\log {\lambda_1}\log
{\lambda_2}}{{\lambda_1}{\lambda_2}}+\frac{\frac{{\lambda_2}\log {\lambda_1}}{{\lambda_1}}-\frac{{\lambda_1}\log
{\lambda_2}}{{\lambda_2}}}{{\lambda_1}^2-{\lambda_2}^2}\\\frac{2\log {\lambda_1}\log {\lambda_2}}{{\lambda_1}{\lambda_2}}+\frac{\frac{{\lambda_2}\log
{\lambda_1}}{{\lambda_1}}-\frac{{\lambda_1}\log {\lambda_2}}{{\lambda_2}}}{{\lambda_1}^2-{\lambda_2}^2}&\frac{2\log^2{\lambda_2}-\log
{\lambda_2}+1}{{\lambda_2}^2}},
\end{align}
respectively.  The first condition of Buliga's criterion is obviously satisfied because of the symmetry and convexity and hence Schur-convexity of the
function $\ell:\R_+^2\to\R$,\ $\ell(\lambda_1,\lambda_2):= g(e^{\lambda_1},e^{\lambda_2})=e^{\lambda_1^2+\lambda_2^2}$ (see Theorem \ref{SchurTh}).

 Hence, the energy is rank-one-convex if and only if the following inequality holds true for all $a_1,a_2\in\R$ and for all ${\lambda_1},{\lambda_2}>0$:
\[
 H_{11}({\lambda_1},{\lambda_2})\,a_1^2+ 2\, H_{12}({\lambda_1},{\lambda_2})\,a_1\,a_2 + H_{22}({\lambda_1},{\lambda_2})\, a_2^2+ 2\,
 G_{12}({\lambda_1},{\lambda_2})\, |a_1\,a_2|\geq 0.
\]

Applied to our function (and upon division by $2\,e^{\log^2{\lambda_1}+\log^2{\lambda_2}}>0$) this corresponds to
\begin{align}
&\left(\frac{2\log^2{\lambda_1}-\log {\lambda_1}+1}{{\lambda_1}^2}\right)a_1^2 + \left(\frac{2\log^2{\lambda_2}-\log
{\lambda_2}+1}{{\lambda_2}^2}\right)a_2^2\\
&\quad+ \left(\frac{2\log {\lambda_1}\log {\lambda_2}}{{\lambda_1}{\lambda_2}}+\frac{\frac{{\lambda_2}\log {\lambda_1}}{{\lambda_1}}-\frac{{\lambda_1}\log
{\lambda_2}}{{\lambda_2}}}{{\lambda_1}^2-{\lambda_2}^2}\right) 2a_1a_2+\frac{\log {\lambda_1}-\log
{\lambda_2}}{{\lambda_1}^2-{\lambda_2}^2}|2\,a_1\,a_2|\geq 0, \forall\, {\lambda_1},{\lambda_2}>0\; \forall\, a_1,a_2\in \R.\notag
\end{align}

To see that this does not hold true, we set
\[\lambda_1=e^2,\quad \lambda_2=e^{11},\quad a_1=-e^{15},\quad a_2=e^{22}.\]
Upon these choices, the inequality turns into
\begin{align*}
0\leq&\frac{2\cdot
2^2-2+1}{e^4}e^{30}+\frac{2\cdot11^2-11+1}{e^{22}}e^{44}-\left(\frac{2\cdot2\cdot11}{e^{13}}+\frac{2e^9-11e^{-9}}{e^4-e^{22}}\right)\cdot2\cdot
e^{37}+\frac{2-11}{e^4-e^{22}}2e^{37}\\
=&7e^{26}+111e^{22}-88e^{24}+\frac{4e^9-22e^{-9}}{e^{18}-1}e^{33}+\frac{18}{e^{18}-1}e^{33}\leq7e^{26}+111e^{22}-88e^{24}+4e^{33+9-17}+18e^{16}\\
\leq&7e^{26}+4e^{25}+112e^{22}-88e^{24}=e^{22}(7e^4+4e^3+112-88e^2)<-75e^{22},
\end{align*}
and it is obviously not satisfied.

In view of Theorem  \ref{Buligacrit}, we conclude that $F\mapsto e^{\|\log U\|^2}$  is not rank-one convex in 2D. Of course, this shows that $F\mapsto
e^{\|\log U\|^2}$ is also not rank-one convex in $3D$.

\begin{conjecture}\label{conj:notsep_rk1} It seems that the function $F\mapsto e^{\|\log U\|^2-\frac{\alpha}{2}{\rm tr}(\log U)^2}$ is
\begin{itemize}
\item[i)]  not separately convex (which implies it is not rank-one convex) for $\alpha> 1$;
\item[ii)] is not rank-one convex for $\alpha<1$.
\end{itemize}
If this conjecture is true, then the function $F\mapsto e^{\|\log U\|^2-\frac{\alpha}{2} {\rm tr}(\log U)^2}$ is rank-one convex in 2D if and only if
$\alpha=1$, i.e. only for the function $F\mapsto e^{\|\dev_2\log U\|^2}$ .
\end{conjecture}
Hence,  the form of energies \eqref{the} may not just  be an arbitrary choice, but additively splitting into isochoric and volumetric parts seems to be the
only useful version of an additive split in plane elasto-statics. The reason to believe that {Conjecture \ref{conj:notsep_rk1}} is true consists in the
fact that for $\lambda_1=e^{n}$ and $\lambda_2=e^{n-3}$, the last inequality from  Knowles and Sternberg's  Theorem \ref{silhavy318} seems to be satisfied
in the limit $n\rightarrow\infty$ only if $\alpha=1$.

\section{Outlook for three dimensions}
\setcounter{equation}{0}

The 3D-case is, as usual, much more involved. In this section we show that a similar calculus as in $2D$ can be  applied in principle. However,  while we
consider the obvious generalization of the 2D result, the answer is in general negative: the necessary conditions from Knowles and Sternberg{'s} Theorem
\ref{silhavy318} or Dacorogna{'s} Theorem \ref{dacorogna5} are not satisfied for the energy
\begin{align}
\widehat{W}(U)=e^{k\,\|\dev_3\log U\|^2}.
\end{align}
This  {implies} that this energy is {\bf not rank-one convex} \cite{Raoult86,Neff_Diss00}.  We have already shown  that $F\mapsto \|\dev_3\log U\|^2$ is
not rank-one convex even in the case of incompressible materials (see Proposition \ref{neffdis} or \cite{Neff_Diss00}, page 197).
\begin{lemma}
\label{thm:w3}
 Let $F\in{\rm GL}^+(3)$  with singular values $\lam_1,\lam_2,\lam_3$. Then
 \begin{align}\label{wf2}
W(F)=g(\lambda_1,\lambda_2,\lambda_3),\ \ \text{where}\  g:\mathbb{R}^3_+\rightarrow\mathbb{R},\ \  g(\lambda_1,\lambda_2,\lambda_3):=
e^{\frac{k}{3}\left[\log^2\frac{\lam_1}{\lam_2}+
\log^2\frac{\lam_2}{\lam_3}+\log^2\frac{\lam_3}{\lam_1}\right]}.
\end{align}
\end{lemma}
{\it Proof.} The proof follows from relation \eqref{exprimaredev3}. \hfill$\Box$\\
This  {l}emma remains true in all dimension $n\in \mathbb{N}$, {see App}endix \ref{identitiesapp}.

\subsection{$F\mapsto e^{k\,\|\dev_3\log U\|^2}$ is not  rank-one convex}

 We begin our 3D investigation by proving that
\begin{lemma}
 For all $k>0$ the function
 \begin{align}
F\mapsto e^{k\,\|\dev_3\log U\|^2},\quad F \in{\rm GL}^+(3)
\end{align}
is not rank-one convex.
\end{lemma}
\begin{proof}
In the following we prove that two necessary conditions given {by Knowle}s and Sternberg's criterion are not satisfied for the function $g$ defined by
\eqref{wf2}.
Our goal is to prove that there does not exist a number $k>0$ such {that} the inequalities
\begin{align}\label{dev3c}
\frac{\partial ^2g}{\partial \lambda_1^2}\geq 0,\qquad\frac{\partial ^2g}{\partial \lambda_2^2}\geq 0,\qquad  \sqrt{\frac{\partial ^2g}{\partial
\lambda_1^2} \frac{\partial ^2g}{\partial \lambda_2^2}}-\frac{\partial ^2g}{\partial \lambda_2\, \partial \lambda_1}+\frac{\frac{\partial g}{\partial
\lambda_1}+\frac{\partial g}{\partial \lambda_2}}{\lambda_1+\lambda_2}\geq 0
\end{align}
are simultaneously  satisfied. The inequalities \eqref{dev3c} are  equivalent {to}
\begin{align}
\frac{2 \frac{k}{3} \,e^{\frac{k}{3} \left(\log ^2\frac{\lambda_1}{\lambda_2}+\log ^2 \frac{\lambda_3}{\lambda_1}+\log
^2\frac{\lambda_2}{\lambda_3}\right)}}{\lambda_1^2}\,&g_1(\lambda_1,\lambda_2,\lambda_3)\geq 0,\qquad
\frac{2 \frac{k}{3}\, e^{\frac{k}{3}\left(\log ^2\frac{\lambda_1}{\lambda_2}+\log ^2 \frac{\lambda_3}{\lambda_1}+\log
^2\frac{\lambda_2}{\lambda_3}\right)}}{{\lambda_2^2}}\,g_2(\lambda_1,\lambda_2,\lambda_3) \geq 0,\notag\\
&\frac{2 \frac{k}{3}\, e^{ \frac{k}{3} \left(\log ^2\frac{\lambda_1}{\lambda_2}+\log ^2 \frac{\lambda_3}{\lambda_1}+\log
^2\frac{\lambda_2}{\lambda_3}\right)}}{\lambda_1 \lambda_2}\,g_3(\lambda_1,\lambda_2,\lambda_3)\geq 0,
\end{align}
where
\begin{align*}
g_1(\lambda_1,\lambda_2,\lambda_3)&=2 \frac{k}{3} \left(\log \frac{\lambda_1}{\lambda_2}-\log \frac{\lambda_3}{\lambda_1}\right)^2+\log
\frac{\lambda_3}{\lambda_1}-\log \frac{\lambda_1}{\lambda_2}+2\notag,\\
g_2(\lambda_1,\lambda_2,\lambda_3)&= 2 \frac{k}{3} \left(\log \frac{\lambda_1}{\lambda_2}-\log \frac{\lambda_2}{\lambda_3}\right)^2+\log
\frac{\lambda_1}{\lambda_2}-\log \frac{\lambda_2}{\lambda_3}+2,\notag\\
g_3(\lambda_1,\lambda_2,\lambda_3)&=2 \frac{k}{3} \left(\log \frac{\lambda_3}{\lambda_1} \log \frac{\lambda_2}{\lambda_3}+2 \log
^2\frac{\lambda_1}{\lambda_2}\right)+1+\frac{\lambda_2 \left(\log \frac{\lambda_1}{\lambda_2}-\log
\frac{\lambda_3}{\lambda_1}\right)}{\lambda_1+\lambda_2}+\frac{\lambda_1 \left(\log \frac{\lambda_2}{\lambda_3}-\log
\frac{\lambda_1}{\lambda_2}\right)}{\lambda_1+\lambda_2}\notag\\&\notag+2\, \frac{k}{3}\, \sqrt{  \left[\left(\log \frac{\lambda_1}{\lambda_2}-\log
\frac{\lambda_2}{\lambda_3}\right)^2+\log \frac{\lambda_1}{\lambda_2}-\log \frac{\lambda_2}{\lambda_3}+2\right]\left[ \left(\log \frac{\lambda_1}{\lambda_2}-\log \frac{\lambda_3}{\lambda_1}\right)^2-\log \frac{\lambda_1}{\lambda_2}+\log
\frac{\lambda_3}{\lambda_1}+2\right]}.
\end{align*}
We compute that, for extremely large principal stretches $(\lambda_1,\lambda_2,\lambda_2)=(e^{11},e^7,e^{-1})$
\begin{align}\label{countdaco}
g_1(e^{11},e^7,e^{-1})&=2(256 \frac{k}{3}-7),\qquad g_2(e^{11},e^7,e^{-1})=2(16 \frac{k}{3}-1),\\\notag
g_3(e^{11},e^7,e^{-1})&=-128 \frac{k}{3}+\frac{12}{1+e^4}+5+2 \sqrt{(16 \frac{k}{3}-1) (256 \frac{k}{3}-7)},
\end{align}
and we  remark that
\begin{align}
&g_1(e^{11},e^7,e^{-1})>0\quad \Leftrightarrow \quad \frac{k}{3}>\frac{7}{256}\qquad \text{and}\qquad g_2(e^{11},e^7,e^{-1})>0\quad \Leftrightarrow \quad
\frac{k}{3}>\frac{1}{16}.
\end{align}
For $\frac{k}{3}>\frac{1}{16}$ we have $-5-\frac{12}{1+e^4}+128 \frac{k}{3}>0$. Hence,
 $
g_3(e^{11},e^7,e^{-1})\geq 0
$
is equivalent to
\begin{align*}
 &4(16\, \frac{k}{3}-1) \,(256 \,\frac{k}{3}-7)-(-5-\frac{12}{1+e^4}+128 \frac{k}{3})^2\geq 0
\  \Leftrightarrow \ -64 \left(e^4-15\right) \left(1+e^4\right) \frac{k}{3}+e^8-38 e^4-87\geq 0,
\end{align*}
which is not satisfied for $\frac{k}{3}>\frac{1}{16}$. Hence, for $0<k\leq \frac{3}{16}$ the function is not separately convex, while for
$\frac{k}{3}>\frac{1}{16}$ one of the condition \eqref{dev3c}$_3$ given {by Knowle}s and Sternberg's criterion is also not satisfied. Thus, the proof is
complete.
\end{proof}

However,  the function
$
g$ defined by \eqref{wf2} satisfies the Baker-Ericksen (BE) inequalities
\begin{align}
\frac{\lambda_i \frac{\partial g}{\partial {\lambda_i}}-\lambda_j \frac{\partial g}{\partial \lambda_j}}{\lambda_i-{\lambda_j}
}&=2 \frac{k}{3}\,e^{\frac{k}{3} \left(\log ^2\frac{\lambda_i}{\lambda_j}+\log ^2\frac{\lambda_r}{\lambda_i}+\log
^2\frac{\lambda_j}{\lambda_r}\right)}\frac{ 2 \log \frac{\lambda_i}{\lambda_j}-\log \frac{\lambda_r}{\lambda_i}-\log
\frac{\lambda_j}{\lambda_r}}{\lambda_i-\lambda_j}\\
&=2\, k\,e^{\frac{k}{3} \left(\log ^2\frac{\lambda_i}{\lambda_j}+\log ^2\frac{\lambda_r}{\lambda_i}+\log ^2\frac{\lambda_j}{\lambda_r}\right)}\frac{  \log
\frac{\lambda_i}{\lambda_j}}{\lambda_i-\lambda_j}>0, \notag
\end{align}
for any permutation of $i,j,r$. Moreover,
\begin{align}
\frac{\partial ^2 g}{\partial \lambda_i^2}=\frac{2 \frac{k}{3} e^{\frac{k}{3} \left(\log ^2\frac{\lambda_i}{\lambda_j}+\log
^2\frac{\lambda_r}{\lambda_i}+\log ^2\frac{\lambda_j}{\lambda_r}\right)} }{\lambda_i^2}\left[2 \frac{k}{3} \left(\log \frac{\lambda_r}{\lambda_i}-\log
\frac{\lambda_i}{\lambda_j}\right)^2+\log \frac{\lambda_r}{\lambda_i}-\log \frac{\lambda_i}{\lambda_j}+2\right]\geq 0
\end{align}
for any permutation of $i,j,r$ and for all $\frac{k}{3}\geq\frac{1}{16}$. Thus, $g$ is separately convex for  $\frac{k}{3}\geq\frac{1}{16}$. This is not in
contradiction to the 2D result where $k\geq \frac{1}{4}$ was needed for separate convexity since the function $g$ in 2D is not obtained by choosing
$\lambda_3=1$ in the 3D expression of the function $g$.

It is easy to see that the condition
\begin{align}\label{cecdaco}
\frac{\partial^2 g}{\partial \lambda_1^2 }\geq 0,\qquad\quad \frac{\partial^2 g}{\partial \lambda_2^2 }\geq 0, \qquad\quad \sqrt{\frac{\partial^2
g}{\partial \lambda_1^2 }\frac{\partial^2 g}{\partial \lambda_2^2 }}+m_{12}^\varepsilon\geq 0
\end{align}
from Dacorogna's criterion (Theorem \ref{dacorogna5}) are also not simultaneously satisfied for the values considered in \eqref{countdaco}. Let us recall
that for $\varepsilon_1,\varepsilon_2\in \{\pm1\}$
\begin{align}\label{valsilh}
 m_{12}^\varepsilon= \varepsilon_1\varepsilon_2 \frac{\partial^2 g}{\partial \lambda_1 \partial \lambda_2}
+\frac{\frac{\partial g}{\partial \lambda_1}-\varepsilon_1\varepsilon_2 \frac{\partial g}{\partial \lambda_2}}{ \lambda_1 -\varepsilon_1\varepsilon_2
\lambda_2}&\qquad \text{if}\quad \lambda_1\neq\lambda_2 \quad \text{or}\quad \varepsilon_1\varepsilon_2\neq 1.
\end{align}
We choose $\varepsilon_1=1$ and $\varepsilon_2=-1$. For these values, the inequalities \eqref{valsilh} become
\begin{align}\label{cecdaco1}
&\frac{\partial^2 g}{\partial \lambda_1^2 }\geq 0,\qquad\quad \frac{\partial^2 g}{\partial \lambda_2^2 }\geq 0,\qquad\quad \sqrt{\frac{\partial^2
g}{\partial \lambda_1^2 }\frac{\partial^2 g}{\partial \lambda_2^2 }}-
\frac{\partial^2 g}{\partial \lambda_1 \partial \lambda_2}
+\frac{\frac{\partial g}{\partial \lambda_1}+ \frac{\partial g}{\partial \lambda_2}}{ \lambda_1 +\lambda_2}
\geq 0.
 \end{align}
We remark that the conditions \eqref{cecdaco} are in fact equivalent to the inequalities \eqref{dev3c} from  Knowles and Sternberg's criterion, and they
cannot be simultaneously satisfied for the values defined in \eqref{countdaco}.

Moreover, direct and similar calculations as above give:

\begin{remark}
\begin{itemize}
\item[]
\item The function
 \begin{align}\label{wf2inc3}
 g:\mathbb{R}^3_+\rightarrow\mathbb{R},\ \  g(\lambda_1,\lambda_2,\lambda_3):=
e^{\frac{k}{3}\left[\log^2\frac{\lam_1}{\lam_2}+
\log^2\frac{\lam_2}{\lam_3}+\log^2\frac{\lam_3}{\lam_1}\right]}+\frac{\kappa}{2}e^{\widehat{k}\log^2(\lam_1\lam_2\lam_3)}
\end{align}
{does} not satisfy the inequalities from  Knowles and Sternberg's criterion because it does not even satisfy  Zubov's criterion for incompressible
elastic materials as we prove in the next subsection.
\item While we have shown ellipticity of $F\mapsto e^{k\,\|\dev_2\log U\|^2}\!\!,$ we cannot infer (and it does not hold) that $F\mapsto
    e^{k\,\|\dev_3\log U\|^2}$, evaluated and restricted to plane strain deformation $(\lambda_1,\lambda_2,1)$ is elliptic.
\end{itemize}
\end{remark}
Motivated by the preceding negative development, we were inclined to try other, similar Hencky type energies as candidates for an overall elliptic
formulation. However:
\begin{itemize}
\item  The function
 $
 g:\mathbb{R}^3_+\rightarrow\mathbb{R},\ \  g(\lambda_1,\lambda_2,\lambda_3):=
e^{\frac{k}{3}\log^2\frac{\lam_1}{\lam_2}}+
e^{\frac{k}{3}\log^2\frac{\lam_2}{\lam_3}}+e^{\frac{k}{3}\log^2\frac{\lam_3}{\lam_1}}
$
{does} not satisfy the inequalities from  Knowles and Sternberg's criterion.
\item  The function
 $
 g:\mathbb{R}^3_+\rightarrow\mathbb{R},\ \  g(\lambda_1,\lambda_2,\lambda_3):=\mu\left(
e^{\frac{k}{3}\log^2\frac{\lam_1}{\lam_2}}+
e^{\frac{k}{3}\log^2\frac{\lam_2}{\lam_3}}+e^{\frac{k}{3}\log^2\frac{\lam_3}{\lam_1}}\right)+\frac{\kappa}{2}e^{\widehat{k}\log^2(\lam_1\lam_2\lam_3)}
$
{ does} not satisfy the inequalities from  Knowles and Sternberg's criterion because it does not satisfy  Zubov's criterion for incompressible elastic
materials.
\item The function
 $
 g:\mathbb{R}^3_+\rightarrow\mathbb{R},\ \  g(\lambda_1,\lambda_2,\lambda_3):=\mu\left(
e^{k\,\log^2\lam_1}+
e^{k\,\log^2\lam_2}+e^{k\,\log^2\lam_3}\right)+\frac{\kappa}{2}e^{\widehat{k}\log^2(\lam_1\lam_2\lam_3)}
$
{does} not satisfy the inequalities from  Knowles and Sternberg's criterion.
\end{itemize}

\subsection{The ideal nonlinear incompressible elasticity model}

Whereas $e^{k\,\|\dev_3\log U\|^2}$ is not rank-one convex on ${\rm GL}^+(3)$, one might hope that perhaps its restriction to ${\rm SL}(3)$ might be
{rank-one }convex. In the following, for simplicity, we consider only the case $k=1$. Thus, a first open problem is if the following energy $W:{\rm
GL}^+(3)\to\R_+\,,$
\begin{align}\label{Whinc}
 W(F) & = \left\{\begin{array}{lll}
\dd\,e^{\,\|{\rm dev}_3\log U\|^2}&\text{if}& \det\, F=1,\vspace{2mm}\\
+\infty &\text{if} &\det F\neq1 ,
\end{array}\right.
\end{align}
is  rank-one convex. To this aim, we use Zubov's Theorem \ref{zubovinc} to show that this energy is not even rank-{one}-convex on SL(3). According to
\eqref{wf2}, we check the conditions of this theorem for the function defined by \eqref{wf2}.
The answer is negative as can be seen by the counterexample
\begin{align}
\lambda_1=e^{4}, \qquad \lambda_2=e^{-4}, \qquad \lambda_3=1, \qquad \lambda_1\lambda_2\lambda_3=1.
\end{align}
For these values we will prove that the condition
\begin{align}
\sqrt{\delta _1 \delta _2}+\epsilon _3>0,
\end{align}
from   Zubov's Theorem \ref{zubovinc} is not satisfied.
Let us recall that
\begin{align}
\notag\beta_1&=\frac{\partial^2 g}{\partial \lambda_1^2 },\quad \beta_2=\frac{\partial^2 g}{\partial \lambda_2^2 }, \quad \beta_3=\frac{\partial^2
g}{\partial \lambda_3^2 },\quad
\notag\gamma_1^{-}=-\frac{\partial^2 g}{\partial \lambda_2\partial \lambda_3}+\frac{\frac{\partial g}{\partial \lambda_2}+\frac{\partial g}{\partial
\lambda_3}}{\lambda_2+\lambda_3},\quad
\notag\gamma_2^{-}=-\frac{\partial^2 g}{\partial \lambda_1\partial \lambda_3}+\frac{\frac{\partial g}{\partial \lambda_1}+\frac{\partial g}{\partial
\lambda_3}}{\lambda_1+\lambda_3},\\
\gamma_3^{+}&=\frac{\partial^2 g}{\partial \lambda_1\partial \lambda_2}+\frac{\frac{\partial g}{\partial \lambda_1}-\frac{\partial g}{\partial
\lambda_2}}{\lambda_1-\lambda_2},\quad
\notag\dd \delta_1=\beta_2\lambda_2^2+\beta_3 \lambda_3^2+2\gamma_1^-\lambda_2\lambda_3,\quad
\notag\dd \delta_2=\beta_3\lambda_3^2+\beta_1 \lambda_1^2+2\gamma_2^-\lambda_3\lambda_1,\\
\notag\epsilon_3&=\beta_3\lambda_3^2+\gamma_3^+\lambda_1\lambda_2+\gamma_1^-\lambda_3\lambda_2+\gamma_2^-\lambda_3\lambda_1.
\end{align}

In view of \eqref{wf2}, we have
\begin{align}
\beta_i=\frac{ \frac{2}{3} e^{\frac{1}{3} \left(\log ^2\frac{\lambda_i}{\lambda_j}+\log ^2\frac{\lambda_r}{\lambda_i}+\log
^2\frac{\lambda_j}{\lambda_r}\right)} }{\lambda_i^2}\left[ \frac{2}{3} \left(\log \frac{\lambda_r}{\lambda_i}-\log
\frac{\lambda_i}{\lambda_j}\right)^2+\log \frac{\lambda_r}{\lambda_i}-\log \frac{\lambda_i}{\lambda_j}+2\right],
\end{align}
for any permutation of $i,j,r$. Moreover, we have
\begin{align}
\gamma_1^{-}&=\frac{2\, e^{\frac{1}{3} \left(\log ^2\frac{\lambda_1}{\lambda_2}+\log ^2 \frac{\lambda_3}{\lambda_1}+\log
^2\frac{\lambda_2}{\lambda_3}\right)}}{9 \lambda_2 \lambda_3 (\lambda_2+\lambda_3)}
\left[\log \frac{\lambda_3}{\lambda_1} \left(2 (\lambda_2+\lambda_3) \left(\log \frac{\lambda_1}{\lambda_2}-\log \frac{\lambda_2}{\lambda_3}\right)+3
\lambda_2\right)\right.
\\&\qquad\qquad\left.-\left(2 (\lambda_2
+\lambda_3) \log \frac{\lambda_2}{\lambda_3}+3 \lambda_3\right)
\left(\log \frac{\lambda_1}{\lambda_2}-\log \frac{\lambda_2}{\lambda_3}\right)
+3 \left(\lambda_2 \left(-\log \frac{\lambda_2}{\lambda_3}\right)+\lambda_2+\lambda_3\right)\right],\notag\\
\gamma_2^{-}&=\frac{2\, e^{\frac{1}{3} \left(\log ^2\frac{\lambda_1}{\lambda_2}+\log ^2 \frac{\lambda_3}{\lambda_1}
+\log ^2\frac{\lambda_2}{\lambda_3}\right)}}{9 \lambda_1 \lambda_3 (\lambda_1+\lambda_3)}
\left[\log \frac{\lambda_3}{\lambda_1} \left(3 (\lambda_1-\lambda_3)-2 (\lambda_1+\lambda_3) \left(\log \frac{\lambda_2}{\lambda_3}-\log
\frac{\lambda_3}{\lambda_1}\right)\right)
\right.\notag
\\&\qquad\left.+\log \frac{\lambda_1}{\lambda_2}
 \left(2 (\lambda_1+\lambda_3) \left(\log \frac{\lambda_2}{\lambda_3}-\log \frac{\lambda_3}{\lambda_1}\right)+3 \lambda_3\right)
 +3 \left(\lambda_1 \left(-\log \frac{\lambda_2}{\lambda_3}\right)+\lambda_1+\lambda_3\right)\right]),\notag\\
 \gamma_3^{+}&=-\frac{2\, e^{\frac{1}{3} \left(\log ^2\frac{\lambda_1}{\lambda_2}+\log ^2 \frac{\lambda_3}{\lambda_1}
 +\log ^2\frac{\lambda_2}{\lambda_3}\right)}}{9 \lambda_1 \lambda_2 (\lambda_1-\lambda_2)}
  \left[-\log \frac{\lambda_1}{\lambda_2} \left(2 (\lambda_1-\lambda_2) \left(\log \frac{\lambda_3}{\lambda_1}+\log \frac{\lambda_2}{\lambda_3}\right)+3
  (\lambda_1+\lambda_2)\right)
 \right.\notag\\&\qquad\left.+3 \left(\lambda_2 \log \frac{\lambda_3}{\lambda_1}+\lambda_1-\lambda_2\right)+\log \frac{\lambda_2}{\lambda_3}
 \left(2 (\lambda_1-\lambda_2) \log \frac{\lambda_3}{\lambda_1}+3 \lambda_1\right)+2 (\lambda_1-\lambda_2)
  \log ^2\frac{\lambda_1}{\lambda_2}\right].\notag
\end{align}
By direct substitution we deduce
\begin{align}
&\notag\beta_1(e^4,e^{-4},1)=\frac{172\, e^{24}}{3},\qquad \beta_2(e^4,e^{-4},1)=\frac{220\, e^{40}}{3}, \qquad \beta_3(e^4,e^{-4},1)=\frac{4\,
e^{32}}{3},\\
&\notag\gamma_1^{-}(e^4,e^{-4},1)=\frac{2}{3} \left(1-\frac{12}{1+\frac{1}{e^4}}\right) e^{36},\qquad \gamma_2^{-}(e^4,e^{-4},1)=\frac{2\, e^{28}
\left(13+e^4\right)}{3 \left(1+e^4\right)},\quad  \gamma_3^{+}(e^4,e^{-4},1)=\frac{2\, e^{32} \left(109-85 e^8\right)}{3 \,\left(e^8-1\right)},\\
&\notag \delta_1(e^4,e^{-4},1)=\frac{4\, e^{32} \left(19+15 e^4\right)}{1+e^4},\ \ \delta_2(e^4,e^{-4},1)=\frac{4 \, e^{32} \left(19+15
e^4\right)}{1+e^4},\ \ \epsilon_3(e^4,e^{-4},1)=\frac{2\, e^{32} \left(31+8 e^4-31 e^8\right)}{e^8-1},
\end{align}
and
\begin{align}
\sqrt{\delta _1(e^4,e^{-4},1) \delta _2(e^4,e^{-4},1)}+\epsilon _3(e^4,e^{-4},1)=-\frac{2 e^{32} \left(7-16 e^4+e^8\right)}{e^8-1}<0.
\end{align}
This means that the necessary and sufficient conditions from Zubov's Theorem \ref{zubovinc} are not satisfied. Hence, we conclude
\begin{proposition}
The function $F\mapsto e^{\,\|\dev_3\log U\|^2}$ is not rank-one convex  on ${\rm SL}(3)$.
\end{proposition}
\subsection{Rank-one convexity  domains for the  energy $F\mapsto  e^{k\,\|\dev_3\log U\|^2}$}\label{rank-dom}
The understanding of loss of ellipticity may become important for  severe strains and stresses at  crack tips.
The analysis in this subsection is motivated by the results established by Bruhns et al. \cite{Bruhns01,Bruhns02JE} (see also
\cite{knowles1975ellipticity,glugegraphical} in order to compare the domains of ellipticity  obtained in nonlinear elastostatics
for a special material\footnote{For this special material the energy is elliptic for $\rho<\frac{\lambda_1}{\lambda_2}<\frac{1}{\rho}$,
$\rho=2-\sqrt{3}=0.268$.}), in which it is proved
that the quadratic Hencky strain energy function $W_{_{\rm H}}$ with non-negative Lam\'{e} constants, $\mu,\lambda>0$, fulfils the
Legendre-Hadamard condition for all principal stretches with
\begin{align}
\lambda_i\in[0.21162...,\sqrt[3]{e}]=[0.21162...,1.39561...].
\end{align}
The LH-ellipticity of the quadratic Hencky strain energy function $W_{_{\rm H}}$ for all principal stretches in this cube $[0.21162...,1.39561...]^3$ implies (see
Remark \ref{remarSU}) that the exponentiated energy $e^{W_{_{\rm H}}}$ is also LH-elliptic for all principal stretches in this box and for non-negative Lam\'{e}
constants $\mu,\lambda>0$.

\medskip

Let us first remark that the function
$
g:\mathbb{R}^3_+\rightarrow\mathbb{R},\ \  g(\lambda_1,\lambda_2,\lambda_3):=
e^{\frac{k}{3}\left[\log^2\frac{\lam_1}{\lam_2}+
\log^2\frac{\lam_2}{\lam_3}+\log^2\frac{\lam_3}{\lam_1}\right]}
$
corresponding to our energy $F\mapsto  e^{k\,\|\dev_3\log U\|^2}$ is invariant under scaling\footnote{This means $e^{k\,\|\dev_3\, \log(a\,
U)\|^2}=e^{k\,\|\dev_3\log U\|^2}$ for all $a>0$.}:
\begin{align}
  g(a\,\lambda_1,a\,\lambda_2,a\,\lambda_3)=g(\lambda_1,\lambda_2,\lambda_3), \qquad \text{for all}\quad a>0.
\end{align}

In fact, we have:
\begin{remark}
All functions $F\mapsto W(F)=W_1(\|\dev_3\log U\|^2)$ are invariant  under the scaling: $F\mapsto a\, F$, $a>0$.
\end{remark}

Let us consider the substitution $(\widetilde{\lambda}_1,\widetilde{\lambda}_2,\widetilde{\lambda}_3)=(a\,\lambda_1,a\,\lambda_2,a\,\lambda_3), \ \text{for
all}\quad a>0$. For the derivatives, we deduce
\begin{align}
&\frac{\partial}{\partial \widetilde{\lambda}_i} g(\widetilde{\lambda}_1,\widetilde{\lambda}_2,\widetilde{\lambda}_3)=\frac{1}{a}\frac{\partial}{\partial
\lambda_i} g(\lambda_1,\lambda_2,\lambda_3),\quad
\frac{\partial^2}{\partial \widetilde{\lambda}_i\partial \widetilde{\lambda}_j} g(\widetilde{\lambda}_1,\widetilde{\lambda}_2,\widetilde{\lambda}_3)
=\frac{1}{a^2}\frac{\partial^2}{\partial \lambda_i\partial \lambda_j} g(\lambda_1,\lambda_2,\lambda_3).
\end{align}
 Hence, for the function $g$ corresponding to our energy $F\mapsto  e^{k\,\|\dev_3\log U\|^2}$, the inequalities in Dacorogna's criterion are also
 invariant under scaling.
More generally:
\begin{remark}
\begin{itemize}
\item[]
\item[i)]
Let  $F\mapsto W(F)=W_1(\|\dev_3\log U\|^2)$ be a function on ${\rm GL}^+(3)$. Then, the inequalities in Dacorogna's criterion in terms of the
corresponding function $g:\R_+\to\R$ are  invariant under scaling.
\item[ii)] For all functions $F\mapsto W(F)$  (for instance for functions $F\mapsto W(F)=W_{\rm iso}(\frac{F}{\det F^{1/3}})$) which are invariant under
    scaling, the inequalities in Dacorogna's criterion in terms of the corresponding function $g:\R_+\to\R$ are  invariant under scaling.
\end{itemize}
\end{remark}

Therefore,   if the function $g$ does not satisfy the requested inequalities  in  Dacorogna's criterion in a point
$(\lambda_1^{(0)},\lambda_2^{(0)},\lambda_3^{(0)})$, then it also does not satisfy them in the point
$(\widetilde{\lambda}_1^{(0)},\widetilde{\lambda}_2^{(0)},\widetilde{\lambda}_3^{(0)})
 =(a\,\lambda_1^{(0)},a\,\lambda_2^{(0)},a\,\lambda_3^{(0)})$ for arbitrary $a>0$. In the following we will exploit this insight.

In the previous subsections we have proved that there exist a point in which the function $F\mapsto e^{\|\dev_3 \log U\|^2}$ looses the  LH-ellipticity,
namely in
\begin{align}\label{cexci}
\lambda_1^{(0)}&=e^{11},\qquad \lambda_2^{(0)}=e^7,\qquad \lambda_3^{(0)}=e^{-1}
\end{align}
for  compressible materials and in
\begin{align}
\lambda_1^{(0)}&=e^{4},\qquad \lambda_2^{(0)}=e^{-4},\qquad \lambda_3^{(0)}=1
\end{align}
in the case of incompressible materials.
In view of the scaling invariance  discussed above, we have
\begin{lemma}\label{lemmalinie}
If the function $
g:\mathbb{R}^3_+\rightarrow\mathbb{R},\ \  g(\lambda_1,\lambda_2,\lambda_3):=
e^{\frac{k}{3}\left[\log^2\frac{\lam_1}{\lam_2}+
\log^2\frac{\lam_2}{\lam_3}+\log^2\frac{\lam_3}{\lam_1}\right]}
$
 is not elliptic in a point $P^{(0)}=(\lambda_1^{(0)},\lambda_2^{(0)},\lambda_3^{(0)})$, then it is not elliptic in all  points
 $P(\lambda_1,\lambda_2,\lambda_3)$,  $(\lambda_1,\lambda_2,\lambda_3)\neq (0,0,0)$ belonging to the line $OP^{(0)}$, where $O=(0,0,0)$. In other words,
 the ellipticity domain is invariant under scaling.
\end{lemma}
More general:
\begin{remark}
\label{rem:straightlines}
Let  $F\mapsto W(F)$   be an invariant under scaling function (for instance  $F\mapsto W(F)=W_{\rm iso}(\frac{F}{\det F^{1/3}})$ or $F\mapsto
W(F)=W_1(\|\dev_3\log U\|^2)$) defined on ${\rm GL}^+(3)$. If  the corresponding function $g:\R_+^3\to\R$ is not elliptic in a point
$P^{(0)}=(\lambda_1^{(0)},\lambda_2^{(0)},\lambda_3^{(0)})$, then it is not elliptic in all  points $P(\lambda_1,\lambda_2,\lambda_3)$,
$(\lambda_1,\lambda_2,\lambda_3)\neq (0,0,0)$ belonging to the line $OP{^{(0)}},$ where $O=(0,0,0)$. In other words, the ellipticity domain of a function
invariant under scaling function will be invariant under scaling.
\end{remark}

\begin{proposition}\label{prop1lin}
  The energy $F\mapsto e^{\|\dev_3 \log U\|^2}$, $F\in {\rm GL}^+(3)$ cannot be  LH-elliptic in any  cube like domain  $(0,y)\times(0,y)\times (0,y)$,
  $y>0$.
 \end{proposition}
 \begin{proof} If the point $P^{(0)}=(\lambda_1^{(0)},\lambda_2^{(0)},\lambda_3^{(0)})$ given by \eqref{cexci} belongs to the domain
 $(0,y)\times(0,y)\times (0,y)$ then we have nothing more to prove. If
 $P^{(0)}(\lambda_1^{(0)},\lambda_2^{(0)},\lambda_3^{(0)})\not\in(0,y)\times(0,y)\times (0,y)$, then there is a point  $P(\lambda_1,\lambda_2,\lambda_3)$,
 $(\lambda_1,\lambda_2,\lambda_3)\neq (0,0,0)$ belonging to the line $OP$ and $P(\lambda_1,\lambda_2,\lambda_3)\in(0,y)\times(0,y)\times (0,y)$ (see Figure
 \ref{nLH-box-s}). For instance the point  $(\frac{\lambda_1^{(0)}}{a},\frac{\lambda_2^{(0)}}{a},\frac{\lambda_3^{(0)}}{a})$, where
 $a>\max\limits_{i=1,2,3}\left\{\frac{\lambda_i^{(1)}}{y}\right\}$. In view of Lemma \ref{lemmalinie} the proof is complete.
\end{proof}
\begin{proposition}\label{prop2lin}
  The energy $F\mapsto {e^{\|\dev_3\log U\|^2}}$, $F\in {\rm GL}^+(3)$ is not LH-elliptic in any cube like domain  $(x,\infty)\times(x,\infty)\times
  (x,\infty)$, $x>0$.
 \end{proposition}
 \begin{proof}
The proof is similar to the proof of the previous proposition, because for $b<\min\limits_{i=1,2,3}\left\{\frac{\lambda_i^{(0)}}{x}\right\}$, the    point
$(\frac{\lambda_1^{(0)}}{b},\frac{\lambda_2^{(0)}}{b},\frac{\lambda_3^{(0)}}{b})\in OP$ belongs also to the domain $(x,\infty)\times(x,\infty)\times
(x,\infty)$ (see Figure \ref{nLH-box-b}).
\end{proof}

We already can prove this more general result:
\begin{proposition}
Let  $F\mapsto W(F)$   be a function defined on ${\rm GL}^+(3)$  which is
invariant under scaling. If  the corresponding function $g:\R_+^3\to\R$ is not elliptic
in a point $P^{(0)}=(\lambda_1^{(0)},\lambda_2^{(0)},\lambda_3^{(0)})$, then there are no cube-like domains $(0,y)^3$, $y>0$, or $(x,\infty)^3$, $x>0$, on
which $g$ is elliptic.
     \end{proposition}

On the other hand, the energy $F\mapsto e^{k\,\|\dev_3\log U\|^2}$ is invariant under inversion\footnote{The invariance under inversion of an energy $W$ is
the tension-compression symmetry $W(F)=W(F^{-1})$.}, i.e.
\begin{align}
 e^{k\,\|\dev_3\log U\|^2}=e^{k\,\|\dev_3\log U^{-1}\|^2}\quad \Leftrightarrow\quad
 g(\lambda_1,\lambda_2,\lambda_3)=g\left(\frac{1}{\lambda_1},\frac{1}{\lambda_2},\frac{1}{\lambda_3}\right), \quad {\text{ for  all }}\; \,
 (\lambda_1,\lambda_2,\lambda_3)\in{\R_+^3}.
\end{align}
However,  Dacorogna's  ellipticity criterion is not invariant under inversion. This is the reason why  Proposition {\ref{prop2lin}} does not follow
directly from Proposition \ref{prop1lin} using the invariance under inversion.

\begin{figure}[h!]
\hspace*{1cm}
\begin{minipage}[h]{0.4\linewidth}
\centering
\includegraphics[scale=0.4]{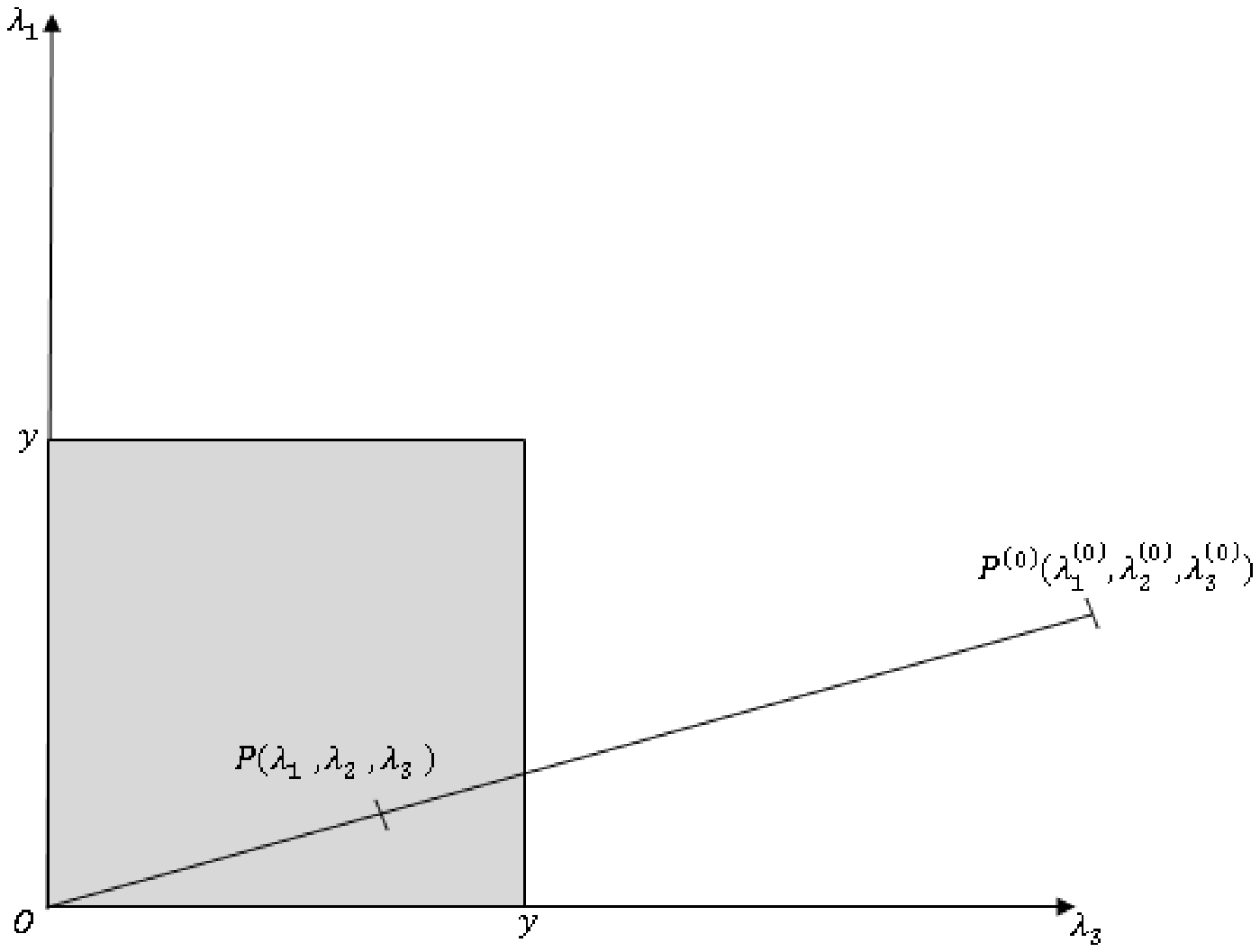}
\centering
\caption{\footnotesize{$F\mapsto e^{k\,\|\dev_3\log U\|^2}$ is not LH-elliptic in the domain  $(0,y)\times(0,y)\times (0,y)$, $y>0$.}}
\label{nLH-box-s}
\end{minipage}
\hspace{1cm}
\begin{minipage}[h]{0.4\linewidth}
\centering
\includegraphics[scale=0.4]{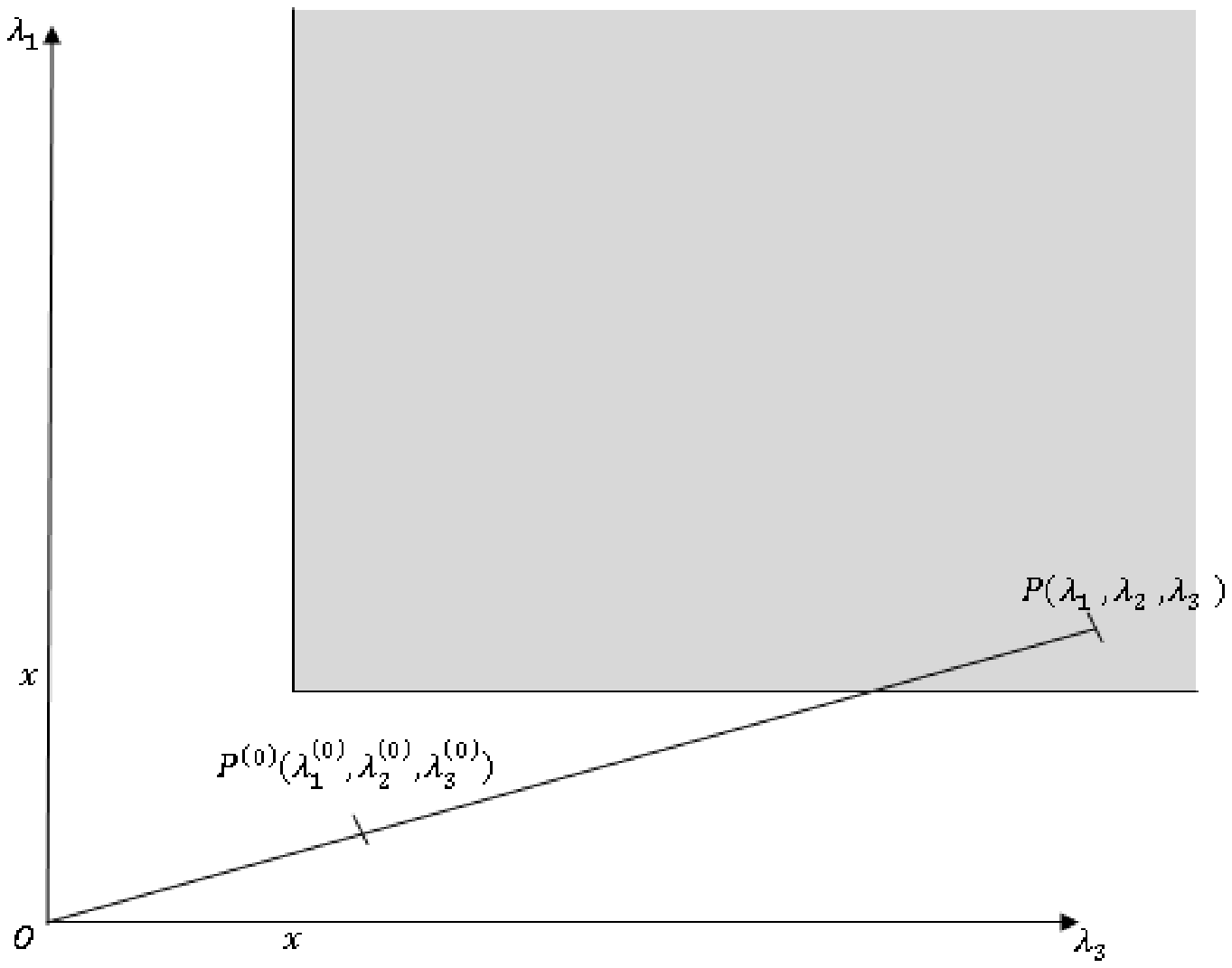}
\centering
\caption{\footnotesize{$F\mapsto e^{k\,\|\dev_3\log U\|^2}$ is not LH-elliptic in the domain  $(x,\infty)\times(x,\infty)\times (x,\infty)$, $x>0$.}}
\label{nLH-box-b}
\end{minipage}
\end{figure}%

\begin{remark}
Looking back to the quadratic Hencky energy $F\mapsto W_{_{\rm B}}(F):=\,\|\dev_3\log U\|^2$ considered by Bruhns et al. \cite{xiao1997logarithmic}
\footnote{Hutchinson and Neale \cite{Hutchinson82} have considered the energy $\|\dev_3\log U\|^{N}$ for $0<N\leq 1$.} and to the corresponding function
$g_{_{\rm B}}:\mathbb{R}^3_+\rightarrow\mathbb{R},$ $ g_{_{\rm B}}(\lambda_1,\lambda_2,\lambda_3):=\!\!
\frac{1}{3}\!\!\left[\log^2\frac{\lam_1}{\lam_2}+
\log^2\frac{\lam_2}{\lam_3}+\log^2\frac{\lam_3}{\lam_1}\!\right]$, we remark:
\begin{itemize}
\item $g_{_{\rm B}}$ is separately convex (see Proposition \ref{TEHencky} and Corollary \ref{TEHenckycor}) only for those $U$ such that the eigenvalues
    $\mu_1,\mu_2,\mu_3$ of $\dev_3\log U$ are smaller than $\frac{2}{3}$, i.e, if and only if  the eigenvalues  $\lambda_1,\lambda_2,\lambda_3$ of $U$
    are such that
\begin{align}
\lambda_1^2\leq e^2 \,\lambda_2\,\lambda_3,\qquad \lambda_2^2\leq e^2 \,\lambda_3\,\lambda_1, \qquad \lambda_3^2\leq e^2 \,\lambda_1\,\lambda_2.
\end{align}
\item $g_{_{\rm B}}$ always satisfies  the BE-inequalities.
\item Numerical computations give us reasons to believe that there exists a number $a_{_{\rm B}}>0$ such that $F\mapsto \,\|\dev_3\log U\|^2$ is LH-elliptic in
    the domain (invariant under scaling)
$$
\widetilde{\mathcal{E}}(W_{B},{\rm LH},U,\frac{4}{3}):=\left\{U\in{\rm PSym}(3)\,|\,\, \|\dev_3\log U\|^2<a_{_{\rm B}}<\frac{4}{3}\right\}.
$$
\item The results already obtained in \cite{Bruhns01} cannot be used for the energy $W_{_{\rm B}}(F)=\|\dev_3\log U\|^2=\|\log U\|^2-\frac{1}{3}[{\rm tr}(\log
    U)]^2$, because they are applicable only for energies $W_{_{\rm H}}(F)=\mu\|\log U\|^2+\frac{\lambda}{2}[{\rm tr}(\log U)]^2$ for which $\mu,\lambda\geq0$.
\end{itemize}
\end{remark}

\begin{remark}
As expected, the above properties are improved by considering the exponentiated Hencky energy  \linebreak $F\mapsto W(F)=\,e^{k\,\|\dev_3\log U\|^2}$. The
corresponding function  $g:\mathbb{R}^3_+\rightarrow\mathbb{R},\ \  g(\lambda_1,\lambda_2,\lambda_3):=e^{
\frac{k}{3}\left[\log^2\frac{\lam_1}{\lam_2}+
\log^2\frac{\lam_2}{\lam_3}+\log^2\frac{\lam_3}{\lam_1}\right]}$
is such that:
\begin{itemize}
\item $g$ is separately convex everywhere if $k>\frac{3}{16}$.
\item $g$ always satisfies  the BE-inequalities.
\item Numerical computations give us reasons to believe that there is a number $a>0$, $a>\frac{4}{3}>a_{_{\rm B}}>0$, such that $F\mapsto \,e^{k\,\|\dev_3\log
    U\|^2}$ is LH-elliptic in the domain $\|\dev_3\log U\|^2<a$. The approximative value which we observed is $a=27$, i.e. $U\in \mathcal{E}(W_{_{\rm eH}},{\rm
    LH},U,27)$, where
 \begin{align}\label{margdevlog}
 \mathcal{E}(W_{_{\rm eH}}^{\rm iso},{\rm LH},U,27):=\{U\in {\rm PSym}(3)\,  \big|\, \|\dev_3 \log U\|^2\leq 27\}.
 \end{align}
Of course,  $\mathcal{E}(W_{_{\rm eH}}^{\rm iso},{\rm LH},U,27)$ contains a neighbourhood of $\id$. Rephrasing the remark of Ogden \cite[page 409]{Ogden83},
``the question whether a constitutive inequality {\rm [LH, TSS-I, TSTS-M, BSTS, BSS etc.]} holds for all deformations of a compressible solid is open.
The applicability of elastic theory outside {\rm [a bounded domain in stretch space]} is itself questionable because, for example, there may exist yield
surfaces beyond which permanent deformation occurs."
\end{itemize}
\end{remark}
\begin{remark}
The major open problems  in this respect are:
 \begin{itemize}
 \item {Do} there exist numbers $x,y>0$ such that the energy function $F\mapsto e^{\|\dev_3 \log U\|^2}$ is  LH-elliptic in $(x,y)\times(x,y)\times(x,y)$
     and $\id\in(x,y)\times(x,y)\times(x,y)$? If true, then  in view of Lemma \ref{lemmalinie} the function $F\mapsto e^{\|\dev_3 \log U\|^2}$ is
     LH-elliptic in the domain given by Figure \ref{LH-box}. In the  three-dimensional representation it is a cone with the angle in the origin.
 \item Is the function $F\mapsto e^{\|\dev_3 \log U\|^2}$ elliptic in a ball containing $\id$? If true, then  in view of Lemma \ref{lemmalinie} the
     function $F\mapsto e^{\|\dev_3 \log U\|^2}$ is LH-elliptic  in the domain given  by Figure \ref{LH-sfera}. In  the two-dimensional representation
     this domain is a corner domain but in the three-dimensional representation it is the interior of an infinite cone (not {necessarily} circular) with
     the angle in the origin.
 \item
In fact it is enough to check where the energy function $F\mapsto e^{\|\dev_3 \log U\|^2}$ looses the ellipticity in all  planes $\lambda_i=1$, meaning
planes $\pi_i$ containing the point $(1,1,1)$ and orthogonal to the axes $O\lambda_i$, respectively. In view of the symmetry in $\lambda_i$, it is enough
to see what  happens in the plane $\pi_1:\ \lambda_1=1$.
\end{itemize}
\end{remark}

\begin{figure}[h!]
\hspace*{1cm}
\begin{minipage}[h]{0.4\linewidth}
\centering
\includegraphics[scale=0.5]{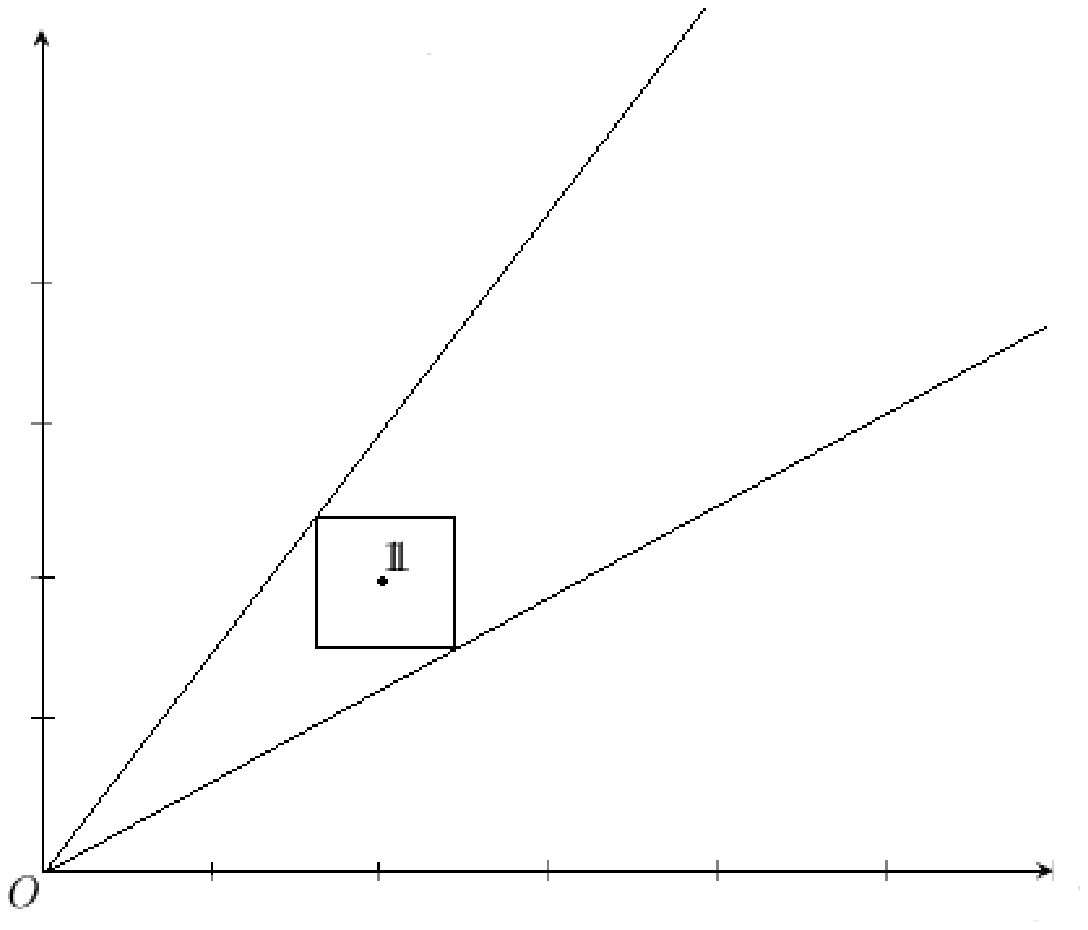}
\centering
\caption{\footnotesize{ A section along the first diagonal along the line containing the origin $O$ and $\id $ of the LH-ellipticity domain of the energy
function $F\mapsto e^{\|\dev_3 \log U\|^2}$ if it is   LH-elliptic in a box $(x,y)\times(x,y)\times(x,y)$.}}
\label{LH-box}
\end{minipage}
\hspace{1cm}
\begin{minipage}[h]{0.4\linewidth}
\centering
\includegraphics[scale=0.5]{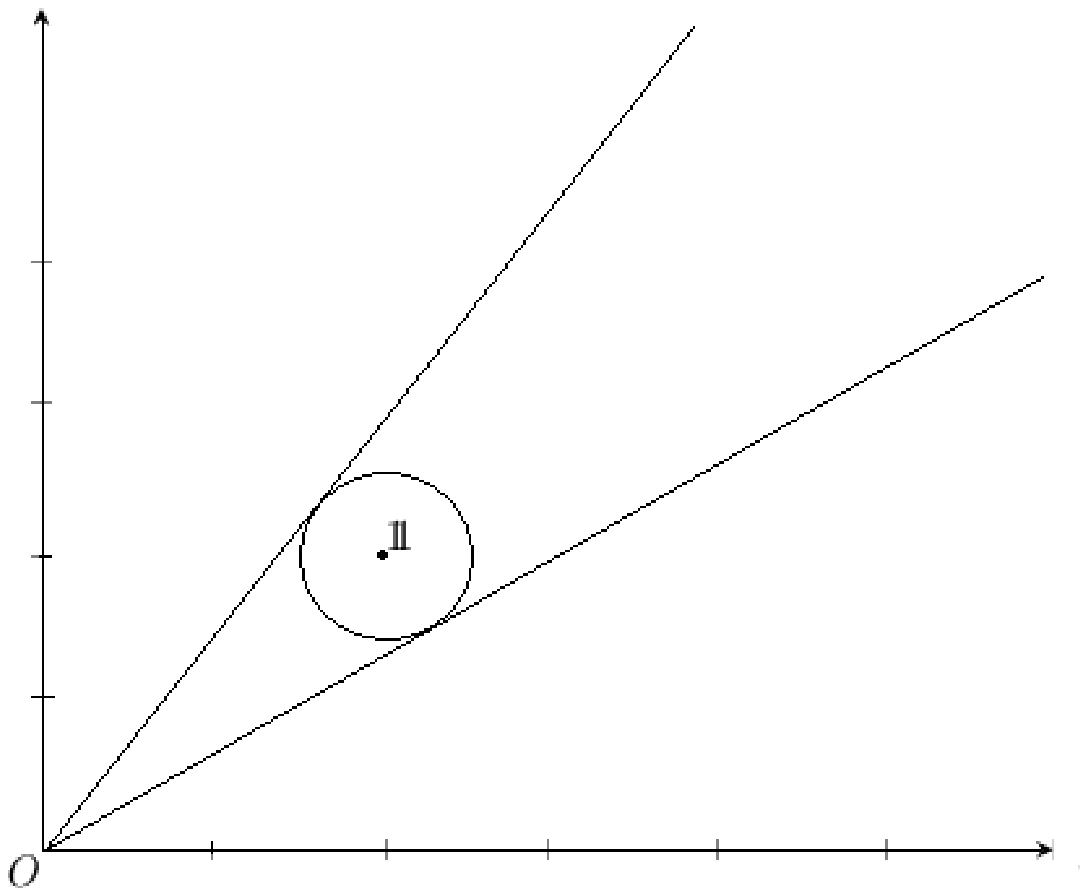}
\centering
\caption{\footnotesize{ A section along the first diagonal along the line containing the origin $O$ and $\id $ of the LH-ellipticity domain of the energy
function $F\mapsto e^{\|\dev_3 \log U\|^2}$ if it is   LH-elliptic in a ball.}}
\label{LH-sfera}
\end{minipage}
\end{figure}

\begin{figure}[h!]\begin{center}
\begin{minipage}[h]{0.6\linewidth}
\centering
\includegraphics[scale=0.5]{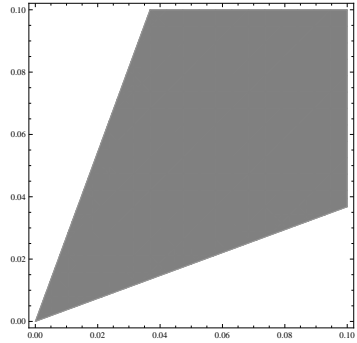}
\centering
\caption{\footnotesize{A section along the first diagonal along the line containing the origin $O$ and $\id $ of the domain ${\|\dev_3 \log U\|^2}<1$. This
indicates that, similar to TSTS-M$^+$, ellipticity might be controlled by the distortional energy.}}
\label{dev-log-corner}
\end{minipage}
\end{center}
\end{figure}%

\section{Summary and open problems}
\setcounter{equation}{0}
To summarize, in the present paper:
\begin{itemize}
\item We have proved that the planar exponentiated Hencky strain energy function
\begin{align}
F\mapsto W_{_{\rm eH}}(F):=\widehat{W}_{_{\rm eH}}(U):&=\left\{\begin{array}{lll}
\frac{\mu}{k}\,e^{k\,\|\dev_2\log\,U\|^2}+\frac{\kappa}{2\widehat{k}}\,e^{\widehat{k}\,(\tr(\log U))^2}&\ \ \ \!\!\text{if}& \det\, F>0,\vspace{2mm}\\
+\infty &\ \ \ \!\! \text{if} &\det F\leq 0\end{array}\right.
\end{align}
is {\bf rank-one convex} for  $\mu>0, \kappa>0$, $k\geq\dd\frac{1}{4}$ and $\widehat{k}\dd\geq \tel8$;
\item We have shown that the exponentiated volumetric energy function
\begin{align}
 F\mapsto \frac{\kappa}{2\widehat{k}}\,e^{\widehat{k}\,(\tr(\log U))^2}, \quad F\in {\rm GL}^+(n)
\end{align}
is {\bf rank-one convex} w.r.t $F$ for the volumetric strain parameter $\widehat{k}\geq \tel{m^{(m+1)}}$. ($m=2: \widehat{k}\geq \tel8$, $m=3:
\widehat{k}\geq \tel{81}$);
\item  We have  shown that for all distortional strain stiffening parameters $k>0$ the energy function
 \begin{align}
F\mapsto \frac{\mu}{k}\,e^{k\,\|\dev_3\log U\|^2},\quad F \in{\rm GL}^+(3)
\end{align}
is {\bf not rank-one convex};
\item Numerical tests suggest that the LH-ellipticity domain of the distortional energy function
$
F\mapsto \frac{\mu}{k}\,e^{k\,\|\dev_3\log U\|^2}$, $ F \in{\rm GL}^+(3),
$
with $k\geq \frac{3}{16}$ (the necessary  condition for separate convexity  (SC) of $e^{k\,\|\dev_3\log U\|^2}$ in 3D) is  an extremely  large cone
\begin{align}
\mathcal{E}(W_{_{\rm eH}},{\rm LH},U, 27)=\{U\in{\rm PSym}(3)\,\big|\, \|\dev_3\log U\|^2<27\};
\end{align}
\item  We have proved that the energy function
 \begin{align}
F\mapsto \frac{\mu}{k}\,e^{k\,\|\log U\|^2},\quad F \in{\rm GL}^+(n), \quad n=2,3
\end{align}
is {\bf not rank-one convex};
\item We have shown that the true-stress-true-strain relation
is invertible for the family of energies $W_{_{\rm eH}}$.
 \item The monotonicity of the Cauchy stress tensor,  as a function of $\log V,$ for our family of exponentiated Hencky energies is true in certain
     domains of bounded distortions
     \begin{align}
 \mathcal{E}(W_{_{\rm eH}},\textrm{TSTS-M}^+, \tau_{_{\rm eH}},\frac{2}{3}\, {\boldsymbol{\sigma}}_{\!\mathbf{y}}^2):=\left\{ \tau\in {\rm Sym}(3) \big|\,\ \|\dev_3
 \tau\|^2\leq\frac{2}{3}\, {\boldsymbol{\sigma}}_{\!\mathbf{y}}^2\right\},
 \end{align}
  superficially similar to the observed ellipticity domains $ \mathcal{E}(W_{_{\rm eH}},\textrm{TSTS-M}^+, \tau_{_{\rm eH}},\frac{2}{3}\,
  {\boldsymbol{\sigma}}_{\!\mathbf{y}}^2)$.
\item
For all exponentiated energies KSTS-M$^+$, KSTS-I, TSTS-I, TSS-I conditions are satisfied everywhere.
\item For $n=3$ among the family $W_{_{\rm eH}}$ we have singled out a special  ($k=\frac{2}{3}\, \widehat{k}\;$) three parameter subset  $$
W_{_{\rm eH}}^{\sharp}(\log V)=\frac{1}{2\,k}\left\{\frac{E}{1+\nu}\,e^{k\,\|\dev_3\log\,V\|^2}+\frac{E}{2(1-2\,\nu)}\,e^{\frac{2}{3}\,k\,({\rm tr}(\log V))^2}\right\}$$
such that uniaxial tension leads to no lateral contraction if and only if $\nu=0$, as in linear elasticity.

\end{itemize}

In  forthcoming papers \cite{Neff_Osterbrink_Martin_hencky13,neff2013hencky,Neff_Nagatsukasa_logpolar13} our geodesic invariants
\begin{itemize}
\item[]``the magnitude-of-dilatation": \quad \quad  $K_1^2=|\tr(\log U)|^2=|\log \det U|^2=|\log \det V|^2=|\log \det F|^2,$
\item[] ``the magnitude-of-distortion":\quad\quad\  \,$K_2^2=\|\dev_3\log U\|^2=\|\dev_3\log V\|^2,$
\end{itemize}
as basic ingredients of  idealized isotropic strain energies  will be  motivated in detail.  As already stated in the introduction, it can be shown that
\cite{Neff_Osterbrink_Martin_hencky13,%
neff2013hencky,Neff_Nagatsukasa_logpolar13}
\begin{align*}
{\rm dist}^2_{{\rm geod}}\left((\det F)^{1/n}\cdot \id, {\rm SO}(n)\right)&={\rm dist}^2_{{\rm geod,\mathbb{R}_+\cdot \id}}\left((\det F)^{1/n}\cdot \id,
\id\right)=|\log \det F|^2 = \widetilde{W}_{_{\rm H}}^{\rm vol}(\det U)\,,\\
{\rm dist}^2_{{\rm geod}}\left( \frac{F}{(\det F)^{1/n}}, {\rm SO}(n)\right)&={\rm dist}^2_{{\rm geod,{\rm SL}(n)}}\left( \frac{F}{(\det F)^{1/n}}, {\rm
SO}(n)\right)=\|\dev_n \log U\|^2 = \widetilde{W}_{_{\rm H}}^{\rm iso}\left(\frac{U}{\det U^{1/n}}\right)\,,
\end{align*}
where ${\rm dist}^2_{{\rm geod,\mathbb{R}_+\cdot \id}}$ and ${\rm dist}^2_{{\rm geod,{\rm SL}(n)}}$ are the canonical left invariant geodesic distances on
the Lie-group ${\rm SL}(n)$ and on the group $\mathbb{R}_+\cdot\id$, respectively (see
\cite{Neff_Osterbrink_Martin_hencky13,neff2013hencky,Neff_Nagatsukasa_logpolar13}). For this investigation new mathematical tools had to be
discovered \cite{Neff_Nagatsukasa_logpolar13,LankeitNeffNakatsukasa} also having consequences for the classical polar  decomposition
\cite{jog2002explicit,jog2002foundations}.

Hence, {using this terminology}, in {the present} paper we have shown rank-one convexity of
\begin{align}
W_{_{\rm eH}}(F):=\frac{\mu}{k}\,e^{k\,{\rm dist}^2_{{\rm geod,{\rm SL}(2)}}\left( \frac{F}{\det F^{1/2}},\, {\rm SO}(2)
\right)}+\frac{\kappa}{2\widehat{k}}\,e^{\widehat{k}\,{\rm dist}^2_{{\rm geod},\mathbb{R}_+\cdot \id}\left(\det F^{1/2}\cdot \id, {\rm SO}(2)\right)}.
\end{align}

Our $W_{_{eH}}$ formulation ignores at first sight yield surfaces and other aspects of a theory of plasticity. Yet, our investigation on the ellipticity conditions in 3D suggests a relation between loss of ellipticity conditions and permanent deformations. We will come back to this point in the near future \cite{NeffGhibaPlasticity}.

Let us finish this paper with some conjectures stemming from our unsuccessful attempts in this direction:

\begin{conjecture}
For n=2,3 the energy $F\mapsto \frac{\mu}{k}e^{k\, \|\dev_n\log U\|^2}$, $k>\frac{3}{16}$ is rank-one convex in a set which contains the large cone
\begin{align}
\mathcal{E}(W_{_{\rm eH}},{\rm LH},U,27)=\{U\in{\rm PSym}(3)\,\big|\, \|\dev_3\log U\|^2<27\}.
\end{align}
\end{conjecture}
Moreover, it would be interesting to know the rank-one convex and quasiconvex envelope of the energy $F\mapsto \frac{\mu}{k}e^{k\, \|\dev_n\log U\|^2}$,
$k>\frac{3}{16}$.
\begin{conjecture}
For n=3 there is no elastic energy expression
\begin{align}
W=W(K_2^2)=W(\|\dev_3\log U\|^2)
\end{align}
such that $F\mapsto W(\|\dev_3\log U\|^2)$ is Legendre-Hadamard elliptic in ${\rm GL}^+(3)$, i.e. over the entire deformation range.
\end{conjecture}

\begin{conjecture}
For n=2,3 there is no elastic energy expression
\begin{align}
W=W(K_2^2)=W(\|\dev_n\log U\|^2)
\end{align}
such that $F\mapsto W(\|\dev_n\log U\|^2)$  satisfies the TSTS-M$^+$ condition in ${\rm GL}^+(n)$, i.e. over the entire deformation range.
\end{conjecture}

A further open problem is to find an energy $F\mapsto W(\|\dev_3\log U\|^2, [\tr(\log U)]^2)$ such that the BSS-I condition is s{atisfied. In a futu}re
contribution we will discuss the application of the family $W_{_{\rm eH}}(F)$ to the description of large strain rubber elasticity for  Treloar's classical data.

\section{Acknowledgement}
This paper is inconceivable without the stimulus of  Albert Tarantola's book  "Elements for Physics" \cite{Tarantola06}. We  would like to thank  {Prof. \!Krzysztof Chelminski} (TU Warsaw) for helping us in the study of rank-one convexity of the function $e^{\|\log U\|^2}$ in
the planar case, Prof.\! David Steigman (UC Berkeley) who indicated to us reference \cite{Hutchinson82},  Prof.
Bernard Dacorogna (EPFL-Lausanne) for sending us  reference
\cite{DacorognaMarechal} and Prof.  Miroslav \v{S}ilhav\'{y} (Academy of Sciences of the Czech Republic, Prague) for
comments on rank-one convexity. The interest in considering nonlinear scalar functions of $\|\dev_3\log U\|^2$ arose after an insightful comment by  {Prof. Reuven
Segev} (Ben-Gurion University of the Negev, Beer-Sheva) on the presentation of the first author at the 4th Canadian Conference on Nonlinear Solid Mechanics
(CanCNSM July 2013) in Montreal. Discussion with Prof. Chandrashekhar S. Jog (Indian Institute of Science, Bangalore) on the TSTS-M$^+$ condition is also
gratefully acknowledged.

\bibliographystyle{plain} 
\addcontentsline{toc}{section}{References}
\begin{footnotesize}

\end{footnotesize}

\appendix
\section*{Appendix}\addcontentsline{toc}{section}{Appendix} \addtocontents{toc}{\protect\setcounter{tocdepth}{-1}}

\setcounter{section}{1}

\subsection{ Some useful identities}
\setcounter{subsection}{1}\label{identitiesapp}
\setcounter{equation}{0}

\begin{itemize}
\item $\tr(B^{-1}X\, B)=\tr(X)$ for any invertible matrix $B$.
\item $\dev_n(B^{-1}X\, B)=B^{-1}X\, B-\frac{1}{n}\,\tr(B^{-1}X\, B)=B^{-1}(\dev_n X)\, B$ for any invertible matrix $B$.
\item  $\norm{\dev_n X}^2=\norm{X-\frac1n\tr X \cdot \id}^2=\norm{X}^2+\frac1{n^2}(\tr X)^2\norm{\id}^2-\frac2n\tr X\langle X,I\rangle=\norm{X}^2-\frac1n(\tr
    X)^2$.
\item The norm of the deviator in $\R^{n\times n}$:
\begin{align}
 \|\dev_n\left(
 \begin{array}{cccc}
 \xi_1&0&\cdots&0\\
 0&\xi_2&\cdots&0\\
 \vdots&\vdots &\ddots &\vdots\\
 0&0&\cdots&\xi_n\\
 \end{array}\right)
\|^2&=\sum\limits_{i=1}^n \xi_i^2-\tel n(\sum\limits_{i=1}^n \xi_i)^2\notag=\frac{n-1}{n}\sum\limits_{i=1}^n \xi_i^2-\frac{2}{n}\sum\limits_{i,j=1,i<j}^n
\xi_i\xi_j\label{formdevn}\\
 &=\frac{1}{n}[(n-1)\sum\limits_{i=1}^n \xi_i^2-2\sum\limits_{i,j=1,i<j}^n \xi_i\xi_j]\\\notag
 &=\frac{1}{n}\sum\limits_{i,j=1,i<j}^n (\xi_i^2-2 \xi_i\xi_j+\xi_j^2)=\frac{1}{n}\sum\limits_{i,j=1,i<j}^n (\xi_i-\xi_j)^2.\notag
\end{align}

\item From \cite[page 200]{Neff_Diss00} we have: $\frac{\|X\|^p}{z^\alpha}$ is convex in $(X,z)$ if  $\frac{\alpha+1}{\alpha}\geq
    \frac{p}{p-1}\Leftrightarrow p\geq \alpha+1$.
   \item $\log U=\sum\limits_{i=1}^n \log \lambda_i\,  N_i\otimes N_i,$ where $N_i$ are the eigenvectors of $U$ and $\lambda_i$ are the eigenvalues of
       $U$.
\item $\log U=(U-\id)-\frac{1}{2}(U-\id)^2+\frac{1}{3}(U-\id)^3-...$,  convergent for $\|U-\id\|<1$.
   \item $\log V=\sum\limits_{i=1}^n \log \widehat{\lambda}_i\,  \widehat{N}_i\otimes \widehat{N}_i,$ where $\widehat{N}_i$ are the eigenvectors of $V$
       and $\widehat{\lambda}_i$ are the eigenvalues of $V$.
\item $\log V=(V-\id)-\frac{1}{2}(V-\id)^2+\frac{1}{3}(V-\id)^3-...$, convergent for $\|V-\id\|<1$.
\item $ \left(
          \begin{array}{cc}
            F_{11} & F_{12} \\
            F_{21} & F_{22} \\
          \end{array}
        \right)^{-1}=\dd\frac{1}{F_{11}F_{22}-F_{12}F_{21}}\left(
          \begin{array}{cc}
            F_{22} & -F_{12} \\
            -F_{21} & F_{22} \\
          \end{array}
        \right)\quad \Rightarrow \quad \|F^{-1}\|^2\,\overset{n=2}{=}\,\dd\frac{1}{(\det F)^2}\|F\|^2.
$
\item $F=\left(
          \begin{array}{cc}
            F_{11} & F_{12} \\
            F_{21} & F_{22} \\
          \end{array}
        \right)$,  $U^2=F^T F=\left(
\begin{array}{cc}
 F_{11}^2+F_{21}^2 & F_{11} F_{12}+F_{21} F_{22} \\
 F_{11} F_{12}+F_{21} F_{22} & F_{12}^2+F_{22}^2 \\
\end{array}
\right)$. \\ The eigenvalues of $U^2$ are:\\
 $\qquad \qquad \mu_1=\frac{1}{2}\left(F_{11}^2+F_{12}^2+F_{21}^2+F_{22}^2-\sqrt{\left(F_{11}^2+F_{12}^2+F_{21}^2+F_{22}^2\right){}^2-4 \left(F_{12}
 F_{21}-F_{11} F_{22}\right){}^2}\right)$\\
 $\hspace*{0.5cm}=\frac{1}{2}\left(\|F\|^2-\sqrt{\|F\|^4-4(\det F)^2}\right)$, \\ $\qquad \qquad
 \mu_2=\frac{1}{2}\left(F_{11}^2+F_{12}^2+F_{21}^2+F_{22}^2+\sqrt{\left(F_{11}^2+F_{12}^2+F_{21}^2+F_{22}^2\right){}^2-4 \left(F_{12} F_{21}-F_{11}
 F_{22}\right){}^2}\right)$\\
  $\hspace*{0.5cm}=\frac{1}{2}\left(\|F\|^2+\sqrt{\|F\|^4-4(\det F)^2}\right)$.\\
 The principal stretches of $F$, i.e. the eigenvalues of $U=\sqrt{F^T F}$, which are the same as the eigenvalues of $V=\sqrt{F F^T}$, are
 $\lambda_1(F)=\sqrt{\mu_1},\ \lambda_2(F)=\sqrt{\mu_2}$.
\item Taking the pure stretch under shear stress $F_1=\left(
                                                        \begin{array}{ccc}
                                                          \cosh \frac{t}{2} & \sinh \frac{t}{2} & 0 \\
                                                          \sinh \frac{t}{2} & \cosh \frac{t}{2} & 0 \\
                                                          0 & 0 & 1 \\
                                                        \end{array}
                                                      \right)$  and  the simple glide $F_2=\left(
                                                        \begin{array}{ccc}
                                                          1 & t& 0 \\
                                                          0 & 1 & 0 \\
                                                          0 & 0 & 1 \\
                                                        \end{array}
                                                      \right)$, the corresponding rates
$L_1(t)=\dd\frac{{\rm d}}{{\rm d t}}F_1\cdot F_1^{-1}\neq \frac{{\rm d}}{{\rm d t}}F_2\cdot F_2^{-1}=L_2(t)$ are different, as is $\log U_1(t)\neq
\log\sqrt{F_2^TF_2}=\log U_2(t)$ and $\dd\frac{{\rm d}}{{\rm d t}}\log U_1\neq \frac{{\rm d}}{{\rm d t}}\log U_2(t)$. However, $D_1(t)={\rm sym}\,
L_1(t)={\rm sym} \,L_2(t)=D_2(t)$. This shows that $\dd\frac{{\rm d}}{{\rm d t}}\log U(t)=D(t)$ is true only for coaxial families $U(t)$.
\end{itemize}

\subsection{Vall\'{e}e's formula}
\label{Appensansour}

\begin{lemma}{\rm (Vall\'{e}e's formula\footnote{In \cite{vallee2008dual} Vall\'{e}e et al.  have given a proof without using a Taylor expansion.} (see
also  \cite{vallee1978lois,vallee2008dual,Kochkin1986,sansour1997theory}))}\\
Let us consider $S\in{\rm Sym}(3)$ and let $\Psi:{\rm Sym}(3)\rightarrow \R$ be a differentiable isotropic  scalar  value function. We define
$W(S)=\Psi({\rm exp}(S)).$ Then, the following chain rules hold:
\begin{align}\label{Sf}
&D_S[\Psi({\rm exp}(S))]={\rm exp}(S)\cdot D\,\Psi({\rm exp}(S)), \qquad D_SW(S)=D\,\Psi({\rm exp}(S))\cdot {\rm exp}(S),\\
&D_C\,\Psi(C)=D\,W(\log C)\cdot C^{-1},\qquad C\cdot D_C\Psi(C)=D\, W(\log C),\notag
\end{align}
while it is generally not true that $D_C[\log C].\, H=\langle C^{-1},H\rangle$.
\end{lemma}
\begin{proof}Let us first remark that
\begin{align}
\exp(X+H)&=\id+(X+H)+\frac{1}{2}(X+H)^2+\frac{1}{6}(X+H)^3+ ...\\
&=\id+(X+H)+\frac{1}{2}(X^2+X\, H+H\, X+H^2)\notag\\&\ \ \ \ + \frac{1}{6}(X^3+X\, H\, X+H\, X\, X+H^2X+X^2H+X\, H^2+H\, X\, H+H^3)+...\notag\\
&=\id+X+\frac{1}{2}X^2+\frac{1}{6}X^3+...+H+\frac{1}{2}(X\, H+H\, X)+\frac{1}{6}(X\,X\, H+X\, H\, X+H\, X\, X)\notag\\
&=\exp(X)+\underbrace{H+\frac{1}{2}(X\, H+H\, X)+\frac{1}{6}(X\, X\, H+X\, H\, X+H\, X\, X)+...}_{D(\exp(X)).\, H}.\notag
\end{align}
Further we consider the Taylor expansion of the function $\Psi(\exp(S))$
\begin{align}
\Psi(\exp(S+H))&=\Psi(\exp(S)+D(\exp(S)).\, H+...)\\
&=\Psi(\exp(S))+\langle D\, \Psi(\exp(S)), D(\exp(S)).\, H\rangle+...\notag\\
&=\Psi(\exp(S))+\langle D\, \Psi(\exp(S)), H+\frac{1}{2}(S\, H+H\, S)\rangle\notag\\&\ \ +\langle D\, \Psi(\exp(S)),\frac{1}{6}(S\,S\,H+S\,H\, S+H\, S\,
S)+...\rangle+...\notag\\
&=\Psi(\exp(S))+\langle D\, \Psi(\exp(S)), H\rangle+\frac{1}{2}\,\langle D\, \Psi(\exp(S)),S\, H+H\, S\rangle\notag\\&\ \ +\frac{1}{6}\,\langle D\,
\Psi(\exp(S)),S\,S\,H+S\,H\, S+H\, S\, S\rangle+...\notag\\
&=\Psi(\exp(S))+\langle D\, \Psi(\exp(S)), H\rangle+\frac{1}{2}\,[\langle S^T\,D\, \Psi(\exp(S)), H\rangle+\langle D\, \Psi(\exp(S))\,
S^T,H\rangle]\notag\\&\ \ +\frac{1}{6}\,[S^T\,S^T\,\langle D\, \Psi(\exp(S)),H\rangle+\langle S^T\, D\, \Psi(\exp(S))\,S^T,H\rangle+\langle D\,
\Psi(\exp(S))\, S^T\, S^T,H\rangle]+...\, .\notag
\end{align}
Since $S\in{\rm Sym}(3)$, it follows
\begin{align}
\Psi(\exp(S+H))&=\Psi(\exp(S))+\langle D\, \Psi(\exp(S)), H\rangle+\frac{1}{2}\,[\langle S\,D\, \Psi(\exp(S)), H\rangle+\langle D\, \Psi(\exp(S))\,
S,H\rangle]\\&\ \ +\frac{1}{6}\,[S\,S\,\langle D\, \Psi(\exp(S)),H\rangle+\langle S\, D\, \Psi(\exp(S))\,S,H\rangle+\langle D\, \Psi(\exp(S))\, S\,
S,H\rangle]+...\, .\notag
\end{align}
On the other hand, since $D\, \Psi$ is a isotropic tensor function and obvious $\exp(S)$ is also isotropic, we have that $D\, \Psi(\exp(S))$ is also a
isotropic tensor function and therefore it holds
\begin{align}
D\, \Psi(\exp(S))\cdot S=S\cdot D\, \Psi(\exp(S)).
\end{align}
Therefore,
\begin{align}
\Psi(\exp(S&+H))=\Psi(\exp(S))+\langle D\, \Psi(\exp(S)), H\rangle+\langle D\, \Psi(\exp(S))\, S,H\rangle+\frac{1}{2}\langle D\, \Psi(\exp(S))\,
S^2,H\rangle+...\\
&=\Psi(\exp(S))+\langle D\, \Psi(\exp(S))[\id+ S+\frac{1}{2}\, S^2+...],H\rangle\notag=\Psi(\exp(S))+\langle D\, \Psi(\exp(S))\cdot \exp(S),H\rangle\,
.\notag
\end{align}
Using again the isotropy of $D\, \Psi(\exp(S))$, we obtain
\begin{align}
\Psi(\exp(S+H))
&=\Psi(\exp(S))+\langle \exp(S)\cdot D\, \Psi(\exp(S)),H\rangle+ ...\, .
\end{align}
We recall that  we simultaneously have
\begin{align}
\Psi(\exp(S+H))
&=\Psi(\exp(S))+\langle  D_S\, \Psi(\exp(S)),H\rangle\,+... \,,
\end{align}
for all $H\in{\rm Sym}(3)$.
Thus, we deduce
\begin{align}
\langle D_S\Psi(\exp(S)),H\rangle=\langle\exp(S)\cdot D\, \Psi(\exp(S)),H\rangle, \qquad \langle D_SW(S),H\rangle=\langle\exp(S)\cdot D\,
\Psi(\exp(S)),H\rangle.
\end{align}
Choosing $S=\log C$, the relations \eqref{Sf}$_3$ also results and the proof is complete.
\end{proof}
\subsection{LH-ellipticity for  functions of the type $F\mapsto h(\det F)$}
\label{Appenhdet}

We consider a function $h:\R\to\R$ and we analyse when the function $F\mapsto h(\det F)$ is LH-elliptic as a function of $F$, $F\in \R^{3\times 3}$.  We
recall that
\begin{align}
D(\det F).H=\det F\cdot \tr(H\, F^{-1})=\langle \Cof F,H\rangle.
\end{align}
Using the first Frech\'{e}t- formal derivative, we compute the derivative
\begin{align}
D(h(\det F)).(H,H)=h^{\prime}(\det F)\cdot \langle \Cof F, H\rangle,
\end{align}
and the second derivative will be
\begin{align}
D^2(h(\det F)).(H,H)&=h^{\prime\prime}(\det F)\cdot \langle \Cof F, H\rangle^2+h^{\prime}(\det F)\langle D(\Cof F).H, H\rangle\\
&=h^{\prime\prime}(\det F)\cdot \langle \Cof F, H\rangle^2+h^{\prime}(\det F) \{\langle \langle \Cof F,H\rangle \, F^{-T},H\rangle+\det F\langle -F^{-T}H^T
F^{-T},H\rangle\},\notag
\\
&=h^{\prime\prime}(\det F)\cdot \langle \Cof F, H\rangle^2+h^{\prime} (\det F)\, \det F \{\langle  F^{-T},H\rangle^2-\langle F^{-T}H^T
F^{-T},H\rangle\}.\notag
\end{align}
Hence, for $\xi,\eta\in \R^3$ we have
\begin{align}
D^2(h(&\det F)).(\xi\otimes\eta)\\
&=h^{\prime\prime}(\det F)\cdot \langle \Cof F, (\xi\otimes\eta)\rangle^2+h^{\prime} (\det F)\, \det F \{\langle  F^{-T},(\xi\otimes\eta)\rangle^2-\langle
F^{-T}(\xi\otimes\eta)^T F^{-T},(\xi\otimes\eta)\rangle\}.\notag
\end{align}
On the other hand
\begin{align}
\langle  F^{-T},(\xi\otimes\eta)\rangle^2-&\langle F^{-T}(\xi\otimes\eta)^T F^{-T},(\xi\otimes\eta)\rangle=
\langle \id, F^{-1}(\xi\otimes\eta)\rangle^2-\langle (\eta\otimes F^{-1}\xi) ,(F^{-1}\xi\otimes\eta)\rangle\notag\\
&=
\langle \id, F^{-1}(\xi\otimes\eta)\rangle^2-\langle (F^{-1}\xi\otimes\eta)^T ,(F^{-1}\xi\otimes\eta)\rangle
\notag=
\langle  F^{-1}\xi,\eta\rangle^2-\langle F^{-1}\xi,\eta \rangle^2=0{.}
\end{align}
This leads to the surprising simplification
\begin{align}
D^2(h(\det F)).&(\xi\otimes\eta,\xi\otimes\eta)=h^{\prime\prime}(\det F)\cdot \langle \Cof F, (\xi\otimes\eta)\rangle^2.
\end{align}
In conclusion, $F\mapsto h(\det F)$ is LH-elliptic  if and only if $t\mapsto h(t)$ is convex since $\langle \Cof F, (\xi\otimes\eta)\rangle^2$ is
positive.

From  \cite[page 213]{Dacorogna08} we know more:
\begin{proposition}\label{propDacdet}
Let $W:\mathbb{R}^{n\times n}\rightarrow\mathbb{R}$ be quasiaffine but not identically constant and $h:\mathbb{R}\rightarrow\mathbb{R}$ be such that
   $
    W(F)=h(\det F).
   $ Then
    \begin{align}\hspace{-4mm}
      W \quad \text{polyconvex}\quad  \Leftrightarrow \quad W \quad \text{quasiconvex}\quad \Leftrightarrow \quad  W \quad \text{rank one convex}\quad
      \Leftrightarrow
     \quad h \quad \text{convex}.
    \end{align}
   \end{proposition}
\subsection{Convexity for  functions of the type $ t\mapsto \xi((\log t)^2)$}

We consider a generic function $\xi:\R_+\to\R_+$ and we find a characterisation of the convexity for the function
$t\mapsto \xi((\log t)^2)$. In the following let $\zeta$ denote the function
$\zeta:\R_+\to\R_+$ , $\zeta(t)=(\log t)^2$.
We deduce
\begin{align}
\notag \frac{\rm d }{\rm d t}{\xi}((\log t)^2)&=\xi^{\prime}((\log t)^2)\, 2 \frac{1}{t} \log t,\\
\notag \frac{{\rm d}^2 }{{\rm d t}^2}\xi((\log t)^2)&=2\frac{\rm d}{\rm dt}\left(\xi^{\prime}((\log t)^2)\, 2 \frac{1}{t} \log t\right)\\
&=4\,\xi^{\prime\prime}((\log t)^2)\,  \frac{1}{t^2} (\log t)^2-2\,\xi^{\prime}((\log t)^2)\,  \frac{1}{t^2} \log t+2\,\xi^{\prime}((\log t)^2)\,
\frac{1}{t^2}
\\
\notag&=2\frac{1}{t^2}\left[{2}\,\xi^{\prime\prime}((\log t)^2)\,   (\log t)^2 +\xi^{\prime}((\log t)^2)(1- \log t)\right],
\end{align}
where $\xi^{\prime}=\frac{d \xi}{d\zeta}$.    Hence, the function $t\mapsto \xi((\log t)^2)$ is
\begin{itemize}
\item convex on $[1,\infty)$ as a function of  $t$ if and only if
$
2\frac{d^2\xi(\zeta)}{d\zeta^2}\,  \zeta +\ {\frac{d\xi(\zeta)}{d\zeta} (}1- \sqrt{\zeta})\geq 0, \qquad \text{for all} \quad \zeta\in\R_+.
$
\item convex on $(0,1)$ as a function of  $t$ if and only if
$
2\frac{d^2\xi(\zeta)}{d\zeta^2}\,   \zeta +\ {\frac{d\xi(\zeta)}{d\zeta}(1}+ \sqrt{\zeta})\geq 0, \qquad \text{for all} \quad \zeta\in\R_+.
$
\end{itemize}

\subsection{Connecting $\dev_3\log U$ with $\dev_2\log U$}

For $U^\sharp\in{\rm GL}(2)$, we define the lifted quantity
\begin{align}\label{liftU}
U=\left(\begin{array}{ccc}
&U^\sharp& 0\\
& & 0\\
0&0&(\det U^\sharp)^{1/2}
\end{array}\right)\in{\rm GL}(3).
\end{align}
We remark that
\begin{align}
\det\left(\begin{array}{ccc}
&U^\sharp& 0\\
& & 0\\
0&0&(\det U^\sharp)^{1/2}
\end{array}\right)=\det U^\sharp \, (\det U^\sharp)^{1/2}=(\det U^\sharp)^{3/2},
\end{align}
which implies 
$
(\det U)^{1/3}=\left[\det U^\sharp \, (\det U^\sharp)^{1/2}\right]^{1/3}=\left[(\det U^\sharp)^{3/2}\right]^{1/3}=(\det U^\sharp)^{1/2}.
$ 
Moreover, we obtain
\begin{align}
\dev_3\log U&=\log
\frac{U}{\det U^{1/3}}=\log
\frac{U}{(\det U^\sharp)^{1/2}}=\log\left(\begin{array}{ccc}
& \frac{U^\sharp}{(\det U^\sharp)^{1/2}}& 0\\
& & 0\\
0&0&1
\end{array}\right)\notag\\&=\left(\begin{array}{ccc}
& \log \frac{U^\sharp}{(\det U^\sharp)^{1/2}}& 0\\
& & 0\\
0&0&0
\end{array}\right)=\left(\begin{array}{ccc}
&\dev_2 \log U^\sharp& 0\\
& & 0\\
0&0&0
\end{array}\right).
\end{align}
In general, for $A^\sharp\in \R^{2\times 2}$ and $\alpha\in \R$ we have
\begin{align}
\|\dev_3\left(\begin{array}{ccc}
&A^\sharp& 0\\
& & 0\\
0&0&\alpha
\end{array}\right)\|^2&=\|\left(\begin{array}{ccc}
&A^\sharp& 0\\
& & 0\\
0&0&\alpha
\end{array}\right)\|^2-\frac{1}{3}\,[\tr[\left(\begin{array}{ccc}
&A^\sharp& 0\\
& & 0\\
0&0&\alpha
\end{array}\right)]^2\\
&=\|A^\sharp\|^2+\alpha^2-\frac{1}{3}\,[\tr(A^\sharp)+\alpha
]^2=\|A^\sharp\|^2-\frac{1}{3}\,[\tr(A^\sharp)]^2-\frac{2}{3}\alpha\,\tr(A^\sharp)
 -\frac{1}{3}\,\alpha
^2+\alpha^2\notag\\&=\|\dev_2 A^\sharp\|^2-\frac{2}{3}\alpha\,\tr(A^\sharp)
 +\frac{2}{3}\,\alpha
^2.\notag
\end{align}
Thus
\begin{align}
\|\dev_3\left(\begin{array}{ccc}
&A^\sharp& 0\\
& & 0\\
0&0&\alpha
\end{array}\right)\|^2&=\|\dev_2 A^\sharp\|^2
\end{align}
if and only if $\alpha=0$ or $\alpha=\tr(A^\sharp)$. Hence, we deduce
$ 
{\|\dev_3\log U\|^2}={\|\dev_2\log U^\sharp\|^2},
 $ 
for $U$ of the form \eqref{liftU}.
Since $U^\sharp \in {\rm PSym}(2)$, we can assume that $U^\sharp =\left(\begin{array}{cc}\lambda_1&0\\
0&\lambda_2\end{array}\right)$, \ $\lambda_1,\lambda_2\in\R_+$. Then, the lifted quantity $U$ lies in $\PSym(3)$ and $U =\left(\begin{array}{ccc}\lambda_1&0&0\\
0&\lambda_2&0\\
0&0&(\lambda_1\,\lambda_2)^{1/2}\end{array}\right)$.

The next problem is if for a given deformation $\varphi^\sharp=(\varphi_1^\sharp, \varphi_2^\sharp):\R^2\rightarrow\R^2$ such that $U^\sharp=\sqrt{(\nabla \varphi^
\sharp)^T\, \nabla \varphi^
\sharp}$ we can construct an ansatz $\varphi:\R^3\rightarrow\R^3$ such that $U=\sqrt{\nabla \varphi^T\, \nabla \varphi}$, where $U$ is the lifted quantity
associated to $U^\sharp$. For this it is necessary to have $\varphi=(\varphi_1(x_1,x_2),\varphi_2(x_1,x_2),x_3\,\alpha(x_1,x_2))$ and $\alpha_{,x_1}=0,\ \alpha_{,x_2}=0$. Checking the compatibility equation we see that this is possible if and only if $\det \nabla \varphi^\sharp=K=const.$, which implies $\varphi_{3,x_3}=K$.
In the incompressible case $\det \nabla\varphi=1$, an appropriate ansatz is therefore
\begin{align}
\varphi(x_1,x_2,x_3)=(\varphi_1^\sharp(x_1,x_2), \varphi_2^\sharp(x_1,x_2),x_3),
\end{align}
since
\begin{align}
U^2&={\nabla \varphi^T\, \nabla \varphi}=\left(\begin{array}{ccc}&(\nabla \varphi^\sharp)^T \, \nabla \varphi^\sharp&0\\
&&0\\
0&0&1\end{array}\right)=\left(\begin{array}{ccc}&(\nabla \varphi^\sharp)^T \, \nabla \varphi^\sharp&0\\
&&0\\
0&0&\det[(\nabla \varphi^\sharp)^T \, \nabla \varphi^\sharp]\end{array}\right)\notag\\
&=\left(\begin{array}{ccc}&(U^\sharp)^2&0\\
&&0\\
0&0&(\det[(U^\sharp)^{1/2}])^2\end{array}\right)=\left(\begin{array}{ccc}&U^\sharp&0\\
&&0\\
0&0&(\det U^\sharp)^{1/2}\end{array}\right)^2,
\end{align}
with $\det U^\sharp=1$.

\end{document}